\documentclass[oneside]{amsart}
\usepackage{amsmath, amsfonts, amssymb}
\usepackage[dvipsnames]{xcolor} 
\usepackage[colorlinks = true,
            linkcolor = Blue,
            urlcolor  = Blue,
            citecolor = Blue,
            anchorcolor = Blue]{hyperref}
            \usepackage{aliascnt}
            \usepackage{amsthm}
\usepackage{cleveref}

\usepackage{imakeidx}
\makeindex[columns=3, title=Notation Index]
\indexsetup{noclearpage}
\newcommand{\INN}[1]{\index{#1|BH}}  % send notation into main index

\usepackage{tikz-cd}

\usepackage{url}
\usepackage{enumerate}
\usepackage{rotating}
 
\newcommand{\xdownarrow}[1]{
  \ensuremath{\begin{turn}{90}{\tiny${#1 }$}\end{turn}
    \left\downarrow\vbox to 0.4cm{}\right.\kern-\nulldelimiterspace
  }
}
\newcommand{\xuparrow}[1]{
  \ensuremath{
    \begin{turn}{90}{\tiny ${#1}$}\end{turn} \left\uparrow\vbox to 0.4cm{}\right. \kern-\nulldelimiterspace
  } 
}

\newcommand{\pft}{\mathrm{Gr}^{\mathrm{pft}}}
\DeclareMathOperator{\F}{F}  
\DeclareMathOperator{\hml}{H}  
\newcommand{\liepp}[1]{\mathrm{Lie}_{#1,\Delta}^{\pi}}
\newcommand{\moduli}{\mathrm{Moduli}}
\newcommand{\liep}{\mathrm{Lie}^{\pi}_{k,\Delta}}
\newcommand{\liedg}{\mathrm{Lie}^{\dg}_{k}}

\newcommand{\lieps}{\mathrm{Lie}_{k,\mathbb{E}_\infty}^{\pi}}
\newcommand{\liepss}[1]{\mathrm{Lie}_{#1, \mathbb{E}_\infty}^{\pi}}
\newcommand{\mdl}{\mathrm{Moduli}}
\newcommand{\art}{\mathrm{art}}
\newcommand{\ZZ}{\mathbb{Z}}

\newcommand{\EE}{\mathbb{E}}
\newcommand{\RR}{\mathbb{R}}
\newcommand{\FF}{\mathbb{F}}

\newcommand{\liepr}[1]{\mathrm{Lie}^{\pi}_{#1, \Delta}}
\newcommand{\fil}{\mathrm{Fil}}

 \newcommand{\HH}{\operatorname{H}} 

\newcommand{\mywedge}[1]{\mathbin{\operatorname*{\wedge}_{#1}}}
 \renewcommand{\hom}{\mathrm{Hom}}

\newcommand{\hobased}{{{\text{h}}}}
\newtheorem*{theorem*}{Theorem}
  
\DeclareMathOperator{\chara}{char}  
\DeclareMathOperator{\Def}{Def}  
\DeclareMathOperator{\Free}{Free}  
\usepackage{upgreek}
\newcommand{\mytimestwo}[2]{\mathbin{\operatorname*{\times}_{#1}^{#2}}}
\newcommand{\myotimestwo}[2]{\mathbin{\operatorname*{\otimes}_{#1}^{#2}}}
\DeclareMathOperator{\coequ}{coequ} 
\DeclareMathOperator{\Sing}{Sing}
\DeclareMathOperator{\sSet}{\mathbf{sSet}}
\newcommand{\vecto}{\mathrm{Vect}^\omega}
\newcommand{\clgaug}{\mathrm{CAlg}^{\mathrm{aug}}}
\newcommand{\filc}{\mathrm{Fil}_{\mathrm{cpl}}}

\newcommand{\clg}{\mathrm{CAlg}}
\newcommand{\free}{\mathrm{free}}
 
\makeatletter
\def\foo#1\endgraf\unskip#2\foo{\def\row@to@buffer{#1\endgraf\unskip\unskip#2}}

\makeatother

\newcommand{\mylim}[1]{\mathbin{\operatorname*{lim \hspace{1pt}}_{#1}^{}}}
\newcommand{\mycolim}[1]{\mathbin{\operatorname*{colim \hspace{1pt}}_{#1}^{}}}

\newcommand{\CC}{\mathbb{C}}

\newcommand{\NN}{\mathbb{N}}
\DeclareMathOperator{\Lie}{Lie} 
\DeclareMathOperator{\forget}{forget} 
 
\DeclareMathOperator{\sgn}{sgn} 
\DeclareMathOperator{\Barr}{Bar} 
\DeclareMathOperator{\N}{N} 
\DeclareMathOperator{\EHP}{EHP}

\DeclareMathOperator{\Spec}{Spec} 
\DeclareMathOperator{\Ind}{Ind} 
\DeclareMathOperator{\cofib}{cofib} 
 
\DeclareMathOperator{\End}{End}

\DeclareMathOperator{\Map}{Map} 
 
\DeclareMathOperator{\GL}{GL} 
\DeclareMathOperator{\dg}{dg} 
\DeclareMathOperator{\aug}{aug}

\DeclareMathOperator{\Exc}{Exc} 
\DeclareMathOperator{\Lan}{Lan} 
\DeclareMathOperator{\cross}{cr} 
\DeclareMathOperator{\Tot}{Tot} 
\DeclareMathOperator{\Br}{Br} 
\DeclareMathOperator{\cN}{cN} 
\DeclareMathOperator{\ft}{ft} 

\newcommand{\myotimes}[1]{\mathbin{\operatorname*{\otimes}_{#1}}}
\DeclareMathOperator{\Alg}{Alg} 
\DeclareMathOperator{\Set}{Set} 
\DeclareMathOperator{\Fin}{Fin}

\DeclareMathOperator{\id}{id}  

\newcommand{\Df}{\mathfrak{D}}

\usepackage{verbatim}
\usepackage{geometry}
 
\renewcommand{\L}{\mathbb{L}}
\newcommand{\coh}{\mathrm{Mod}^{\mathrm{ft}}}
\newcommand{\sym}{\mathrm{Sym}}

\newcommand{\gr}{\mathrm{Gr}} 
\newcommand{\einf}{\mathbb{E}_\infty}

\newcommand{\vect}{\mathrm{Vect}}
\newcommand{\fun}{\mathrm{Fun}}

\newcommand{\perf}{\mathrm{Perf}}
\usepackage{xy}
\input xy
\newcommand{\spec}{\mathrm{Spec}}

\xyoption{all}

\renewcommand{\mod}{\mathrm{Mod}}
\newcommand{\sF}{\mathcal{F}}

\theoremstyle{definition}
\newtheorem{definition}{Definition}[section]

\newaliascnt{cons}{definition}
\newtheorem{cons}[cons]{Construction}
\aliascntresetthe{cons}

\newaliascnt{example}{definition}
\newtheorem{example}[example]{Example}
\aliascntresetthe{example}
\newtheorem*{example*}{Example}

\DeclareMathOperator{\cNoet}{\mathrm{CAlg}_k^{\mathrm{cN}}}

\newaliascnt{notation}{definition}
\newtheorem{notation}[notation]{Notation}
\aliascntresetthe{notation}

\newaliascnt{convention}{definition}
\newtheorem{convention}[convention]{Convention}
\aliascntresetthe{convention}

\newaliascnt{remark}{definition}
\newtheorem{remark}[remark]{Remark}
\aliascntresetthe{remark}

\theoremstyle{theorem}

\newaliascnt{proposition}{definition}
\newtheorem{proposition}[proposition]{Proposition}
\aliascntresetthe{proposition}

\newaliascnt{lemma}{definition}
\newtheorem{lemma}[lemma]{Lemma}
\aliascntresetthe{lemma}

\newaliascnt{corollary}{definition}
\newtheorem{corollary}[corollary]{Corollary}
\aliascntresetthe{corollary}

\newaliascnt{warning}{definition}
\newtheorem{warning}[warning]{Warning}
\aliascntresetthe{warning}

\newaliascnt{observation}{definition}

\aliascntresetthe{observation}

\newaliascnt{claim}{definition}

\aliascntresetthe{claim}

\newaliascnt{exercise}{definition}

\aliascntresetthe{exercise}

\newaliascnt{theorem}{definition}
\newtheorem{theorem}[theorem]{Theorem}
\aliascntresetthe{theorem}

\newaliascnt{superlemma}{definition}

\aliascntresetthe{superlemma}

\newaliascnt{fact}{definition}

\aliascntresetthe{fact}

\newcommand{\SCR}{\mathrm{SCR}}

\newcommand{\md}{\mathrm{Mod}}
\newcommand{\D}{\mathfrak{D}}
\newcommand{\alg}{\mathrm{Alg}}
\begin{document}

\title{Deformation theory and partition Lie algebras}

\author[D.\ Lukas B.\ Brantner]{D.\ Lukas B.\ Brantner}
\address{Oxford University,   Universit\'{e} Paris--Saclay (CNRS)} 
\email{lukas.brantner@maths.ox.ac.uk, lukas.brantner@universite-paris-saclay.fr}

\author[Akhil Mathew]{Akhil Mathew}
\address{University of Chicago}
\email{amathew@math.uchicago.edu}

\maketitle

\begin{abstract}
A theorem of Lurie and Pridham establishes a correspondence 
between formal moduli problems and differential 
graded Lie algebras in characteristic zero, thereby formalising a well-known principle
in deformation theory. 
We introduce a variant of differential
graded Lie algebras, called partition Lie algebras, in arbitrary characteristic. We then
explicitly compute 
the homotopy groups of   free algebras, which parametrise operations. 
Finally, we prove generalisations of the Lurie--Pridham correspondence 
classifying formal moduli problems via partition Lie algebras over an arbitrary field, as well as 
over a complete local base. 
\end{abstract}
 \ 
\\ \\ 
\\
\\
\\
\\
\\
\\ 
\\
\\
\\ 
\\
\\
\\
\tableofcontents 
\newpage
\newcommand{\Fbox}[1]{\fbox{\strut#1}}
\setlength{\fboxsep}{1pt} 
 
\section{Introduction} 
A well-known principle in deformation theory postulates that the infinitesimal structure of any  moduli space in characteristic zero is controlled by a differential graded Lie algebra.\vspace{2pt}

This heuristic can be traced back to Deligne \cite{deligne1986letter}, Drinfeld \cite{drinfeld1988letter}, and Feigin, and was explored further in the work of Goldman-Millson \cite{goldman1988deformation}, Hinich \cite{hinich2001dg}, Kontsevich-Soibelman \cite{kontsevich2002deformation}, Manetti \cite{manetti2009differential}, and many others. 
Eventually, it  was articulated as a precise correspondence by Lurie \cite{lurie2011derivedX} and Pridham \cite{pridham2010unifying}, who constructed an equivalence 
between formal moduli problems  and differential graded Lie algebras in characteristic zero.\vspace{2pt}

Our principal aim is to  generalise this equivalence to  finite and mixed \mbox{characteristic,} thereby giving a Lie algebraic description of formal deformations in these contexts.

\subsection{Background} Before delving into any technical details, we shall recall a classical example.  

\begin{example*}
Given a smooth and proper variety $Z$ over the field $\CC$ of complex numbers, we can study its
infinitesimal deformations over local Artinian $\mathbb{C}$-algebras.
It is well-known that these deformations are closely related to the   lower cohomology groups of the tangent bundle $T_Z$:
\begin{enumerate}[a)]
\item  
$\HH^0(Z, T_Z)$
classifies 
infinitesimal automorphisms 
of the trivial deformation $Z \times_{\spec \mathbb{C}} \spec
(\mathbb{C}[\epsilon]/\epsilon^2)$; infinitesimal automorphisms therefore correspond to
vector fields. 
\item $\HH^1(Z, T_Z)$ classifies isomorphism classes of first-order deformations;
every  such  deformation \mbox{$\mathcal{Z}\rightarrow \Spec(
\CC[\epsilon]/\epsilon^2)$} gives rise to a {Kodaira-Spencer class} $x_\mathcal{Z} $ in
$ H^1(Z,{T}_Z)$, cf.\ \cite{kodaira1958deformations}.\vspace{2pt}
\end{enumerate}

To formulate a criterion for when a given first order deformation $\mathcal{Z}\rightarrow \Spec(\CC[\epsilon]/\epsilon^2)$ extends to higher order, observe that  the classical Dolbeault complex $$\mathfrak{g}_Z : = C^\ast(Z,{T}_Z) =
(\ \mathcal{A}^{0,0}({T}_Z) \rightarrow \mathcal{A}^{0,1}({T}_Z)
\rightarrow  \mathcal{A}^{0,2}({T}_Z)
\rightarrow \  \ldots \  ) $$ computing $\HH^\ast(Z,{T}_Z)$ admits the structure of a
\textit{differential graded Lie algebra}. Its Lie bracket combines the wedge product on differential forms with the commutator
bracket on vector fields.   \vspace{2pt}

\begin{enumerate}[a)]\setcounter{enumi}{2}
\item $\HH^2(Z, T_Z)$ contains the obstructions to extending first-order
deformations: a first-order deformation 
$\mathcal{Z}\rightarrow \Spec(
\CC[\epsilon]/\epsilon^2)$  
extends to $\spec( \mathbb{C}[\epsilon]/\epsilon^3)$ precisely if the self-bracket
$\frac{1}{2}[x_{\mathcal{Z}}, x_{\mathcal{Z}}]$ vanishes.\vspace{-2pt}

\end{enumerate}
\end{example*} 

\theoremstyle{definition}
\newtheorem*{cons*}{Construction}

In fact, the  Kodaira--Spencer differential graded Lie algebra $\mathfrak{g}_Z$  records 
all   {formal
deformations  of $Z$}. More precisely, given any  local  Artinian $\CC$-algebra $A$,  deformations of $Z$ to $A$ correspond to Maurer--Cartan elements $\{x \in \mathfrak{m}_A \otimes (\mathfrak{g}_Z)_{-1}\ | \ dx + \frac{1}{2} [x,x] = 0\}$,  \mbox{considered up to gauge equivalence.}\vspace{5pt}

Other deformation functors of interest are also governed by differential graded  {Lie algebras,  but these} Lie algebras are not uniquely determined by the deformation functor.  For instance,   {there are two  natural  differential graded Lie algebras governing  deformations of a  closed immersion of \mbox{smooth varieties $X\subset Y$} -- one interprets $X$ as a point in a Hilbert scheme; the other as a point in a Quot scheme,  cf.\  e.g.\  \cite{toen2014derived}.  It was a  key insight of Drinfeld \cite{drinfeld1988letter}  that this problem disappears once we also 
consider derived deformations to \textit{simplicial}  local Artinian   $\CC$-algebras. \vspace{2pt}

In our example above, the differential graded Lie algebra $\mathfrak{g}_Z$     not only remembers   infinitesimal deformations, but also  \textit{derived} infinitesimal deformations  of $Z$ via a refined  Maurer--Cartan construction,  cf.\ \cite[Section 1.3]{hinich2001dg}.  In fact,  $\mathfrak{g}_Z$  is  {uniquely} determined by this property,  \mbox{up to equivalence.}\vspace{4pt}
 
This  illustrates the general principle we alluded to in the very beginning:   the derived infinitesimal deformations of an  object in characteristic zero are controlled by a differential graded  Lie algebra. 
A precise formulation was given by Lurie  \cite{lurie2011derivedX} and Pridham \cite{pridham2010unifying} in the language of
\emph{formal moduli problems},  cf.\ Definition \ref{FMP}. Roughly speaking, any
reasonably geometric deformation problem \mbox{(e.g.\  deformations of schemes, of
complexes,  or the formal completion of a suitably geometric stack)} gives rise to a
formal moduli problem. 
In particular, there is a formal moduli problem $\Def_Z$ encoding the derived 
deformations of our variety
$Z$ from above. \vspace{3pt}

We   then have the following correspondence:\vspace{-3pt}
\begin{theorem}[Lurie, Pridham]
If $k$ is a field of characteristic zero, then there is an equivalence of \mbox{$\infty$-categories} between 
 formal moduli problems and differential graded Lie algebras over $k$.\label{LP}
\end{theorem}

This result has been extended in various  directions in characteristic zero, but
an analogue in positive or mixed characteristic remained elusive.  In fact,  formal moduli in this setting are \textit{not} equivalent to any classically known version of Lie algebras such as differential graded Lie algebras or simplicial-cosimplicial (restricted) Lie algebras.\vspace{3pt}

Deformations in finite and mixed characteristic have attracted significant interest in arithmetic geometry and number theory,  ranging from the classical work of Serre--Tate on abelian varieties \cite{serre1966groupes} \cite{tate1967p}
 and of Nygaard--Ogus on K3 surfaces \cite{nygaard1983tate}\cite{nygaard1985tate} to the recent work of Galatius--Venkatesh \cite{GV18} on derived deformations of Galois representations. \vspace{3pt}

In this paper, we formulate and prove a generalisation of  the Lurie--Pridham theorem in
positive and mixed characteristic. For this purpose, we introduce the notion of a \textit{partition
Lie algebra}, which is the correct generalisation of the notion of a  differential graded Lie
algebra to this setting. 
\mbox{Partition Lie algebras} are subtle homotopical objects controlled by the   equivariant topology of {partition complexes},  a family of spaces 
$|\Pi_1|,  |\Pi_2|,  |\Pi_3|, \ldots$ which have been \mbox{studied extensively.}
We compute the homotopy groups of free partition Lie algebras, which parametrise the  natural Dyer--Lashof-like operations  acting on the homotopy groups of any partition Lie algebra.  

We construct partition Lie algebras using the framework of $\infty$-categories and monads, which was indispensable in their discovery.
One can also construct more explicit (and more complicated) point set models for partition Lie algebras -- we refer to  Definition 4.46 and Construction 5.43 in the article \cite{brantner2021pd}, which also 
extends the theory of  partition  Lie algebras to general coherent rings. 
 
Since their   discovery,  partition Lie algebras have already facilitated the proof of the fundamental theorem of purely inseparable Galois theory   \cite{brantner2020purely} and have inspired new insights on deformations of Calabi--Yau varieties \cite{CYDef}; we trust that they will also prove \mbox{useful  for other open problems.}

\begin{convention} Most objects $X$  in this article can be considered as spectra with extra structure.
Accordingly, we   write $\pi_n(X)$ for the $n^{th}$ homotopy \mbox{group of the underlying spectrum  of an object $X$.} 
Unravelling the definitions, this for example means that  the $n^{th}$ homotopy group $\pi_n(V)$ of a  chain complex $V = ( \ldots  \rightarrow V_1 \rightarrow V_0 \rightarrow V_{-1} \rightarrow \ldots )$    is simply given by the $n^{th}$ homology group of $V$.

 \end{convention}

\subsection{Statement of results} Away from characteristic zero, classical algebraic geometry can be generalised in two inequivalent ways (cf.\ \cite{toen2008homotopical} and \cite{lurie2016spectral} for detailed treatments): 
\begin{enumerate}[a)] \item derived algebraic geometry is based on simplicial commutative rings.
\item spectral algebraic geometry is based on (connective) $\EE_\infty$-rings.
\end{enumerate} 

We will prove variants of our main results  in both settings, and will begin with the former. Here, affine schemes over a given field $k$ correspond to simplicial commutative $k$-algebras, and infinitesimal thickenings of $\Spec(k)$  (over $k$)  correspond to the following kind of objects:\vspace{-1pt}
\begin{definition}[Derived Artinian algebras]  \label{daa} \INN{190@$\SCR_k^{\art}$}
A simplicial commutative $k$-algebra 
$A$ is called \mbox{\emph{Artinian} if}
\begin{enumerate}
\item $\pi_0(A)$ is a local Artinian ring with residue field $k$.  
\item $\pi_\ast(A)$ is a finite-dimensional $k$-vector space.
\end{enumerate}
Let $\SCR_k^{\art}$ denote the $\infty$-category of Artinian simplicial
commutative $k$-algebras, defined as a full subcategory of the
$\infty$-category of simplicial commutative $k$-algebras (cf.\ \Cref{setupSCR}). \vspace{-2pt}
\end{definition} 
\begin{notation}Write $\mathcal{S}$ \INN{190@$\mathcal{S}$}   for the $\infty$-category of spaces (cf.\ \cite[Definition 1.2.16.1]{lurie2009higher}). Unless stated otherwise,  limits and colimits will be computed in a higher categorical sense,  and we will heavily rely on the theory of $\infty$-categories developed in \cite{lurie2009higher}.
 \end{notation} \vspace{2pt}

Derived infinitesimal deformations will be described by the following kind \vspace{-1pt}of functors: 
\begin{definition}[Formal moduli problems]   \label{FMP}  
A  \emph{derived formal moduli problem over a field $k$}   is given by a functor \mbox{$X:\SCR_k^{\art} \rightarrow \mathcal{S}$} satisfying the following two properties: \begin{enumerate}
\item the space $X(k) $ is contractible;
\item given a pullback square\vspace{-6pt}
$$ \xymatrix{
\widetilde{A} \ar[d] \ar[r] &  A' \ar[d] \\
A \ar[r] &  A''
}$$
in $\SCR_k^{\art}$ in which $\pi_0(A') \rightarrow \pi_0(A'')$ and $\pi_0(A) \rightarrow \pi_0(A'')$
are surjective, the square\vspace{-4pt}
$$ \xymatrix{
X(\widetilde{A}) \ar[d]  \ar[r] & X(A') \ar[d]  \\
X(A) \ar[r] &  X(A'')
}$$ \vspace{-5pt}
is a (homotopy) pullback  of spaces. \vspace{4pt}
\end{enumerate}
Informally,  (1) says that there are no  nontrivial deformations over the point $\Spec(k)$ and (2) asserts   that   deformations over a union of two infinitesimal thickenings \vspace{3pt}of $\Spec(k)$ \mbox{are obtained by gluing.}

Let $\mdl_{k,\Delta} $ \INN{130@$\mdl_{k,\Delta} $} be the full subcategory of $\fun( \SCR_k^{\art}, \mathcal{S})$ spanned by all  \vspace{3pt}
 formal moduli problems. 
\end{definition} 
Formal moduli problems exist in abundance. Indeed, the formal neighbourhood of any point in a (suitably geometric) derived
stack, as well as various natural deformation problems (such as deformations of varieties or
vector bundles), provide a vast supply of examples. We refer to \cite[Chapter 16]{lurie2016spectral} or \cite{toen2014derived} for a discussion \vspace{3pt} of some of them.

A guiding goal in the subject is to give an ``algebraic'' classification  of formal moduli
problems. Theorem~\ref{LP} realises this aim  when $k$ is a field of characteristic zero by constructing an
equivalence
$$\mdl_{k,\Delta} \xrightarrow{\ \ \simeq \ \ } \mathrm{Alg}_{\liedg} $$
between  $\moduli_{k,\Delta} $   and the $\infty$-category of  differential graded Lie algebras over $k$.
Intuitively, this   correspondence arises as follows. Given  a formal moduli problem $X\in \mdl_{k,\Delta}$, one first constructs its
\textit{tangent complex} \INN{200@$T_X$}  
$T_X$. This chain complex over $k$ is  a derived version of the tangent space, and  is
determined by the values of $X$ on trivial square-zero extensions. The underlying spectrum of $T_X$ is given by the sequence of spaces $(T_X)_n= X(k\oplus k[n])$ with the equivalences \vspace{3pt} $\Omega (T_X)_{n+1} \simeq (T_X)_n$ provided by   \Cref{FMP}(2).}

The Lurie--Pridham correspondence then equips the shift $T_X[-1]$ with the structure of a differential graded
Lie algebra over $k$, and moreover provides a method to functorially recover $X$ from $T_X[-1]$. Hence, it can be interpreted as a variant of formal Lie theory (cf.\ \cite[Chapter 7]{GRII}). 

The correspondence has been extended in several directions, for example by 
Gaitsgory--Rozenblyum  \cite{GRII}, 
Hennion \cite{Hennion}, and Nuiten \cite{Nuiten}. 
These generalisations treat the case of formal moduli problems relative to a base (rather than a  point)  in characteristic zero.  \vspace{4pt}

To generalise Theorem~\ref{LP} to general base fields, we will introduce the new algebraic and
homotopy-theoretic structure of a    \emph{partition Lie algebra}. In characteristic zero,  partition Lie algebras are equivalent to differential graded Lie algebras;   in finite characteristic, this is \mbox{no longer true.}
Partition Lie algebras are  closely related to the genuine equivariant topology of the following   spaces:
\begin{definition}[Partition complexes]\label{partitioncomplex} \INN{160@$\Pi_n$} \INN{190@$\Sigma |\Pi_n|^\diamond$} 
Given an integer $n\geq 1$, the  \textit{$n^{th}$ partition complex} $|\Pi_n|$ is the genuine $\Sigma_n$-space given by the  geometric realisation of the following simplicial $\Sigma_n$-set:
$$  \Pi_n := \N_\bullet \left(  \left\{ \mbox{Poset of partitions of } \{1,\ldots,n\} \right\} - \{\hat{0},\hat{1}\} \right).$$
Here $\N_\bullet$ denotes the nerve\vspace{1pt} construction, $\hat{0} = \Fbox{1}\hspace{-0.3pt}\Fbox{2}\ldots \Fbox{n}$ \ is the discrete partition, $\hat{1} = \Fbox{12 \ldots n}$ is
the indiscrete partition, and all partitions are  ordered under refinement.    
\end{definition}
Partition complexes were linked to ordinary Lie algebras in the   work of  Barcelo \cite{barcelo1990action}, Hanlon \cite{hanlon}, Joyal \cite{Joyal}, and Stanley \cite{stanley}, 
 who constructed and examined an isomorphism of  \mbox{$\Sigma_n$-representations} $\widetilde{\HH}^{n-1}(\Sigma |\Pi_n|^\diamond,\ZZ)\cong \Lie_n \otimes \sgn_n$. Here $\Sigma |\Pi_n|^\diamond$ is the reduced suspension of the unreduced suspension of $|\Pi_n|$  and  $\sgn_n$ denotes the sign representation of  $\Sigma_n$. Moreover, $\Lie_n$  \INN{120@$\Lie_n$} \INN{190@$\sgn_n$} is the quotient of the free abelian group on all Lie words in  letters $x_1,\ldots, x_n$ involving each $x_i$ exactly once  by the usual antisymmetry and   Jacobi relations.

This above isomorphism between representations of the symmetric group was later refined in work of Salvatore \cite{salvatore1998configuration} and Ching \cite{ching2005bar}, who constructed the Lie operad in the $\infty$-category of spectra; algebras over this operad are called \textit{spectral Lie algebras}.
\begin{remark}
Spectral Lie algebras offer excellent computational and conceptual opportunities in unstable chromatic  homotopy theory, as is exploited, for example, in \mbox{\cite{behrens2015bousfield}, \cite{brantnerthesis}, \cite{Gijsbert},\cite{bhk}.}
\end{remark}
Over a field of characteristic zero, spectral Lie algebras are equivalent to differential graded Lie algebras. In contrast, they are \textit{not} the correct structure for the purposes of deformation theory in characteristic $p$, where it
will in fact not be possible to define partition Lie algebras as algebras \mbox{over any $\infty$-operad.}
Instead, we will need to use the language of monads.
\begin{notation} 
Write  $\md_k$ for the derived $\infty$-category of $k$; its objects are chain complexes of $k$-vector spaces, or equivalently $k$-module\vspace{-2pt} spectra (cf.\ \cite[Definition 1.3.5.8; Remark 7.1.1.16]{lurie2014higher}). \INN{130@$\md_k$}  

Recall that a \textit{monad} on $\md_k$ consists of an endofunctor $T$ together with natural transformations $\id \rightarrow T$, $T\circ T \rightarrow T$,   and an infinite set of coherence data. Every monad $T$ on $\md_k$ gives rise to an $\infty$-category $\Alg_T$ of $T$-algebras (sometimes also called $T$-modules); an  object in  $\Alg_T$ is informally given by a chain complex $M\in \md_k$, a natural transformation $T(M) \rightarrow M$, and an infinite set of coherence data. We refer to  \cite[Section 4.7]{lurie2014higher} for precise definitions.\vspace{-2pt}
\end{notation}
\begin{example}If $k$ is a field of characteristic zero, then the $\infty$-category $\mathrm{Alg}_{\liedg} $  \INN{011@$\mathrm{Alg}_{\liedg} $ }of differential graded Lie algebras can be described as algebras over a  certain monad $\liedg$  \INN{120@$\liedg$} on $\md_k$, which sends a chain complex $V $ to the chain complex $\liedg(V) = \bigoplus_n (\Lie_n \otimes_{} V^{\otimes n})_{\Sigma_n}$. Here the tensor product is computed in complexes, and $(-)_{\Sigma_n}$ denotes  $\Sigma_n$-orbits (which are equivalent to \mbox{$\Sigma_n$-homotopy orbits).}\end{example}
In  \Cref{defliep} below, we will construct the \textit{partition Lie algebra monad} $\liep$ on $\md_k$.  \vspace{-2pt}
\begin{cons}[Partition Lie algebras] \label{conspla}
	The monad $\liep$  satisfies the following properties: \vspace{-6pt} \INN{120@$\liep$}  \vspace{-4pt}
\begin{enumerate}
\item If $V$ is a finite-dimensional $k$-vector space (considered as a discrete
$k$-module spectrum), then $\liep(V)$ is the linear dual of the (algebraic)
cotangent fibre\footnote{In particular, the theory of the cotangent complex in
the context of simplicial commutative rings, as in \cite{Qui67HA,
Illusie, Andre74} (see \cite[Sec.~25.3]{lurie2016spectral} for a modern account), plays a central role throughout this paper.}  of $k \oplus
V^{\vee}$, the trivial square-zero extension of $k$ by $V^\vee$.
In fact, this remains true for any coconnective $k$-module spectrum $V$ for which $\pi_i(V)$ is finite-dimensional for all $i$.
\item If $V\simeq \Tot(V^\bullet)\in \mod_{k,\leq 0}$ is represented by a 
cosimplicial $k$-vector space $V^\bullet$, then\vspace{-2pt}  $$ \liep(V) \simeq \bigoplus_{n} \Tot \left(\widetilde{C}^\bullet(\Sigma |\Pi_n|^\diamond,k) \otimes (V^\bullet)^{\otimes n}\right)^{\Sigma_n}.\vspace{-2pt}$$ Here $\widetilde{C}^\bullet(\Sigma |\Pi_n|^\diamond,k)$ denotes the    $k$-valued cosimplices on the space $\Sigma |\Pi_n|^\diamond$, the functor 
$(-)^{\Sigma_n}$ takes strict fixed points,  and the tensor product is computed in cosimplicial $k$-modules.

\item The functor $\liep$ commutes with filtered colimits and
geometric realisations. 
\item The tangent fibre $T_X$ of any   $X \in \mdl_{k,\Delta}$ has the structure of a
$\liep$-algebra.
\end{enumerate}
We write $\mathrm{Alg}_{\liep}$ \INN{011@$\mathrm{Alg}_{\liep}$}
for the $\infty$-category of $\liep$-algebras in
$\md_k$. \vspace{-1pt} \vspace{-1pt} \vspace{-1pt} 
\end{cons}
\begin{remark}
Given any partition Lie algebra $\mathfrak{g}\in \mathrm{Alg}_{\liep}$, the homotopy groups $\pi_\ast(\mathfrak{g})$ form a graded Lie algebra {in the shifted sense}. This means that given  $x\in \pi_i(\mathfrak{g})$ and $ y\in \pi_j(\mathfrak{g})$, we have a bracket $[x,y] \in \pi_{i+j-1}(\mathfrak{g})$. This shift is merely a matter of convention, but we have decided to adopt it as it seems more natural for our applications.\vspace{-1pt} \vspace{-1pt} 
\end{remark}
When $k$ has characteristic zero,   $\liep$ can be identified with the shifted differential graded Lie algebra monad (cf.\ \Cref{deriveddg}). 
For general fields,  partition Lie algebras  provide a new generalisation of differential graded Lie algebras. While $\liep$ looks somewhat similar to the  restricted Lie algebra monad (cf.\ e.g.\ \cite{fresselie}), it 
behaves in a substantially different way. For example, 
there is no homological degree $i$ such that the free partition Lie algebra on  a module $M$ in degree $i$ is also concentrated in degree $i$.

Most importantly, partition Lie algebras have the following   application  in deformation theory:\vspace{-2pt}
\begin{theorem}[Main theorem] 
\label{mainthm}
If $k$ is a  field, there is an equivalence\vspace{-1pt} of \mbox{$\infty$-categories} $$\mdl_{k,\Delta} \simeq \mathrm{Alg}_{\liep}\vspace{-1pt} $$ between 
 formal moduli problems and 
partition  Lie algebras over $k$. It sends a formal moduli problem  $X \in \mdl_{k,\Delta}$ to its tangent
fibre $T_X$ equipped  with a suitable partition Lie algebra structure. \vspace{-3pt}
\end{theorem} 
This means that locally, moduli spaces are still governed by an appropriate Lie\vspace{3pt} algebraic  \mbox{structure.}
\begin{example}
While it can be challenging to fully describe the partition Lie algebra associated to a given formal moduli problem, it is  easy to compute its underlying chain complex by considering the infinitesimal automorphisms of the trivial deformation over the trivial \mbox{square-zero extensions $k\oplus k[n]$.}
As expected,  the    formal deformations of a smooth and proper variety $Z$ over a field $k$ 
are controlled by a partition Lie algebra structure \mbox{on the chain 
complex $C^\ast(Z,T_Z)[1]$,} and deformations of a vector bundle $\mathcal{E}$ on $Z$ are governed by a 
partition Lie algebra \mbox{structure on $R\Gamma(Z,\End(\mathcal{E}))[1]$.}
We refer to \cite[Section 16.5, 19.4]{lurie2016spectral} and  \cite[Section 3]{CYDef} for details.
\end{example}

As in \cite{lurie2011derivedX, pridham2010unifying}, the correspondence between formal moduli problems and partition Lie
algebras arises from a form of    {Koszul duality} for algebras, which we shall formulate next. 

Write $\mathrm{SCR}_{k}^{\aug}$ for  the $\infty$-category of augmented simplicial commutative
$k$-algebras,  i.e.  elements in the overcategory $A\in (\mathrm{SCR}_{k})_{/k}$.\INN{190@$\mathrm{SCR}_{k}^{\aug}$} 
We will construct  a Koszul duality functor
$$ \Df: \mathrm{SCR}_{k}^{\aug}  \rightarrow \mathrm{Alg}_{\liep}^{op}, $$
which sends an augmented simplicial commutative $k$-algebra $A$ to 
the dual of its (algebraic) tangent fibre $(k \otimes_A L^{\Delta}_{A/k})^{\vee}$, equipped with its
natural $\liep$-algebra structure. 
A key step in the \mbox{proof of} \Cref{mainthm} is to show that $\Df$ restricts to an equivalence on the following subcategory of $\mathrm{SCR}_{k}^{\aug}$:
\begin{definition}  \label{cnscr}
An augmented object  $A\in \mathrm{SCR}_{k}^{\aug}$ is \emph{complete local Noetherian} if \begin{enumerate}\INN{190@$\mathrm{SCR}^{\cN}$}  
\item $\pi_0(A)$ is a complete local Noetherian ring;
\item each $\pi_i(A)$ is a finitely generated $\pi_0(A)$-module.
\end{enumerate}
\end{definition}  
 {Let  $\mathrm{SCR}_{k}^{\cN}$ be the full subcategory \INN{190@$\mathrm{SCR}_{k}^{\cN}$}
spanned by all complete local Noetherian $k$-algebras. Then:}
\begin{theorem} 
\label{koszuldualfun}
The Koszul duality functor
$\Df$ restricts to a contravariant equivalence
between $\mathrm{SCR}_{k}^{\cN}$ and  the full subcategory
of $\mathrm{Alg}_{\liep}$
spanned by those partition Lie algebras $\mathfrak{g}$ for which
$\pi_i(\mathfrak{g})$ is finite-dimensional for each $i$ and
vanishes for $i > 0$. \end{theorem}

We prove Theorem~\ref{koszuldualfun} ``by hand'' at the level of simplicial commutative
rings, by working carefully with filtered objects and exploiting the fact that
all rings that one encounters in this way are Noetherian. 
To deduce \Cref{mainthm}, one also needs to prove that the Koszul duality functor
carries appropriate pullbacks of simplicial commutative rings to pushouts of
$\liep$-algebras.\vspace{3pt}

Partition Lie algebras are subtle homotopical objects, and we need tools to study them. One can, for example,  consider the natural operations acting on the homotopy groups of a partition Lie algebra, which can   obstruct  the extendability of the corresponding formal moduli \mbox{problem to $\EE_n$-algebras.}

These operations  are parametrised by the homotopy groups of \textit{free} partition Lie algebras, which we will compute  by
using techniques from the work of the first author and Arone \cite{arone2018action},  
relying on discrete Morse theory and an argument inspired by earlier work of Arone and Mahowald \cite{arone1999goodwillie}. \vspace{2pt}

To state our result, we  need  the following classical notion  (cf.\ \cite{shirshov1958free}, \cite{chen1958free}):
\begin{definition}\label{Lyndon word}   
A word $w$ in letters $x_1,\dots,x_m$ is  said to be a  \textit{Lyndon word} if it is  smaller than any of its cyclic rotations in the lexicographic order with $x_1<\dots < x_m$. Write  $B(n_1,\dots,n_m)$ \INN{020@$B(n_1,\ldots,n_m)$}
 for the set of   Lyndon words which involve  the letter $x_i$  precisely $n_i$ times. 
\end{definition}

\begin{remark}
It is well-known that the set of Lyndon words in letters $x_1,\ldots x_m$ forms a basis for the free (ungraded) Lie algebra over $\ZZ$ on $m$ letters (cf.\ e.g.\ \cite{reutenauer2003free}).
\end{remark}
\begin{theorem}\label{finalthecohomology} 
The $\FF_p$-vector space $\pi_\ast(\Lie_{\FF_p,\Delta}^\pi  (\Sigma^{\ell_1} \FF_p  \oplus \ldots \oplus \Sigma^{\ell_m} \FF_p  ))$ 
 has a basis indexed by sequences $(i_1, \ldots, i_k, e,w)$. Here $w\in B(n_1,\ldots,n_m)$ is a  Lyndon word in $x_1,\ldots, x_m$. We have $e \in \{0,\epsilon\}$, where
 $\epsilon = 1$ if $p$ is odd and $\deg(w):= \sum_i (\ell_i-1)n_i+1$ is even. Otherwise,   $\epsilon = 0$.

  The integers $i_1, \ldots, i_k$ satisfy: 
\begin{enumerate}
\item each $|i_j|$ is congruent to $0$ or $1$ modulo $2(p-1)$; \label{congruence}
\item for all $1\le j<k$, we have   
$p i_{j+1} <i_j < -1$    or $  0 \leq i_j < pi_{j+1}$;
\item we have  $ (p-1)(1+e)\deg(w)-\epsilon \leq i_k < -1$ or $0 \leq i_k \leq (p-1) (1+e)\deg(w) -\epsilon$.
\end{enumerate}
The  sequence $(i_1, \ldots, i_k, e,w)$ sits  in homological degree  \mbox{$\left((1+e)\deg(w) -e\right) + i_1+\cdots +i_k-k$.} 
The decomposition in \Cref{conspla}(2) induces a weight decomposition  on 
the above vector space,  placing   $(i_1, \ldots, i_k, e,w)$ in multi-weight $(n_1 p^k (1+e) ,\ldots,n_m  p^k (1+e))$. \vspace{-2pt}
\end{theorem}
The input of this computation is  the homotopy of free simplicial and cosimplicial commutative rings
 as computed by Dold \cite{dold1958homology},  Nakaoka \cite{nakaoka1957cohomology} \cite{nakaoka1957cohomology2}, Milgram \cite{milgram1969homology}, and  Priddy \cite{priddy1973mod}. 
The case where $\ell_i\leq 0$ for all $i$  follows immediately from \cite[Theorem 8.14]{arone2018action}. For $p=2$ and  $\ell_i\leq 0$ for all $i$, our result can also be read off from the work of Goerss \cite{goerss1990andre}, who computed the algebraic Andr\'{e}--Quillen homology of trivial square-zero extensions at $p=2$.
\begin{remark}
Note that Pridham (cf.\ \cite[Section 5.3]{pridham2010unifying}) also considers the operations acting on the tangent spaces of formal moduli problems. He abstractly identifies the operations on  the coconnective part with the operations on Andr\'e--Quillen homology. \end{remark}

Up to now, we have stated our results   in the context of  simplicial commutative rings, which was indicated by the subscript $``\Delta"$.
We can obtain parallel results in the context of spectral algebraic geometry, hence describing deformations parametrised by connective $\EE_\infty$-rings over a given \mbox{field $k$.}

More precisely,  define the $\infty$-category $\clg_k^{\art}$ of spectral Artinian $k$-algebras and the $\infty$-category $\moduli_{k,\EE_\infty}$ \INN{130@$\moduli_{k,\EE_\infty}$} of spectral formal moduli problems by   replacing the term ``simplicial commutative $k$-algebra" by the term ``connective $\EE_\infty$-$k$-algebra" in \Cref{daa} and \Cref{FMP}, respectively.
In  \Cref{spla} below, we   construct the \textit{spectral partition Lie algebra monad} $\lieps$ on $\md_k$.  
\begin{cons}[Spectral partition Lie algebras] \label{conssppla}
	The monad $\lieps$  satisfies the following:
\begin{enumerate}
\item If $V$ is a finite-dimensional $k$-vector space (considered as a discrete
$k$-module spectrum), then $\lieps(V)$ is the linear dual of the topological cotangent fibre of $k \oplus
V^{\vee}$, the trivial square-zero extension of $k$ by $V^\vee$.
In fact, this remains true for any coconnective $k$-module spectrum $V$ for which $\pi_i(V)$ is finite-dimensional for all $i$.
\item If $V\in \md_{k,\leq N}$ is truncated above, then  $$ \lieps(V) \simeq \bigoplus_{n} \left(\widetilde{C}^\bullet(\Sigma |\Pi_n|^\diamond,k) \otimes V^{\otimes n}\right)^{h\Sigma_n},$$ 
where $(-)^{h\Sigma_n}$ denotes homotopy fixed points and the other notation is as above.
\item The functor $\lieps$ commutes with filtered colimits and
geometric realisations. 
\item The tangent fibre $T_X$ of any   $X \in \mdl_{k,\EE_\infty}$ has the structure of a
$\lieps$-algebra.
\end{enumerate}
We write $\mathrm{Alg}_{\lieps}$ for the $\infty$-category of $\lieps$-algebras in
$\md_k$. 
\end{cons}

We then have a variant of the previous equivalence:
\begin{theorem}
\label{spectralmainthm}
If $k$ is a  field, then there is an equivalence of \mbox{$\infty$-categories} $$\mdl_{k,\EE_\infty} \simeq \mathrm{Alg}_{\lieps}$$ between 
 spectral formal moduli problems and spectral
partition  Lie algebras over $k$, sending a formal moduli problem  $X \in \mdl_{k,\EE_\infty}$ to its tangent
fibre $T_X$. 
\end{theorem} 
Proving this equivalence again requires constructing a Koszul duality functor, and showing that it restricts to an equivalence between complete local Noetherian $\EE_\infty$-$k$-algebras and spectral partition Lie algebras $\mathfrak{g}$ which are coconnective and have degreewise finite-dimensional homotopy groups.
\Cref{mainthm} and \Cref{spectralmainthm} are proven with similar methods, which is why we  present much of the argument in an axiomatic way (cf.\ \Cref{axiomaticsec}) applying to both of these contexts at once.\vspace{3pt}

The natural operations on $  \Lie_{\FF_p,\EE_\infty}^\pi$-algebras are parametrised by the homotopy groups of free  {spectral} partition Lie algebras, which we compute by a similar method as before:
 \begin{theorem}\label{finalthecohomologyspectral} 
The $\FF_p$-vector space $\pi_\ast( \Lie_{\FF_p,\EE_\infty}^\pi(\Sigma^{\ell_1} \FF_p  \oplus \ldots \oplus \Sigma^{\ell_m} \FF_p  ))$ 
 has a basis indexed by sequences $(i_1, \ldots, i_k, e,w)$. Here $w\in B(n_1,\ldots,n_m)$ is a  Lyndon word in $x_1,\ldots, x_m$. We have $e \in \{0,\epsilon\}$, where
 $\epsilon = 1$ if $p$ is odd and $\deg(w):= \sum_i (\ell_i-1)n_i+1$ is even. Otherwise,   $\epsilon = 0$.

  The integers $i_1, \ldots, i_k$ satisfy: 
\begin{enumerate}
\item\hspace{-5pt}' \ each $i_j$ is congruent to $0$ or $1$ modulo $2(p-1)$; \label{congruence}
\item\hspace{-5pt}' \ for all $1\le j<k$, we have   $i_j<p i_{j+1}$;
\item\hspace{-5pt}' \  we have  $i_k \leq (p-1) (1+e)\deg(w) -\epsilon $.
\end{enumerate}
The homological degree of $(i_1, \ldots, i_k,e,w)$ is $\left((1+e)\deg(w) -e \right)+i_1+\cdots+i_k-k$.
The decomposition in \Cref{conssppla}(2) induces a weight decomposition  on 
the above vector space,  placing   $(i_1, \ldots, i_k, e,w)$ in multi-weight $(n_1 p^k (1+e) ,\ldots,n_m  p^k (1+e))$.
\end{theorem}
The input to this computation is the  homotopy of free $\EE_\infty$-rings computed by Adem \cite{MR0050278},  
Serre \cite{MR0060234}, Araki--Kudo \cite{kudo1956topology},  Cartan \cite{MR0065161}  \cite{MR0068219}, Dyer--Lashof \cite{MR0141112}, May, and Steinberger \cite{MR836132}.  It is again inspired by Arone--Mahowald's classical work \cite{arone1999goodwillie}. \vspace{5pt}

Finally, we prove variants of \Cref{mainthm} and \Cref{spectralmainthm} in mixed characteristic. More precisely, let $A$ be complete local Noetherian  with residue field $k$, either in  simplicial commutative rings or in  $\EE_\infty$-rings. There is a natural notion of formal moduli problems in these mixed contexts (cf.\ \Cref{fmpmixed} below); we write $\moduli_{A//k, \Delta}$ and $\moduli_{A//k, \einf}$ for the respective $\infty$-categories. For example, we can describe the formal neighbourhood of a $k$-point inside a (suitably geometric) stack  defined over $\Spec(A)$ by one of  these ``mixed" formal moduli problems.

In \Cref{Akcomparestuff} and \Cref{pliemixed}, we construct relative versions of partition Lie algebras and spectral partition Lie algebras. The resulting $\infty$-categories are denoted by $\alg_{\liepr{A}}$ and  $\alg_{\liepss{A}}$, respectively. Finally, we prove:

\begin{theorem} \label{mostgeneral}Let $k$ be a field.
\begin{enumerate}
\item
If $A$ is a complete local Noetherian simplicial commutative ring with residue field $k$, there is an equivalence \vspace{3pt} of $\infty$-categories
$\moduli_{A//k, \Delta} \simeq \alg_{\liepr{A}}$. 
\item If $A$ is a complete local Noetherian $\einf$-ring with residue field $k$, then there is an equivalence of $\infty$-categories
$\moduli_{A//k, \einf} \simeq \alg_{\liepss{A}}$. 
\end{enumerate}
In both cases, these  equivalences send  a formal moduli problem
to its tangent fibre. 
\end{theorem} 

Hence the various variants of partition Lie algebras   provide an algebraic description of formal deformation theory in finite and mixed characteristic.

\subsection*{Acknowledgments} 
It is a pleasure to thank Greg Arone and Jacob Lurie for their very generous support and consistent encouragement.
We are also grateful to  Bhargav Bhatt, Dennis Gaitsgory, Saul
Glasman, Thomas Nikolaus, Sam Raskin, Nick Rozenblyum, and Peter Teichner for helpful conversations related to this project. 
We thank the referee for many helpful suggestions and corrections to an earlier version of this paper, and Sof\'{i}a Marlasca Aparicio and Sam Moore for comments on the manuscript.

The first author wishes to thank the Max Planck Institute for Mathematics, Bonn, and Merton College and the Mathematical Institute at  Oxford University, for their support.
This work was done while the second author was a Clay Research Fellow, and the
second author would like to thank the Clay Mathematics Institute and the
University of Chicago for their 
support. 

Both authors would also like to thank the Mathematical Sciences Research Institute  in Berkeley, California, for its hospitality, where their work was  supported by the National Science Foundation under Grant No. DMS-1440140 during the  Spring 2019 semester.

\newpage

\section{Preliminaries}

\newcommand{\tr}{\mathrm{tr}}

Let $\mathcal{C}$ be a presentable stable $\infty$-category. 
In this section, we will briefly review various preliminaries involving filtered and graded
objects in $\mathcal{C}$, and moreover fix some notation for the remainder of this paper. 
A convenient reference for this material (with   slightly different
notation) is \cite{GwilliamPavlov}. 

For notational convenience, we will usually work with filtrations and gradings concentrated in degrees $1$ and above. This choice reflects that in the sequel (in
particular in \Cref{axiomaticsec}), we will often work with {nonunital} commutative algebras.  We will   use the notation  {$\fil$ and $\gr$ in this context.} 

When we discuss unital
commutative algebras in \Cref{einfsec}, we will use the notation $\fil^+$ and $ \gr^+$ for
filtrations and gradings that start in degree $0$.

\begin{definition}[Filtered objects] \label{filtereddef} 
Consider the nerve $\N(\mathbb{Z}_{\geq 1})$ of the partially ordered  \INN{060@$\fil(\mathcal{C})$,  $\fil^+(\mathcal{C})$}
set $\mathbb{Z}_{\geq 1}$ and its opposite $\N(\mathbb{Z}_{\geq 1})^{op}$. 
We define the $\infty$-category $\fil(\mathcal{C})$ of \emph{filtered objects}
of $\mathcal{C}$ as
$$ \fil(\mathcal{C} ) := \fun( \N(\mathbb{Z}_{\geq 1})^{op}, \mathcal{C}). $$
We will often write a filtered object $X \in \fil(\mathcal{C})$ as a system $\left\{F^i X\right\}_{i \geq
1}$ of objects in $\mathcal{C}$, i.e.\ a sequence $\dots \rightarrow F^i X \rightarrow F^{i-1} X \to
\dots \rightarrow F^1 X$.  
We call $F^1 X$ \INN{060@$F^1$}the \emph{underlying object} of $X$, and obtain a \INN{060@$F^1$} functor $$F^1:
\fil(\mathcal{C}) \rightarrow \mathcal{C}.$$ 

Similarly,  set  $\fil^+(\mathcal{C}) := \fun( \N( \mathbb{Z}_{\geq 0})^{op}, \mathcal{C})$. The objects of $\fil^+(\mathcal{C})$ are 
filtered objects where the filtration starts in degree $0$
instead, and we 
  write $F^0:
\fil^+(\mathcal{C}) \rightarrow \mathcal{C}$ for the   \mbox{underlying object functor.}

\end{definition} 
\begin{example} 
\label{filadjunctionunderlying}
The functor $F^1: \fil(\mathcal{C}) \rightarrow \mathcal{C}$ sending
$\left\{F^i X\right\}$ to $F^1 X \in \mathcal{C}$ admits a 
left adjoint, which sends an object
$Y \in \mathcal{C}$ to the filtered object $(\dots \rightarrow 0 \rightarrow 0 \rightarrow Y)$. 
\end{example}

\begin{definition}[Graded objects] \INN{070@$\gr(\mathcal{C})$,  $\gr^+(\mathcal{C})$}
Let $\mathbb{Z}_{\geq 1}^{\mathrm{ds}}$ denote the category with one object for every positive integer and only identity morphisms.
Define the $\infty$-category $\gr(\mathcal{C})$ \mbox{of \textit{graded objects} of  
$\mathcal{C}$ as}  $$\gr(\mathcal{C}) = \fun(\N( \mathbb{Z}_{\geq 1}^{\mathrm{ds}}),
\mathcal{C}).$$ 
We    write objects of $\gr(\mathcal{C})$ as
$X_{\star}$ whenever we want to emphasize the grading. Given a 
graded object $X_{\star}$,  the direct sum $\bigoplus_{i \geq 1} X_i$. 
is referred to as the 
\emph{underlying object} of $X_{\star}$.

Similarly as for filtered objects, we 
define a variant $\gr^+(\mathcal{C})$ as 
$\gr^+(\mathcal{C}) = \fun(\N( \mathbb{Z}_{\geq 0}^{\mathrm{ds}}),
\mathcal{C}),$
where $\mathbb{Z}_{\geq  0 }^{\mathrm{ds}}$ \INN{260@$\mathbb{Z}_{\geq  0 }$,  $\mathbb{Z}_{\geq  0 }^{\mathrm{ds}}$, $\mathbb{Z}_{\geq  1 }$,  $\mathbb{Z}_{\geq  1 }^{\mathrm{ds}}$} denotes the discrete category on
$\mathbb{Z}_{\geq 0}$. 
\end{definition}

\begin{definition}[Associated gradeds]  \INN{070@$\gr X $}
We have a functor 
$$\gr: \fil(\mathcal{C}) \rightarrow \gr(\mathcal{C}),$$
which sends $X = \left\{F^i X\right\}_{i \geq 1}$ to the associated graded
object
$\gr X $ satisfying $(\gr X)_i = F^i X/F^{i+1} X $ for all $i \geq 1$. 
Similarly, we have a natural   functor 
$ \gr: \fil^+(\mathcal{C}) \rightarrow \gr^+(\mathcal{C}).$
\end{definition} 

\begin{definition}[Symmetric monoidal structures]\label{filmonoidal}
Suppose that $\mathcal{C}$ is nonunital presentably symmetric monoidal, by which we mean that
$\mathcal{C}$ is presentable and the tensor product preserves colimits in each
variable. 
Using Day convolution (cf.\ \cite{glasman2016day}), one can equip both $\fil(\mathcal{C})$ and
$\gr(\mathcal{C})$ with the structure of presentably nonunital symmetric 
monoidal $\infty$-categories. Furthermore, the associated graded functor 
$\gr: \fil(\mathcal{C}) \rightarrow \gr(\mathcal{C})$ is (nonunital) symmetric monoidal (cf.
\cite[Sec. 2.23]{GwilliamPavlov}).  
\end{definition}

\begin{definition}[Gradings in degree $\geq a$] 
Given an integer $a \geq 1$, 
\begin{enumerate}
\item
let $\gr_{\geq a}(\mathcal{C})$ be the full subcategory \INN{070@$\gr_{\geq a}(\mathcal{C})$}
of $\gr(\mathcal{C})$ spanned by all  $X_{\star}$ with $X_j \simeq 0$ for $j< a$.

\item  let $\tr_{\leq a}: \gr(\mathcal{C}) \rightarrow \gr(\mathcal{C})$ denote the functor\INN{200@$\tr_{\leq a}$}
which sends a graded object $X_{\star}$ to the graded object $Y_{\star}$ with $Y_j
= X_j$ for $j \leq a$ and $Y_j = 0$ for $j > a$.

\end{enumerate}
\end{definition}

Next, we will review the process of completion. 
\begin{definition}[Complete and constant filtered objects]\label{defcomplete}\ 
\begin{enumerate} 
\item
A filtered object $Y = \left\{F^i Y\right\}_{i \geq 1}$ is said to be \emph{constant} 
if
the maps $F^{i+1} Y \rightarrow F^{i} Y$ are   equivalences for all $i \geq 1$, or \vspace{2pt} equivalently, if $\gr(Y) =0$. 
\item A filtered object $Z = \left\{F^i Z\right\}_{i \geq 1}$ is 
\emph{complete} if $\varprojlim_i F^i Z = 0$, i.e.\ if for each constant filtered object $Y \in \fil(\mathcal{C})$,
we have $\hom_{\fil(\mathcal{C})}(Y, Z) = 0$. 
Let $\filc(\mathcal{C}) \subset \fil(\mathcal{C})$ denote the full subcategory of \INN{060@$\filc(\mathcal{C})$}
complete objects. We have a similar notion for objects of $\fil^+(\mathcal{C})$. \end{enumerate}
\end{definition}
\begin{definition}[Completions]\label{defcompletion}
The inclusion $\filc(\mathcal{C}) \subset \fil(\mathcal{C})$
is the right adjoint of a Bousfield localisation, which we will refer to as the
\emph{completion} functor $\fil(\mathcal{C}) \rightarrow \filc(\mathcal{C})$. 
Given a filtered object $X = \left\{F^i X\right\}_{i \geq 1}$, 
we can form its  \emph{completion} $\hat{X}$,   
which is  given by the cofibre of the constant filtered object on $\lim_i F^iX$ into $X$.
The canonical map $X \rightarrow \widehat{X}$ induces an equivalence on $\gr(-)$. 
\end{definition}

\begin{remark} \label{detect}We can detect whether  a given morphism $X \xrightarrow{f} Y$
induces an equivalence after completion by  passing to associated gradeds. Indeed, if \mbox{$\gr(f):
\gr(X) \rightarrow \gr(Y)$} is an equivalence, then $\gr(\cofib(f))\simeq
\cofib(\gr(f)) \simeq 0$. This implies that $\cofib(\widehat{f})\simeq \widehat{\cofib(f)}$ is both constant and
complete, and therefore vanishes. 
\end{remark}

In general, $\filc(\mathcal{C}) \subset \fil(\mathcal{C})$ is {not} closed under
colimits. However completions are preserved
by geometric realisations under suitable connectivity hypotheses.
 More precisely, 
let $R$ be a connective $\einf$-ring. 
Write $\md_{R, \geq 0} \subset \md_R$ for the full subcategory 
of connective $R$-modules.

\begin{definition}[Connective filtered and graded objects] We define:  
\begin{enumerate}
\item
 $\fil \md_{R, \geq 0} \subset \fil \md_R$ \INN{060@$\fil \md_{R, \geq 0} $, $\filc \md_{R, \geq 0} $} is the  subcategory
spanned by all  
$X = \left\{F^i X\right\}_{i \geq 1}$ for which each $F^i X$ belongs to
$\md_{R, \geq 0}$; we have an \vspace{1pt}analogous subcategory \mbox{$\fil^+ \md_{R, \geq 0}
\subset \fil^+ \md_R$.}
\item   $\filc \md_{R, \geq 0} \subset \fil \md_{R, \geq 0}$   is the   \vspace{1pt} \vspace{1pt} \vspace{1pt}subcategory of complete objects (and similarly \mbox{for
$\fil^+$).}
\item  $\gr \md_{R, \geq 0} \subset \gr \md_R$ \INN{070@$\gr \md_{R, \geq 0}$} is the full subcategory
of objects $X_{\star}$ with $X_i$ connective for each $i \geq 1$. 
\end{enumerate}
\end{definition}

In the sequel, we will make frequent use of the following observation:
\begin{proposition}\label{geomrealcomplete}
\mbox{The full subcategory $\filc \md_{R, \geq 0} \subset \fil \md_R$ is closed under
geometric realisations.}
 \end{proposition}
 \begin{proof}
Let $X_\bullet$ be a simplicial object in $\filc \md_{R, \geq 0}$  and let
$Y = |X_\bullet|$ denote its geometric realisation (computed in $\fil \md_R$). 
We need to see that 
$Y$ is complete, i.e.\   $\varprojlim_i F^i Y = 0$ in $\md_R$. 
By the Milnor short exact sequence, this is equivalent to the assertion that
for each $j$, we have
\begin{equation} \label{2limvanish}  \varprojlim_i \pi_j (F^i Y) = \varprojlim{}^1_i
\pi_j( F^i Y) = 0,  \end{equation}
in the category of abelian groups. 

We observe that $F^i Y = | F^i X_\bullet|$. Since all modules in question are
connective, we have
$\pi_j ( F^i Y) \cong\pi_j (  \mathrm{sk}_n |F^i X_\bullet|)$ for $n >  j$ and all $i$. 
Thus, in verifying \eqref{2limvanish} for a given $j$, we may replace $Y$ with
the filtered object $\mathrm{sk}_{j+1} |X_\bullet| \in \fil\md_R$. 
This can be expressed as a \emph{finite} colimit of a diagram in the $X_s$,
and since $\filc \md_R \subset \fil \md_R$ is   closed under finite
colimits, we deduce
$\mathrm{sk}_{j+1} |X_\bullet| \in \filc\md_R$. Applying   the Milnor exact sequence  again
to
$\mathrm{sk}_{j+1} |X_\bullet| \in \filc\md_R$
 \mbox{shows \eqref{2limvanish}.} 
\end{proof}

In the sequel, it will be important to understand how functors $F$ interact with 
the
internal grading. To this end, we will use the following natural definition. 

\begin{definition}[Increasing functors]  
We say that a functor  $F: \gr \md_{k, \geq 0} \longrightarrow \gr \md_{k, \geq 0}$ is  {\emph{$i$-increasing}} if: 
\begin{enumerate}
\item The functor $ \tr_{\leq n} F$ factors through $ \tr_{\leq n-i+1} : \gr \md_{k, \geq 0} \rightarrow \gr
\md_{k, \geq 0}$.   
\item Given any $X \in \gr \md_{k, \geq 0}$, we have $F(X)_j = 0$ for all $j < i$. 
That is, $F$ takes values in objects which have contractible components in
internal grading less than $i$. 
\end{enumerate}
\end{definition}

\begin{example} 
The functor $V \mapsto V^{\otimes i}, \gr \md_{k, \geq 0} \rightarrow \gr \md_{k, \geq
0}$ is $i$-increasing. 
Similarly, the functor $V \mapsto (V^{\otimes i})_{h \Sigma_i}$ (which appears
in the expression for the free $\einf$-algebra) is $i$-increasing.
\end{example}

 Next, we record several finiteness conditions which will be useful in the sequel. 
 Mostly the following serves to record some notation.

\newcommand{\grcoh}{\gr^{\mathrm{ft}} \md}
\newcommand{\filcoh}{\fil^{\mathrm{ft}} \md}
\begin{definition}[Finiteness conditions]  \label{finitenessconditions}
Let $k$ be a field.  
\begin{enumerate}
\item Let $\coh_k \subset \md_k$ \INN{130@$\coh_k$, $\coh_{k, \geq 0}, \coh_{k, \leq 0}$} 
denote the full  subcategory spanned by those
objects $X \in \md_k$ such that each homotopy group $\pi_i(X), i \in \mathbb{Z}$ is a
finite-dimensional vector space. We say that these objects are of \emph{finite type.}  
Define $\coh_{k, \geq 0}, \coh_{k, \leq 0}
\subset \coh_k$ as the full subcategories spanned by  \vspace{3pt}connective and coconnective
objects, respectively. 
\item Let $\grcoh_k  \subset \gr \md_k$ denote the  subcategory
spanned by all objects $X_{\star} \in \gr \md_k$ for which 
$\bigoplus_{i \geq 1} X_i$ belongs to $\coh_k$. We define $\gr^+
\coh_k$ in a similar way. We let $\grcoh_{k, \geq 0} \subset \grcoh_k$  be the full subcategory
spanned by  \vspace{3pt}connective objects.  \INN{070@$\grcoh_{k, \geq 0}$}
\item Let $\filcoh_k \subset \fil \md_k$ denote the full subcategory of
filtered objects $X =\left\{F^i X\right\}_{i \geq 1}$ which are complete and such
that $\gr(X) \in \grcoh_k$. Similarly, we denote the full subcategory of connective objects by
$\filcoh_{k, \geq 0} \subset \filcoh_k$. \INN{060@$\filcoh_k$, $\filcoh_{k, \geq 0}$}
(As an example, we could take a  (discrete) finite-dimensional $k$-vector space with a finite classical
filtration by subspaces).
\end{enumerate}
\end{definition}

\newpage

\section{Functors of $k$-modules}\label{extending} 
Let $k$ be a field and write $\md_k$ for the $\infty$-category of $k$-module
spectra. In this section, we will discuss functors $\md_k\rightarrow \md_k$
which preserve sifted colimits, which is equivalent to preserving filtered colimits and geometric
realisations. 

Below, we will need to construct various such functors 
$\md_k \rightarrow \md_k$, e.g.\ the free partition Lie algebra functor. 
One typically cannot write down such a functor easily by hand on all of $\md_k$. 
However, it will be easy 
to describe the
functor on a suitable
subcategory of $\md_k$, often in particular the \emph{coconnective} perfect
$k$-module spectra. For the general theory we will need our functors on all of
$\md_k$, though, and the primary purpose of this section is to discuss some
abstract 
homological algebra 
which will enable us to construct the extension to all of $\md_k$. 

\subsection{Extending functors} 
In the following, we will freely use the theory of  Kan extensions along fully faithful
inclusions, as in  \cite[Section 4.3.2]{lurie2009higher}.

\begin{notation} 
Let $\mathcal{C}, \mathcal{D}$ be $\infty$-categories with sifted colimits. 
Write $\fun_{\Sigma}(\mathcal{C},
\mathcal{D}) $ \INN{060@$\fun_{\Sigma}(\mathcal{C},
\mathcal{D}) $} for the full subcategory of $ \fun(\mathcal{C}, \mathcal{D})$ spanned
by all functors which preserve sifted colimits; set \mbox{$\End_\Sigma(\mathcal{C}) := \fun_{\Sigma}(\mathcal{C},\mathcal{C})$.}
\end{notation} 

The first observation is that it is easy to describe sifted-colimit-preserving
functors out of
the full subcategory $\md_{k, \geq 0} \subset \md_k$ of connective $k$-module
spectra.
  In fact, $\md_{k, \geq 0}$ can be
characterised by a universal property (cf.\ \cite[Sec. 7.2.2]{lurie2014higher}):
\begin{proposition}[{The universal property of $\md_{k, \geq 0}$}]
\label{modulesarepsigma}
Given any $\infty$-category $\mathcal{D}$ with sifted colimits, 
restriction induces  \mbox{an equivalence}
of $\infty$-categories
$$\fun_{\Sigma}( \md_{k, \geq 0}, \mathcal{D}) \xrightarrow{\ \ \simeq \ \ } \fun(
\vect^\omega_k, \mathcal{D})$$ 
whose inverse is given by left Kan extension.
Here $\vect^\omega_k$ denotes the full subcategory spanned by all
finite-dimensional discrete $k$-module spectra, which is equivalent to the (nerve of the) usual category
of  finite-dimensional $k$-vector spaces.
\end{proposition} 
One   therefore has the following construction (which goes back to the work of \mbox{Dold-Puppe \cite{dold1961homologie}):}

\begin{cons}[Nonabelian derived functors] 
\label{derivedfun}
Fix a functor $F: \vect^\omega_k \rightarrow \vect_k$ from the category $\vect^\omega_k$ of
finite-dimensional $k$-vector spaces to the category $\vect_k$ of all $k$-vector
spaces. Using that $\vect_k$ is equivalent to the full subcategory of $\md_k$
spanned by all discrete $k$-module spectra, we can extend  $F$ 
 to a sifted-colimit-preserving  functor $\L F: \md_{k, \geq 0}
\rightarrow \md_{k, \geq 0}$. 
The functor $\L F$ is often called the \emph{nonabelian derived functor} of $F_0$. 
\end{cons}
\begin{example}[] We recall the following classical examples:
\label{tensorpower} 
\begin{enumerate}
\item  
Given  an integer $i\geq 0$, consider the functor $\bigotimes^i: \vect^\omega_k \rightarrow \vect^\omega_k$ which 
sends a vector space $V_0$ to $V_0^{\otimes i}$. This canonically extends to a
functor $\md_{k, \geq 0} \rightarrow \md_{k, \geq 0}$, which is necessarily just the
iterated tensor power functor $\bigotimes^i: \md_{k, \geq 0} \rightarrow \md_{k, \geq 0}$ coming from the
symmetric monoidal structure on $\md_{k, \geq 0}$. 

\item
Given $i \geq 0$, we consider the functors\INN{190@$\sym^i$} \INN{120@$\L\sym$} 
\INN{120@$ \bigwedge^i}
\INN{120@$ \L \bigwedge^i$} 
\INN{070@$\Gamma^i$}
\INN{120@$\L \Gamma^i$}
$\sym^i,\  \bigwedge^i, \ \Gamma^i : \vect^\omega_k \rightarrow \vect_k$ which send a finite-dimensional vector space $V$ to its $n^{th}$ symmetric, exterior, or divided power, respectively. 
 The nonabelian derived functor \Cref{derivedfun} again allows us to define
 canonical extensions $\L \sym^i, \L \bigwedge^i , \L \Gamma^i : \mod_{k, \geq 0} \rightarrow
 \mod_{k, \geq 0}$ of these three functors. 

\end{enumerate}
\end{example} 
There is a basic asymmetry 
between these two examples. 
The extended functor  of $\bigotimes^i$  in \Cref{derivedfun}(1)
arises naturally  from the symmetric monoidal
structure on $\md_{k, \geq 0}$; consequently, the functor $\bigotimes^i$ is naturally defined on all of $\md_k$, not only on 
the connective $k$-module spectra. By contrast, if $k$ is of \mbox{characteristic $p > 0$,} the
functors $\L \sym^i$ and $ \L \bigwedge^i$ in  \Cref{derivedfun}(2) generally cannot be described directly in terms of
the symmetric monoidal structure on $\md_{k, \geq 0}$. It is correspondingly less
clear that $\L \sym^i$ and $ \L \bigwedge^i$ naturally extend to all of $\md_k$, though this
was first shown in work of Illusie \cite[Sec. I-4]{Illusie}. 

In this section, we will establish two generalisations of Proposition
\ref{modulesarepsigma} to functors defined on  \textit{all} $k$-module spectra,
and help bridge the above asymmetry: in particular, we will describe 
functors such as $\L \sym^i$ and $ \L  \bigwedge^i$ on all of $\md_k$. 

We begin by  reviewing some basic  facts about compact generation and perfect
modules, and refer to \cite[Sec. 7.2.4]{lurie2014higher} for a detailed treatment. 
\begin{notation}[Perfect modules]   \INN{160@$\perf_k$,  $\perf_{k, [n_1,n_2]}$,  $\perf_{k, \geq n}$, $\perf_{k, \leq n}$}
We write $\perf_k \subset \md_k$ for the full subcategory of $\mod_k$ spanned by
all \textit{perfect} $k$-module spectra, i.e.\ $k$-module spectra $M$ with $\dim_k(\pi_\ast(M))<\infty$.
Let $\perf_{k, [n_1,n_2]}$ be the full subcategory of $\perf_k$ spanned by
$k$-module spectra whose homotopy groups are concentrated between degrees $n_1$
and $n_2$. Set $\perf_{k, \geq n}:= \perf_{k, [n,\infty]}$ and
\mbox{$\perf_{k, \leq n}:= \perf_{k, [-\infty, n]}$.} 
\end{notation}

The $\infty$-category $\md_k$ is a compactly generated $\infty$-category, and 
a  module spectrum $M\in \mod_k$ is compact if and only if it is perfect.
The $\infty$-category $\md_k$ can therefore be identified with the $\mathrm{Ind}$-completion
(cf.\ \cite[Sec. 5.3.5]{lurie2009higher}) of
$\perf_k$. We deduce that for any $\infty$-category $\mathcal{D}$ with filtered colimits, restriction and left Kan extension give mutually inverse \mbox{equivalences}
\begin{equation} \label{indequiv} \fun_{\omega}( \md_k, \mathcal{D}) \simeq  \fun(\perf_k, \mathcal{D})
\end{equation}
between $\fun(\perf_k, \mathcal{D})$ and the $\infty$-category  $\fun_{\omega}( \md_k, \mathcal{D})$ of functors $\md_k\rightarrow \mathcal{D}$ which preserve filtered colimits.  
\begin{notation} 
Given a simplicial diagram $X_\bullet \in \fun(\Delta^{op}, \mathcal{C})$ in some $\infty$-category $\mathcal{C}$, we write $|X_\bullet| :=\mycolim{\Delta^{op}}(X_\bullet)$ for its geometric realisation. \INN{000@$|-|$}  
If $X^\bullet \in \fun(\Delta, \mathcal{C})$ is instead a cosimplicial diagram, we write $\Tot(X^\bullet ):=\mylim{\Delta}(X_\bullet)$ for its \textit{totalisation}.\INN{200@$\Tot$}  

A simplicial object $X_\bullet$ is  said to be \emph{$m$-skeletal}  if it is the left Kan extension of its restriction to $\Delta^{op}_{\leq m}$, the full subcategory of $\Delta^{op}$ spanned by $[0],[1],\ldots,[m]$.   
We recall that $\Delta^{op}_{\leq m}$ is cofinal to a finite simplicial set
(see \cite[1.2.4.17]{lurie2014higher}), so that geometric realisations of $m$-skeletal simplicial
objects behave like finite colimits. 
\end{notation}

\begin{example} If $X_{\bullet}= (  \ldots \  \substack{\longrightarrow \vspace{-4pt} \\ \leftarrow   \vspace{-4pt} \ \\ \longrightarrow  \vspace{-4pt} \\ \leftarrow  \vspace{-4pt} \ \\ \longrightarrow} \   X_1\   \substack{  \longrightarrow  \vspace{-4pt} \\ \leftarrow  \vspace{-4pt} \ \\ \longrightarrow} \ X_0)$ is a simplicial diagram 
  valued in discrete $k$-vector spaces, its realisation $|X_{\bullet}| \in \mod_k$  is the connective chain complex $\ldots  \rightarrow X_1 \rightarrow X_0 \rightarrow 0 \rightarrow \ldots $ obtained by taking the alternating sum of the face maps of $X_{\bullet}$.  Dually, the totalisation of a cosimplicial diagram in discrete $k$-vector spaces is a coconnective chain complex, \mbox{which can be computed  in a similar way.}
\end{example}

\begin{definition}[Finite geometric realisations] 
Let $\mathcal{D}$ be an $\infty$-category 
admitting geometric realisations. 
We say that a functor $F: \perf_k \rightarrow \mathcal{D}$ \emph{preserves finite
geometric realisations} if for every simplicial object $X_\bullet$ of $\perf_k$
which is $m$-skeletal for some $m$ (so that $|X_\bullet|$ belongs to
$\perf_k$),  the natural map  
$|F(X_\bullet)| \rightarrow F(|X_\bullet|)$ is an equivalence. 
We write \INN{060@$\fun_\sigma( \perf_k, \mathcal{D})$} $\fun_\sigma( \perf_k, \mathcal{D}) \subset \fun( \perf_k,
\mathcal{D})$ for the full subcategory spanned by functors which preserve finite
geometric realisations. 
\end{definition} 

We can now state our first generalisation of Proposition \ref{modulesarepsigma}: 
\begin{proposition} 
\label{littlesigma}
Given any $\infty$-category $\mathcal{D}$ with sifted colimits, restriction
and left Kan extension induce mutually inverse equivalences 
$$ \fun_\Sigma( \md_k, \mathcal{D})  \xrightarrow{\ \ \simeq \ \ }\fun_\sigma( \perf_k, \mathcal{D}).$$
\end{proposition} 
\begin{proof} 
This follows from the equivalence  $\fun_\omega( \md_k, \mathcal{D}) \simeq \fun( \perf_k,
\mathcal{D})$ stated in  \eqref{indequiv}. 
It suffices to show that a functor $F: \md_k \rightarrow \mathcal{D}$ which preserves
filtered colimits 
additionally preserves sifted colimits if and only if its restriction
$F|_{\perf_k}$ preserves finite geometric realisations. This follows easily from
the following facts: 
\begin{enumerate}
\item  
The functor $F$ preserves sifted colimits if and only if it preserves filtered colimits (which it does by assumption) and geometric realisations (cf.\ \cite[Corollary 5.5.8.17]{lurie2009higher}).
\item
Every simplicial object in $\md_k$ is a filtered colimit
of simplicial objects which are $m$-skeletal for various $m$ (take the left Kan
extensions from truncations). 
\item Every $m$-skeletal simplicial object in $\md_k$ is a filtered colimit 
of $m$-skeletal simplicial objects which take values in $\perf_k$ (this follows  
as the hom-sets in $\Delta^{op}_{\leq m}$ are finite).\vspace{-15pt}
\end{enumerate}
\end{proof} 
For later applications, we need to refine the above result from $\perf_k$ to a smaller subcategory. 

\begin{definition}[Finite coconnective  geometric realisations]  
\label{cfgr}
Let $\mathcal{D}$ be an $\infty$-category 
admitting geometric realisations. 
We say that a functor $F: \perf_{k, \leq 0} \rightarrow \mathcal{D}$ \emph{preserves 
finite coconnective geometric realisations} if for every simplicial object
$X_\bullet$ of $\perf_{k, \leq 0}$
which is $m$-skeletal for some $m$  and such that 
$|X_\bullet|$ belongs to
$\perf_{k, \leq 0}$,  the  natural map  
$|F(X_\bullet)| \rightarrow F(|X_\bullet|)$ is an equivalence. 
We write $\fun_\sigma( \perf_{k, \leq 0}, \mathcal{D}) \subset \fun(
\perf_{k, \leq 0}, \INN{060@$\fun_\sigma( \perf_{k, \leq 0}, \mathcal{D})$}
\mathcal{D})$ for the full subcategory spanned by functors which preserve
finite coconnective
geometric realisations. 

\end{definition} 

We will now interpret the condition in Definition \ref{cfgr} in terms of left Kan
extensions (albeit with an infinite number of such conditions).

\begin{proposition} 
Let $\mathcal{D}$ be an $\infty$-category with sifted colimits. 
\begin{enumerate}
\item 
Let 
$F \in \fun( \perf_k, \mathcal{D})$. Then $F$ preserves finite
geometric realisations if and only if, 
 for
each $n \geq 0$, the
restriction  $F|_{\perf_{k, \geq -n}}$ is left Kan extended from
$\vecto_k[-n]$. 
\item  
Let $F \in \fun( \perf_{k, \leq 0}, \mathcal{D})$. Then 
$F $ preserves finite coconnective geometric realisations if and only if, for
any $n \geq 0$, the
restriction   $F|_{\perf_{k, [-n, 0]}}$ is left \mbox{Kan extended from
$\vecto_k[-n]$.}
\end{enumerate}
\label{LKanleqzero}
\end{proposition} 
\begin{proof} 
We shall only prove statement (2); the proof of (1) is similar. 
Suppose 
$F \in \fun_\sigma( \perf_{k, \leq 0}, \mathcal{D})$. 
We claim that $F_n = F|_{\perf_{k, [-n, 0]}}$ is left Kan extended from
$\vect^\omega_k[-n]$. 
Analogously to Proposition~\ref{modulesarepsigma}, the left Kan extension of $F_{n}^0 := F|_{\vect^\omega_k[-n]}$ to
$\perf_{k, [-n, 0]}$ (and indeed to all of $\md_{k, \geq -n}$) can be computed 
as follows: 
given $X \in \perf_{k, [-n, 0]}$, we find an $n$-skeletal simplicial object $Y_\bullet$ with each $Y_i \in \vect^\omega_k[-n]$ such that  $|Y_\bullet|\simeq X$.  The value of the left Kan extension of $F_n^0$ on $X$ is then given by $|F_{n}^0(Y_\bullet)|$. Since $F$ preserves
finite coconnective geometric realisations, it follows that this agrees with
$F_n(X)$ and $F_n$ is
indeed left Kan extended
from
$\vect^\omega_k[-n]$, as desired.

Conversely, suppose $F \in \fun( \perf_{k,  \leq 0}, \mathcal{D})$ has the
property that 
$F_n = F|_{\perf_{k, [-n, 0]}}$ is left Kan extended from
$\vect^\omega_k[-n]$ for all $n \geq 0$. 
For each $n$, it then follows 
(as in Proposition~\ref{modulesarepsigma})
that $F$ preserves geometric realisations of
simplicial objects in
$\perf_{k, [-n, 0]}$ whose realisation also belongs to $\perf_{k, [-n, 0]}$.
Since any $m$-skeletal simplicial object in $\perf_{k, \leq 0}$ is   a
simplicial object in $\perf_{k, [-n, 0]}$ for $n$ sufficiently large, we deduce that 
$F$ preserves all coconnective  finite geometric realisations. 
\end{proof}

\begin{corollary} 
\label{lKanleq02}
If $F:  \perf_k \rightarrow \mathcal{D}$ preserves finite geometric realisations, then $F$ is left Kan extended
from $\perf_{k, \leq 0}$. 
\end{corollary} 
\begin{proof} 
The statement follows from part (1) of \Cref{LKanleqzero} by ``taking the limit as $n \to
\infty$''. More precisely,  since $F|_{\perf_{k, \geq -n}}$ is left Kan extended from
$\vect^\omega_k[-n]$, it is   also left Kan extended from the larger
subcategory $ \perf_{k, [-n,
0]}$. The claim then follows from the remark below. \end{proof} 

\begin{remark} 
Let 
$\mathcal{C}= \bigcup_n \mathcal{C}^n$ be the union of an increasing chain
$\mathcal{C}^1 \subset \mathcal{C}^2 \subset \dots $ of full subcategories of $\mathcal{C}$ and suppose that $F:
\mathcal{C} \rightarrow \mathcal{D}$ has the property that $F|_{\mathcal{C}^n }$ is left
Kan extended from $\mathcal{C}_0^n :=  \mathcal{C}_0 \cap \mathcal{C}^n$ for all $n$.
Then $F$ is left Kan extended from $\mathcal{C}_0$. 
This follows from the definition of a Kan extension because for any $x \in \mathcal{C}$, we have an equivalence of  $\infty$-categories 
$(\mathcal{C}_0)_{/x} = \mycolim{n} ( \mathcal{C}_0^n )_{/x}$. 

\end{remark}

We arrive at our second generalisation of Proposition \ref{modulesarepsigma}:
\begin{proposition}\label{perfectgeneralisation} Given any $\infty$-category
$\mathcal{D}$ with sifted colimits, restriction   induces an equivalence
$$\fun_{\Sigma}(\md_k, \mathcal{D}) \xrightarrow{\ \ \simeq \ \ } \fun_\sigma(
\perf_{k, \leq 0},
\mathcal{D})$$
between $\fun_{\Sigma}(\md_k, \mathcal{D})$ and 
the full subcategory   of  $\fun(
\perf_{k, \leq 0},
\mathcal{D})$
spanned by all functors which preserve finite coconnective
geometric realisations. 
The inverse   is given by left Kan extension.
\end{proposition} 
\begin{proof} 
In view of \Cref{littlesigma}, it suffices to show that the restriction functor
$\fun_\sigma( \perf_k, \mathcal{D}) \rightarrow \fun_{\sigma}( \perf_{k, \leq 0},
\mathcal{D})$ is an equivalence whose inverse is given by taking the left Kan extension. 
By \Cref{lKanleq02}, this restriction functor is fully faithful. 

For essential surjectivity, we will check that if  $G: \perf_{k, \leq 0} \to
\mathcal{D}$  
preserves finite coconnective geometric realisations, then its left Kan
extension
$\widetilde{G}: \perf_k \rightarrow \mathcal{D}$ preserves finite \mbox{geometric realisations.} 
For this, let $\widetilde{G}_n$ denote the left Kan extension of $G|_{\perf_{k, [-n, 0]}}$ to  $\perf_k$.
The various functors $ \widetilde{G}_n$ are linked by  natural transformations $\widetilde{G}_0 \rightarrow \widetilde{G}_1 \rightarrow\widetilde{G}_2 \rightarrow \ldots $. 
By  \Cref{LKanleqzero}(2), the restriction of  $\widetilde{G}_n$ to
$\perf_{k, \geq -n}$ preserves finite geometric realisations, as it is left Kan
extended from $\vect^\omega_k[-n]$. 
Any simplicial object in $\perf_k$ which is $m$-skeletal for some $m$
belongs to $\perf_{k, \geq -n}$ for $n \gg 0$, and thus its geometric
realisation is preserved by
$\widetilde{G}_n$ for $n$ sufficiently large. 
The result then follows from the equivalence
$\widetilde{G} \simeq \mycolim{n} \widetilde{G}_n$.
\end{proof}

Proposition \ref{perfectgeneralisation}  characterises sifted-colimit-preserving
functors $F:\mod_k \rightarrow \mathcal{D}$ in terms of their restriction to
$\perf_{k, \leq 0}$. Setting $\mathcal{D} = \mod_k$, we can deduce:
\begin{corollary} 
\label{extendmonadperfect}
Let $\End_\Sigma^{\perf_{\leq 0}}( \md_k) $ be the full subcategory of $ \End_\Sigma( \md_k) $
spanned by those functors which preserve $\perf_{k, \leq 0}$. Then the
monoidal restriction functor
$$ \End_\Sigma^{\perf_{\leq 0}}( \md_k)    \rightarrow \End_\sigma (\perf_k^{\leq 0})  $$
is an equivalence. Here $\End_\sigma (\perf_{k, \leq 0}) $ denotes the
$\infty$-category of endofunctors of $\perf_{k, \leq 0}$ which preserve finite coconnective geometric realisations. 
\end{corollary} 
Corollary \ref{extendmonadperfect} allows us to extend functors $F:\perf_{k,
\leq 0}\rightarrow \perf_{k, \leq 0}$ which preserve suitable colimits to sifted-colimit-preserving endofunctors of $\mod_k$ in a monoidal fashion. 
However, the functors which we will want to extend  later will usually
\textit{not} preserve $\perf_k$. Instead, they will preserve the following larger subcategory of $\mod_k$.  
{We will now record  slight variants of the above results
extending  from $\coh_{k, \leq 0}$ instead of $\perf_{k, \leq 0}$.}

\begin{definition} 
Let $\mathcal{D}$ be an $\infty$-category with sifted colimits. 
A functor $F: \coh_{k, \leq 0} \rightarrow \mathcal{D}$ is said to \emph{preserve finite
coconnective geometric realisations} if 
for every simplicial
object \mbox{$X_\bullet \in \coh_{k, \leq 0}$} which is $m$-skeletal for some $m$ and with $|X_\bullet| \in
\coh_{k, \leq 0}$, the natural map $|F(X_\bullet)| \rightarrow F(|X_\bullet|)$ is an
equivalence. We say that $F$ is \emph{right complete}  
if for any $X \in \coh_{k, \leq 0}$, the natural map $\mycolim{n}  F(\tau_{\geq -n} X) \to
F(X)$ is an equivalence. We write $\fun'_\sigma( \coh_{k, \leq 0},
\mathcal{D})\subset \fun( \coh_{k, \leq 0},
\mathcal{D})$ \INN{060@$\fun'_\sigma( \coh_{k, \leq 0}, \mathcal{D})$}  for the full subcategory spanned by all functors  
which preserve finite coconnective geometric realisations and which are right
complete. 
\end{definition}

We can now deduce the following ``finite type variant'' of Proposition \ref{perfectgeneralisation}:
\begin{proposition} 
\label{critextendfun}
Given any $\infty$-category $\mathcal{D}$ with sifted colimits, restriction induces an equivalence
$$ \fun_{\Sigma}( \md_k, \mathcal{D}) \longrightarrow \fun'_\sigma( \coh_{k, \leq 0},
\mathcal{D})  $$
between $\fun_{\Sigma}( \md_k, \mathcal{D})$ and the full subcategory of $\fun(
\coh_{k, \leq 0},
\mathcal{D})$ spanned by all functors which preserve finite coconnective geometric realisations and are right
complete. 
\end{proposition}
\begin{proof} 
Using \Cref{perfectgeneralisation}, we observe that it suffices to check that 
the restriction functor $\fun'_\sigma( \coh_{k, \leq 0}, \mathcal{D}) \rightarrow
\fun_\sigma( \perf_{k, \leq 0}, \mathcal{D})$ is an equivalence with inverse given by left \mbox{Kan extension.}

Given $F\in \fun'_\sigma( \coh_{k, \leq 0}, \mathcal{D})$, the restriction
$F|_{\perf_{k, \leq 0}}$ preserves finite coconnective geometric realisations.
\Cref{perfectgeneralisation}  implies that the left Kan extension of
$F|_{\perf_{k, \leq 0}}$ to $\mod_k$ preserves all sifted colimits, which in turn
shows that  the left Kan extension $\widetilde{F}$ of $F|_{\perf_{k, \leq 0}}$ to
$\coh_{k, \leq 0}$ is  right complete. We deduce that $\widetilde{F} \rightarrow F$ is a transformation between right-complete functors $
\coh_{k, \leq 0} \rightarrow \mathcal{D}$  which is an equivalence on
$\perf_{k, \leq 0}$. 
This transformation  is therefore an equivalence and  $F$ is left Kan extended
from $\perf_{k, \leq 0}$.
Hence, the restriction $\fun'_\sigma( \coh_{k, \leq 0}, \mathcal{D}) \rightarrow
\fun_\sigma( \perf_{k, \leq 0}, \mathcal{D})$ is  fully faithful. 
It is also essentially surjective since the left Kan extension of any $G\in
\fun_\sigma( \perf_{k, \leq 0}, \mathcal{D})$  to $\mod_k$ preserves sifted
colimits by \Cref{perfectgeneralisation}, which implies that the left Kan
extension of $G$ to $\coh_{k, \leq 0}$ is right complete and  preserves finite geometric realisations.
\end{proof} 
Setting $\mathcal{D} =  \mod_k $ in  Proposition~\ref{critextendfun}, we can 
 deduce the following result, which will be crucial in our later
applications to extending monads.

\begin{corollary} 
\label{extendmonad}
Let $\End_\Sigma^{\coh_{\leq 0}}( \md_k) $ be the full subcategory of $\End_\Sigma(  \md_k)$
spanned by all functors which preserve $\coh_{k, \leq 0}$. Then the
monoidal restriction functor
$$ \End_\Sigma^{\coh_{\leq 0}}( \md_k) \rightarrow \End'_\sigma( \coh_{k, \leq 0})  $$
is an equivalence. 
 Here $\End_\sigma'(\coh_{k, \leq 0}) $ denotes the $\infty$-category of
 endofunctors of $\coh_{k, \leq 0}$ which preserve finite coconnective geometric realisations and are right-complete. 
\end{corollary}

\subsection{Right-left extension}\label{rlext}
We shall now apply the tools developed in  the previous
subsection and  build an array of extended functors. 

\begin{remark} 
Closely related ideas appear in the work of Illusie \cite[Sec. I-4]{Illusie}, 
 and more recently in the work of 
Kaledin \cite[Sec. 3]{Kaledin}.
\end{remark}

Throughout this subsection, we fix a stable $\infty$-category $\mathcal{D}$
admitting all limits and colimits and a field $k$. 
Our basic procedure first extends a functor on finite-dimensional $k$-vector spaces in the coconnective direction and then in the connective direction. 

More precisely, let $F: \vecto_k  \rightarrow \mathcal{D}$ be a functor. Our goal is to extend $F$ to a
sifted-colimit-preserving functor $\md_k\rightarrow \mathcal{D}$. 
In a first step, we take the right Kan extension  $F^R: \perf_{k, \leq 0} \to
\mathcal{C}$ of $F$  along the inclusion
$\vecto_k \subset \perf_{k, \leq 0}$.

\begin{remark} Using linear duality and  \Cref{modulesarepsigma}, we see that the right Kan extension $F^R$ of a functor $F$ as above can be computed as follows: given $X \in
\perf_{k,  \leq 0}$, we write $X \simeq \mathrm{Tot}(V^\bullet)$ for $V^\bullet$
a cosimplicial object of $\vecto_k$. Then  $F^R(X) \simeq \mathrm{Tot}(
F(V^\bullet))$. 
\end{remark}

In order to  further extend $F^R: \perf_{k, \leq 0} \rightarrow \mathcal{D}$ to a sifted-colimit-preserving functor $\mod_k \rightarrow \mathcal{D}$ as in  \Cref{perfectgeneralisation}, we need to assume the following condition: 
\begin{definition} \label{rightext} 
A functor $F: \vecto_k \rightarrow \mathcal{D}$ is said to be
\emph{right-extendable}
if the right Kan extension $F^R: \perf_{k, \leq 0} \rightarrow \mathcal{D}$ commutes
with finite coconnective geometric realisations (cf.\ \Cref{cfgr}). 
\end{definition} 
\begin{cons}[Right-left extension]
\label{rightleft}
The  \emph{right-left extension}  
$F^{RL}: \mod_k \rightarrow \mathcal{D}$ of a right-extendable functor\INN{060@$F^{RL}$}
$F: \vecto_k \rightarrow \mathcal{D}$ is given by the left Kan extension of \mbox{$F^R:
\perf_{k, \leq 0}\rightarrow \mathcal{D}$ to $\mod_k$.}\end{cons}
\begin{remark}\label{RLremark} By  \Cref{perfectgeneralisation}, the  right-left Kan extension $F^{RL}$ of any right-extendable functor $F$ preserves sifted colimits. Hence, $F^{RL}$ restricts on $\md_{k, \geq 0}$ to the left Kan
extension $\mathbb{L} F$. 
\end{remark}
Let $\vect_k \subset \md_k$ \INN{220@$\vect_k$}  be the   subcategory of discrete $k$-module spectra, i.e.   ordinary $k$-vector spaces.
\begin{proposition}\label{RKEisright}
Let $F: \vect_k \rightarrow \md_{k,\leq 0}$ be a filtered-colimit-preserving\vspace{2pt} functor. Suppose that 
the composite 	$\vect_k^\omega  \rightarrow \vect_k \xrightarrow{F} \mod_{k,\leq 0} \hookrightarrow \mod_k$ 	is right-extendable with 
right-left extension $\widetilde{F}$ (cf.\ \Cref{rightleft}).
 If $M^\bullet$ is a cosimplicial $k$-vector space, then $\widetilde{F}$ is determined by the formula
$$\widetilde{F}(\Tot(M^\bullet)) \simeq \Tot({F}(M^\bullet)).$$
\end{proposition}
\begin{proof}
Since $F$ preserves filtered colimits, we obtain a natural equivalence $\widetilde{F}|_{\vect_k}\simeq F$. 
Given a cosimplicial $k$-vector space $M^\bullet$ as above, we can write  $$M^\bullet \simeq  \mycolim{i\in I} (M_i^\bullet),$$ where each $M_i^n$ belongs to $\vect_k^\omega$ and every  $M_i^\bullet$ is a  finite cosimplicial diagram for all $i$.

Since each $M_i^n$ is coconnective, every truncation $\tau_{\geq m}(\Tot(M^\bullet))$ is only affected by the $m^{th}$ coskeleton of the appearing totalisations. We can therefore commute the filtered colimit past the totalisation to obtain \vspace{-4pt} equivalences
$\displaystyle \Tot(M^\bullet)\simeq  \Tot(\mycolim{i\in I} (M_i^\bullet)) \simeq \mycolim{i\in I} (\Tot(M_i^\bullet)) $. The same argument applies to the diagram $F (M_i^\bullet)$.

Combining this observation with \Cref{RLremark} and the defining property of $\widetilde{F}$, we obtain  equivalences  $\widetilde{F}(\Tot(M^\bullet)) \simeq \mycolim{i\in I} \left(\Tot(F (M_i^\bullet))\right)
 \simeq \Tot(\mycolim{i\in I} \left(F (M_i^\bullet))\right)$.
Since $F$ preserves filtered colimits, this is then equivalent to $\Tot( F(\mycolim{i\in I} M_i^\bullet)) \simeq\Tot( F(M^\bullet))$, as desired.
\end{proof}

Let $\fun_{RL}( \vecto_k, \mathcal{D}) \subset \fun( \vecto_k, \mathcal{D})$
denote the full subcategory of right-extendable functors. 
Right-left extension establishes a fully faithful embedding 
$\fun_{RL}( \vecto_k, \mathcal{D})  \subset \fun_{\Sigma}( \md_k,
\mathcal{D})$ whose 
image consists of all functors $F $ whose restriction to $F|_{\perf_{k, \leq 0}}$
commutes with totalisations which are $m$-coskeletal for some $m$ (equivalently,
$F|_{\perf_{k, \leq 0}}$
is right Kan \mbox{extended from $\vect_k^\omega$).}

\Cref{rightleft} will be our basic tool for building sifted-colimit-preserving functors on $\md_k$. 
In order to proceed, we will need a criterion for right-extendability. 
We begin by recalling the following definition, originally due to Eilenberg--MacLane \cite{eilenberg1954groups}:
\begin{definition}[Functors of finite degree] \label{spf} A functor $F: \vect_k^\omega \rightarrow \mathcal{D}$ is said to be  
\begin{enumerate} 
\item   \emph{of degree $0$} if $F$ is constant. 
\item   \emph{of degree $n$} with $n\geq 1$ if for any $X \in \vect_k^\omega$,
the difference functor $D_X F : \vect_k^\omega \rightarrow \mathcal{D}$ defined via $D_X
F(Y) = \mathrm{fib}( F(X \oplus Y) \rightarrow F(Y))$ is of degree $n-1$. 
\item   \emph{of finite degree} if it is  of degree $n$ for some  $n\geq 0$. 
\end{enumerate}
\end{definition} 

\begin{example} 
The  $n^{th}$ symmetric, exterior, and divided power functors $\sym^n, \bigwedge^n, \Gamma^n $ on $ \vecto_k$ are all 
of degree $n$. 
\end{example}

We can now state the main result of this section:
\begin{theorem} 
\label{polyrightext}
Let $F: \vect_k^\omega \rightarrow \mathcal{D}$ be a functor of finite degree. Then $F$ is
right-extendable. 
In particular, we obtain a canonical sifted-colimit-preserving extension $F: \md_k \rightarrow \mathcal{D}$. 
\end{theorem}  
Our proof of  the above 
result will rely on Goodwillie's  calculus of functors. 
We recall several of the key definitions from Goodwillie's work (cf.\ \cite{goodwillie1991calculus} \cite{goodwillie2003calculus}) in their $\infty$-categorical incarnation, which is described  in detail in  \cite[Chapter 6]{lurie2014higher}.
 
For the rest of this section, we fix a pointed $\infty$-category $\mathcal{A}$  with finite colimits and a stable $\infty$-category $\mathcal{D}$ 
with small colimits.  
\begin{definition}[$n$-excisive functors] \label{excisive}\ 
\begin{enumerate}
\item  
An \emph{$(n+1)$-cube} in $\mathcal{A}$ is a functor $\mathcal{P}( \left\{0,   \dots,
n\right\}) \rightarrow \mathcal{A}$, where 
$\mathcal{P}( \left\{0, \dots,
n\right\})$ denotes the poset of finite subsets of
$\left\{0,   \dots, n\right\}$. Such a cube is 
\begin{itemize}
\item  \emph{strongly coCartesian}  
if it is left Kan extended from subsets of cardinality at most $1$.
\item  \emph{coCartesian} if it is a colimit diagram, i.e.\ its value on
$\left\{0,   \dots, n\right\}$ is determined by its values on proper subsets.
\end{itemize}
\item
A functor   $F: \mathcal{A} \to
\mathcal{D}$ is \emph{$n$-excisive}
if 
$F$ carries strongly coCartesian $(n+1)$-cubes to coCartesian cubes (recall that $\mathcal{D}$ is assumed to be stable). 
\end{enumerate}
Let $\Exc^{n}(\mathcal{A}, \mathcal{D})$ denote the full subcategory of $\fun(\mathcal{A}, \mathcal{D})$ spanned by all \INN{050@$\Exc^{n}(\mathcal{A}, \mathcal{D})$}
$n$-excisive functors $\mathcal{A} \rightarrow \mathcal{D}$.  
\end{definition}

\begin{example} 
Let $G: \mathcal{A}^n \rightarrow \mathcal{D}$ be a functor which preserves finite
colimits in each variable. Then the diagonal functor $F$ defined by $F(X) = G(X,
\dots, X)$ is $n$-excisive.
\end{example}

\begin{remark} 
The full subcategory $\Exc^{n}(\mathcal{A},\mathcal{D}) \subset
\fun(\mathcal{A}, \mathcal{D})$ is closed under arbitrary limits and colimits:
that is, limits and colimits of 
$n$-excisive functors are $n$-excisive. Here and in the preceding Example, we have used that $\mathcal{D}$ is stable.
\end{remark}

\begin{remark} 
Suppose $\mathcal{A}$ is small.  
Given an
$n$-excisive functor $F: \mathcal{A} \rightarrow \mathcal{D}$, we can canonically extend $F$ to
a filtered-colimit-preserving  
 functor
$\mathrm{Ind}(\mathcal{A}) \rightarrow \mathcal{D}$, where $\mathrm{Ind}(\mathcal{A})$ is the $\mathrm{Ind}$-completion of $\mathcal{A}$.  \INN{090@$\mathrm{Ind}(\mathcal{A})$} It is 
not hard to see that this functor is also $n$-excisive (cf.\ \cite[Proposition 6.1.5.4]{lurie2014higher}).
In fact, restriction and left Kan extension establish an equivalence $\Exc^n(\mathcal{A},\mathcal{D}) \simeq \Exc^n_c(\Ind(\mathcal{A}),\mathcal{D}) $ between $n$-excisive functors $\mathcal{A}\rightarrow\mathcal{D}$ and $n$-excisive filtered-colimit-preserving functors $\Ind(\mathcal{A})\rightarrow \mathcal{D}$.
\end{remark}

A theorem of Goodwillie allows us to universally approximate functors by $n$-excisive  functors:

\begin{proposition}[The $n$-excisive approximation]\label{napp}
The inclusion $\Exc^{n}(\mathcal{A}, \mathcal{D})
\subset \fun(\mathcal{A}, \mathcal{D})$ admits a left adjoint 
$P_n: \fun(\mathcal{A}, \mathcal{D}) \rightarrow 
\Exc^{n}(\mathcal{A}, \mathcal{D})$.\end{proposition}

Goodwillie has in fact given a more explicit description of the functor $P_n(F)$ as a sequential colimit $ P_n(F) = \mycolim{n} \left(F \rightarrow T_n(F) \rightarrow T_n(T_n(F)) \rightarrow \ldots \right)$. Here $G\mapsto  T_n(G)$ is a certain construction on functors $G: \mathcal{A} \rightarrow \mathcal{D}$ with the property that the value
 $T_n(G)(X)$  is obtained
as a finite limit of copies of $G$ evaluated on various direct sums of
suspensions of $X$. 
We will not need to know the precise formula for $T_n$, except that it has the following implications:
\begin{proposition}
 If $F$ preserves filtered colimits, then so does $P_nF$. 
\end{proposition}
\begin{proposition} \label{resandPncommute}   Given a right exact functor $ \mathcal{A} \rightarrow  \mathcal{B}$ between two pointed $\infty$-categories with finite colimits,  the following diagram commutes
$$ \xymatrix{
\fun(\mathcal{B}, \mathcal{D}) \ar[d] \ar[r]^{P_n\ } & \Exc^{n}(\mathcal{B},
\mathcal{D}) \ar[d] \\
\fun(\mathcal{A}, \mathcal{D}) \ar[r]^{P_n\ } &  \Exc^{n}(\mathcal{A}, \mathcal{D})
}.$$

Here the vertical maps are simply restrictions. 
\end{proposition}
\begin{proposition}[{Johnson-McCarthy 
\cite[Proposition 5.10]{johnson1999taylor}}]
\label{JMcC}
If $F:  \vecto_k \rightarrow \mathcal{D}$ is   of  degree $n$, then its nonabelian derived functor $\mathbb{L} F:
\md_k^{\geq 0}
\rightarrow \mathcal{D}$ is
$n$-excisive.  
\end{proposition} 
\begin{proof} 
Given a collection of maps
$Y \rightarrow X_i$ for $0 \leq i \leq n$,   we can form a 
strongly coCartesian cube $c:\mathcal{P}(\{0,\ldots,n\}) \rightarrow \md_k^{\geq 0}$ by left Kan extension. It suffices to show that the functor $LF$ carries every such cube $c$ to 
a coCartesian cube $\mathbb{L}F\circ c$. 

We first assume that each map $Y \rightarrow X_i \simeq Y \oplus Z_i$ is a (split) injection
between  discrete finite-dimensional $k$-vector spaces. In this case, the fact that $F$ is  of degree $n$ immediately implies that $F\circ c$ is a coCartesian cube.
 But we can write any collection 
$\mathfrak{C} = \left\{Y, X_i, Y \rightarrow X_i\right\}_{0 \leq i \leq n}$ as a 
sifted colimit of  collections $\mathfrak{C}'$
with all $Y', X_i'\in \vect_k^\omega$  and  each $Y' \rightarrow X_i'$ (split)
injective. 
To prove this, we first note that  the evaluation functor $\fun(\{0,\ldots, n\}, (\mod_{k,\geq 0})_{Y/}) \rightarrow \fun(\{0,\ldots, n\}, \mod_{k,\geq 0})$ admits a left adjoint, sending   $\{Z_i\}_{0 \leq i \leq n}$ to the collection $\{Y \rightarrow Y \oplus Z_i\}_{0 \leq i \leq n}$.  Here $\{0,\ldots, n\}$ denotes the category with objects $0,\ldots , n$ and only identity morphisms.

The evaluation functor  preserves geometric realisations and is therefore monadic, which shows that we can write  $\mathfrak{C} $ as a realisation of a simplicial diagram   whose terms are collections of the form $\{Y \rightarrow Y \oplus Z_i\}_{0 \leq i \leq n}$.
We now write each of these collections as a sifted colimit of 
 collections $\mathfrak{C}'$, \mbox{using that $\mod_{k,\geq 0}$ is the sifted cocompletion of $\vect_k^{\omega}$.}
 
Finally, as $\mathbb{ L}F$ preserves 
sifted colimits, the  assertion for each
$\mathfrak{C}'$ implies the \mbox{result for $\mathfrak{C}$.}
\end{proof}

\begin{theorem} \label{extensionofexcisives}
Let $\mathcal{A}$ be  small. Suppose that $\mathcal{A}_0
$ is  a full subcategory of $\mathcal{A}$ which is  closed under finite colimits such that for any $X \in \mathcal{A}$, we have $\Sigma^m X \in \mathcal{A}_0$ for
$m\gg 0$ sufficiently large. The  restriction functor induces an equivalence $$\Exc^{n}( \mathcal{A}, \mathcal{D}) \xrightarrow{\ \ \simeq\ \ }
\Exc^{n}(\mathcal{A}_{0}, \mathcal{D})$$
whose inverse is given by $F\mapsto P_n(\Lan_{\mathcal{A}_0}^{\mathcal{A}}(F))$.
\end{theorem} 
\begin{proof} 
The composite functor $\Exc^n(\mathcal{A}, \mathcal{D}) \rightarrow \fun(\mathcal{A},\mathcal{D}) \rightarrow \fun(\mathcal{A}_0,\mathcal{D})$ admits a left adjoint $P_n \circ \Lan_{\mathcal{A}_0}^\mathcal{A}$ by \cite[Proposition 4.3.2.17]{lurie2009higher}, \cite[Lemma 4.3.2.13]{lurie2009higher}, and
\Cref{napp}. Given that  $\Exc^n(\mathcal{A}_0,\mathcal{D}) \hookrightarrow \fun(\mathcal{A}_0,\mathcal{D}) $ is fully faithful, this implies that  $\Exc^n(\mathcal{A},\mathcal{D}) \rightarrow \Exc^n(\mathcal{A}_0,\mathcal{D})$ also admits a left adjoint given by $F \mapsto P_n(\Lan_{\mathcal{A}_0}^\mathcal{A} F)$.

This left adjoint is fully faithful. Indeed, given an  $n$-excisive functor  $F_0 : \mathcal{A}_0 \rightarrow \mathcal{D}$, we first observe that 
$\Lan_{\mathcal{A}_0}^{\mathcal{A}}(F_0)|_{\mathcal{A}_0} \simeq  F_0$ since $\mathcal{A}_0
\subset \mathcal{A}$ is a full subcategory. By \Cref{resandPncommute}, 
Goodwillie's explicit construction of the 
$n$-excisive approximation allows us to recover $F_0$ from $P_n(\Lan_{\mathcal{A}_0}^{\mathcal{A}}(F_0))$:
$$P_n(\Lan_{\mathcal{A}_0}^{\mathcal{A}}(F_0))|_{\mathcal{A}_0} \simeq P_n(\Lan_{\mathcal{A}_0}^{\mathcal{A}}(F_0)|_{\mathcal{A}_0}) \simeq  P_n(F_0) \simeq  F_0.$$ 
 
To conclude the proof, it  suffices to show that the right adjoint $\Exc^{n}( \mathcal{A}, \mathcal{D}) \rightarrow
\Exc^{n}(\mathcal{A}_{0}, \mathcal{D})$ is conservative. 
So let $F \rightarrow G$ be a natural transformation between functors $F, G \in \Exc^{n}( \mathcal{A}, \mathcal{D})$ which becomes an equivalence upon restriction to $\mathcal{A}_{0}$.  Given $X\in \mathcal{A}$ and $r\geq 0$, we let $S_{r,X}$ be the statement
that $F(\Sigma^r
X^{\oplus m}) \rightarrow G(\Sigma^r
X^{\oplus m})$ is an equivalence  for all $m \geq 0$.

Since $F$ is $n$-excisive, statement $S_{r,X}$ implies statement $S_{r-1,X}$ as we can use the strongly
coCartesian $(n+1)$-cubes obtained from the maps $\left\{X \rightarrow 0\right\}$ and $\left\{Y \rightarrow 0\right\}$ to recover $F(X)\rightarrow G(X)$ 
 from the maps
$F(0) \rightarrow G(0)$, 
$F(\Sigma X) \rightarrow  G(\Sigma X)$,
$F(\Sigma X \oplus \Sigma X) \rightarrow G(\Sigma X \oplus \Sigma X). \ldots $.
By assumption, we know that $S_{r,X}$ holds true for $r \gg 0$.  A descending induction shows  that  $S_0$ is true, i.e.\  that $F(X) \rightarrow G(X)$ is an equivalence.
\end{proof} 

\begin{remark}
	In a previous version of this paper, we had stated \Cref{extensionofexcisives} only for stable $\mathcal{A}$. We thank Greg Arone for pointing out that the statement also holds for unstable $\mathcal{A}$ (with the same proof).
\end{remark}

\begin{proposition} \label{finitereals}
Let $F: \mathcal{A} \rightarrow \mathcal{D}$ be a functor between cocomplete stable
$\infty$-categories which is $n$-excisive and preserves filtered colimits. 
Then $F$ preserves totalisations which are $m$-skeletal for some $m$ and all 
geometric realisations. 
\end{proposition} 
\begin{proof} 
For an $n$-excisive functor  $F$, the image of an $m$-skeletal cosimplicial
diagram in $\mathcal{A}$ is an $nm$-skeletal cosimplicial diagram in
$\mathcal{D}$, cf.~\cite[Prop.~2.10]{BGMN}. 
Now limits over $\Delta_{\leq nm}$ commute with filtered colimits in
$\mathcal{D}$. It follows that the collection of $n$-excisive functors $F$ that
satisfy the conclusion of the theorem is closed under cofibre sequences and
filtered colimits. 

We use the fibre sequence $D_n(F)\rightarrow P_n(F)\rightarrow P_{n-1}(F).$ By induction, it suffices to prove the claim for the $n$-homogeneous functor $D_n(F)$.
We can find a symmetric functor
$G=\cross_n(F): \mathcal{A}^n \rightarrow \mathcal{D}$ which preserves colimits in each variable such that $D_n(F) \simeq G(X,\ldots,X)_{h\Sigma_n}$ (cf.\ \cite[Proposition 6.1.4.14., Corollary 6.1.4.15]{lurie2014higher}).
As  $(-)_{h\Sigma_n} : \fun(B\Sigma_n, \mathcal{D}) \rightarrow \mathcal{D}$
is  exact  and limits and colimits in functor categories are computed pointwise,
it suffices (by the previous paragraph) to check that the functor $\mathcal{A} \rightarrow \mathcal{D}$,  $X\mapsto G(X,\ldots,X)$  preserves  geometric \mbox{realisations and finite totalisations.}

For realisations, this follows immediately as $\Delta^{op}\rightarrow ( \Delta^{op})^n$ is left cofinal by \cite[Lemma 5.5.8.4]{lurie2009higher}.
If a cosimplicial object $X^\bullet : \Delta \rightarrow \mathcal{A}$   is right Kan extended from $\Delta _{\leq m}$, we note that 
$(X^\bullet, \ldots, X^\bullet) : \Delta^n \rightarrow \mathcal{A}^n$ is right Kan extended from $(\Delta _{\leq m})^n$. As  this is a finite limit condition and $G$ is exact in each variable, the multisimplicial object
$G(X^\bullet, \ldots, X^\bullet)$ is also right Kan extended from $(\Delta _{\leq m})^n$. Hence $G(\Tot(X^\bullet), \ldots, \Tot(X^\bullet)) \simeq \lim_{\Delta^n}G (X^\bullet, \ldots, X^\bullet)) \simeq \Tot \left(G (X^\bullet, \ldots, X^\bullet)\right)$.
For the final equivalence,  we have used  that the diagonal   \mbox{$\Delta  \rightarrow \Delta^n$ is right cofinal.}
\end{proof} 

\begin{proof}[Proof of  \Cref{polyrightext}]
If $F:\vect_k^\omega \rightarrow \mathcal{D}$ is of degree $n$, we know by
\Cref{JMcC} that the sifted-colimit-preserving left Kan extension $\L F :
\mod_k^{\geq 0} \rightarrow \mathcal{D}$ is $n$-excisive. 

The functor $\widetilde{F} := P_n(\Lan_{\mod_k^{\geq
0}}^{\mod_k} \L F) \simeq P_n( \L F \circ \tau_{\geq 0})$ is evidently $n$-excisive, and it preserves filtered colimits as the $t$-structure on $\mod_k$ is compatible with filtered colimits (cf.\ \cite[Proposition 1.3.5.21]{lurie2014higher}). Combining this observation with \Cref{resandPncommute} and \Cref{extensionofexcisives}, we can deduce  that the restriction of $\widetilde{F}$ to $\mod_{k,\geq 0}$ agrees with $\L F$.

\Cref{finitereals} then implies that $\widetilde{F}$ preserves 
 geometric realisations and finite totalisations. Hence, $\widetilde{F}|_{\perf_k^{\leq 0}}$ is right Kan extended from $F$ and preserves finite coconnective \mbox{realisations.} 
\end{proof}

\begin{example} 
As a consequence, we obtain $n$-excisive functors 
$$ \mathbb{L}\sym^n, \  \L \bigwedge^n,  \ \L \Gamma^n: \md_k \longrightarrow \md_k.  $$
extending the ordinary symmetric, exterior, and divided  power functors from Example \ref{tensorpower}
\end{example}

We have the following duality phenomenon:

\begin{proposition}[Duality]\label{duality}
Let $F:\vect_k^\omega\longrightarrow \mod_k$ be a functor of finite degree. Then the functor $F^\vee : \vect_k^\omega\longrightarrow \mod_k $ given by $M \mapsto (F(M)^\vee)^\vee$ is right-extendable and its extension $\widetilde{F^\vee}$ satisfies
$\widetilde{F^\vee}(M) \simeq (\widetilde{F}(M^\vee))^\vee $
for all $M\in \perf_k$.
\end{proposition}
\begin{proof} Since the functor  $(-)^\vee: \mod_k \rightarrow \mod_k^{op}$ preserves colimits, it is exact.
As $\widetilde{F}$ preserves geometric realisations and finite totalisations, the functor $G=  (\widetilde{F}(M^\vee))^\vee $ preserves finite realisations and finite totalisations.
This in turn implies that $G|_{\perf_k^{\leq 0}}$ is right Kan extended from $\vect_k^{\omega}$ and that $G|_{\perf_k^{\leq 0}}$ preserves finite coconnective geometric realisations. Hence, the functor $F^\vee$ is right-extendable.
The second claim follows since $\widetilde{F^\vee}$ and $G$ agree on $\perf_k^{\leq 0}$ and preserve finite geometric realisations.
\end{proof}

\begin{example} 
Since the $n^{th}$ symmetric power functor $M\mapsto (M^{\otimes})_{\Sigma_n}$ is dual to the $n^{th}$ divided power functor $M\mapsto (M^{\otimes})^{\Sigma_n}$ for $M\in \vect_k^{\omega}$, we conclude from the above  proposition that the extended functors satisfy $\Gamma^n(M) \simeq (\sym^n(M^\vee))^\vee $
for all $M\in \perf_k$.
\end{example}

\subsection{Extended functors and Bredon homology}\label{bredon}
We shall now attach $n$-excisive functors $\mod_k\rightarrow \mod_k$ to genuine $\Sigma_n$-spaces and  construct a spectral sequence to compute their values.

Given a field $k$ and a finite pointed $\Sigma_n$-set $(X,x)$, we write $k[X]$ \INN{110@$k[X]$}
for the quotient of the free $k$-vector space on $X$  by the relation $x\simeq 0$. 
We define  functors  \INN{060@$F_X,  F_X^h$}
$F_X, F_X^h: \vect_k^\omega \rightarrow  \mod_k$  by setting
$$ F_X(M) :=   (k[X] \myotimes{} M^{\otimes n})_{\Sigma_n} \mbox{ \ \ \ \ and\ \ \ \  }  F_X^h(M) :=   (k[X] \myotimes{} M^{\otimes n})_{h\Sigma_n}.\vspace{-3pt}$$

\begin{proposition} 
The functors $F_X$  and $F_X^h$ \vspace{-3pt}  are of degree $n$ in the sense of \Cref{spf}.
\end{proposition}
\begin{proof}
To show  that $F_X$ has  degree $n$, we prove the more general claim that  for any subgroup $H\subset \Sigma_m \times \Sigma_n$ and  $\Sigma_m$-vector space $V$ in $\vect_k^\omega$, 
the functor  $G$ given by 
$Y \mapsto (V \otimes Y^{\otimes n})_H $ is of degree $n$. 
This proves the  claim for $F_X$ by taking $m=n$,  $V = k[X]$, and $H = \Sigma_n$ \mbox{diagonally embedded.}

The general claim is evident for $n=0$. For $n>0$,
the binomial formula shows that $D_XG(Y) = \mathrm{fib}( G(X \oplus Y) \rightarrow G(Y))$ sends $Y \in \vect^\omega_k$ to $Y \mapsto \bigoplus_{j=0}^{n-1} \left(V \otimes \Ind_{\Sigma_{n-j} \times \Sigma_j}^{\Sigma_n}(X^{\otimes (n-j)} \otimes Y^{\otimes j})\right)_H.$
Using the projection formula, we deduce that  the functor $D_XG(Y)$ is in fact given by  a sum of functors   $Y \mapsto (V \otimes X^{\otimes (n-j)} \otimes Y^{\otimes j})_{H'}$ for $H'$ a subgroup of $(\Sigma_m \times \Sigma_{n-j}) \times \Sigma_j$ with $j\leq n-1$.  The claim follows by induction. A similar argument shows that $F_X^h$ is of degree $n$.
\end{proof}

By \Cref{polyrightext},   $F_X$ and $F_X^h$ admit  canonical $n$-excisive sifted-colimit-preserving extensions  $\mod_k\rightarrow \mod_k$, which we  denote by the same names.
Extending the assignments $(\Set^{\Fin}_\ast)^{\Sigma_n}\rightarrow \End^n_\Sigma(\mod_k)$ given by $X \mapsto F_X$,  $X \mapsto F_X^h$  in a sifted-colimit-preserving way, we obtain   functors 
$$ F_{(-)}, F_{(-)}^h\ : \ \mathcal{S}_\ast^{\Sigma_n} \simeq \mathcal{P}_{\Sigma}((\Set^{\Fin}_\ast)^{\Sigma_n}) \ \longrightarrow\ \End^n_\Sigma(\mod_k) $$ \INN{060@$F_X,  F_X^h$}
from genuine $\Sigma_n$-spaces to  $n$-excisive sifted-colimit-preserving endofunctors of $\mod_k$.

We will now describe a method for computing the value of $F_X$ and  $F_X^h$   {on a given $k$-module spectrum $M$.} First, we recall the following notion:
\begin{definition}\INN{080@$\widetilde{H}^{\Br}_\ast(X,\mu) $}\INN{030@$\widetilde{C}_\ast(-,\mu) $}  
Given an additive functor $\mu:(\Set^{\Fin}_\ast)^G \rightarrow \mod_k$ from  finite  pointed $G$-sets to $k$-module spectra, we define the \textit{(reduced) Bredon chains} $\widetilde{C}_\ast(-,\mu): \mathcal{S}_\ast^G  \rightarrow \mod_k$ as the left Kan extension of $\mu$ to the $\infty$-category of pointed  genuine $G$-spaces $\mathcal{S}_\ast^G  \simeq \mathcal{P}_{\Sigma}((\Set^{\Fin}_\ast)^G)$ . The \textit{(reduced) Bredon homology groups} of $X\in \mathcal{S}_\ast^G $ with coefficients in $\mu$ are given by 
$\widetilde{H}^{\Br}_\ast(X,\mu) : = \pi_\ast(\widetilde{C}_\ast(X,\mu))$.
\end{definition}
If  $X\in \mathcal{S}_\ast^{\Sigma_n} $ is the geometric realisation of a simplicial diagram $X_\bullet$ taking values in   finite pointed ${\Sigma_n}$-sets, 
  we have $F_X(M)\simeq |F_{X_\bullet}(M)|$ and $F^h_X(M)\simeq |F^h_{X_\bullet}(M)|$.
Using the skeletal filtration of the simplicial $k$-module spectra $F_{X_\bullet}(M)$ and   $F^h_{X_\bullet}(M)$  gives convergent half-plane spectral sequences 
$$E^2_{s,t} = \pi_s\left(\pi_t(F_{X_\bullet}(M))\right)  \ \   \ \ \Rightarrow \ \   \ \ \pi_{s+t}(F_X(M)) $$
$$E^{2,h}_{s,t} = \pi_s\left(\pi_t(F^h_{X_\bullet}(M))\right)  \ \   \ \ \Rightarrow \ \   \ \ \pi_{s+t}(F^h_X(M)) .$$
Thus, we have $E^2_{s,t}  = \widetilde{H}^{\Br}_s(X, \mu_t^M)$  and $E^{2,h}_{s,t}  = \widetilde{H}^{\Br}_s(X, \mu_t^{M,h})$, where $\mu_t^M, \mu_t^{M,h}: (\Set^{\Fin}_\ast)^{\Sigma_n}  \rightarrow \mod_k$ 
are given by $X\mapsto \pi_t(F_X(M))$ and $X\mapsto \pi_t(F^h_X(M))$, respectively.

The functors $\mu_t^M$ and $\mu_t^{M,h}$  are particularly computable when $M$ is perfect and coconnective. In this case, we can write $M=\Tot(M^\bullet)$ as a finite totalisation of a cosimplicial finite-dimensional $k$-vector space $M^\bullet$. For $X \in (\Set^{\Fin}_\ast)^{\Sigma_n}$ a finite pointed $\Sigma_n$-set, the functor  $(F_X)|_{\perf_k^{\leq 0}}$ is right Kan extended from $\vect_k^\omega$, which in turn implies 
that $F_X(M) \simeq \Tot(F_X(M^\bullet)) = \Tot(k[X] \myotimes{} (M^\bullet)^{\otimes n})_{\Sigma_n}$ and $F_X^h(M) \simeq \Tot(F_X^h(M^\bullet)) = \Tot(k[X] \myotimes{} (M^\bullet)^{\otimes n})_{h\Sigma_n}$
(by the dual of  \Cref{modulesarepsigma}). Dual remarks apply for $M$ connective and of finite type.
\INN{130@$\mu_\ast^M, \mu_\ast^{M,h}$}
In both cases, we can use the standard wrong-way maps to upgrade  $\mu_t^M$ and $\mu_t^{M,h}$ to   \textit{Mackey functors}. We recall that a Mackey functor consists of a pair of functors
$$(\mu :(\Set^{\Fin})^G \rightarrow \vect_k \  ,\  \mu^\natural: (\Set^{\Fin})^G \rightarrow \vect_k^{op})$$  from finite $\Sigma_n$-sets to $k$-vector spaces which agree on objects and such that whenever 
the left square below is a pullback of finite pointed $G$-sets, then the right hand square commutes:
\[
\begin{tikzcd}
	A \arrow[r, "f"] \arrow[d, "h"'] & B \arrow[d, "g"] \\
	C \arrow[r, "k"'] & D
\end{tikzcd}
\qquad\qquad
\begin{tikzcd}
	\mu(A) \arrow[r, "\mu(f)"] & \mu(B) \\
	\mu(C) \arrow[u, "\mu^\natural(h)"] \arrow[r, "\mu(k)"'] & \mu(D) \arrow[u, "\mu^\natural(g)"']
\end{tikzcd}
\]
We will later use the above spectral sequence to compute the homotopy of free partition Lie algebras.\\

\subsection{Admissible functors}
The main purpose of this subsection is to isolate a class of functors which preserve certain
totalisations in $\coh_{k, \geq 0}$. 
This will play a technical role at various stages in the axiomatic argument in
the following section. 

\newcommand{\grp}{\mathrm{Gr}^{\mathrm{pft}}}
We first need the following elementary and classical observation asserting that
endofunctors of $\md_{k, \geq 0}$ which preserve sifted colimits naturally commute with limits
of Postnikov towers. 
\begin{proposition} 
\label{truncbehaveswell}
Let $F: \md_{k,\geq 0} \rightarrow \md_{k,\geq 0}$ be a functor which preserves sifted
colimits. 
Suppose $V \rightarrow V'$ is a map in $\md_{k,\geq 0}$ such that $\tau_{\leq n} V
\simeq \tau_{\leq n} V'$. Then $\tau_{\leq n} F(V) \simeq \tau_{\leq n} F(V')$. 
\end{proposition} 
\begin{proof} 
Since the functor  $\Omega^\infty: \md_{k,\geq 0} \rightarrow \mathcal{S}$ preserves products and sifted
colimits, the functor $\Omega^\infty \circ F: \md_{k,\geq 0} \rightarrow \mathcal{S}$ has the same properties.
Hence $\widetilde{F}:= \Omega^\infty \circ F$ is left Kan extended from a  
product-preserving functor $\vect_k^\omega \rightarrow \mathcal{S}$, since
$\md_{k,\geq 0} =\mathcal{P}_\Sigma( \vect_k^\omega)$. 
The general theory of $\mathcal{P}_\Sigma$ shows now that any
such functor can be written as a sifted colimit of functors of the form 
$$ h_{V_0}(-) = \hom_{\md_k}(V_0, -) : \md_{k,\geq 0} \rightarrow \mathcal{S},  $$
for $V_0 \in \vect_k^\omega$. 
It is clear that $h_{V_0}$ has the desired property:   if 
$V \rightarrow V'$ is a map in $\md_{k,\geq 0}$ which  induces an equivalence on
$\tau_{\leq n}$, then $\tau_{\leq n} h_{V_0}(V) \rightarrow \tau_{\leq n} h_{V_0}(V')$
is an equivalence. 
The result follows
as the collection of all functors 
$\md_{k,\geq 0} \rightarrow \mathcal{S}$ 
satisfying this property is closed under colimits. 
\end{proof}

\begin{definition}[Admissible functors]\label{admissiblefunctors}  
Let $F: \md_{k,\geq 0} \rightarrow \md_{k,\geq 0}$ be a functor which preserves sifted
colimits. 
We will say that $F$ is \emph{admissible} if the
following hold: 
\begin{enumerate}
\item 
The functor $F$ preserves the full subcategory $\coh_{k, \geq 0} \subset
\md_{k,\geq 0}$  of finite type $k$-modules. 
\item 
If $X^\bullet$ is a cosimplicial object of $\coh_{k, \geq 0}$ such that 
 $\mathrm{Tot}(X^\bullet)$ (computed in $ \md_k$) belongs to 
$\coh_{k, \geq 0}$, then 
$F( \mathrm{Tot}(X^\bullet)) \longrightarrow \mathrm{Tot}( F(X^\bullet))$ \vspace{2pt} is an equivalence, where $\mathrm{Tot}( F(X^\bullet))$ is also computed in $\mod_k$.
In particular, this means that $\mathrm{Tot}( F(X^\bullet))$ is connective.
\end{enumerate}
\end{definition} 

\begin{proposition} 
Let $F: \md_{k,\geq 0} \rightarrow \md_{k,\geq 0}$ be a functor which preserves sifted
colimits and preserves the full subcategory $\coh_{k, \geq 0}$. 
Then we can define a functor  \INN{060@$F^\vee$}
$F^\vee: \perf_{k, \leq 0} \rightarrow \coh_{k, \leq 0}$ by the formula  $V \mapsto
F(V^{\vee})^{\vee}$. 
The functor $F$ is admissible if and only if $F^\vee$ preserves finite
coconnective geometric realisations (\Cref{cfgr}). 
\label{critadmissible}
\end{proposition}
\begin{proof} 
If $F^\vee$ preserves finite coconnective geometric realisations, then
$F^\vee$ extends uniquely to a sifted-colimit-preserving functor $\widetilde{F^\vee}: \md_k \to
\md_k$.  For $V \in \perf_{k, \leq 0}$, we have a natural equivalence
$\widetilde{F^\vee}(V) \simeq F(V^{\vee})^{\vee}$. 

This formula in fact holds for all $V \in \coh_{k, \leq 0}$. Indeed,
\Cref{truncbehaveswell} implies that for all $i$, the sequence
$\pi_i F((\tau_{\geq 0} V)^\vee) \leftarrow \pi_i F((\tau_{\geq -1} V)^\vee) \leftarrow \ldots $ stabilises to $\pi_i F(V^\vee)$. The canonical map $$  \widetilde{F^\vee}(V)  \simeq \mycolim{n}( F((\tau_{\geq -n} V)^\vee))^\vee \longrightarrow F(V^\vee)^\vee$$
is therefore an equivalence.
 
Now suppose that $W^\bullet$ is an augmented cosimplicial object in $\coh_{k, \geq 0}$ which is a limit diagram in 
$\md_k$.
Dualising, we obtain an augmented simplicial object $(W^{\vee})_\bullet$ of
$\coh_{k, \leq
0}$ which is a colimit diagram. Here we have used that linear duality is conservative. 
Now $\widetilde{F^\vee}( (W^{\vee})_\bullet)$ is a colimit diagram. 
Dualising again, we find that 
$F(W^\bullet)$ is a limit diagram, as desired.  
The reverse implication follows by a similar argument.
\end{proof} 

\begin{proposition}\label{excadm}
Let $F: \md_{k,\geq 0}\rightarrow \md_k$ be a functor which commutes with sifted
colimits. Suppose that $F$ 
carries $\coh_{k, \geq 0}$ into $\coh_{k, \geq 0}$ and that $F$
is $n$-excisive. Then $F$ is admissible. 
\end{proposition} 
\begin{proof} 
We first observe that $F$ extends uniquely to an $n$-excisive functor on $\md_k$
(\Cref{extensionofexcisives}), and that it therefore 
preserves all finite totalisations by \Cref{finitereals}. 
The functor $F^\vee$ on $\perf_k$ given by $V\mapsto  
F(V^{\vee})^{\vee}$  is also $n$-excisive  and therefore preserves finite geometric realisations. 
By \Cref{critadmissible}, it follows that $F$ is admissible. 
\end{proof}

\Cref{excadm} provides a large supply of examples of admissible functors. However, 
we will also need to work with functors, e.g.\ the free $\einf$-algebra functor, which
are not admissible. However, they will become 
admissible in the graded setting, and this observation will be crucial. 
\begin{definition}[Pointwise finite type] \label{pftobjects} 
Let $\grp \md_{k,\geq 0}$ \INN{070@$\grp \md_{k,\geq 0}$}
be the full subcategory 
of $\gr \md_k$ consisting of those objects $X_{\star}$ such that $X_i \in
\coh_{k, \geq 0}$ for all $i > 0$. 
We shall refer to these objects as \emph{pointwise of finite type}. Note that $\grcoh_{k, \geq 0}$ is a
subcategory of $\grp \md_{k,\geq 0}$. 
\end{definition} 
\begin{remark}
This is a weaker notion than being of finite type in the sense of \Cref{finitenessconditions}.
\end{remark}
We can give the following graded variant of \Cref{admissiblefunctors}:
\begin{definition}[Graded admissible functors]  
\label{def:admissible} 
Let $F: \gr \md_{k, \geq 0} \rightarrow \gr \md_{k, \geq 0}$ be a functor which
preserves sifted colimits. We will say that $F$ is \emph{admissible} if the
following conditions hold: 
\begin{enumerate}
\item 
The functor $F$ preserves the full subcategory $\grp \md_{k, \geq 0} \subset \gr
\md_{k, \geq 0}$. 
\item 
If $X^\bullet$ is a cosimplicial object of $\grp \md_{k, \geq 0}$ such that 
the totalisation $\mathrm{Tot}(X^\bullet)$ (computed in $\gr \md_k$) belongs to 
$\grp \md_{k, \geq 0}$, then 
$F( \mathrm{Tot}(X^\bullet)) \rightarrow \mathrm{Tot}( F(X^\bullet))$ is an equivalence.
In particular, the right-hand side is connective.
\end{enumerate}
\end{definition} 
We illustrate the technical advantage of working in the graded setting: 
\begin{example} \label{admissibleexample}
For some $n > 0$, consider the functor $\md_{k, \geq 0} \rightarrow \md_{k,\geq 0}$ given by $V \mapsto
V^{\otimes n}$. Since this is $n$-excisive, \Cref{excadm} implies that this functor is admissible. 
However, the functor $V \mapsto \bigoplus_{n> 0} V^{\otimes n}$ \textit{is not}
admissible as a functor $\md_{k, \geq 0} \rightarrow \md_{k, \geq 0}$. 

It is, however, easy to see that the assignment
$V \mapsto \bigoplus_{n > 0} V^{\otimes n}$ \textit{is} admissible as a functor 
$\gr \md_{k, \geq 0} \rightarrow \gr \md_{k, \geq 0}$. This holds  because each summand is admissible, and the
summands live in higher and higher internal degree. 
\end{example}

\newpage

\newcommand{\sqz}{\mathrm{sqz}}
\newcommand{\AS}{\mathfrak{A}^{\mathrm{fil}}}
\newcommand{\hcot}{\widehat{\cot}}
\newcommand{\adic}{\mathrm{adic}}
\newcommand{\ffil}{\mathrm{FilMod^\omega(k)}_{\geq 0}}

\section{The axiomatic argument}\label{filteredandgraded}
\label{axiomaticsec}
Given an augmented monad acting on $\mod_k$, we can consider the $\infty$-category of formal moduli problems based in algebras over this monad (cf.\ \Cref{{Cfmp}}).
In  this section, we shall prove that under certain conditions specified in \Cref{filteredrefinementdef}, this $\infty$-category
admits a ``Lie algebraic" description (cf.\ \Cref{fmpequiv}) as algebras over a monad constructed in terms of the  monadic bar construction.

\subsection{An informal overview}
\label{sec:informaloverview}  
We briefly recall some axiomatic aspects of  
{bar-cobar} duality, which goes back to the classical work of Moore \cite{MR0436178}. For recent modern treatments, we refer to \cite[Chapter 4]{ching2005bar}, 
\cite[Sections 3,4]{francis2012chiral}, or 
\cite[Chapter 6]{GRII}, \cite[Chapters 2,6]{LV}. 

Given an $\EE_\infty$-ring spectrum $k$, we write 
$\md_k$ for the $\infty$-category of $k$-module spectra. 
If  $\mathcal{O}$ is an $\infty$-operad in $\md_k$ (cf.\ e.g.\ \cite[Definition 4.1.4]{brantnerthesis}, then we can 
consider the $\infty$-category $\mathrm{Alg}_{\mathcal{O}}( \md_k)$ of
\emph{$\mathcal{O}$-algebras} in $\md_k$.
Suppose now that $\mathcal{O}(0) = 0, \mathcal{O}(1) = k$. 
Restriction  along the canonical map   from $\mathcal{O}$ to the trivial operad gives rise to the \emph{square-zero extension}  functor \INN{190@$\sqz$}
$$ \sqz: \md_k \rightarrow \mathrm{Alg}_{\mathcal{O}}(\md_k), $$
which turns an object $V \in \md_k$ into an $\mathcal{O}$-algebra whose  operadic multiplication maps are trivial. This  functor admits a left adjoint, which we will 
denote by\INN{032@$\cot$}   \INN{190@$\sqz$}
$$ \cot: \mathrm{Alg}_{\mathcal{O}}(\md_k) \rightarrow \md_k . $$
It is often called the \emph{cotangent fibre} or
\emph{$\mathcal{O}$-algebra homology} and can be thought of as the derived indecomposables functor.

\begin{remark} If $k$ is a field of characteristic zero, this was
discussed by Hinich \cite{HinichHomotopy}, \mbox{who studied} operads in chain complexes of $k$-vector spaces. For general $\einf$-ring spectra $R$,  modelled as commutative algebras in symmetric spectra,  these functors were also  studied in  \cite{Harpermodel1, Harpermodel2, Harpermodel3}.
\end{remark}

Informally speaking,   bar-cobar duality aims to recover an $\mathcal{O}$-algebra $A$
from $\cot(A)$ together with some additional structure placed on it. 
More formally, we  observe that  $C=\cot \circ \sqz$ defines a comonad on $\md_k$,  and abstract nonsense
gives rise to a functor
\begin{equation} \label{cotsqzadj} \mathrm{Alg}_{\mathcal{O}}( \md_k) \rightarrow \mathrm{coAlg}_{\cot \circ \sqz}(
\md_k).\end{equation}
Furthermore, the functor $\cot \circ \sqz$ 
can be identified with $V \mapsto \bigoplus_{n > 0} (\mathcal{K}(n) \otimes
V^{\otimes n})_{h \Sigma_n}$, where $\mathcal{K}(n) \in
\fun( B\Sigma_n, 
\md_k)$ is  a symmetric sequence of $k$-module spectra. 
This symmetric sequence is in fact the underlying 
symmetric sequence of the \emph{Koszul dual} cooperad $\mathcal{K} = \mathrm{Bar}(\mathcal{O})$, which one can
often identify explicitly, cf.\ e.g.\ \cite{GK, GKerr, GJ}. 
Taking $k$-linear duals, it follows that $\mathrm{Bar}(\mathcal{O})^{\vee}$ is an operad in
$\md_k$, and the functor $\cot^{\vee}$ takes values in algebras over 
$\mathrm{Bar}(\mathcal{O})^{\vee}$. 
Under suitable conditions, one may hope  that  
\eqref{cotsqzadj} will restrict to an equivalence on appropriate subcategories.

We mention some well-known results in this direction. 
First of all, the following general comparison result   shows that
one has an equivalence under connectivity hypotheses:
\begin{theorem}[Ching-Harper \cite{ChingHarper}] 
\label{CHtheorem}
Let $k$ be a connective commutative symmetric ring spectrum and let $\mathcal{O}$ be
an operad $\mod_k$ for which 
$\mathcal{O}(0) \simeq 0, \mathcal{O}(1) \simeq k$, and $\mathcal{O}(i)$ is
connective for all \mbox{$i \geq 0$.}
Let $\mathcal{K}$ be the associated cooperad on $\md_k$. 
Then the duality functor
$\mathrm{Alg}_{\mathcal{O}}( \md_k) \rightleftarrows 
\mathrm{coAlg}_{\mathcal{K}}( \md_k)$ restricts to an equivalence between
$0$-connected objects on both sides. 

\end{theorem} 
The hypotheses of 0-connectedness (which means that  $\pi_n(X) = 0$ for all $n\leq 0$)
is crucial and 
essential for the convergence of certain filtrations. 
This result has many predecessors, including the result of Moore
\cite{MR0436178}, who proves \Cref{CHtheorem} for $\mathcal{O}$ the
nonunital associative operad (over a discrete ring $k$) by explicitly constructing the  adjunctions
on chain complexes. Quillen \cite{quillen1969rational} proves it for
$\mathcal{O}$ the Lie operad over $\mathbb{Q}$, where $\mathrm{coAlg}_{\mathcal{K}}$ becomes
(up to a shift)  the $\infty$-category of cocommutative coalgebras.  \vspace{5pt}

The   theorem of  Lurie \cite{lurie2011derivedX} and Pridham
\cite{pridham2010unifying} 
classifying formal moduli problems  in {characteristic zero}
can be interpreted 
as a result in this vein, albeit with some additional hypotheses and
conclusions. In particular, one needs to slightly extend the equivalence beyond the assumption of $0$-connectedness, and one needs to prove that  certain pullbacks are carried to pushouts.

We recall the work of Lurie and Pridham in more detail. Let $k$ be a field of characteristic zero, and 
write $\clgaug_{k}$ for the $\infty$-category of 
augmented $\einf$-algebras over $k$. Note that we can also regard this as the
$\infty$-category of algebras over the \emph{nonunital} $\einf$-operad in
$\md_k$,  so this is an instance of the situation described above.

\begin{definition}  \INN{032@$\cNoet$}
We say that $A \in \clgaug_k$ is \emph{complete local
Noetherian} if $A$ is connective, $\pi_0(A)$ is a 
Noetherian ring which is complete with respect to the augmentation ideal, and each $\pi_i(A)$ is a finitely generated $\pi_0(A)$-module.  
 
Let {$\cNoet $} denote the full subcategory of $ \clgaug_k$ spanned by such
algebras. 
\end{definition}  

It is well-known that the Koszul dual to the nonunital $\einf$-operad is the
shifted Lie
operad. 
It follows that if $A \in \clg_k^{\mathrm{aug}}$, then $\cot(A)^{\vee}[-1]$ is naturally equipped with the structure
of a differential graded Lie algebra, and we  obtain a functor
$$\mathfrak{D}: (\clg_k^{\mathrm{aug}})^{op} \rightarrow \mathrm{Alg}_{\liedg} (\mod_{k}).$$ Here
$ \mathrm{Alg}_{\liedg} (\mod_{k})$ denotes the $\infty$-category of Lie
algebras in $\mod_k$, which is presented by the model category differential
graded Lie algebras. 

The Lurie--Pridham \Cref{LP}  is   essentially equivalent to the following result:
\begin{theorem}[Lurie, Pridham]
\label{antiLiecomm}
The functor $\mathfrak{D}$ restricts to an anti-equivalence
$$ \mathfrak{D} : (\cNoet )^{op} \simeq   \mathrm{Alg}_{\liedg} (\mod_{k,\leq -1}^{\ft})$$
between $\cNoet $
and the $\infty$-category 
$ \mathrm{Alg}_{\liedg} (\mod_{k,\leq -1}^{\ft})$
of
differential graded Lie algebras whose homotopy groups are finite-dimensional in each degree and concentrated in negative degrees. 

Furthermore, 
given maps $A\rightarrow A''$ and $A' \rightarrow A''$ in $\cNoet $ which induce surjections on $\pi_0$,  
\mbox{$\mathfrak{D}$ takes} the associated pullback square to a pushout square of
differential graded Lie algebras. 
\end{theorem} 
The   deduction of \Cref{LP} from 
\Cref{antiLiecomm} is   explained in Sections $6$ and $7$ of the survey paper \cite{LurieICM},
 or, using the   language of deformation theories, in \cite[Theorem 1.3.12]{lurie2011derivedX}.\vspace{5pt}
 
The main ``formal'' contribution of this paper is  a new method for proving results like \Cref{antiLiecomm} for $\infty$-categories of algebras more general than 
$\clg_k^{\aug}$ with $\chara(k)=0$. 
We will prove the relevant equivalence using Lurie's higher categorical version of the Barr--Beck comonadicity theorem (cf. Theorem 4.7.3.5 in \cite{lurie2014higher}), which we shall briefly recall for the reader's convenience:
\begin{theorem}[Barr--Beck--Lurie]\label{BBL}   
The adjunction  $F:\mathcal{C} \rightleftarrows \mathcal{D}:G$ is comonadic if and only if the following conditions hold true:
\begin{enumerate}
\item the functor $F$ is conservative, i.e.\ a morphism  $f$ in $\mathcal{C}$ is an equivalence if and only if $F(f)$ is an equivalence;
\item given an $F$-split cosimplicial object $X^\bullet$ in $\mathcal{C}$ (cf.\ \cite[Section 4.7.2]{lurie2014higher}), the diagram  $X^\bullet$ admits a limit in $\mathcal{C}$, which is preserved by $F$.
\end{enumerate}
\end{theorem} 

The Barr--Beck--Lurie theorem is a powerful tool in higher category theory which, just like its classical counterpart, 
has been used to establish various descent equivalences by checking certain
convergence results. We give an instructive  example   \mbox{(cf.\ \cite[Sec. 3]{MGaloisgroups}):}
\begin{example}[Nilpotent descent] 
\label{nilpdescent}
Fix a field $k$ and consider the algebra object $k[\epsilon]/\epsilon^2$,  equipped with the
obvious augmentation map   to $k$. 
We claim that the natural functor
$$ \mod_{k[\epsilon]/\epsilon^2}  \rightarrow \mod_k, \quad M
\mapsto k \otimes_{k[\epsilon]/\epsilon^2} M $$
is comonadic. To see this, we need to check that the functor $  k
\otimes_{k[\epsilon]/\epsilon^2} (-)$ 
is (1) conservative and (2) commutes with totalisations of cosimplicial objects in 
$\mod_{k[\epsilon]/\epsilon^2}$ which become split after applying $k
\otimes_{k[\epsilon]/\epsilon^2} (-)$. 
Both of these facts follow easily from the fibre sequence of $k[\epsilon]/\epsilon^2$-modules
$$k \rightarrow k[\epsilon]/\epsilon^2 \rightarrow k.  $$ 
Indeed, if $M \in \mod_{k[\epsilon]/\epsilon^2}$ satisfies $k
\otimes_{k[\epsilon]/\epsilon^2} M = 0$, then the above fibre sequence implies $M \simeq  0$. This implies (1).
Similarly, we can prove  statement (2) about commutativity with
totalisations by observing that if $M^\bullet$ is a cosimplicial object in
$\md_{k[\epsilon]/\epsilon^2}$ such that $k \otimes_{k[\epsilon]/\epsilon^2} M^\bullet$ is split, then the $\mathrm{Tot}$-tower defined by $M$  is a constant pro-object. \vspace{3pt}
\end{example} 
Since  we will be working with algebras rather than modules, the arguments
are more involved. The basic observation is that the convergence criterion appearing in the Barr--Beck--Lurie \Cref{BBL}(2) is far easier to check when working in the context of connected graded objects 
 (compare also the notion of pro-nilpotence in \cite{francis2012chiral}). 
As a consequence,  it becomes much  easier to establish a Koszul duality statement for connected  graded
$\mathbb{E}_\infty$-algebras.

Any augmented $\EE_\infty$-algebra is automatically endowed with a canonical $\mathfrak{m}$-adic filtration. With some care, we can use this fact to  transfer all convergence questions into the simpler connected graded setting. 
We note that the use of filtrations in this type of argument is
standard in the literature on operadic Koszul duality, \mbox{cf. for example
\cite{kuhn2003mccord, HarperHess, kuhn2017operad,  GRII,  ChingHarper}.}\vspace{3pt}

In \Cref{filteredrefinementdef} below, we  will isolate  axiomatic properties of a monad which will guarantee that the strategy outlined above passes through. Indeed, we prove in  \Cref{thm:mainaxiomatic}  that if these axioms are satisfied, then a generalised version of \Cref{antiLiecomm} holds true.

\subsection{The axiomatic setup}
Let $k$ be a field and suppose that $T$ is a sifted-colimit-preserving monad acting on $\mod_k$. In order to set up a cotangent formalism for $T$-algebras and apply it to $T$-based formal moduli problems, we shall need to assume that $T$ is augmented over the identity monad. 
In fact, we will adopt an equivalent formalism based in adjunctions rather than monads; this will later facilitate our  treatment of filtrations:

\newcommand{\afp}{\mathrm{afp}}
\begin{definition}  
\label{basicsit}An \textit{augmented monadic adjunction} consists of a pair of adjunctions 
\[
\begin{tikzcd}[column sep=5em]
	\operatorname{Mod}_{k,\ge 0}
	\arrow[r, shift left=1.4ex, "\mathrm{free}"]
	\arrow[r, shift right=1.4ex, "\mathrm{forget}"', leftarrow]
	\arrow[r, phantom, "\perp" description]
	 & \mathcal{C}
	\arrow[r, shift left=1.4ex, "\mathrm{cot}"]
	\arrow[r, shift right=1.4ex, "\mathrm{sqz}"', leftarrow]
	\arrow[r, phantom, "\perp" description]
	& \operatorname{Mod}_{k,\ge 0}
\end{tikzcd}
\]
whose composite is the identity  such that 
$\mathcal{C}$ is pointed and presentable and 
$(\mathrm{free}\dashv \mathrm{forget})$ is   monadic  with sifted-colimit-preserving right adjoint.
\end{definition}
\begin{remark}
Given an augmented monadic adjunction as above, we can identify $\mathcal{C}$ with the $\infty$-category of algebras for a monad $T=\mathrm{forget} \circ \mathrm{free}  $ on $\md_{k,\geq 0}$, and this monad is canonically augmented via $T = \mathrm{forget} \circ \mathrm{free}  \longrightarrow
 \mathrm{forget} \circ \mathrm{sqz} \circ\mathrm{cot} \circ \mathrm{free}   = \id$. 
\end{remark}
 
\begin{example} 
\label{aex:SCR}
The main example of interest to us is the case where $\mathcal{C}$ is the $\infty$-category of
augmented simplicial commutative $k$-algebras.
In this case, the above adjunctions are defined as expected: 
$\mathrm{free}$ builds the free simplicial commutative $k$-algebra,
$\mathrm{forget}$ sends an augmented simplicial commutative $k$-algebra to its augmentation ideal, 
$\mathrm{cot}$ takes the cotangent fibre, and $\mathrm{sqz}$ is the trivial square-zero construction functor. 

The augmented monad $\L\sym =\mathrm{forget} \circ \mathrm{free}$  \INN{120@$\L\sym$} sends $M$ to $\bigoplus_{i\geq 1} \L\sym^i(M)$, where $\L\sym^i$ is the left derived functor of the $i^{th}$ symmetric power \mbox{functor $(-)^{\otimes{i}}_{\Sigma_i}$.} The functor $\cot$ can be computed explicitly as $R\mapsto \Barr(1,\L\sym,IA)$, where $IA$ is the augmentation ideal of $A$.
\end{example}  
\begin{example} 
\label{aex:Einf}
There is also a variant of the above example when we consider $\mathcal{C}$ to
be the $\infty$-category of connective augmented $\mathbb{E}_\infty$-algebras over $k$. 
\end{example} 

\begin{example} 
\label{aex:op}
In fact, given any $\infty$-operad $\mathcal{O}$  in $\md_{k, \geq 0}$ with $\mathcal{O}(0) = 0$ and $
\mathcal{O}(1) \simeq k$, we can take $\mathcal{C}$ to be the $\infty$-category of
connective $\mathcal{O}$-algebras.  
For  $k$ a field of characteristic zero, we recover the desired adjunctions as in the discussion at the beginning of
\Cref{sec:informaloverview}. 
\end{example}

\begin{remark} 
We will often suppress the notation $\forget$ when it will not cause
confusion. For instance, given $A \in \mathcal{C}$, we shall write \INN{160@$\pi_i$}$\pi_i(A) =
\pi_i( \forget(A))$. 
\end{remark}

In the situation
of Definition~\ref{basicsit}, we hope to establish a version of
\Cref{antiLiecomm}.  
For this,  need   a   subcategory
$\mathcal{C}_{\afp} \subset \mathcal{C}$ of \emph{complete almost finitely presented
objects} satisfying\vspace{-2pt} \mbox{the following desiderata:}
\begin{enumerate} \label{desired}
\item  the adjunction $(\cot \dashv \sqz)$ restricts to a \vspace{-2pt}comonadic adjunction 
\begin{tikzpicture}[baseline=(C.base)]
	\node (C) at (0,0) {$\mathcal{C}_{\afp}$};
	\node (M) at (3,0) {$\mod^{\ft}_{k,\geq 0}$};
	% adjunction arrows
	\draw[->] ([yshift=2.5pt]C.east) -- node[above] { $\cot$}
	([yshift=2.5pt]M.west);
	\draw[<-] ([yshift=-2.5pt]C.east) -- node[below] {$\sqz$}
	([yshift=-2.5pt]M.west); 
\end{tikzpicture};
\item the monad  $  (M\mapsto  \cot (\sqz (M^\vee))^\vee)$ on $\mod^{\ft}_{k, \leq 0}  $ extends uniquely to a sifted-colimit-preserving monad $T^\vee$ on\vspace{5pt}  $\mod_k$ (cf.\ \Cref{extendmonad}). \\
 We write $\mathfrak{D}: \mathcal{C}_{\afp}^{op} \rightarrow \mathrm{Alg}_{T^{\vee}}$ for \vspace{5pt} the fully faithful embedding sending $A$ to $\cot(A)^\vee$;
\item given $A, A', A'' \in \mathcal{C}_{\afp}$ and $\pi_0$-surjective maps $A \rightarrow A''$ and $A' \rightarrow A''$, the fibre product $A \times_{A''} A' \in
\mathcal{C} $ also belongs to $\mathcal{C}_{\afp}$, and the following square is a pushout in $T^{\vee}$-algebras: 
$$ \xymatrix{
\mathfrak{D} (A'') \ar[d]  \ar[r] &  \mathfrak{D}( A') \ar[d]  \\
\mathfrak{D}(A) \ar[r] &  \mathfrak{D}( A \times_{A''} A')  \ .
}$$
\end{enumerate}
\begin{remark} Writing $T = \cot \circ \sqz : \mod^{\ft}_{k,\geq 0} \rightarrow \mod^{\ft}_{k,\geq 0}$ for 
the  comonad induced by the above adjunction, the above conditions give   equivalences 
$\mathcal{C}_{\afp} \simeq \mathrm{coAlg}_T( \mod^{\ft}_{k,\geq 0}) \simeq \mathrm{Alg}_{T^\vee}( \mod^{\ft}_{k,\leq 0})^{op}.$
This is a version of bar-cobar duality. 
\end{remark}  

\begin{remark}
In all our \Cref{aex:SCR,,aex:Einf} above, we will be able to explicitly identify the target subcategory
$\mathcal{C}_{\afp}$. More precisely,
$\mathcal{C}_{\afp}$ will be the $\infty$-category of augmented simplicial
commutative rings (resp. connective $\mathbb{E}_\infty$-rings) $A$
such that $\pi_0(A)$ is complete local Noetherian, and such that $\pi_i(A)$ is
finitely generated as a $\pi_0(A)$-module for all $i \geq 0$.  
\end{remark}

To construct a subcategory $\mathcal{C}_{\afp} \subset \mathcal{C}$  satisfying the desirable properties  above, we will need to work in a more refined context. 
Thinking of $\mathcal{C}$ as some $\infty$-category of algebras, we will  specify $\infty$-categories $\mathcal{C}^{\fil}$ and $\mathcal{C}^{\gr}$ of filtered and graded algebras, and assume that  these  are suitably linked by 
 several natural functors. Further technical assumptions,  which are readily checked in practice,   then allow us to give a Koszul dual description of $\mathcal{C}$-based \mbox{formal moduli problems.}
We summarise all required data in the following central definition, which is a \mbox{filtered enhancement of \Cref{basicsit}:}
\begin{definition}  
\label{filteredrefinementdef}
A \textit{filtered augmented monadic adjunction}    \vspace{-5pt} consists of a diagram of left adjoints \vspace{0pt} 
$$ \begin{tikzcd} 
	\md_{k,\geq 0} \arrow[r, "\mathrm{free}"] \arrow[d, "(-)_1"] & \mathcal{C} \arrow[r, "\mathrm{cot}"] \arrow[d, "\adic"] & \md_{k,\geq 0} \arrow[d, "(-)_1"] \\
	\fil \md_{k,\geq 0} \arrow[r, "\mathrm{free}"] \arrow[d, "\gr"] & \mathcal{C}^{\fil} \arrow[r, "\mathrm{cot}"] \arrow[d] & \fil \md_{k,\geq 0} \arrow[d, "\gr"] \\
	\gr \md_{k,\geq 0} \arrow[r, "\mathrm{free}"] & \mathcal{C}^{\gr} \arrow[r, "\mathrm{cot}"] & \gr \md_{k,\geq 0}
\end{tikzcd} $$
where the  vertical arrows $(-)_1: \md_{k,\geq 0} \to\fil \md_{k,\geq 0}$
 send $V $ to $(V)_1 = (\dots \rightarrow 0 \rightarrow 0 \rightarrow V) $ 
(cf.\ \Cref{filadjunctionunderlying}) and  the functor $\gr$ takes the associated graded. 
Moreover, we shall assume: 
\begin{enumerate} 
\item (Augmented Monadicity) 
All horizontal composites give the identity, and the   adjunctions $(\md_{k,\geq 0}  \rightleftarrows \mathcal{C})$, $(\fil \md_{k,\geq 0} \rightleftarrows \mathcal{C}^{\fil})$,  $
( \gr \md_{k,\geq 0} \rightleftarrows \mathcal{C}^{\gr})$
are monadic with sifted-colimit-preserving right adjoints.
\item (Adjointability)  
Taking right adjoints of the top vertical arrows gives   commutative squares \vspace{-1pt}
$$\begin{tikzcd}
	\md_{k,\geq 0} \arrow[r, "\mathrm{free}"] & \mathcal{C} \arrow[r, "\mathrm{cot}"] & \md_{k,\geq 0} \\
	\fil \md_{k,\geq 0} \arrow[r, "\mathrm{free}"] \arrow[u, "\F^1"] & \mathcal{C}^{\fil} \arrow[r, "\mathrm{cot}"] \arrow[u, "\F^1"] & \fil \md_{k,\geq 0} \arrow[u, "\F^1"]
\end{tikzcd}\vspace{-2pt}  $$
and the unit $\id_{\mathcal{C}} \rightarrow F^1 \circ \adic$ is an equivalence.

Taking right adjoints of the middle and lower horizontal maps gives  
commutative squares\vspace{-2pt} 
$$\begin{tikzcd}
	\fil \md_{k,\geq 0} \arrow[d, "\gr"'] 
	& \mathcal{C}^{\fil} \arrow[d, "\gr"] \arrow[l, below, "\mathrm{forget}"'] 
	& \fil \md_{k,\geq 0} \arrow[d, "\gr"] \arrow[l, below, "\mathrm{sqz}"']   \\
	\gr \md_{k,\geq 0} 
	& \mathcal{C}^{\gr} \arrow[l, below, "\mathrm{forget}"'] 
	& \gr \md_{k,\geq 0}  \arrow[l, below, "\mathrm{sqz}"'] 
\end{tikzcd}$$

\item (Admissibility) \label{grtower}
 The map $  \gr \md_{k,\geq 0} \xrightarrow{\forget \circ \free}  \gr \md_{k,\geq 0}$  
is  admissible  \mbox{(cf.\ \Cref{def:admissible}).} \\
Given $A \in\mathcal{C}^{\gr}$, there is a functorial tower
$\{A^{(i)}\}_{i \geq 1}$ in $\mathcal{C}^{\gr}$ 
and compatible maps $A \to
A^{(i)}$ satisfying:
\begin{enumerate}
\item there is a natural isomorphism $A^{(1)} \simeq \sqz \circ \cot(A)$ of objects over $A$ in
$\mathcal{C}^{\gr}$;
\item for $i>1$, there is a natural isomorphism $
\forget(A^{(i-1)})/\forget (A^{(i)})  \simeq G_i ( \cot(A))$ for an admissible and $i$-increasing
functor $G_i:  \gr \md_{k,\geq 0} \rightarrow \gr \md_{k,\geq 0}$;
\item  the map $\forget(A) \rightarrow  \forget(A^{(i)})$ in $  \gr \md_{k,\geq 0}$
induces an equivalence on graded pieces \vspace{-7pt}of (internal) degree $\leq i$. \\
\end{enumerate}
\hspace{-22pt} Let $\mathcal{C}^{\gr}_{\afp}\subset \mathcal{C}^{\gr}$\INN{030@$\mathcal{C}^{\gr}_{\afp}$}
 consist of all $A$  with $\cot(A) \in \gr^{\ft}\mod_{k,\geq 0}$. Write $\mathcal{C}^{\fil}_{\afp}\subset \mathcal{C}^{\fil}$ \INN{030@$\mathcal{C}^{\fil}_{\afp}$} for the full  subcategory of all complete
$A $ with $\gr(A) \in \mathcal{C}^{\gr }_{\afp}$. Let $\mathcal{C}_{\afp}\subset \mathcal{C}$ 
\INN{030@$\mathcal{C}_{\afp}$}
consist of all $A$ \mbox{with $\adic(A) \in \mathcal{C}^{\fil}_{\afp}$.}   \vspace{-9pt} 
\item (Coherence) The following conditions hold:
\label{coherencehyp} 
\begin{enumerate}
\item  
when $A, A', A'' \in \mathcal{C}^{\gr }_{\afp}$  and 
we have maps $A \rightarrow A'', A' \rightarrow A''$ which induce surjections on $\pi_0$, then
the fibre product $A \times_{A''} A' \in \mathcal{C}^{\gr}$ also belongs to 
$\mathcal{C}^{\gr }_{\afp}$;
\item
if $V \in \gr^{\ft} \mod_{k,\geq 0}$, then $\sqz(V) \in  \mathcal{C}^{\gr }_{\afp}$. 
\end{enumerate}

\item (Completeness)  \label{goodfc}
If $A \in \mathcal{C}^{\fil}_{\afp}$, then the following conditions hold true:
\begin{enumerate}[a)]
\item  
the cotangent complex $\cot(A) $ is complete;
\item
the adic filtration $\adic( F^1 A)$ on $F^1 A \in \mathcal{C}$ is complete. 
\end{enumerate}
\end{enumerate}

\end{definition}

\begin{remark}
The first part of the adjointability axiom (2) asserts that taking free algebras and
taking the cotangent fibre commutes with taking underlying objects. 
The second  part ensures that taking underlying objects or taking a square-zero extension  commutes
with passage \mbox{to associated gradeds.}
The functor $\adic: \mathcal{C} \rightarrow \mathcal{C}^{\fil}$ is an abstraction of the construction
which sends an augmented commutative ring $R$ with augmentation ideal
$\mathfrak{m}$ to its (derived) $\mathfrak{m}$-adic filtration.
In practice, the tower $A^{(i)}$ appearing in the admissibility axiom (3) will always be constructed via an adic filtration in graded objects.
\end{remark}

\begin{remark} 
The reader is encouraged to keep in mind the following example (which is
discussed in more detail starting in \Cref{setupSCR}). Let
$\mathcal{C}$ be the $\infty$-category of augmented simplicial commutative
$k$-algebras, 
as in \Cref{aex:SCR}. Then, we can let $\mathcal{C}^{\fil}$  be the
$\infty$-category of augmented filtered simplicial commutative $k$-algebras and
$\mathcal{C}^{\gr}$ the  $\infty$-category
of augmented graded simplicial commutative $k$-algebras. 
In this case, one can show the following facts: 
\begin{enumerate}
\item   \INN{030@$\mathcal{C}^{\gr}_{\afp}$}
$\mathcal{C}^{\gr}_{\afp}$ consists of graded simplicial
commutative rings $A$ such that the bigraded ring $\pi_*(A)$ has the 
property that $\pi_0(A)$ is Noetherian and $\pi_i(A)$ is finitely generated \mbox{over
$\pi_0(A)$ for all $i$;}
\item\INN{030@$\mathcal{C}^{\fil}_{\afp}$}
$\mathcal{C}^{\fil}_{\afp}$ consists of those complete filtered simplicial
commutative rings $A$ whose associated graded $\gr(A)_*$ is almost finitely
presented as the previous item. This   implies that $\pi_0(A)$ is a 
complete local Noetherian ring and each $\pi_i(A)$ is finitely generated as a
$\pi_0(A)$-module;
\item $\mathcal{C}_{\afp}$ consists of all complete local Noetherian objects in the sense of  \Cref{cnscr}
\INN{030@$\mathcal{C}_{\afp}$.}
\end{enumerate}
\end{remark} 

\begin{remark} 
Given any augmented $\infty$-operad $\mathcal{O}$ in $\mod_{k,\geq 0}$,  there is a natural $\infty$-category of filtered $\mathcal{O}$-algebras (i.e.\ $\mathcal{O}$-algebras in filtered objects), as well as one of graded $\mathcal{O}$-algebras.  
If $\mathcal{O}$ satisfies reasonable finiteness properties, then these $\infty$-categories are linked by a diagram satisfying conditions $(1)-(3)$ of \Cref{filteredrefinementdef}, where $(3)$ can be handled using the homotopy completion tower for graded $k$-modules (cf.\ \cite{HarperHess}).
More generally, it is not hard to check that  a similar statement holds for  any  augmented monad in the category of strict polynomial functors (cf.\ \cite{friedlander1997cohomology}), with the identity functor as weight one component. 
\end{remark} 

\begin{remark}
The admittedly clumsy axiomatisation adopted in \Cref{filteredrefinementdef} allows us to simultaneously handle the cases of $\EE_\infty$-$k$-algebras (which are algebras over an operad) and simplicial commutative $k$-algebras (for which this is not true).
\end{remark}

We can state and prove\vspace{-3pt} the main formal result in this paper:
\begin{theorem} 
\label{thm:mainaxiomatic}
Let $\mathcal{C}, \mathcal{C}^{\fil}, \ldots$ be part of a filtered augmented monadic adjunction \mbox{(cf. 
\Cref{filteredrefinementdef}).} 
\begin{enumerate}
\item  The adjunction $(\cot \dashv \sqz)$ restricts to a \vspace{-3pt}comonadic adjunction \begin{tikzpicture}[baseline=(C.base)]
	\node (C) at (0,0) {$\mathcal{C}_{\afp}$};
	\node (M) at (3,0) {$\mod^{\ft}_{k,\geq 0}$};
	% adjunction arrows
	\draw[->] ([yshift=2.5pt]C.east) -- node[above] { $\cot$}
	([yshift=2.5pt]M.west);
	\draw[<-] ([yshift=-2.5pt]C.east) -- node[below] {$\sqz$}
	([yshift=-2.5pt]M.west); 
\end{tikzpicture},  where $\mathcal{C}_{\afp}$ is defined as in \Cref{filteredrefinementdef} above or \Cref{cafp:defaxiomatic} below.

\item The monad  $  (M\mapsto  \cot (\sqz (M^\vee))^\vee)$ on $\mod^{\ft}_{k, \leq 0}  $ extends uniquely to a sifted-colimit-preserving \vspace{0pt}   monad $T^\vee$ \INN{200@$T^\vee$}
 on $\mod_k$ (cf.\ \Cref{extendmonad}). \vspace{5pt}  \\
 We write \INN{040@$\mathfrak{D}$}$\mathfrak{D}: \mathcal{C}_{\afp}^{op} \rightarrow \mathrm{Alg}_{T^{\vee}}$ \INN{012@$\mathrm{Alg}_{T^{\vee}}$}
 for \vspace{5pt} the fully faithful embedding sending $A$ to $\cot(A)^\vee$.

\item Given $A, A', A'' \in \mathcal{C}_{\afp}$ and $\pi_0$-surjective maps $A \rightarrow A''$ and $A' \rightarrow A''$, the fibre product $A \times_{A''} A' \in
\mathcal{C} $ also belongs to $\mathcal{C}_{\afp}$, and 
the following square is a pushout\vspace{-3pt} in $T^{\vee}$-algebras:
$$ \xymatrix{
\mathfrak{D} (A'') \ar[d]  \ar[r] &  \mathfrak{D}( A') \ar[d]  \\
\mathfrak{D}(A) \ar[r] &  \mathfrak{D}( A \times_{A''} A')
}$$
\end{enumerate}

\end{theorem}  
This  will allow us to prove that $\mathcal{C}$-based formal deformations are in fact governed by $T^\vee$-algebras. First, we   specify the  small objects which will parametrise our \vspace{-1pt} deformations \mbox{(cf.\ \cite[Definition 1.1.8]{lurie2011derivedX}):}
\begin{definition}[Artinian objects] \label{axiomaticart} 
Let $\mathcal{C}_{\art} \subset \mathcal{C}_{\afp}$   \INN{030@$\mathcal{C}_{\art}$}
denote the smallest full
subcategory \vspace{-3pt} satisfying: 
\begin{enumerate}
\item  
$\mathcal{C}_{\art}$ contains the terminal object $\ast$;
\item 
given $A \in \mathcal{C}_{\art}$ and a morphism $ A \rightarrow \sqz( k[n])$ with
$n> 0$, then the fibre product \mbox{$A' = A \times_{\sqz(k[n])} (\ast)$} also belongs to
$\mathcal{C}_{\art}$. \vspace{0pt}
\end{enumerate}
Observe that $(1)$ and $(2)$ together imply that  
$\mathcal{C}_{\art}$ contains  $\sqz( k[n])$ for any $n \geq 0$. 
\end{definition} 
We spell out  the definition of a $\mathcal{C}$-based formal moduli problem \mbox{(cf.\ \cite[Definition 1.1.14]{lurie2011derivedX})}:
\begin{definition} \label{Cfmp} 
A \emph{$\mathcal{C}$-based formal moduli problem} is a functor
$X: \mathcal{C}_{\art} \rightarrow \mathcal{S}$ with the properties: 
\begin{enumerate}
\item $X$ carries the  
terminal object $\ast$ to a contractible space;  
\item given $A \in \mathcal{C}_{\art}$ and a morphism $A \rightarrow \sqz( k[n])$ with $n  >
0$, the functor $X$ sends the fibre product  $\displaystyle  A' \simeq A \times_{\sqz(k[n])} (\ast)$   to a fibre product in  spaces
\mbox{$X(A') \simeq X(A) \times_{X(\sqz(k[n]))} X(\ast)$.}
\end{enumerate}
Let \INN{130@$\moduli_{\mathcal{C}}$} $\moduli_{\mathcal{C}} \subset \fun(\mathcal{C}_{\art}, \mathcal{S})$ be the $\infty$-category of formal moduli problems. 

Given $X \in \moduli_{\mathcal{C}}$, the functor
$X \circ \sqz: \perf_{k,\geq 0} \rightarrow \mathcal{S}$ is  excisive, and therefore defines an essentially unique 
$k$-module $T_X \in \md_k$ with $\Omega^{\infty}(V \otimes T_X) \simeq X(\sqz(V))  $ for $V\in \perf_{k,\geq 0}$.
We call    $T_X\in \md_k$    the \emph{tangent
complex} to $X$. Its  \INN{200@$T_X$}underlying spectrum \mbox{satisfies $\Omega^{\infty-n}T_X \simeq X(\sqz(k[n]))$.}
\end{definition}

We then have the following consequence of \Cref{thm:mainaxiomatic}:
\begin{theorem} \label{fmpequiv}
If $\mathcal{C}$ is part of a filtered augmented monadic adjunction (cf. 
\Cref{filteredrefinementdef}), there is 
an equivalence of $\infty$-categories 
$ \mathrm{Alg}_{T^{\vee}}\xrightarrow{\simeq }\moduli_{\mathcal{C}} $ with
$\mathfrak{g} \mapsto ( R \mapsto \Map_{\mathrm{Alg}_{T^{\vee}}}(\mathfrak{D}(R), \mathfrak{g}) ) $
such that the composite  $\moduli_{\mathcal{C}} \to\mathrm{Alg}_{T^{\vee}} \to
\md_k$ is equivalent to the tangent fibre functor $X \mapsto T_X$. 
\end{theorem}

\subsection{Graded objects}
Let $\mathcal{C}, \mathcal{C}^{\fil}, \mathcal{C}^{\gr}, \ldots$ be part of a filtered augmented monadic adjunction in the sense of \Cref{filteredrefinementdef}.
Our first goal is to study the  adjunction 
 $$\begin{tikzpicture}[baseline=(C.base)]
	\node (C) at (0,0) {$\cot:\mathcal{C}^{\gr}$};
	\node (M) at (3,0) {$\gr \mod_{k,\geq 0}: \sqz$};
	% adjunction arrows
	\draw[->] ([yshift=2.5pt]C.east) -- node[above] { }
	([yshift=2.5pt]M.west);
	\draw[<-] ([yshift=-2.5pt]C.east) -- node[below] { }
	([yshift=-2.5pt]M.west); 
\end{tikzpicture}$$
 on graded objects.
The admissibility axiom (\ref{grtower}) will allow us to argue in a straightforward
manner that this adjunction is in fact comonadic on a large class of objects. 
The rest of the argument required to prove \Cref{thm:mainaxiomatic} will then 
amount to reducing everything to this case. 
We begin with several basic observations on graded objects:
\begin{remark}[The bar construction]    \INN{020@$\Barr_\bullet$}
\label{thebarconstruction}
Let $A \in \mathcal{C}^{\gr}$ (resp. $\mathcal{C}, \mathcal{C}^{\fil}$).
The augmented simplicial object $\Barr_\bullet(\free, \free, A) \rightarrow A$  admits an extra degeneracy in 
$\gr\mod_{k,\geq 0}$,
and therefore induces an equivalence $|\Barr_\bullet(\free, \free, A)| \simeq A$. As the forgetful functor preserves geometric realisations, we deduce that $\Barr_\bullet(\free, \free, A) \rightarrow A$ is  in fact a colimit diagram in $\mathcal{C}^{\gr}$ (resp. $\mathcal{C}, \mathcal{C}^{\fil}$).

Since $\cot(-)$ preserves colimits,  it follows that $\cot(A)$ is the geometric realisation of the simplicial object $\Barr_\bullet(\id, \free, A)$ whose value in  degree $i$ is $\free^{\circ i} (A)$. 

In addition, we have a natural map $A \rightarrow \cot(A)$
in $ \gr \md_{k,\geq 0}$ (resp. $\md_{k,\geq 0},  \fil \md_{k,\geq 0}$), which is obtained 
by observing that $\mathrm{Bar}_0(A) = A$ and that we have a map
$\mathrm{Bar}_0(A) \rightarrow |\mathrm{Bar}_\bullet(A)|$.
Note that for each $i$, we obtain a (degeneracy) map
$A \rightarrow \mathrm{Bar}_i(A) = \free^{\circ i}(A)$; this is simply a composite of
unit maps.
\end{remark}

\begin{remark}[Conservativity of $\cot$ in the graded case] 
\label{conservativecotgraded}
Let $A \in \mathcal{C}^{\gr}$. Then the map $A \rightarrow \cot(A)$ induces an
equivalence $A_1 \simeq \cot(A)_1$ in degree $1$ by
assumption $(3a)$ in   \Cref{filteredrefinementdef}. 
\label{cotdeg1}

If a map $A \rightarrow B$ in $\mathcal{C}^{\gr}$ induces an
equivalence on $\cot(-)$, then $A \rightarrow B$ is an equivalence. 
Indeed, it follows inductively from assumption \eqref{grtower} of
\Cref{filteredrefinementdef} that $A^{(i)} \rightarrow B^{(i)}$ is an equivalence for all $i
\geq 1$. Letting $i \rightarrow \infty$ and considering graded pieces, it follows that
$A \rightarrow B$ is an equivalence. 
\end{remark}  
 We will now prove that the associated graded of $\adic(A)$ is canonically a free algebra for \mbox{any
$A \in \mathcal{C}$.} Heuristically, we can think of the functor $\adic(-)$ as a method of interpolating between a general algebra and a free algebra. 
\begin{proposition} 
\label{grofadicisfree}
There is a canonical isomorphism of functors
$\mathcal{C} \rightarrow \mathcal{C}^{\gr}$ given by 
$$ \gr ( \adic(A)) \simeq \free( [\cot(A)]_1) . $$
Here we place $\cot(A)  \in \md_{k,\geq 0} $ in graded degree $1$ 
to construct $[\cot(A)]_1 \in  \gr \md_{k,\geq 0}$. 
\end{proposition}  
\begin{proof} 
Note that for any $B \in \mathcal{C}^{\gr}$, there is a natural map from 
$\free( [B_1]_1) \rightarrow B$. This map is an equivalence in internal degree $1$: by assumption
\ref{grtower} of  \Cref{filteredrefinementdef}, it suffices to check this after applying $\cot$. 
Here, the map becomes  $B_1 \rightarrow \cot(B)_1$, which again is an equivalence
by \mbox{assumption \ref{grtower}.}

Let $A \in \mathcal{C}$. Then $\cot( \adic(A))$ is the filtered object 
$\ldots \rightarrow 0 \rightarrow 0 \rightarrow \cot(A)$, and therefore $\cot( \gr( \adic(A)))$
is the graded object $[\cot(A)]_1$ with $\cot(A)$ concentrated in degree $1$. 
It follows that we obtain a map 
$\free( [\cot(A)]_1) \rightarrow \gr ( \adic(A))$ in $\mathcal{C}^{\gr}$.  This map is an equivalence when $A =
\free(X)$ for $X \in \md_{k,\geq 0}$ as in this case, it is simply the identity map on $\free( [X]_1)$.
It must therefore be an equivalence in general since both sides preserve geometric realisations.
\end{proof} 

We invite the reader to recall the notion of graded modules pointwise of finite type introduced in  \Cref{pftobjects}.
This finiteness property can be detected using the cotangent fibre functor:
\begin{proposition} 
\label{testgrcoherence}
Let $A \in \mathcal{C}^{\gr}$. Then the following are equivalent: 
\begin{enumerate}
\item $\forget(A) \in \pft \md_{k,\geq 0}$; 
\item $\cot(A) \in \pft \md_{k,\geq 0}$. 
\end{enumerate}
\end{proposition} 
\begin{proof} 
As before, we will omit the forgetful functor  from our notation. 

Suppose first that $\cot(A) \in \pft \md_{k,\geq 0}$, i.e.\ that each graded piece $\cot(A)_i \in
\mod^{\ft}_{k,\geq 0}$ is of finite type. 
The filtration appearing in assumption \eqref{grtower} of \Cref{filteredrefinementdef}
shows by induction that 
each $A^{(i)}$ lies in $\pft \md_{k,\geq 0}$, because the functors  $G_j$ preserve $\pft
\md_{k,\geq 0}$ by assumption. 
Since   $A \rightarrow A^{(i)}$ is an
equivalence in graded degrees below $i$, letting $i \rightarrow \infty$ implies   that
$A \in \pft \md_{k,\geq 0}$. 

Conversely, suppose $A \in \pft \md_{k,\geq 0}$ is a $k$-module spectrum of pointwise finite type. We  can again proceed by induction.
By  \Cref{cotdeg1}, we know that $\cot(A)_1\simeq A_1 $ is of finite type.
Suppose $\cot(A)_1, \dots, \cot(A)_{i-1} \in \mod_{k,\geq 0}^{\ft}$.  
Since the functors $G_j$ are increasing
and preserve $\pft \md_{k,\geq 0}$ for all $j>1$, it follows that
$G_j( \cot(A))_i = G_j( \tr_{<i}(\cot(A)))_i$ belongs to $\mod_{k,\geq 0}^{\ft}$ for all
$1 < j \leq i$. 
Then the filtration of assumption~\ref{grtower}
shows inductively that 
 the graded piece $A_i  = A^{(i)}_i$ of $A^{(i)} \in \mathcal{C}^{\gr}$
has a finite filtration involving $\cot(A)_i$ and terms in $\mod^{\ft}_{k,\geq 0}$ (namely, $G_j(
\cot(A))_i$ for $1 < j \leq i$).
Since we assumed that $A_i \in \mod^{\ft}_{k,\geq 0}$, it follows   that
$\cot(A)_i$ lies in $\mod_{k,\geq 0}^{\ft}$.
\end{proof} 

We can now establish  the basic tool for commuting $\cot$ and totalisations,
which is the heart of the convergence arguments needed in this work. First we
need a basic notion. 

\begin{definition} 
The $\infty$-categories $\mathcal{C}, \mathcal{C}^{\fil},\mathcal{C}^{\gr}$ are
presentable, and hence have all limits. These 
are computed at the level of objects in $\md_{k ,\geq 0}$ (or $\gr \md_{k, \geq
0}, \fil \md_{k, \geq 0}$). 
We will say that a limit in 
$\mathcal{C}, \mathcal{C}^{\fil},\mathcal{C}^{\gr}$ \emph{connectively exists}
if the limit is also preserved in $\md_k, \gr \md_k, \text{ or }\fil \md_k$. In particular,
the limit in $\md_k, \gr \md_k, \text{ or }\fil \md_k$ is connective. 
\end{definition} 

\begin{proposition}[Convergence criterion in $\mathcal{C}^\gr$]  
Let $A^\bullet$ be a cosimplicial   object of $\mathcal{C}^{\gr}$ such that for each $i$, we have $\forget(A^i) \in \pft \md_{k,\geq 0}$. 
Then the following are equivalent: 
\begin{enumerate}
\item the totalisation $\mathrm{Tot}( \forget(A^\bullet))$  (computed in $\gr
\md_k$) belongs to $\pft \md_{k,\geq 0}$;
\item the totalisation 
$\mathrm{Tot}( \cot(A^\bullet))$
(computed in $\gr
\md_k$) belongs to $\pft \md_{k,\geq 0}$. 
\end{enumerate}
If these conditions are satisfied, then the limit $\mathrm{Tot}( A^\bullet)$
connectively exists in $\mathcal{C}^{\gr}$,  and the canonical map 
$ \cot( \mathrm{Tot}( A^\bullet)) \rightarrow \mathrm{Tot}( \cot(A^\bullet)) $
in $\gr \md_{k,\geq 0}$ is an equivalence. 
\label{grconvcrit}
\end{proposition} 
 
\begin{proof} 
In fact, we shall prove the following more refined statement: 
\begingroup
\addtolength\leftmargini{-0.2in}
\begin{quote}
Let $B^\bullet$ be an augmented cosimplicial object of $\mathcal{C}^{\gr}$ with
$B^j\in \pft \md_{k,\geq 0}$ for all $j \geq 0$.
Suppose that $B^\bullet_1, \dots, B^\bullet_{i-1}$ are limit diagrams in $\md_k$
with $B^{-1}_1, \dots, B^{-1}_{i-1} \in \mod^{\ft}_{k,\geq 0}$. Then the following
are equivalent: 
\begin{enumerate}
\item  
$B^\bullet_i$ is a limit diagram with $B^{-1}_i \in \mod^{\ft}_{k,\geq 0}$; 
\item
$\cot(B^\bullet)_i$ is a
limit diagram in $\md_k$ and $\cot(B^{-1})_{i} \in  \mod^{\ft}_{k,\geq 0}$. 
\end{enumerate}
\end{quote}
\endgroup

By induction, this implies the equivalence of $(1)$ and $(2)$ in the proposition, as well as 
the asserted convergence. Here we use that the forgetful functor
$\mathcal{C}^{\gr}\rightarrow \gr\mod_{k,\geq 0}$ creates limits.

For $i  = 1$, our refined claim  follows  from the equivalence $B^\bullet_1 \simeq \cot(B^\bullet)_1$. 
In general, we observe that $B^\bullet_i$ (considered as an augmented cosimplicial
object of $\md_k$) admits a finite filtration
whose associated graded terms are given by $\cot(B^\bullet)_i$ and  $G_j(
\cot(B^\bullet))_i$ for $1 < j \leq i$. By assumption, the functor $G_j$ is admissible and
increasing, and so the augmented cosimplicial 
object   
$G_j(\cot(B^\bullet))_i \simeq G_j(\tr_{i-j+1}\cot(B^\bullet))_i$
is a limit diagram by the hypothesis. It  follows that $B^\bullet_i$ is a limit diagram with
$B^{-1} \in \pft \md_{k,\geq 0}$ if and
only if $\cot(B^\bullet)_i$ is a limit diagram with
$\cot(B^{-1}) \in \pft \md_{k,\geq 0}$. 
\end{proof}

Let $\mathcal{C}^{\pft} \subset \mathcal{C}^{\gr}$\INN{030@$\mathcal{C}^{\pft}$} be the full subcategory
spanned by objects whose underlying graded module belongs to $\pft \md_{k,\geq
0}$.  
\begin{proposition}  \INN{032@$\cot$} \INN{190@$\sqz$}
\label{monadicgradedinfinite}
Restriction gives rise to a comonadic adjunction
$$ \begin{tikzpicture}[baseline=(C.base)] 
	\node (C) at (0,0) {${\cot}: \mathcal{C}^{\pft}$};
	\node (M) at (3,0) {$  \pft
		\md_{k,\geq 0}: \sqz .$};
	% adjunction arrows
	\draw[->] ([yshift=2.5pt]C.east) -- node[above] { }
	([yshift=2.5pt]M.west);
	\draw[<-] ([yshift=-2.5pt]C.east) -- node[below] { }
	([yshift=-2.5pt]M.west); 
\end{tikzpicture}$$
\end{proposition} 
\begin{proof} 
The adjunction is well-defined by \Cref{testgrcoherence}. It is in fact comonadic by \Cref{BBL},  whose conditions are satisfied by  \Cref{conservativecotgraded} and
\Cref{grconvcrit}. 
\end{proof} 
This comonadicity result for $ \mathcal{C}^{\pft} $ will later allow us to check convergence results in $\mathcal{C}$ by first lifting cosimplicial diagrams to filtered objects and then taking associated gradeds everywhere.\vspace{3pt}

However, $ \mathcal{C}^{\pft} $ is \textit{not} the correct graded analogue of the  $\infty$-category $\mathcal{C}^{\afp}$ of almost finitely presented objects
because the cotangent complex need not be 
finite type (only pointwise finite type).  We therefore introduce a more restrictive notion of finiteness.

For this, we first recall \Cref{finitenessconditions}, which introduces
$\gr^{\ft} \mod_{k,\geq 0}$ as the full subcategory of $\gr\mod_k$ spanned by
all $X_\star$  for which the underlying module $\bigoplus_{i\geq 1} X_i$ is of
finite type.  Note that $\gr^{\ft} \mod_{k,\geq 0}\subset \pft \md_{k,\geq 0}$
is a proper inclusion.  

We recall the corresponding  full \mbox{subcategory of
$\mathcal{C}^{\gr} $:}
\begin{definition}[The subcategory $\mathcal{C}^{\gr}_{\afp}$]  \INN{030@$\mathcal{C}^{\gr}_{\afp}$}
\label{def:grafp} 
The   category $\mathcal{C}^{\gr}_{\afp}$ of \emph{almost finitely
presented} objects  consists of all  
$A \in \mathcal{C}^{\gr}$ whose cotangent complex $\cot(A) \in \gr^{\ft} \mod_{k,\geq 0}$ has finite type. 
\end{definition} 
We should think of this as a finite generation condition (at least after any
truncation).  Note that by \Cref{testgrcoherence}, we have an inclusion $\mathcal{C}^{\gr}_{\afp}\subset \mathcal{C}^{ {\pft}}$.
\begin{remark} 
While $\mathcal{C}^{\pft}$ is evidently closed under fibre products of maps which are
surjective on $\pi_0$ (as these can be computed pointwise), the corresponding claim in $\mathcal{C}^{\gr}_{\afp}$ only holds by our coherence axiom $(4)$ in \Cref{filteredrefinementdef}, which will be easy to check in the examples of interest.

For instance, this ensures that if $V, V', V'' \in \gr^{\ft} \mod_{k,\geq 0}$ and we
have maps $V \rightarrow V'', V' \rightarrow V''$ which induce surjections on $\pi_0$, then
the fibre product
$A = \free(V) \times_{\free(V'')} \free(V')$ 
has the property that 
$\cot(A) \in \gr^{\ft} \mod_{k,\geq 0}$. \Cref{testgrcoherence} only shows  that
$\cot(A) \in \pft \md_{k, \geq 0}$, so we need to postulate this stronger
statement.  

\end{remark}

The coherence axiom $(4)$ in \Cref{filteredrefinementdef} implies:
\begin{cons}   \INN{032@$\cot$} \INN{190@$\sqz$}
\label{grafpadj}
The $(\cot,\sqz)$-adjunction  between $\mathcal{C}^\gr$ and $\gr \mod_{k,\geq 0}$ restricts to an \mbox{adjunction} 
 $$\begin{tikzpicture}[baseline=(C.base)]
	\node (C) at (0,0) {$\cot:\mathcal{C}^{\gr}_{\afp}$};
	\node (M) at (3,0) {$\gr \mod^{\ft}_{k,\geq
			0}: \sqz.$};
	% adjunction arrows
	\draw[->] ([yshift=2.5pt]C.east) -- node[above] { }
	([yshift=2.5pt]M.west);
	\draw[<-] ([yshift=-2.5pt]C.east) -- node[below] { }
	([yshift=-2.5pt]M.west); 
\end{tikzpicture}$$
\end{cons} 

\subsection{Filtered objects}
We will now transfer some of the above results to the $\infty$-category  $\mathcal{C}^{\fil}$.
In order to obtain similarly strong statements, we  need to restrict attention to \textit{complete} objects.
Recall that  a filtered connective $k$-module $(\ldots \rightarrow F^2M \rightarrow F^1M)  \in \fil \mod_{k,\geq 0}$ is said to be complete if the inverse limit $\lim_i  (F^iM)$ vanishes.  The inclusion $\filc \md_{k,\geq 0} \subset \fil \md_{k,\geq 0}$ of complete filtered connective $k$-modules into all 
filtered connective $k$-modules admits a left adjoint called \textit{completion} (cf.\ \Cref{defcompletion}).

In order to lift this completion functor to $\mathcal{C}^{\fil}$, we need an elementary\vspace{-1pt}  categorical observation.
\begin{remark}[Adjunctions and localisations]  
\label{adjcommute}
Let \ 
$F : \mathcal{A} \rightleftarrows \mathcal{B} : G$ \ be an adjunction of
presentable $\infty$-categories and suppose that we are given  a Bousfield localisation
$L_{\mathcal{A}}: \mathcal{A} \rightarrow \mathcal{A}$ of 
$\mathcal{A}$ with corresponding  strongly saturated class of morphisms  
$\mathcal{W}_{\mathcal{A}}$ in
$\mathcal{A} $ (cf.\ \cite[Def. 5.5.4.5]{lurie2009higher}). 

To produce a corresponding  Bousfield localisation on $\mathcal{B}$, we 
let $\mathcal{W}_{\mathcal{B}}$ be the class of morphisms $f : B_1 \rightarrow B_2$ satisfying $G(f) \in \mathcal{W}_{\mathcal{A}}$. Assume that
$\mathcal{W}_{\mathcal{B}}$ is strongly saturated and contains
$F(\mathcal{W}_{\mathcal{A}})$.
The  localisation functor  $L_{\mathcal{B}}$ for $\mathcal{W}_{\mathcal{B}}$ (cf.\ \cite[Sec.5.5.4]{lurie2009higher}) sits in\vspace{-3pt} a commutative square: 
$$ \xymatrix{
\mathcal{B} \ar[d]^G \ar[r]^{L_{\mathcal{B}}} &  \mathcal{B} \ar[d]^G \\
\mathcal{A}\ar[r]^{L_{\mathcal{A}}} &   \mathcal{A}
}$$ 
Indeed, given  any $B \in \mathcal{B}$, the unit $B \rightarrow L_{\mathcal{B}} B$ lies in
$\mathcal{W}_{\mathcal{B}}$. It follows that if $B$  is
$\mathcal{W}_{\mathcal{B}}$-local, then $G(B) \in \mathcal{A}$ is
$\mathcal{W}_{\mathcal{A}}$-local.  Since $G$ sends  $\mathcal{W}_{\mathcal{B}}$
 to $\mathcal{W}_{\mathcal{A}}$, the commutativity   follows. 
 We say that the adjunction is \emph{compatible with localisations.} 
\end{remark} 

Using this, we can lift the notion of completeness to $\mathcal{C}^\fil$:
\begin{definition}[Completions in $\mathcal{C}^\fil$]\label{C-completions}
An object $A \in \mathcal{C}^{\fil}$ is \textit{complete}  if $\forget(A) \in \fil \md_{k,\geq 0}$ is complete.   
We let \INN{030@$\widehat{\mathcal{C}^{\fil}}$}$\widehat{\mathcal{C}^{\fil}} \subset \mathcal{C}^{\fil}$ be the full subcategory spanned by all complete objects. 
\end{definition}  
The strongly saturated class associated with the localisation $\filc \md_{k,\geq 0} \subset \fil \md_{k,\geq 0}$ consists of all maps which induce equivalences on associated gradeds.
By the assumptions in \Cref{filteredrefinementdef}, the free-forgetful adjunction $\fil \md_{k,\geq 0} \rightleftarrows\mathcal{C}^{\fil}$ is compatible with completions in the sense of \Cref{adjcommute}. We therefore obtain a completion functor $\mathcal{C}^{\fil} \rightarrow \widehat{\mathcal{C}^{\fil}}$ which is the left adjoint of a Bousfield localisation.
Any $A \in \mathcal{C}^{\fil}$ comes equipped with a natural morphism $A \rightarrow \widehat{A}$ to its completion, and the underlying object $\forget( \widehat{A})\in \fil
\md_{k,\geq 0}$ of  $\widehat{A} $   is simply given by
the completion of the filtered object $\forget(A)$. 

Note that if a morphism $A \rightarrow B $ in $\mathcal{C}^{\fil}$ induces an equivalence on 
associated gradeds, then so does \mbox{$\cot(A) \rightarrow \cot(B)$.}
We deduce that for any $A \in \mathcal{C}$, the associated map on cotangent fibres $\cot(A) \rightarrow \cot( \widehat{A})$ is also
an equivalence on associated gradeds. 
We  obtain a diagram of left adjoints:
\[
\begin{tikzcd}
	\fil \md_{k,\geq 0} 
	\arrow[r, "\mathrm{free}"] 
	\arrow[d] 
	& \mathcal{C}^{\fil} 
	\arrow[r, "\mathrm{cot}"] 
	\arrow[d] 
	& \fil \md_{k,\geq 0} 
	\arrow[d] \\
	\filc \md_{k,\geq 0} 
	\arrow[r, "\widehat{\mathrm{free}}"] 
	& \widehat{\mathcal{C}^{\fil}} 
	\arrow[r, "\widehat{\mathrm{cot}}"] 
	& \filc \md_{k,\geq 0}
\end{tikzcd}
\]
in which the vertical arrows are given by completion functors.

\begin{remark}[The completed cotangent fibre  adjunction] \INN{030@$\widehat{\cot}$}
The functor $\widehat{\cot} = \widehat{(-)} \circ \cot$ appearing in the lower right hand part of the above diagram is part of an  adjunction 
$ \widehat{\cot} : \  \widehat{\mathcal{C}^{\fil}} \rightleftarrows \filc
\md_{k,\geq 0}\ :  \sqz$. 
Its left adjoint sends a complete object $A$ of $\mathcal{C}$ to its completed
cotangent fibre $\widehat{\cot(A)}$, whereas the right adjoint sends $V \in
\filc \md_{k,\geq 0}$ to $\sqz(V) \in \widehat{\mathcal{C}} \subset \mathcal{C}
$. 
\end{remark}

We can next make a direct translation of \Cref{grconvcrit} to the filtered setting:

\begin{proposition}[Convergence criterion in $\mathcal{C}^{\fil}$] 
\label{convcritCfil}
Let $X^\bullet$ be a cosimplicial complete object of $\mathcal{C}^{\fil}$ such that for each $i$, we have $\gr(\forget(X^i)) \in  \pft \md_{k,\geq 0}$. 
Then the following are equivalent: 
\begin{enumerate}
\item the associated graded $\gr( \mathrm{Tot}( \forget(X^\bullet)) )$  of the totalisation $\mathrm{Tot}( \forget(X^\bullet))$ (computed in $\fil \md_k$) belongs to $\pft \md_{k,\geq 0}$;
\item the associated graded $\gr ( \mathrm{Tot}( \cot(X^\bullet)))$ of the totalisation $\mathrm{Tot}( \cot(X^\bullet))$ (computed in $\fil \md_k$) belongs to $\pft \md_{k,\geq 0}$.
\end{enumerate}
Under these assumptions, the limit $\Tot(X^\bullet)$ connectively exists,
and the  map 
$\widehat{\cot}( \mathrm{Tot}( X^\bullet)) \rightarrow \mathrm{Tot}( \widehat{\cot}(X^\bullet)) $ 
is an equivalence. 
\end{proposition} 
\begin{proof} 
The equivalence of $(1)$ and $(2)$ follows immediately from \Cref{grconvcrit} by taking associated gradeds (using the adjointability axiom $(2)$ in \Cref{filteredrefinementdef} together with the fact  that $\gr:\fil\mod_k\rightarrow \gr\mod_k$ commutes with totalisations). 

The filtered module $\mathrm{Tot}(\forget(X^\bullet))$ (computed in $\fil \mod_k$) is complete (as completeness is a limit condition) and has associated graded in $\mod_{k,\geq 0}$. The Milnor sequence  implies that $\mathrm{Tot}(\forget(X^\bullet))$ in fact belongs to $\fil \mod_{k,\geq 0}$. It follows that $X^\bullet$ admits a limit in $\mathcal{C}^{\fil}$. The  arrow $\widehat{\cot}( \mathrm{Tot}( X^\bullet)) \rightarrow \mathrm{Tot}( \widehat{\cot}(X^\bullet))$ induces an equivalence after passing to associated gradeds by  \Cref{grconvcrit}. Hence, it    is   itself an equivalence since both domain and target are complete.
\end{proof} 

\begin{remark}
Note that both conditions $(1)$ and $(2)$ in \Cref{convcritCfil} contain the nontrivial assertion that the respective totalisations are  connective.
\end{remark}

We can now formulate a notion of almost finite presentation in the complete filtered context:
\begin{definition}[The subcategory  $\mathcal{C}^{\fil}_{\afp}$] \INN{030@$\mathcal{C}^{\fil}_{\afp}$}
\label{def:filafp}
Let $\mathcal{C}^{\fil}_{\afp}$ denote the full  subcategory of objects 
$A \in \mathcal{C}^{\fil}$ which are complete and 
satisfy $\gr(A) \in \mathcal{C}^{\gr }_{\afp}$, i.e.\
$\gr ( \cot(A)) \in \gr^{\ft} \mod_{k,\geq 0}$. 
We will refer to these objects as \emph{complete almost finitely presented.}
\end{definition} 
\begin{example}[Completed-free algebras in $\mathcal{C}_{\afp}^{\fil}$]  
\label{filteredfree}
Completed-free algebras on complete filtered (connective) modules of finite type belong to $\mathcal{C}^{\fil}_{\afp}$. Indeed, if 
$V \in \fil^{\ft} \mod_{k,\geq 0}$ is a complete filtered object with $\gr(V) \in \gr^{\ft} \mod_{k,\geq 0}$
 (cf.\ \Cref{finitenessconditions}), then $\free(V) \in \mathcal{C}^{\fil}$ has cotangent fibre $V$. 
Since ${\free(V)} \rightarrow \widehat{\free(V)}$ induces an equivalence on associated gradeds, the observations made after \Cref{C-completions} imply that   $\gr(\cot(\widehat{\free(V)})) \simeq \gr(\cot( {\free(V)})) \simeq \gr(V)$ belongs to $\gr^{\ft} \mod_{k,\geq 0}$, and hence $\widehat{\free(V)}\in \mathcal{C}^{\fil}_{\afp}$.

In fact, the completeness assumption $(5a)$ in \Cref{filteredrefinementdef} implies a stronger assertion. We know that $ \cot(\widehat{\free(V)})$ is complete, and we can therefore conclude that  $ \cot(\widehat{\free(V)})  \simeq V$. In other words, the morphism $\free(V) \rightarrow \widehat{\free}(V)$ induces an
equivalence on $\cot$. 
This can be thought of  as a generalisation of the classical fact that a polynomial ring
and a power series ring on finitely many variables have the same cotangent
fibre. 
\end{example}

\begin{cons}   \INN{032@$\cot$} \INN{190@$\sqz$}
If $V \in \fil^{\ft} \mod_{k,\geq 0}$, then the coherence axiom (4b) of \Cref{filteredrefinementdef} implies that $\sqz(V) \in
\mathcal{C}^{\fil}_{\afp}$. 
As in Construction~\ref{grafpadj}, 
we obtain the following adjunction by restriction:
 $$\begin{tikzpicture}[baseline=(C.base)]
	\node (C) at (0,0) {$\widehat{\cot}:\mathcal{C}^{\fil}_{\afp}$};
	\node (M) at (3,0) {$ \fil^{\ft} \mod_{k,\geq
			0}: \sqz.$};
	% adjunction arrows
	\draw[->] ([yshift=2.5pt]C.east) -- node[above] { }
	([yshift=2.5pt]M.west);
	\draw[<-] ([yshift=-2.5pt]C.east) -- node[below] { }
	([yshift=-2.5pt]M.west); 
\end{tikzpicture}$$

\end{cons}

\begin{remark} 
\label{rem:Cfilafpgeom}
The full subcategory $\mathcal{C}^{\fil}_{\afp} \subset \mathcal{C}^{\fil}$ is closed under geometric
realisations. This follows as if $X_\bullet$ is a simplicial diagram in $\mathcal{C}^{\fil}_{\afp}$, then the underlying module $\forget(|X_\bullet|) \simeq |\forget(X_\bullet)|$ is complete by  \Cref{geomrealcomplete}, and  moreover
$\gr(\cot(|X_\bullet|)) \simeq |\gr(\cot(X_\bullet|))|$ lies in $\gr^{\ft}\mod_{k,\geq 0}$ since $\mod_{k,\geq 0}^{\ft}\subset \mod_k$ is closed under geometric realisations.

Hence any geometric realisation of completed-free objects 
$\widehat{\free(V)}$ with $V \in \fil^{\ft}\mod_{k,\geq 0}$ \mbox{lies in 
$\mathcal{C}^{\fil}_{\afp}$.\vspace{5pt}}
\end{remark} 
We will now pass from filtered to non-filtered objects:  \INN{030@$\mathcal{C}_{\afp}$}
\begin{definition}[The subcategory $\mathcal{C}_{\afp}$] 
\label{cafp:defaxiomatic}
The full subcategory $\mathcal{C}_{\afp} \subset \mathcal{C}$ of \emph{complete almost finitely presented} objects in $\mathcal{C}$ consists of all  
$A $ for which $\adic(A) \in \mathcal{C}^{\fil}$ is complete and
 $\cot(A) \in \mod^{\ft}_{k,\geq 0}$.  
\end{definition}

In \Cref{aex:SCR,,aex:Einf}, the respective subcategories $\mathcal{C}_{\afp}$ will be as expected, i.e.\ consist of  all augmented simplicial commutative
$k$-algebras (resp. connective  $\mathbb{E}_\infty$-k-algebras) $A$ for which  $\pi_0(A)$ is complete local
Noetherian and $\pi_i(A)$ finitely generated over $\pi_0(A)$ for all $i$.

\begin{remark} 
\label{forgetcompletefilt}
By the completeness axiom $(5)$ in \Cref{filteredrefinementdef}, we know that 
if $A \in \mathcal{C}^{\fil}_{\afp}$, then $F^1 A \in \mathcal{C}_{\afp}$. 
Indeed,  assumption $(5b)$ implies that $F^1 A$ is complete, and  $\cot(F^1 A)\simeq F^1(\cot(A))$ lies in $\mod^{\ft}_{k,\geq 0}$ since $\cot(A)$ is complete by $(5a)$ and $\gr(\cot(A))$ lies in $\gr^{\ft} \mod_{k,\geq 0}$ by definition. 
\end{remark} 

\begin{example}[Completed-free algebras in $\mathcal{C}_{\afp}$] 
Completed-free algebras on  (connective) modules of finite type belong to $\mathcal{C}_{\afp}$.
Indeed, if $V \in \mod^{\ft}_{k,\geq 0}$, then we consider $\widetilde{V} = (\dots \rightarrow 0 \rightarrow 0 \rightarrow V)$ in $\fil^{\ft} \mod_{k,\geq 0}$. We have $\widehat{\free(\widetilde{V}) }\in \mathcal{C}^{\fil}_{\afp}$ by \Cref{filteredfree}, and this implies that 
$\widehat{\free(V)} \simeq F^1\widehat{ \free(\widetilde{V})}$ lies in $\mathcal{C}_{\afp}$ by \Cref{forgetcompletefilt}. 
\end{example}  

\begin{remark}[Closure properties of $\mathcal{C}_{\afp} $]
\label{closureCafp}
The full subcategory $\mathcal{C}_{\afp} \subset \mathcal{C}$ is closed under
geometric realisations. 
This follows from \Cref{rem:Cfilafpgeom} by noting that the left adjoints $\adic$ and $\cot$ preserve realisations and the full subcategory $\mod_{k,\geq 0}^{\ft}\subset \mod_k$ is closed under realisations.

Moreover, if $A, A', A'' \in \mathcal{C}_{\afp}$ and we are given  maps 
$A \rightarrow A'', A' \rightarrow A''$ which induce surjections on $\pi_0$, then the pullback
$A \times_{A''} A'$ also belongs to $\mathcal{C}_{\afp}$. Indeed,  note that $\adic(A),
\adic(A'), \adic(A'')$ belong to $\mathcal{C}^{\fil}_{\afp}$, and that both  maps $\adic(A) \to
\adic(A'')$ and $\adic(A') \rightarrow \adic(A'')$ are \mbox{surjective on $\pi_0$.} By the coherence axiom $(4a)$ in \Cref{filteredrefinementdef}, we deduce that $\gr(\adic(A) ) \times_{\gr(\adic(A''))} \gr(\adic(A'))$ belongs to $\mathcal{C}^{\gr}_{\afp}$. The canonical arrow $$\gr\left(\adic(A)   \times_{ \adic(A'') } \adic(A')\right)\rightarrow \gr(\adic(A) ) \times_{\gr(\adic(A''))} \gr(\adic(A'))$$
induces an equivalence after applying $\forget$ by the adjointability axiom $(2)$ in \Cref{filteredrefinementdef}. We deduce that $\adic(A) \times_{\adic(A'')} \adic(A') \in \mathcal{C}_{\afp}^{\fil}$  (as it is evidently complete).
Applying the right adjoint $F^1$ now shows that $A \times_{A'} A'' \in \mathcal{C}_{\afp}$ by \Cref{forgetcompletefilt}.
\end{remark} 
\begin{remark}
If $V \in \mod^{\ft}_{k,\geq 0}$, then $\sqz(V)$ belongs to $\mathcal{C}_{\afp}$. To see this, we  lift  $V$  to a complete filtered module $\widetilde{V} = (\ldots \rightarrow 0 \rightarrow 0 \rightarrow V)$ in $\fil^{\ft} \mod_{k,\geq 0}$. We then observe that $\sqz(\widetilde{V})$ is evidently complete and it therefore lies in $\mathcal{C}_{\afp}^{\gr}$ by the coherence axiom $(4b)$ in  \Cref{filteredrefinementdef}. By \Cref{forgetcompletefilt}, this implies that $F^1(\sqz(\widetilde{V})) \simeq \sqz(V)$ lies in $\mathcal{C}_{\afp}$.
\end{remark}

Using this observation, we can  establish a version of \Cref{grafpadj} in the unfiltered context: 
\begin{cons}\label{resadjCafp}   \INN{032@$\cot$} \INN{190@$\sqz$}
The $(\cot,\sqz)$-adjunction  between $\mathcal{C}$ and $\mod_{k,\geq 0} $ restricts to an \mbox{adjunction}
$$ \begin{tikzpicture}[baseline=(C.base)] 
	\node (C) at (0,0) {${\cot}: \mathcal{C}_{\afp}$};
	\node (M) at (3,0) {$  \mod^{\ft}_{k,\geq
			0}: \sqz .$};
	% adjunction arrows
	\draw[->] ([yshift=2.5pt]C.east) -- node[above] { }
	([yshift=2.5pt]M.west);
	\draw[<-] ([yshift=-2.5pt]C.east) -- node[below] { }
	([yshift=-2.5pt]M.west); 
\end{tikzpicture}$$
\end{cons}

We now wish to show that this adjunction 
satisfies the desirable conditions stated in \Cref{thm:mainaxiomatic}. 
For this, we will need a convergence criterion for cosimplicial objects in
$\mathcal{C}_{\afp}$. 
Indeed, our criterion says that if a cosimplicial object in $\mathcal{C}_{\afp}$ admits a suitable lift to $\mathcal{C}_{\afp}^{\fil}$, then 
taking the cotangent fibre commutes with totalisation. More precisely:
\begin{proposition}[Convergence criterion in $\mathcal{C}$]  
\label{convcritCafp}
Let $X^\bullet$ be a cosimplicial object of $\mathcal{C}_{\afp}$. 
Suppose that there exists a lift $\widetilde{X}^\bullet$ of $X^\bullet$ to 
$\mathcal{C}^{\fil}_{\afp}$
which satisfies the equivalent conditions of \Cref{convcritCfil} and 
moreover has complete almost finitely presented totalisation  $\mathrm{Tot}(\widetilde{X}^\bullet) \in \mathcal{C}^{\fil}_{\afp}$.

Then the limit $\mathrm{Tot}(X^\bullet)$ of $X^\bullet$ connectively exists in $\mathcal{C}$, belongs to  $\mathcal{C}_{\afp}$,  and the following map is an
equivalence:
$$\cot( \mathrm{Tot}(X^\bullet)) \rightarrow \mathrm{Tot}( \cot(X^\bullet)).$$ 
\end{proposition}

\begin{proof} 
By \Cref{convcritCfil}, the totalisation $\widetilde{X}^{-1}:=
\mathrm{Tot}(\widetilde{X}^\bullet)$ connectively exists in $\mathcal{C}^{\fil}$ and we have an equivalence
$\widehat{\cot}( \widetilde{X}^{-1}) \xrightarrow{\simeq} \mathrm{Tot}(
\widehat{\cot}(\widetilde{X}^\bullet)). $
The right adjoint $F^1:\mathcal{C}^{\fil} \rightarrow \mathcal{C}$ preserves limits, and so \mbox{$ X^{-1}:=
\mathrm{Tot}(X^\bullet)$} connectively exists in $\mathcal{C}$
and we have $X^{-1}\simeq F^1( \widetilde{X}^{-1})$. 
Our assumption  $ \widetilde{X}^{-1}=\mathrm{Tot}(\widetilde{X}^\bullet) \in \mathcal{C}^{\fil}_{\afp}$  implies that $X^{-1}\simeq  F^1( \widetilde{X}^{-1}) $ belongs to $ \mathcal{C}_{\afp}$ by \Cref{forgetcompletefilt}.

By the completeness axiom $(5a)$ of \Cref{filteredrefinementdef}, we know that   $ \cot(\widetilde{X}^i)  \xrightarrow{\simeq} \widehat{\cot}(\widetilde{X}^i)$ is an equivalence for all $i \geq -1$. \Cref{convcritCfil} therefore shows that
${\cot}(\widetilde{X}^{-1}) \simeq \mathrm{Tot}(
{\cot}(\widetilde{X}^\bullet)) $. We conclude the proof by applying $F^1$  and using the adjointability axiom $(2)$ of \Cref{filteredrefinementdef}.
\end{proof} 
We can now proceed to the proof of the main result of this axiomatic \vspace{-5pt}section:
\begin{proof}[Proof of \Cref{thm:mainaxiomatic}]
We\vspace{-2.4pt} constructed the adjunction 
\begin{tikzpicture}[baseline=(C.base)]
	\node (C) at (0,0) {$\mathcal{C}_{\afp}$};
	\node (M) at (2,0) {$\mod^{\ft}_{k,\geq 0}$};
	% adjunction arrows
	\draw[->] ([yshift=2.5pt]C.east) -- node[above] { $\cot$}
	([yshift=2.5pt]M.west);
	\draw[<-] ([yshift=-2.5pt]C.east) -- node[below] {$\sqz$}
	([yshift=-2.5pt]M.west); 
\end{tikzpicture}
in  \Cref{resadjCafp}.  To prove that this adjunction is comonadic, we will  verify the conditions of  \Cref{BBL}.

First, we check that the functor $\cot$ is conservative. Let $A \rightarrow B$ be a map in $\mathcal{C}_{\afp}$ which induces an equivalence $\cot(A)\xrightarrow{\simeq}
\cot(B)$ on  cotangent fibres. 
Then $\adic(A) \rightarrow \adic(B)$ also induces an equivalence on cotangent fibres, and hence also on
associated gradeds by \Cref{grofadicisfree}. Since $\adic(A)$ and $\adic(B)$ are both complete, 
it follows from \Cref{detect} 
that $\adic(A) \xrightarrow{\simeq} \adic(B)$ is an equivalence, and hence the same holds true for 
$A \xrightarrow{\simeq} B$.

To check the second condition of \Cref{BBL}, we fix a cosimplicial object 
$X^\bullet$ in $\mathcal{C}_{\afp}$ and assume that
$\cot(X^\bullet)$ admits a splitting in $\mod_{k, \geq 0}^{\ft}$. 
We pick the filtered\vspace{2pt} lift $\widetilde{X}^\bullet :=\adic( X^\bullet)$ of $X^\bullet$, which is a cosimplicial object in $\mathcal{C}^{\fil}_{\afp}$ by definition. \Cref{grofadicisfree},
implies that \mbox{$\gr(\widetilde{X}^\bullet) \simeq \free([\cot(X^\bullet)]_1)$} admits a splitting in $\mathcal{C}^{\gr}_{\afp}$. Using the adjointability axiom $(2)$ in \Cref{filteredrefinementdef} and \Cref{testgrcoherence}, we see that $\gr(\Tot(\forget(\widetilde{X}^\bullet)))
\simeq \forget(\Tot(\gr(\widetilde{X}^\bullet)))$ belongs to\vspace{2pt}
${\pft} \mod_{k,\geq 0}$.   \Cref{convcritCfil} therefore shows that the limit
$\widetilde{X}^{-1}:=\Tot(\widetilde{X}^\bullet)$ connectively exists in $\mathcal{C}^{\fil}$,  and that the natural  map  
\mbox{$ \widehat{\cot}(\widetilde{X}^{-1}) \xrightarrow{\simeq}
\mathrm{Tot}( \widehat{\cot}( \widetilde{X}^\bullet))$} is an equivalence. 
Since the cosimplicial diagram $\widehat{\cot}( \widetilde{X}^\bullet) \simeq (\cot(X^\bullet))_1$ is split in $\fil^{\ft}\mod_{k, \geq 0}$, it follows that $\widehat{\cot}( \widetilde{X}^{-1})$ belongs to $\fil^{\ft}\mod_{k, \geq 0}$ as well. This shows that $\gr({\cot}( \widetilde{X}^{-1}) \simeq \gr(\widehat{\cot}( \widetilde{X}^{-1}) \in \gr^{\ft}\mod_k $, which allows us to conclude that $\widetilde{X}^{-1}$ belongs to $\mathcal{C}^{\fil}_{\afp}$ (it is evidently complete as this is a limit condition).
 
The convergence criterion \Cref{convcritCafp} then implies that the limit $X^{-1} := \mathrm{Tot}(X^\bullet)$ belongs to $\mathcal{C}_{\afp}$, and that the canonical map 
$ {\cot}( {X}^{-1}) \xrightarrow{\simeq}
\mathrm{Tot}( {\cot}( {X}^\bullet))$ is an equivalence. This proves
comonadicity, i.e.\ statement (1) of the theorem. \vspace{3pt}

Before proceeding further, we record that  comonadicity implies that 
any $A \in \mathcal{C}_{\afp}$ is the totalisation of its canonical cobar resolution $\left((\sqz \circ \cot)(A) \rightrightarrows (\sqz \circ \cot) \circ (\sqz \circ
\cot)(A)  \substack{\longrightarrow{}\vspace{-3pt} \\ \longrightarrow \vspace{-3pt}  \\ \longrightarrow}\dots \right)$.  In particular, $A$ is a totalisation of a cosimplicial object in
$\mathcal{C}_{\afp}$ which at each level is square-zero.  \vspace{10pt}

We will now verify part $(2)$  of the theorem using \Cref{critextendfun}. 
For this, let $T = \mathrm{cot} \circ \sqz$ be the comonad  on
$\mod^{\ft}_{k,\geq 0}$ induced by the adjunction  and define $T^{\vee}$ as the monad induced on $\mod^{\ft}_{k,\leq 0}$ by linear duality. 
Let $V^\bullet$ be a cosimplicial object in $\mod^{\ft}_{k,\geq 0}$ which is $m$-coskeletal for some $m$, and assume that 
$V^{-1} := \mathrm{Tot}(V^\bullet)$ belongs to $\mod^{\ft}_{k,\geq 0}$. 
To prove statement $(2)$ of the theorem, we need to first verify that the following map is an equivalence:
\begin{equation} \label{auxTcottot} T(V^{-1}) \xrightarrow{\ \ \simeq\ \ } \mathrm{Tot}(
T(V^\bullet)). \end{equation} 
Via duality, this implies that $T^{\vee}$ commutes with finite coconnective
geometric realisations \mbox{in $\mod_{k,\leq 0}^{\ft}$.} To prove the equivalence \eqref{auxTcottot}  above, we apply 
\Cref{convcritCafp} 
to the cosimplicial object \mbox{$X^\bullet = \sqz( V^\bullet)$} together with its filtered lift $\widetilde{X}^\bullet = \sqz(\widetilde{V}^\bullet)$, where  $\widetilde{V}^\bullet = (\dots \to
0 \rightarrow 0 \rightarrow V^\bullet)$.
Using axioms $(2)$ and $(4b)$ in \Cref{filteredrefinementdef}, we see that the filtered lift $\widetilde{X}^\bullet$ 
is a cosimplicial object in $\mathcal{C}^{\fil}_{\afp}$ and also satisfies the other assumptions of \Cref{convcritCafp}.
It follows that 
$\cot( \mathrm{Tot}(X^\bullet)) \xrightarrow{\simeq} \mathrm{Tot}( \cot(X^\bullet))$ is an equivalence, which proves 
\eqref{auxTcottot} since $\sqz$ preserves limits. 
Moreover, if $V \in \mod^{\ft}_{k,\geq 0}$ is of finite type, then $T(V) \simeq \varprojlim_n T(
\tau_{\leq n} V)$, and the inverse limit stabilises in any finite range of homological degrees
by \Cref{truncbehaveswell}. We deduce that $T^{\vee}$ is right complete. 
Thus, it follows that the criteria of \Cref{critextendfun} are satisfied, which shows that $T^{\vee}$ admits the  sifted colimit-preserving extension $\md_k \rightarrow \md_k$ asserted  \vspace{10pt} in $(2)$.

Finally, we establish part $(3)$ of the theorem. For this, let $A, A', A'' \in \mathcal{C}_{\afp}$ and suppose that we are given maps
$A \rightarrow A', A \rightarrow A''$ which induce surjections on $\pi_0$. 
We need to show that the natural map 
$\mathfrak{D}( A) \sqcup_{\mathfrak{D}(A'')} \mathfrak{D}(A') \xrightarrow{  \simeq } \mathfrak{D}(A \times_{A'} A')$ is an equivalence.
This is easy to check if everything is square-zero.  That is,  if we are given $V, V', V'' \in \mod^{\ft}_{k,\geq 0}$ together with $\pi_0$-surjective maps $V \rightarrow V''$ and $V' \rightarrow V''$, then the following map of $T^{\vee}$-algebras is an equivalence:
$$ \mathfrak{D}( \sqz(V)) \sqcup_{\mathfrak{D}( \sqz(V''))} \mathfrak{D}( \sqz(V'))\xrightarrow{ \ \  \simeq \ \ } \mathfrak{D}( \sqz(V \times_{V''} V')).
$$
Indeed, the left-hand-side is the pushout of the free $T^{\vee}$-algebras on $V^{\vee},
V''^{\vee},$ and $V'^{\vee}$, respectively, whereas the right-hand side is the free $T^{\vee}$-algebra on $(V \times_{V''} V')^{\vee}$.
Our strategy now is to reduce the general case to the square-zero case by using 
cobar resolutions.

For this, let $X^\bullet, X'^\bullet, X''^\bullet$ be the canonical
cobar resolutions of $A, A', A''$, respectively. For example, we have
$X^0 = (\sqz \circ \cot) (A)$ and $X^1 = (\sqz \circ \cot) \circ (\sqz \circ \cot)
(A)$. Note that these are cosimplicial objects of $\mathcal{C}_{\afp}$. 
The maps $X^\bullet \rightarrow X''^\bullet, X'^\bullet \rightarrow X''^\bullet$ induce surjections on $\pi_0$   at each level, and so we can also form  the cosimplicial object in $\mathcal{C}_{\afp}$ given by $Y^\bullet := X^\bullet \times_{X''^\bullet} X'^\bullet$ by \Cref{closureCafp}.
Comonadicity implies
that 
$A \simeq  \mathrm{Tot}(X^\bullet), A' \simeq   \mathrm{Tot}(X'^\bullet), A''
\simeq  \mathrm{Tot}(X''^\bullet)$. Therefore,  we have 
$\mathrm{Tot}(Y^\bullet)\simeq   A \times_{A'}
A''$. Since we have already verified claim $(3)$ in the case of square-zero extensions, we have the following equivalence of $T^{\vee}$-algebras
for all $i\geq 0$:
$$  \mathfrak{D}(X^i) \sqcup_{\mathfrak{D}(X''^i)} \mathfrak{D}(X'^i) \xrightarrow{\ \  \simeq\ \ }  \mathfrak{D}(Y^i) .$$
To deduce that 
$ \mathfrak{D}( A) \sqcup_{\mathfrak{D}(A'')} \mathfrak{D}(A')\simeq \mathfrak{D}(A \times_{A'} A'')  $, it therefore 
suffices to verify the \mbox{following facts:}
\begin{enumerate}[a)]
\item $| \mathfrak{D}(X^\bullet)|\simeq  \mathfrak{D}(A)  $ and $ | \mathfrak{D}(X'^\bullet)|\simeq \mathfrak{D}(A') $ and $ | \mathfrak{D}(X''^\bullet)| \simeq \mathfrak{D}(A'')  $;
\item $| \mathfrak{D}(Y^\bullet)| \simeq  \mathfrak{D}(A \times_{A'} A'')  $. 
\end{enumerate}
Geometric realisations in $T^{\vee}$-algebras can be computed
in $\mod_k$ as $T^{\vee}$ \mbox{preserves sifted colimits.}

Claim  $(a)$ follows immediately by applying linear duality to the comonadicity  established above. 

Claim $(b)$ will follow  by applying  the  convergence criterion established in  \Cref{convcritCfil}.

First, we form the cosimplicial object $\widetilde{Y}^\bullet := \displaystyle \adic(X^\bullet) \times_{\adic(X''^\bullet)}
\adic(X'^\bullet) $, which is a filtered lift of $Y^\bullet$ to $\mathcal{C}^{\fil}_{\afp}$ by \Cref{closureCafp}.

Second, we will check that the totalisation of $\gr( \widetilde{Y}^\bullet)\in \mathcal{C}^{\gr}$ belongs to $\pft \md_{k,\geq 0}$.
For this, we first observe that by construction, the cosimplicial objects $\cot(X^\bullet)$, $\cot(X'^\bullet)$, $\cot(X''^\bullet)$ are all split  in $\mod_{k,\geq 0}^{\ft}$.  
\Cref{grofadicisfree} then shows that  $\gr( \adic(X^\bullet)), \gr( \adic(X'^\bullet))$, and $\gr(\adic(X''^\bullet)$ have splittings and therefore admit totalisations in $\pft \md_{k,\geq 0}$. Since the maps between these objects are levelwise  surjective on $\pi_0$,   we see that $\gr( \widetilde{Y}^\bullet)$ also admits a totalisation in $\pft \md_{k,\geq 0}$. 

Third, we need to show that $\Tot( \widetilde{Y}^\bullet) $ admits a totalisation in $\mathcal{C}_{\afp}^{\fil}$.
For this, we first observe that the map $\adic(A) \rightarrow \mathrm{Tot}(\adic(X^\bullet))$ induces an equivalence. By completeness, it suffices to check this after applying $\gr$. Here, it is true because the functor 
$\gr \circ \adic$ is equivalent to $\free \circ [-]_1$ by  \Cref{grofadicisfree}, the natural map 
$\cot(A)\xrightarrow{\simeq}\mathrm{Tot}( \cot(X^\bullet))$ is an equivalence by the construction of the cobar resolution, and the functor 
 $\forget \circ\free $ is admissible by axiom $(3)$ of \Cref{filteredrefinementdef}.
A similar argument gives equivalences $\adic(A') \rightarrow \mathrm{Tot}(\adic(X'^\bullet))$ and $\adic(A'') \rightarrow \mathrm{Tot}(\adic(X''^\bullet))$.
We deduce $\Tot( \widetilde{Y}^\bullet) 
\simeq  \displaystyle \Tot(\adic(X^\bullet)) \times_{\Tot(\adic(X''^\bullet))}\Tot(\adic(X'^\bullet)) \simeq  \adic(A)
\times_{\adic(A'')} \adic(A''),$ which belongs to $\mathcal{C}^{\fil}_{\afp}$ by \Cref{closureCafp}.

We can now apply \Cref{convcritCafp} to conclude that $\cot(A \times_{A''} A') \cong \cot(\Tot(Y^\bullet))  \xrightarrow{\simeq} \Tot(\cot(Y^\bullet))$ is an equivalence. Using that $\cot(Y) \in \mod_{k,\geq 0}^{\ft}$ is of finite type and therefore equivalent to its own bidual, we can therefore apply duality and deduce that
the natural map  $|\mathfrak{D}(Y^\bullet)| \xrightarrow{\simeq} \mathfrak{D}(A \times_{A'} A'') $ is an equivalence.
This completes the verification of claim $(b)$ above.
\end{proof}

\subsection{Deformation theories}
We shall now explain how \Cref{thm:mainaxiomatic} translates into the
language of deformation theories studied in \cite[Ch. 12]{lurie2016spectral} or
\cite{lurie2011derivedX}.

As before, we fix a filtered augmented monadic adjunction (cf.\ \Cref{filteredrefinementdef}) \mbox{throughout.
Write} $T^\vee$ for the monad on $\mod_k$  constructed in \Cref{thm:mainaxiomatic} (2), i.e.\ the unique sifted-colimit-preserving extension 
of the monad $ M\mapsto (\cot \sqz (M^\vee))^{\vee} $ acting on $\mod_{k,\leq 0}^{\ft}$. Our  aim in this section is to prove \Cref{fmpequiv}, which asserts that $\mathcal{C}$-based formal moduli problems are equivalent to $T^\vee$-algebras.

We begin by constructing the required deformation functor. First, observe that 
composing the
$(\cot \dashv \sqz)$-adjunction with linear duality in fact  \vspace{-2pt}  gives rise to an adjunction
\begin{equation} 
 \begin{tikzpicture}[baseline=(C.base)] \label{dualadjunction}
	\node (C) at (0,0) {$\cot^{\vee}: \mathcal{C}_{\afp}$};
	\node (M) at (4,0) {$  \mod^{\ft}_{k,\geq
			0}: \sqz \circ  \tau_{\geq 0}\circ (-)^{\vee} .$};
	% adjunction arrows
	\draw[->] ([yshift=2.5pt]C.east) -- node[above] { }
	([yshift=2.5pt]M.west);
	\draw[<-] ([yshift=-2.5pt]C.east) -- node[below] { }
	([yshift=-2.5pt]M.west); 
\end{tikzpicture}\end{equation}

Its left adjoint sends $A \in \mathcal{C}$ to $\cot(A)^{\vee}$, i.e.\ the
linear dual of the cotangent fibre, whereas its right adjoint maps $V \in \mod_k$
to the trivial square-zero extension on $\tau_{\geq 0}(V^{\vee})$. \vspace{5pt}

\newcommand{\wafp}{\mathrm{wafp}}
We can extend  the functor $\mathfrak{D}: \mathcal{C}_{\afp} \rightarrow \mathrm{Alg}_{T^{\vee}}^{op}$ from  \Cref{thm:mainaxiomatic} to all of $\mathcal{C}$ in such a way that postcomposing with the forgetful
functor to $\md_k$ recovers $\cot^{\vee}$. Let $\mathcal{C}_{\wafp} \subset \mathcal{C}$ be
the full subcategory of all \INN{030@$\mathcal{C}_{\wafp}$}
$A$ with $\cot(A) \in \mod^{\ft}_{k,\geq 0}$. 
\begin{cons}[The Koszul duality adjunction]\label{theadjunction}
Since the action of $T^\vee$  on $\mod_{k,\leq 0}^{\ft}$
agrees with the monad induced  by adjunction (\ref{dualadjunction}), we have a natural functor 
$$\mathfrak{D}: \mathcal{C}_{\wafp} \rightarrow \mathrm{Alg}_{T^{\vee}}^{op} \vspace{-2pt}  $$
which forgets to $\cot^{\vee}$ in $\md_{k,\leq 0}$. 
Observe that this functor is left Kan extended from the compact objects of
$\mathcal{C}$, since $\cot^{\vee}: \mathcal{C} \rightarrow \md_k^{op}$ has this
property. 

We can left Kan extend further to $\mathcal{C}$ to finally obtain the deformation functor\vspace{-2pt} 
$$ \mathfrak{D}: \mathcal{C} \longrightarrow \mathrm{Alg}_{T^{\vee}}^{op}.  $$

By \cite[Corollary 4.1.4]{nguyen2020adjoint}, the functor $\mathcal{D}$ admits a right adjoint 
\INN{032@$C^*$}
$$C^*: \mathrm{Alg}_{T^{\vee}}^{op} \longrightarrow \mathcal{C}.\vspace{-20pt}$$ \label{KDadj}
\end{cons}
\begin{remark} \INN{040@$\mathfrak{D}$}
It should not be a surprise that we can extend $\mathfrak{D}: \mathcal{C}_{\afp} \rightarrow \mathrm{Alg}_{T^{\vee}}^{op}$ to all of $\mathcal{C}$. Indeed, 
writing $\widetilde{T^\vee}$ for the monad on $\mod_k$ associated with the above adjunction (\ref{dualadjunction}), the monad 
$T^\vee$ is defined as the unique sifted-colimit-preserving extension of the restriction $\widetilde{T^\vee}|_{\mod_{k,\leq 0}^{\ft}}$.
Since this extension is obtained by left Kan extension, there is a natural transformation of monads $T^\vee \rightarrow \widetilde{T^\vee}$; we may think of $T^\vee$ as an ``uncompletion" of $\widetilde{T^\vee}$.

The   functor $\mathfrak{D}: \mathcal{C} \rightarrow \mathrm{Alg}_{T^{\vee}}^{op}$ is then simply obtained as the composite
$\mathcal{C} \longrightarrow \mathrm{Alg}_{\widetilde{T^{\vee}}}^{op}\longrightarrow \mathrm{Alg}_{{T^{\vee}}}^{op}.$
Here the first map comes from adjunction \eqref{dualadjunction}, whereas  the second uses  \mbox{the map of monads $T^\vee \rightarrow \widetilde{T^\vee}$.}
\end{remark}

Construction~\ref{KDadj} factors
\eqref{dualadjunction} through the free-forgetful 
adjunction 
$\forget: \mathrm{Alg}_{T^{\vee}}^{op} \rightleftarrows \md_k^{op} : \free$.
Unwinding the above definitions, we can observe the following natural equivalences:
\begin{gather}  C^*( \free_{T^{\vee}}(V)) \simeq  \sqz( \tau_{\geq 0} V^{\vee}
), \quad V \in \md_k \\
\mathfrak{D}( \free(W)) \simeq W^{\vee}, \quad W \in \md_{k,\geq 0} , \label{Doffree} \\
\mathfrak{D}( \sqz(W)) \simeq \free_{T^{\vee}}(W^{\vee}), \quad W \in \mod^{\ft}_{k,\geq 0}.
\end{gather}

The statement \eqref{Doffree} is to be interpreted as on the level of objects of
$\md_k$; 
informally the $T^{\vee}$-algebra-structure should be square-zero, but we do not
attempt to make this precise.

Combining these basic facts with \Cref{thm:mainaxiomatic}, we can conclude that the adjunction
 $(\mathfrak{D} \dashv C^*)$ restricts to a pair of inverse equivalences between 
$\mathcal{C}_{\afp}$ and $\mathrm{Alg}_{T^{\vee}}(\mod^{\ft}_{k,\leq 0})^{op}$:
\begin{proposition}  Let $\mathcal{C}, \mathcal{C}_{\afp}, \mathfrak{D}, C^\ast, \ldots$ be defined as above.
\label{KDequivalences}
\begin{enumerate}
\item  
Given any $A \in \mathcal{C}_{\afp}$, the natural map $A \rightarrow C^*( \mathfrak{D}( A))$ is an
equivalence. 
\item
Given  any $T^{\vee}$-algebra $\mathfrak{g}$ such that the underlying $k$-module belongs
to $\mod^{\ft}_{k,\leq 0}$, the natural map $\mathfrak{g} \rightarrow \mathfrak{D}(C^*(\mathfrak{g}))$
is an equivalence. 
\end{enumerate}
\end{proposition} 
\begin{proof} 
If $A = \sqz(W)\in \mathcal{C}_{\afp}$ is a trivial square-zero extension on some $W \in \mod^{\ft}_{k,\geq 0}$, the first claim follows from observations $(6)$ and $(8)$ above. Given a general $A\in \mathcal{C}_{\afp}$, the 
comonadicity claim in \Cref{thm:mainaxiomatic}$(1)$ shows that $A$
 can be written as a totalisation of a cosimplicial object $A^\bullet$ in $\mathcal{C}$ consisting of square-zero extensions, and that moreover $\mathfrak{D}$ preserves this totalisation (i.e.\ carries it to a geometric realisation of $T^{\vee}$-algebras). Of course, the right adjoint $C^*$ also preserves all totalisations. Claim $(1)$ therefore follows.

For claim $(2)$, we use that by \Cref{thm:mainaxiomatic}$(1)$,  any  $T^{\vee}$-algebra $\mathfrak{g}$ with underlying $k$-module in $\mod^{\ft}_{k,\leq 0}$
can be written as  $\mathfrak{g}= \mathfrak{D}( A)$ for some $A \in\mathcal{C}_{\afp}$. The statement then follows from  $(1)$.
\end{proof}

\begin{proposition} 
\label{pushoutproperty}
Let $\mathfrak{g}, \mathfrak{g}', \mathfrak{g}'' \in \mathrm{Alg}_{T^{\vee}}(\mod_{k,\leq 0}^{\ft})$ be $T^\vee$-algebras with underlying module in $\mod_{k,\leq 0}^{\ft}$ and suppose that we are given maps $\mathfrak{g}'' \rightarrow \mathfrak{g}, \mathfrak{g}'' \to
\mathfrak{g}'$ which induce injections on $\pi_0$. 
 
Then the pushout $\displaystyle \mathfrak{g} \sqcup_{\mathfrak{g}''} \mathfrak{g}'$ (computed
in $T^{\vee}$-algebras) has underlying $k$-module in $\mod^{\ft}_{k,\leq 0}$ \mbox{as well.}
\end{proposition} 
\begin{proof} 
Define $A, A', A''\in \mathcal{C}_{\afp}$  as $C^*(\mathfrak{g}), C^*(\mathfrak{g}'),$ and $C^*(\mathfrak{g}'')$. 
We now observe that the induced maps $A \rightarrow A'', A' \rightarrow A''$ induce surjections on $\pi_0$. 
Indeed, this follows immediately from   the $\adic$ filtration.
For example, $A$ has a complete filtration 
with associated graded given by $\free( [\mathfrak{g}^{\vee}]_1)$, where
$[\mathfrak{g}^{\vee}]_1$ is the graded $k$-module with $\mathfrak{g}^{\vee}$ in internal
degree $1$ (cf.\ \Cref{grofadicisfree}). 
The maps 
$\free( [\mathfrak{g}^{\vee}]_1) \rightarrow 
\free( [\mathfrak{g}''^{\vee}]_1)$ 
and 
$\free( [\mathfrak{g}'^{\vee}]_1) \rightarrow 
\free( [\mathfrak{g}''^{\vee}]_1)$ then
induce surjections on $\pi_0$, and passing to the filtered objects shows that
$A \rightarrow A'', A' \rightarrow A''$ have the same property. 

 \Cref{thm:mainaxiomatic}$(3)$ now   shows that 
$\mathfrak{D}( A \times_{A''} A') \simeq \mathfrak{g} \sqcup_{\mathfrak{g}''} \mathfrak{g}
$ has the asserted properties. 
\end{proof}

We are now ready to show that one can obtain a deformation theory in the sense
of Lurie. 
Recall the following definition from 
\cite[Definition 1.3.1, 1.3.9]{lurie2011derivedX} (see also \cite[Section 2]{CalaqueGrivaux} for a
treatment). Roughly, it expresses the idea of a bar-cobar duality, together with
suitable subcategories on which one obtains an inverse equivalence, albeit
translated into \mbox{more abstract language. }

\begin{definition}[Lurie]  \label{defth} 
A \emph{deformation theory} consists of a presentable $\infty$-category
$\mathcal{A}$, a set of objects  $ \left\{E_\alpha\right\}_{\alpha \in T}$ in the
stabilisation of $\mathcal{A}$, and an adjunction 
$\mathfrak{D}: \mathcal{A} \rightleftarrows \mathcal{B}^{op}: C^*$ with $\mathcal{B}$ presentable. 
Moreover, we require that    there exists a full subcategory $\mathcal{B}_0 \subset
\mathcal{B}$ satisfying the following conditions:
\begin{enumerate}
\item  for $B \in \mathcal{B}_0$, the natural map $B \rightarrow \mathfrak{D} C^*( B)$ in
$\mathcal{B}$  is an equivalence; 
\item the subcategory $\mathcal{B}_0 \subset \mathcal{B}$ contains 
the initial object $\emptyset$ of $\mathcal{B}$. Moreover, for any $\alpha \in T$ and $n \geq 1$, there is an object $K_{\alpha, n}\in \mathcal{B}_0$ such
that
$\Omega^{\infty -n} E_\alpha =  C^*( K_{\alpha, n})$; 
\item 
given an object  $K \in \mathcal{B}_0$ and maps $K_{\alpha, n} \rightarrow  K$ and $K_{\alpha, n} \to
\emptyset$, the pushout $K \sqcup_{K_{\alpha, n}} \emptyset$ (computed in $\mathcal{B}$) is
contained in $\mathcal{B}_0$;
\item for each $\alpha \in T$ 
and any $n \geq 2$, assumptions (a) through (c) above imply 
equivalences $\Sigma K_{\alpha, n} \simeq K_{\alpha, n-1}$.  
Using these equivalences, we can define a functor $f_\alpha: \mathcal{B} \rightarrow \mathrm{Sp}$ 
with $$\Omega^{\infty - n} f_\alpha(X) =  \hom_{\mathcal{B}}(K_{\alpha, n}, X).$$ 
We then assume that each functor $f_\alpha$ is conservative and commutes with sifted colimits. 
\end{enumerate} 
\end{definition} 

We now wish to define a deformation theory in the above sense from  the filtered augmented monadic adjunction (cf.\ \Cref{filteredrefinementdef}) fixed throughout this section. 

For this, we consider the   adjunction $\mathfrak{D} : \mathcal{C} \rightleftarrows
\mathrm{Alg}_{T^{\vee}}^{op} : C^\ast$ constructed in the beginning of this subsection.
Observe that since
$\sqz$ is a right adjoint functor $\md_{k,\geq 0} \rightarrow \mathcal{C}$,  it
naturally lifts to the stabilisation \mbox{of $\mathcal{C}$:} given any $V
\in \md_{k,\geq 0}$,  the object $\sqz(V)$ defines an object of
$\mathrm{Stab}(\mathcal{C})$ corresponding to the sequence
$\left\{ \sqz( V[n])\right\}_{n \geq 0}$. 

\begin{proposition} 
\label{isdeformationtheory}
The $\infty$-category $\mathcal{C}$, 
the infinite loop object $ \{\sqz( k[n])\}_{n\geq 0}$   in  $\mathrm{Stab}(\mathcal{C})$, 
and the functor $\mathfrak{D}^{op}: \mathcal{C}^{op}
\rightarrow \mathrm{Alg}_T$ together define a deformation theory in the sense of \Cref{defth}. \end{proposition} 
\begin{proof} 
We define $\mathcal{B}_0 \subset \mathcal{B} = \mathrm{Alg}_{T^{\vee}}$ as the
subcategory of those objects whose underlying $k$-module belongs to
$\mod^{\ft}_{k,\leq 0}$.  We let $K_{n}$ be the free $T^{\vee}$-algebra on $k[-n]$; 
recall that $C^*(K_n) \simeq \sqz(k[n])$ and $\mathfrak{D}( \sqz(k[n])) \simeq K_n$ by
construction of these adjunctions. 
Assumptions 
(a) through (c) now follow from \Cref{KDequivalences} and \Cref{pushoutproperty}.
Assumption (d) follows because $T^{\vee}$ commutes with sifted colimits by construction, and so 
the forgetful functor from $T^{\vee}$-algebras to $\md_k$ (which is the functor
$f_\alpha$ described above) also commutes with sifted
colimits. 
\end{proof} 
We can now deduce the  classification of formal moduli problems through $T^\vee$-algebras which was asserted in the beginning of this section: 
\begin{proof}[Proof of \Cref{fmpequiv}]
Indeed, this follows  from \cite[Theorem 12.3.3.5]{lurie2016spectral}, 
since we know that we have a deformation theory by \Cref{isdeformationtheory}. 
\end{proof}

\newpage

\newcommand{\nou}{\mathrm{nu}}
\newcommand{\enu}{\mathbb{E}_\infty^{\nou}}

\section{Deformations over a   field}
\label{einfsec}
In this section we consider some concrete examples of the general argument in
the previous section. In particular, we verify that one can apply the argument
for connective $\einf$-algebras or simplicial commutative rings augmented over a
field. 

\subsection{$\einf$-algebras}
\label{einfdefsectiona}
We begin by examining deformations parametrised by $\EE_\infty$-algebras. 

\subsubsection*{Preliminaries on $\einf$-algebras}
To set the stage, we will  briefly review some of the basic facts about $\einf$-algebras; a 
comprehensive treatment of this theory can be found in \cite{lurie2014higher}.

Let $(\mathcal{A},\otimes, \mathbf{1})$ be a presentably symmetric monoidal stable $\infty$-category.

\begin{enumerate}
\item  
We can associate the $\infty$-category $\clg( \mathcal{A})$ of
\emph{$\einf$-algebras in $\mathcal{A}$}, and there is a natural free-forgetful adjunction \INN{060@$\free$}
\INN{060@$\free^{\nou}$}
\[
\begin{tikzcd}
	\free: \mathcal{A}
	\arrow[r, shift left=0.6ex]
	& \clg(\mathcal{A}) : \forget
	\arrow[l, shift left=0.6ex]
\end{tikzcd}
\]
The free $\einf$-algebra on an object $X \in \mathcal{A}$ is given by  
$\displaystyle\free(X) \simeq \bigoplus_{n \geq 0} (X^{\otimes n})_{h \Sigma_n}$.

\item
We write $\clg^{\aug}(\mathcal{A}) = \clg(\mathcal{A})_{\mathbf{1}// \mathbf{1}}$ \INN{032@$\clg^{\aug}(\mathcal{A})$, $\clg^{\nou}(\mathcal{A})$}
for the 
$\infty$-category of \emph{augmented $\einf$-algebras in $\mathcal{A}$}; its objects are  $\einf$-algebras $A$  in $\mathcal{A}$ equipped with an augmentation  map $A \rightarrow \mathbf{1}$ to the unit. 
\item  
Let $\clg^{\nou}(\mathcal{A})$ denote the $\infty$-category of
\emph{nonunital} $\einf$-algebras in $\mathcal{A}$.  Since $\mathcal{A}$ is assumed to be stable, we 
have an equivalence $\clg^{\aug}(\mathcal{A}) \simeq
\clg^{\nou}(\mathcal{A})$ which sends an augmented $\einf$-algebra $A$ to the
fibre $\mathfrak{m}_A$ of the augmentation  map $A \rightarrow \mathbf{1}$. 
As expected, there is a free-forgetful adjunction 
\[
\begin{tikzcd} \label{freeforgetnu} 
	\free^{\nou}: \mathcal{A}
	\arrow[r, shift left=0.6ex]
	& \clg^{\nou}(\mathcal{A}) : \forget ,
	\arrow[l, shift left=0.6ex]
\end{tikzcd}
\]

and the free nonunital algebra on an object $X \in \mathcal{A}$ is given by 
$\displaystyle \free^{\nou}(X) \simeq  \bigoplus_{n > 0} (X^{\otimes n})_{h
\Sigma_n}.$
\item Since the monad $ \forget  \circ  \free^{\nou}$ is naturally augmented over the identity monad, and this gives rise to an adjunction  \INN{032@$\cot$} \INN{190@$\sqz$} 
\[
\begin{tikzcd}\label{stablecotsqz}
	\cot: \clg^{\nou}(\mathcal{A})
	\arrow[r, shift left=0.6ex]
	&  \mathcal{A} : \sqz  ,
	\arrow[l, shift left=0.6ex]
\end{tikzcd}
\]

where the functor $\sqz$ sends an object of $\mathcal{A}$ to the associated  nonunital $\einf$-algebra with
square-zero multiplication and the left adjoint $\cot$ is called the
\emph{cotangent fibre}. 

Under the identification
$\clg^{\aug}(\mathcal{A}) \simeq \clg^{\nou}( \mathcal{A})$,
we have an equivalence
$$ \cot(A) \simeq \Omega L_{\mathbf{1}/A},$$
where $L_{-/A}$ \INN{120@$L_{B/A}$} denotes the cotangent complex 
 of an
$\mathbb{E}_\infty$-$A$-algebra in $\mathcal{A}$ (cf. 
\cite[Section 7.3--7.4]{lurie2014higher} or in the original setting
\cite{MR1732625}). 
 Alternatively, we have   $\cot(A) \simeq \mathbf{1}\otimes_A L_{A/\mathbf{1}}. $
\end{enumerate}

\begin{remark} 
The definition of nonunital $\einf$-rings and the construction of the cotangent fibre adjunction \eqref{stablecotsqz}
do not require the unit in $\mathcal{A}$; they therefore both make sense for 
nonunital symmetric monoidal $\infty$-categories. This is relevant for us as we will later study  
the $\infty$-category $\fil \md_k$ of $k$-modules filtered by positive integers (cf.\ \Cref{filtereddef}). 
\end{remark} 

Finally, we will review the adic filtration and how it allows us to approximate every augmented $\einf$-algebra by extended powers of its cotangent fibre. 
This is also
discussed in \cite[Section 4.2]{GaitsgoryLurie2019}, and in fact a 
 special case of the homotopy completion tower studied in
\cite{HarperHess}.

\begin{cons}[The functor $\adic$ and the completion tower]\INN{1@$\adic$} 
\label{consadicforeinfinitygen} 
As explained in \Cref{filmonoidal}, the functor 
$ F^1: \fil(\mathcal{A}) \rightarrow \mathcal{A}$
is (nonunital) symmetric monoidal. It therefore lifts to  a functor on algebras
$\clg^{\nou}( \fil(\mathcal{A})) \rightarrow \clg^{\nou}(\mathcal{A})$, and this functor preserves limits (as these are computed on underlying $\infty$-categories.

Taking the left adjoint now gives rise to a functor
$$ \adic: \clg^{\nou} (\mathcal{A}) \rightarrow   \clg^{\nou}( \fil(\mathcal{A})).  $$  
This construction refines any nonunital commutative algebra object
in $\mathcal{A}$ to a filtered one, and we therefore  get a natural tower. 
More explicitly, we see that $\adic$ carries 
the free nonunital $\mathbb{E}_\infty$-algebra $\bigoplus_{i >0 } (V^{\otimes i})_{h \Sigma_i}$ on $V \in \mathcal{A}$  to the 
free filtered nonunital $\mathbb{E}_\infty$-algebra on the filtered object
$( \dots \rightarrow  0 \rightarrow 0 \rightarrow V)$. Unwinding the definitions of the tensor product
in $\fil(\mathcal{A})$, we see that 
for each $n$, there is an equivalence
\begin{equation} 
\label{adicexpress} 
F^n \adic( \free(V)) \simeq  \bigoplus_{i \geq n} (V^{\otimes i})_{h \Sigma_i}.
\end{equation}

\end{cons}

\begin{example} 
Let $\mathcal{A}  = \md_k$. Let $I \in \clg^{\nou}( \md_k)$ be a connective 
nonunital $\einf$-algebra, so that $\pi_0(I)$ is an ordinary nonunital $k$-algebra. 
Then the  image of $\pi_0(F^n \adic(I)) \rightarrow \pi_0( I)$ is given by the $n$th power
ideal of $\pi_0(I)$. This is evident in the free case, and the general case follows by taking sifted colimits. 
\end{example}

We return to the general case where $\mathcal{A}$ is any presentably symmetric monoidal stable $\infty$-category.
\begin{remark}[The cotangent fibre in degree $1$] 
\label{cotindeg1}
For any $A \in \clg^{\nou}(\gr(\mathcal{A}))$, the natural map
$A \rightarrow \cot(A)$ in $\gr(\mathcal{A})$  
induces 
an equivalence in (internal) degree $1$. 
Indeed, this follows by considering the free case and then observing that everything
commutes with sifted colimits. 
We  obtain \mbox{a natural
map} $$\free^{\nou}( [\cot(A)_1]_1) \rightarrow A$$
in $\clg^{\nou}( \gr(\mathcal{A}))$, 
where  $[\cot(A)_1]_1$ denotes the graded \INN{000@$[-]_1$}
object  given by placing $\cot(A)_1$  in \mbox{degree $1$.}
\end{remark} 

We will now verify some basic properties of   the construction $\adic$. 
\begin{proposition} 
\label{adiceinfgeneralprop}
For any nonunital $\EE_\infty$-algebra $A \in \clg^{\nou}(\mathcal{A})$, we have:
\begin{enumerate}
\item  the natural unit map $A \rightarrow F^1 \adic(A)$ is an equivalence; 
\item \mbox{there is a natural equivalence $\cot( \adic(A)) \simeq (\cot(A))_1 = ( \dots   \rightarrow 0
\rightarrow \cot(A))$ in $\fil(\mathcal{A})$;}
\item 
there is a functorial identification 
$\free^{\nou}( [\cot(A)]_1) \simeq \gr \circ \adic (A)
$ in $\clg^{\nou}( \gr(\mathcal{A}))$, where $[\cot(A)]_1$ denotes the graded \INN{000@$[-]_1$} object obtained by placing $\cot(A)$  
in  degree $1$. 
\end{enumerate}
\end{proposition}  
\begin{proof} 
If $A$ is free, then $(1)$ follows by our explicit computation in (\ref{adicexpress}). From this, we can deduce the general case by observing that both $F^1$ and $\adic$ preserve geometric realisations.
 
For $(2)$, we observe that the following square of right adjoints evidently commutes: 
\[
\begin{tikzcd}
	\fil(\mathcal{A}) 
	\arrow[r, "F^1"] 
	\arrow[d, "\sqz"]
	& \mathcal{A} 
	\arrow[d, "\sqz"] \\
	\clg^{\nou}(\fil(\mathcal{A})) 
	\arrow[r, "F^1"]
	& \clg^{\nou}(\mathcal{A})
\end{tikzcd}
\]
 
Finally, statement $(3)$ follows (just like \Cref{grofadicisfree} above) by directly checking   the free case
and then taking geometric realisations.  
\end{proof}

\subsubsection*{The setup for $\mathbb{E}_\infty$-algebras} \label{setupeinf}
We shall now  define a filtered augmented monadic adjunction (cf.\ \Cref{filteredrefinementdef}) for $\EE_\infty$-algebras over a given field $k$.  By our previous work, this will allow us to deduce a version of \Cref{thm:mainaxiomatic} and thereby give a Lie algebraic description of deformation theory in this context.

We write $\clg_k$ \INN{032@$\clg_k$,  $\clg_{k}^{\nou}$,  $\mathcal{C} = \clg^{\nou}_{k, \geq 0}$}  
for  the $\infty$-category of $\einf$-$k$-algebras and
$\clg_{k}^{\nou}$ for its nonunital version. 
We let $\mathcal{C} = \clg^{\nou}_{k, \geq 0}$ denote the full subcategory of  connective objects in
$\clg^{\nou}_{k} \simeq
\clg^{\aug}_k$.  
\begin{remark} 
For simplicity, we will generally state our results in terms of \emph{augmented}
(rather than nonunital) algebras in this section. 
\end{remark}

\begin{cons}[The setup for $\mathbb{E}_\infty$-algebras]
\label{setupeinfinity}Let $k$ be a field.
\begin{enumerate}[a)]
\item  
Let $\mathcal{C} = \clg^{\aug}_{k, \geq 0} \simeq \clg^{\nou}( \md_{k,\geq 0})$   be the $\infty$-category of
augmented (or equivalently nonunital) connective $\einf$-algebras over $k$. 
\INN{030@$\mathcal{C}^{\fil}$}
\item Let $\mathcal{C}^{\fil} = \clg^{\nou}( \fil \md_{k,\geq 0})$ be the $\infty$-category of nonunital $\EE_\infty$-algebras in the nonunital symmetric monoidal $\infty$-category $\fil \md_{k,\geq 0}$. Note that $\mathcal{C}^{\fil} $ is equivalent to the full subcategory of $\clg^{\aug}( \fil^+
\md_{k,\geq 0})$ spanned by \INN{030@$\mathcal{C}^{\gr}$}all augmented $\einf$-algebra objects $A$ with $F^0 A/F^1 A \simeq k$. 

\item Let $\mathcal{C}^{\gr} = \clg^{\nou}( \gr
\md_{k,\geq 0})$ denote the $\infty$-category  of nonunital $\EE_\infty$-algebras in the nonunital symmetric monoidal $\infty$-category 
$\clg^{\nou}( \gr
\md_{k,\geq 0})$. 
Equivalently, $\mathcal{C}^{\gr}$ is the full subcategory of $\clg^{\aug}( \gr
\md_{k,\geq 0})$ spanned by those objects $A_{\star}$ such that $A_0 \simeq k$.

\item 
We obtain the free-forgetful adjunction 
$\free:  \md_{k,\geq 0} \rightleftarrows \mathcal{C} : \forget$ 
and the cotangent fibre adjunction $\cot : \mathcal{C} \rightleftarrows
\md_{k,\geq 0}:  \sqz$
from \eqref{freeforgetnu} and \eqref{stablecotsqz} above by restricting to connective objects.
Note that if we describe $\mathcal{C}$ as augmented $\EE_\infty$-algebras, then the forgetful
functor $\forget$ sends an augmented $\einf$-algebra to its augmentation ideal. 

Taking $\mathcal{A}  = \fil \md_k$ (or $\mathcal{A}  = \gr \md_k$) instead, we obtain  a similar  \INN{060@$\free$} \INN{060@$\forget$} pair of adjunctions $(\cot \dashv \sqz), (\free  \dashv \forget)$ between  $\mathcal{C}^{\fil}$ and $\fil \md_{k,\geq 0}$ (or $\mathcal{C}^{\gr}$ and $\gr\md_{k,\geq 0}$).
\item The functor \INN{060@$F^1$} \INN{1@$\adic$}\INN{032@$\cot$} \INN{190@$\sqz$}
$F^1: \mathcal{C}^{\fil} \rightarrow \mathcal{C}$ forgets the filtration on an object;  its left
adjoint   is the functor $\mathrm{adic}$ (cf.
Construction~\ref{consadicforeinfinitygen}). 
\item The (nonunital) symmetric monoidal functor $\gr: \fil \md_{k,\geq
0} \rightarrow \gr \md_{k,\geq 0}$ 
induces a functor $\gr: \mathcal{C}^{\fil} \rightarrow \mathcal{C}^{\gr}$ on the level of algebras.

\end{enumerate}
\end{cons}

\begin{proposition}\label{onetothree}
The setup of connective $\mathbb{E}_\infty$-algebras  in 
Construction~\ref{setupeinfinity} satisfies conditions $(1)-(3)$ of
\Cref{filteredrefinementdef}. 
\end{proposition} 
\begin{proof} 
Conditions $(1)$ and $(2)$ of \Cref{filteredrefinementdef} both follow from straightforward formal arguments.
For $(3)$, we first observe that the free nonunital $\einf$-algebra functor
$\gr  \md_{k, \geq 0} \rightarrow \gr \md_{k, \geq 0}$ is given by $X \mapsto \bigoplus_{i > 0} (X^{\otimes i})_{h \Sigma_i}$ and therefore admissible by  \Cref{admissibleexample}.
 
To construct the filtration required in condition~\eqref{grtower}, we use the adic filtration.
Given $X \in \clg^{\nou}( \gr \md_{k,\geq 0})$, 
we  form $\adic(X) \in \clg^{\nou}( \fil (\gr \md_k))$ and set \mbox{$X^{(i)}  = F^i\adic(X) \in \gr \md_k$.}
It follows that we have the tower $\left\{X^{(i)}\right\}_{i\geq 1}$, naturally in $X$,
and natural isomorphisms
$$X^{(i)}/X^{(i+1)} \simeq (\cot(X)^{\otimes i})_{h \Sigma_i}$$ in $\gr \md_k$. 
By taking  simplicial resolutions and thereby reducing to the free case, it follows that $X^{(i)}$
is concentrated in internal degrees $\geq i$ for any $X$. We deduce that the tower 
$\left\{X^{(i)}\right\}_{i\geq 1}$ converges in $\gr \md_k$. Setting $A^{(n)}=X/X^{(n+1)}$ gives the required filtration  in condition~\eqref{grtower}.
\end{proof}

\subsubsection*{Finiteness conditions}
In order to verify the remaining axioms $\eqref{coherencehyp}$ and $\eqref{goodfc}$ of \Cref{filteredrefinementdef}, we will need to exploit the Noetherian and 
finiteness properties of $\einf$-ring spectra; we refer to \cite[Chapter 7]{lurie2014higher} or \cite[Chapter II.4]{lurie2016spectral} for a detailed study of these notions.

\begin{definition}[{cf.\ \cite[Definition 7.2.4.30]{lurie2014higher}}]  \label{def:noeth} 
An $\einf$-ring spectrum $R$ is   \emph{Noetherian} if
\begin{enumerate}
\item $R$ is connective;
\item $\pi_0(R)$ is Noetherian;
\item for each $i \geq 0$, $\pi_i(R)$ is a finitely generated $\pi_0(R)$-module. 
\end{enumerate}
\end{definition} 

We will also need the following notion: 
\begin{definition}\label{Eafp}
An $\EE_\infty$-$k$-algebra $R$ is said to be   \emph{almost finitely presented}   if $R$ is
Noetherian and $\pi_0(R)$ is a finitely
generated $k$-algebra. 
\end{definition} 
The almost finitely presented $\einf$-algebras over $k$ are precisely those connective $\einf$-$k$-algebras $R$ for which the functor $\Map_{\clg_{k}}(R, -)$ commutes with filtered
colimits of connective, $n$-truncated $\einf$-algebras. In fact, this is used
as the definition when one works over a non-Noetherian base (cf. 
\cite[Definition 7.2.4.26]{lurie2014higher} and \cite[Proposition
7.2.4.31]{lurie2014higher}). \vspace{4pt}

One can define analogous finiteness properties on the level of modules.  
If $R \in \clg_k$ is a connective $\einf$-algebra over $k$, then there is a  notion of an
\emph{almost perfect} $R$-module (cf.\ \cite[Definition 7.2.4.10]{lurie2014higher}, \cite[Section 2.7]{lurie2016spectral}).  
More generally, for each $n \in \mathbb{Z}$, one has the notion of an $R$-module which is
\emph{perfect to order $n$}; an $R$-module $M$ is almost perfect if and only
if it is perfect to order $n$ for each $n$. 
When $R$ is Noetherian,  then this notion simplifies by \cite[Proposition 7.2.4.17]{lurie2014higher}. Indeed, an $R$-module $M$ is then almost perfect if and only if
$M$ is bounded below and each homotopy group $\pi_i(M)$ is a finitely generated $\pi_0(R)$-module. 

The theory of Noetherian $\einf$-rings is well-behaved and robust; for instance,
one has a version of Hilbert's basis theorem (cf.\ \cite[Proposition 7.2.4.31]{lurie2014higher}) which, when combined with \cite[Proposition 7.4.3.18]{lurie2014higher}, implies:
\begin{proposition}\label{noethcotap}
The cotangent fibre of any augmented Noetherian $k$-algebra $R$ lies in $\mod_{k,\geq 0}^{\ft}$.
\end{proposition} 

One can detect almost finite presentation of $k$-algebras using the cotangent complex by the following special case of \cite[Theorem 7.4.3.18]{lurie2014higher}:
\begin{theorem} 
\label{critalmostfinitetype}
Let $R$ be a connective $\einf$-algebra over $k$. Suppose $\pi_0 (R)$ is a
finitely generated $k$-algebra. Then the following are equivalent: 
\begin{enumerate}
\item $R$ is almost finitely presented; 
\item the cotangent complex $L_{R/k}$ is almost perfect as an $R$-module. 
\end{enumerate}
\end{theorem}

We shall now discuss graded versions of the above definitions. 
\begin{definition} \label{finiteness} Let $k$ be a field.
\begin{enumerate}
\item  
A graded $\einf$-$k$-algebra $R_{\star} \in \clg( \gr^+( \md_k))$  is \emph{almost finitely presented} if
the underlying $\mathbb{E}_\infty$-$k$-algebra
$\bigoplus_{i \geq 0} R_i$ is almost finitely presented (cf.\ \Cref{Eafp}).
\item
If $R_{\star} \in \clg( \gr^+( \mod_{k,\geq 0}))$ is a
connective graded $\einf$-ring  and $M_{\star}$ is a graded $R_{\star}$-module,
we   say that $M_{\star}$ is \emph{almost perfect} (respectively \emph{perfect
to order $n$}) if the underlying ungraded $\bigoplus_{i \geq 0} R_i$-module
$\bigoplus_{i \geq 0} M_i$
is almost perfect (respectively perfect to order $n$). 
\end{enumerate}
\label{def:almostftalg}
\end{definition} 

If $R_0=k$, then  we have the following straightforward criterion 
for almost perfectness:
\begin{proposition} 
\label{grcrit:aperf}
Let $R_{\star}$ be a degreewise connective graded $\einf$-$k$-algebra with $R_0 = k$. 
Let $M_{\star}$ be an $R_{\star}$-module which is bounded below in each degree. 
Then the following are equivalent: 
\begin{enumerate}
\item $M_{\star}$ is almost perfect;
\item $\pi_i( M_{\star} \otimes_{R_{\star}} k)$ is a  finite-dimensional  $k$-vector space for all $i$. \mbox{We use the augmentation $R_{\star}\to
k$.}
\end{enumerate}
\end{proposition}

\begin{proof}
Condition $(1)$ implies $(2)$ since almost perfect modules are preserved under base-change.

For the converse direction, we may assume without restriction that $M_{\star}$ is connective in each degree.  
We will show by induction that if $(2)$ holds, then $M_{\star}$ is perfect to order $n$ for each $n \geq 0$.

The graded $k$-module $M_{\star}$ is perfect to order $0$ precisely if $\pi_0(M_{\star})$ is finitely generated. 
This follows from the graded variant of Nakayama's lemma since 
$\pi_0( M_{\star} \otimes_{R_{\star}} k) = \pi_0(M_{\star})
\otimes_{\pi_0(R_{\star})} k$
is finite-dimensional.

Now suppose we know that $M_{\star}$ is perfect to order $n-1$ for some $n \geq
1$. 
Choose a degreewise connective, finitely generated free graded $R_{\star}$-module
$P_{\star}$,
together with a map $P_{\star} \rightarrow M_{\star}$ inducing a surjection  of graded $k$-vector spaces on $\pi_0$. 
Let $F_{\star}$ be the homotopy fibre of $P_{\star} \rightarrow M_{\star}$. By \cite[Proposition
2.7.2.1]{lurie2016spectral}, $M_{\star}$ is perfect to order
$n$ if and only if $F_{\star}$ is perfect to order $n-1$. 
Since $F_{\star} \otimes_{R_{\star}} k$ is almost perfect over $k$, 
the inductive hypothesis  shows that $F_{\star}$ is perfect to
order $n-1$, which in turn implies that $M_{\star}$ is perfect to order $n$.
\end{proof}

The finiteness notion for graded $\EE_\infty$-rings introduced in \Cref{finiteness} recovers the ``axiomatic" notion  of graded finiteness given  \Cref{def:grafp}:
\begin{proposition} 
\label{grafpidentify} \INN{030@$\mathcal{C}^{\gr}_{\afp}$}
Let $R_{\star} \in  \clg^{\aug}( \gr^+( \md_{k, \geq 0}))$ be a degreewise
connective, augmented graded $\einf$-$k$-algebra with $R_0 = k$. Then the following are equivalent: 
\begin{enumerate}
\item $R_{\star}$ is almost finitely presented;
\item $\cot(R_{\star}) \in \gr( \md_{k, \geq 0})$ belongs to $\grcoh_{k, \geq
0}$.
\end{enumerate}
\end{proposition} 

\begin{proof}  
If $R_{\star}$ is almost finitely presented, then 
 it follows from \Cref{critalmostfinitetype} that
the cotangent complex $L_{R_{\star}/k}$ is almost perfect as an
$R_{\star}$-module. 
The cofibre sequence for a triple of rings then implies that 
the cotangent complex $L_{k/R_{\star}}$ belongs to $\grcoh_{k, \geq 0}$. 
 
Conversely, suppose that $\mathrm{cot}(R_{\star})$ belongs to $\grcoh_{k, \geq 0}$. 
Recall that $\pi_0( \mathrm{cot}(R_{\star}))$  is the module of indecomposables
of the augmented graded algebra $\pi_0(R_{\star})$. 
Since $\pi_0( \mathrm{cot}(R_{\star}))$ is a finite-dimensional vector space, it follows immediately that $\pi_0(R_{\star})$
is a finitely generated graded algebra; here we use that all   elements in the
augmentation ideal  sit in positive degrees.  
Next, we observe that 
the cotangent complex $L_{R_{\star}/k} \in \md_{R_{\star}}$ has the property that
$L_{R_{\star}/k} \otimes_{R_{\star}} k \simeq \cot(R_{\star})$ is almost perfect
as a graded $k$-module, which in turn implies that $L_{R_{\star}/k}$ is almost perfect by
\Cref{grcrit:aperf}. 
Thus, we see that $L_{R_{\star}/k}$ is almost perfect as an $R_{\star}$-module. 
By \Cref{critalmostfinitetype}, it follows that $R_{\star}$ is almost 
finitely presented.  
\end{proof} 

Next, we record analogous finiteness notions in the filtered case. 
\begin{definition} \label{cafpdef} 
Let $R \in \clg^{\aug}( \fil^+ ( \md_{k, \geq 0}))$ be a degreewise connective, filtered
augmented $\einf$-algebra over $k$ such that $F^0R / F^1 R \simeq k$. We say that $R$ is \mbox{\emph{complete almost
finitely presented} if:}
\begin{enumerate}
\item  
$R$ is complete as a filtered
object;
\item the associated graded algebra $\gr(R) \in \clg^{\aug}( \gr^+( \md_{k, \geq
0}))$ is almost finitely presented in the sense of \Cref{def:almostftalg}. 
\end{enumerate}
\end{definition}

\begin{remark} \INN{030@$\mathcal{C}^{\fil}_{\afp}$}
Again by \Cref{grafpidentify}, we can see that $R\in \clg^{\aug}( \fil^+ (
\md_{k, \geq 0}))$ is  complete almost finitely presented precisely if it  belongs to
  the $\infty$-category $\mathcal{C}^{\fil}_{\afp}$ 
defined in \Cref{def:filafp} of the axiomatic section. 
\end{remark}

We shall now discuss the completed finiteness condition  in the absence of a grading or filtration. This will later allow us to give a concrete description of the full subcategory $\mathcal{C}_{\afp} \subset \mathcal{C}$ from \Cref{cafp:defaxiomatic} in our context.
For future reference, we  start with a slightly more general notion:  

\begin{definition}[Complete local Noetherian $\einf$-rings] \label{clndef} 
An $\einf$-ring spectrum $R$ is said to be \emph{complete local Noetherian} if $R$ is Noetherian (cf.\ \Cref{def:noeth}) and $\pi_0(R)$ is a complete local   ring. We will refer to the
residue field of $\pi_0(R)$ as the residue field of $R$ itself. 
We write $\clg^{\cN} \subset \clg$ \INN{032@$\clg^{\cN}$} for the full subcategory spanned by  complete
local Noetherian $\einf$-rings.
\end{definition}

\begin{remark}[Topological finite generation]\label{tafp}
Let $R \in \clg_k^{\aug}$ be complete local Noetherian.  
By assumption, the  local ring $\pi_0(R)$ has residue field $k$.  If $x_1, \dots, x_n \in \pi_0(R)$ are generators for the
maximal ideal, then we can use the Cohen structure theorem  
 to write $\pi_0(R)$ as a quotient of the formal
power series ring $k[[t_1, \dots,t_n]]$ via the map $t_i \mapsto x_i$. 
Moreover, if $R \rightarrow R'$ is a map in $\clg_k^{\aug}$ between complete local
Noetherian $\einf$-rings, then $\pi_0(R) \rightarrow \pi_0(R')$ is
a local homomorphism. 
\end{remark}

\begin{example}[Completions of almost finitely presented algebras]
\label{complofafp}
Let $R \in \clg_k$ be a connective $\einf$-$k$-algebra which is
almost finitely presented in the sense of \Cref{Eafp}. 
Let $\mathfrak{m}$ be a maximal ideal of $\pi_0(R)$ whose residue field is $k$. 
Then the completion $\widehat{R}_{\mathfrak{m}}$ of $R$ along $\mathfrak{m}$,
together with its canonical augmentation to $k$,  
is a complete local Noetherian with residue field $k$. 
In fact,
$\pi_0( \widehat{R}_{\mathfrak{m}})$ is the (algebraic) completion of $\pi_0(R)$
along $\mathfrak{m}$, and  \INN{180@$\widehat{R}_{\mathfrak{m}}$}
$\pi_i( \widehat{R}_{\mathfrak{m}}) \simeq \pi_i(R) \otimes_{\pi_0(R)}
\pi_0(\widehat{R}_{\mathfrak{m}})$.  
We refer to \cite[Section 7.3]{lurie2016spectral} for a general reference on completions in
the context of $\einf$-ring spectra. 
\end{example}

We will now prove that the cotangent fibre functor is conservative on
complete local Noetherian $\EE_\infty$-algebras augmented over $k$. We need the following straightforward observation:

\begin{lemma} 
\label{almostperfnak}
Let $R $ be a complete local Noetherian $\einf$-ring with residue field $k$. Let $M $ be an $R$-module which is almost perfect. If $M \otimes_R k = 0$, then
$M  = 0$. 
\end{lemma} 
\begin{proof} 
Suppose that $M$ is nonzero. We then look at the smallest integer $n$ for which $\pi_n(M) \neq 0$. Since $\pi_{n}( M \otimes_R
k ) = \pi_n(M) \otimes_{\pi_0(R)} k = 0$ and $\pi_n(M)$ is finitely generated over $\pi_0(R)$, Nakayama's lemma gives
a contradiction.
\end{proof} 

\begin{proposition} 
\label{detectequivcafp}
Let $f: R \rightarrow R'$ be a map between complete  local Noetherian $\EE_\infty$-$k$-algebras.
If $f$ induces an equivalence after applying 
$\cot(-)$, then $f$ is itself an equivalence. 
\end{proposition}

\begin{proof} 
First, we recall that $\pi_0( \cot(R))$ is given by the cotangent space of the augmented $k$-algebra $\pi_0(R)$, i.e.\  the  quotient of indecomposables of  the maximal ideal. A similar statement holds for $R'$. 
Since both $\pi_0(R)$ and $\pi_0(R')$ are complete local rings, it follows that $\pi_0(R)\rightarrow \pi_0(R')$ is surjective. This  in turn implies that the map $R\rightarrow R'$ makes $R'$ into an $R$-module of almost finite presentation.

The triple of maps $R \rightarrow R' \rightarrow k$ gives rise to a cofibre sequence 
$k\otimes_{R'} L_{R'/R}  \rightarrow L_{k/R} \rightarrow L_{k/R'}$. 
By assumption, the second map  is an equivalence, which shows that $k  \otimes_{R'} L_{R'/R} $ is contractible.
Since $L_{R'/R}$ is almost perfect by \cite[Theorem 7.4.3.18]{lurie2014higher}, we can deduce that  $L_{R'/R}  \simeq 0 $ by \Cref{almostperfnak}.
By \cite[Lemma
4.6.2.4]{lurie2016spectral}, this in turn proves that the surjection $\pi_0(R)\rightarrow \pi_0(R')$ is \'{e}tale, and hence an isomorphism (cf.\ e.g.\ \cite[\href{https://stacks.math.columbia.edu/tag/0257}{0257}]{stacks-project}). The claim  now follows from \mbox{\cite[Corollary 7.4.3.4]{lurie2014higher}.} 
\end{proof}

We now wish to show that  $\mathcal{C}_{\afp}$ from \Cref{cafp:defaxiomatic} in the axiomatic setting 
precisely  consists of  all  complete local Noetherian $\EE_\infty$-algebras, and that the
remaining  axioms of \Cref{filteredrefinementdef}
are satisfied. To this end, we first show that complete
almost finitely \mbox{presented \emph{filtered} algebras behave well:}

\begin{proposition} 
\label{cplfilafpimpliesafp}
Let $R$ be an augmented filtered $\einf$-algebra over $k$ which is complete almost finitely presented in the sense of \Cref{cafpdef}. 
Then the underlying augmented $\einf$-$k$-algebra $F^0 R$ is a complete local Noetherian $\EE_\infty$-ring (cf.\ \Cref{clndef}). 
\end{proposition}

\begin{proof} 
Choose a system of generators $\overline{x}_1, \dots, \overline{x}_n$ of $\pi_0( \gr(R))$  in positive internal degrees.
Lifting them to $\pi_0(F^0R)$, we get a system of elements $x_1, \dots,x_n \in \pi_0(F^0 R)$ living in positive filtration. 

By the Milnor exact sequence, we know that $\varprojlim_n \pi_0(F^nR)$ vanishes, which in turn implies that the commutative ring $\pi_0(R)$ is $(x_1, \dots, x_n)$-adically complete.
We therefore obtain a map  $k[[t_1, \dots, t_n]] \rightarrow \pi_0(F^0 R)$,  which is readily seen to be 
surjective by passing to associated gradeds. 
It follows that $\pi_0(F^0 R)$ is a quotient of a formal power series ring in
finitely many variables over $k$. It is therefore complete, local, and Noetherian.

It remains to check that the homotopy groups of $F^0 R$ are finitely generated $\pi_0(F^0 R)$-modules. 
To this end, we observe that since the filtration on $R$ is complete,  \cite[Corollary 7.3.3.3]{lurie2016spectral}  shows that  $F^0 R$ is in fact an $(x_1, \dots, x_n)$-adically complete $\einf$-ring (cf.\ \cite[Definition 7.3.1.1]{lurie2016spectral}), which by  \cite[Theorem 7.3.4.1]{lurie2016spectral} implies that $\pi_i(F^0 R)$ is a derived complete $\pi_0(F^0R)$-module for the ideal $(x_1,\ldots,x_n)$ (cf.\ \cite[Theorem 7.3.0.5]{lurie2016spectral}).
By \cite[\href{https://stacks.math.columbia.edu/tag/091N}{091N}]{stacks-project}, it therefore suffices to prove that 
$\pi_i(F^0 R)/(x_1, \dots, x_n)$ is a  finitely generated $\pi_0(F^0 R)/(x_1, \dots, x_n)$-module for all $i$.

For this, we first upgrade the $F^0R$-module $F^0R/(x_1, \dots, x_n)$ to a filtered $R$-module $R/(x_1, \dots, x_n)$ by taking the iterated  cofibre of multiplication by each $x_i$ in filtered $R$-modules, thereby placing $x_i$ in the appropriate filtration degree. 
Using that $\gr(R)$ is almost finitely presented, a standard induction argument shows that the  $\gr(R)$-module $\gr ( R/(x_1, \dots, x_n)) \simeq \gr(R)/(\overline{x}_1, \dots, \overline{x}_n)$ has homotopy groups which are  finite-dimensional $k$-vector spaces in each degree $i$.  Since $R/(x_1, \dots, x_n)$ is  complete, we deduce that $F^0R/(x_1, \dots, x_n)$ belongs to $\mod_{k,\geq 0}^{\ft}$.
\end{proof} 

We proceed to verify the completeness axiom $(5)$ in   \Cref{filteredrefinementdef}, indicating that 
the cotangent complex of a complete
almost finitely presented  $\einf$-ring     is well-behaved.
We first examine the filtered setting, and start with the easy 0-connected
case where $\pi_0$ is simply $k$. 
\begin{proposition} 
\label{connectedconvergencefilafp}
Let $R \in \clg^{\aug}( \fil^+( \md_{k, \geq 0}))$ be complete almost of finite
presentation.
Suppose that $\pi_0( \gr(R)) = k$. 
Then 
$\cot(R) \in \fil ( \md_{k, \geq 0})$ belongs to $\filcoh_{k, \geq 0}$. 
\end{proposition} 
\begin{proof} 
For all $i$, our assumptions imply that $\pi_i( \gr(R))$ is a finite-dimensional $k = \pi_0( \gr(R))$-module. Therefore
 the augmentation ideal of $R$, which belongs to $\clg^{\nou}( \fil(
\md_{k, \geq 0}))$, 
has underlying object in $\filcoh_{k, \geq 1}$. 
We observe  that 
if $V \in \filcoh_{k, \geq 1}$, then the free filtered algebra $\bigoplus_{i \geq 1} (V^{\otimes i})_{h
\Sigma_i}$ on $V$ also belongs to $ \filcoh_{k, \geq 1}$; this follows since the $i^{th}$ summand is $i$-connective.
By \Cref{thebarconstruction}, we conclude that $\cot(R)$ 
is a geometric realisation of objects in  $\filcoh_{k, \geq 0}$, which implies the claim by \Cref{geomrealcomplete}.
\end{proof}

To extend the above result, we will need the following notion:

\begin{definition}  
Let $R \in \clg^{\aug}( \fil^+ ( \md_{k, \geq 0}))$ be complete almost finitely presented. 
Let $M$ be an $R$-module in $\fil^+( \md_{k})$ for which the following conditions hold true:
\begin{enumerate}
\item $\gr(M)$ is almost perfect as a $\gr(R)$-module (cf.\ \Cref{finiteness});  
\item $M$ is complete. 
\end{enumerate}
Then we say that $M$ is \textit{almost perfect.}
\end{definition}

The $\infty$-category of almost perfect $R$-modules 
behaves analogously to the $\infty$-category of almost perfect modules over 
a connective ring spectrum. 
Rather than verifying all  details (which we leave as an exercise to the reader), we simply observe
the following:

\begin{proposition} 
\label{tensorofalmostperfect}
Let $R \in \clg^{\aug}( \fil^+ ( \md_{k, \geq 0}))$ be complete almost
finitely presented. If  $M$ and $N$ are almost perfect $R$-modules, then so is  
$M \otimes_R N$.
\end{proposition}

\begin{proof} 
We first note that almost perfect $R$-modules are necessarily bounded-below. 
Since tensor products of almost perfect modules over a connective $\einf$-ring
are almost perfect, 
it only remains to show that $M \otimes_R N$ is complete as a filtered object.
Observe that this is clear when $M$ is a free $R$-module.
Suppose we know that $M \otimes_R N$ is complete in (homotopical) degrees $\leq n$
 (i.e.\ that \mbox{$\pi_i(\varprojlim(M \otimes_R N))$} vanishes for all $i \leq n$). 
The assumptions imply that there exists a finite free $R$-module $F$ together with a map 
$F \rightarrow M$ inducing a surjection on $\pi_0$. 
The homotopy fibre $F' $ remains connective and is still almost perfect. The  
cofibre sequence
$F \rightarrow M \rightarrow \Sigma F'$ and induction then imply
that $M \otimes_R N$ is complete in degree $n$. 
\end{proof} 

\begin{remark} 
An elaboration of the above argument shows that if 
$M$ is connective, then it arises as the  geometric realisation of a simplicial diagram of finite
free $R$-modules. In the unfiltered
case, this is established in \cite[Proposition 7.2.4.11]{lurie2014higher}. 
\end{remark}

We can now generalise \Cref{connectedconvergencefilafp}:
\begin{proposition} 
\label{cotoffilafpiscomplete}
Let $R \in \clg^{\aug}( \fil^+( \md_{k, \geq 0}))$
be complete almost finitely presented.
Then $\cot(R) \in \filcoh_{k, \geq 0}$. 
\end{proposition} 
\begin{proof} 
First, $\pi_0( \gr(R))$ is finitely generated, and we can lift generators to 
various filtered pieces. 
It follows that there is a finite-dimensional vector space $V$ equipped with a
finite filtration $$ \ldots \subset 0 \subset 0 \subset V_n \subset V_{n-1} \subset \ldots \subset V_1= V,  $$
together with a map of filtered augmented $\einf$-algebras
\( \free(V) \rightarrow R  \)
inducing a surjection on $\pi_0( \gr(-))$. 
Since $R$ is complete, this map factors through a map $\widehat{\free(V)}\rightarrow R$, which turns  $R$ into an almost perfect  $\widehat{\free(V)}$-module. 

We now argue that the cotangent fibre of
$\widehat{\free(V)}$ is just $V \in \fil \md_{k, \geq 0}$: that is, the map
$\free(V) \rightarrow \widehat{\free(V)}$ induces an equivalence  on cotangent fibres. 
This is clearly true on associated gradeds, so it suffices to see this on
underlying objects. 
But on underlying objects, 
$F^0 \widehat{\free(V)} = \prod_{i \geq 0} (V^{\otimes i})_{h \Sigma_i}$. 
%We see that this has an evident complete filtration whose associated graded is
%$\bigoplus_{i \geq 0} (V^{\otimes i})_{h \Sigma_i}$, 
%which is free on $V$ and hence Noetherian. 
%According to \Cref{cplfilafpimpliesafp}, it follows that $F^0
%\widehat{\free(V)}$ is complete Noetherian. 
In fact, the $\mathbb{E}_\infty$-ring $F^0 \widehat{\free(V)}$ is the completion of 
the augmented $\einf$-algebra $\bigoplus_{i \geq 0}
(V^{\otimes i})_{h \Sigma_i}$ at the augmentation ideal which does not change
the cotangent fibre. This follows from the following algebraic fact: if $M_\ast$
is a nonegatively graded module over $k[x_1, \dots, x_s]$ where $|x_i|> 0$ which
is finitely generated, then the $(x_1, \dots, x_s)$-adic completion of $M_\ast$
is $\prod_j M_j$. We apply the algebraic fact to each homotopy group of  
$\bigoplus_{i \geq 0}
(V^{\otimes i})_{h \Sigma_i}$, using also that this $\mathbb{E}_\infty$-ring is
Noetherian since it is finitely presented over $k$. 

By \Cref{tensorofalmostperfect}, it follows that 
$$ R'= R \otimes_{\widehat{\free(V)}} k \in \clg^{\aug}( \fil^+( \md_{k, \geq 0}))  $$
is also an almost perfect  $\widehat{\free(V)}$-module, and hence  complete.
Since we have an equivalence of filtered objects $\cot(R') \simeq
\mathrm{cofib} ( V \rightarrow \cot(R)) $, it suffices to show that
$\cot(R')$ belongs to $\filcoh_{k, \geq 0}$. Indeed, this follows from the
special case given by \Cref{connectedconvergencefilafp} since $\pi_0(\gr(R')) =
k$. 
\end{proof}

\begin{proposition} 
\label{adicfiltcomplete}
If $R \in \clg^{\aug}_k$ is complete  local Noetherian, then $\adic(R)$ is complete. 
\end{proposition} 
\begin{proof} 
Since $\adic(R) \rightarrow \widehat{\adic(R)}$ is an equivalence on associated gradeds, it suffices to check that it also induces an equivalence
on underlying objects to conclude that $\adic(R)$ is equivalent to $\widehat{\adic(R)}$ and thus complete.

We begin by observing that by \Cref{adiceinfgeneralprop}, the cotangent fibre  $\cot(\adic(R))$ is given by the filtered module
$\ldots \rightarrow 0 \rightarrow 0 \rightarrow \cot(R)$. Furthermore, the natural map  $\cot( \adic(R)) \rightarrow \cot( \widehat{\adic(R)})$ induces an equivalence on associated
gradeds. 

Now $\widehat{\adic(R)}$ is complete almost finitely presented as a filtered $\einf$-ring, 
since  its associated graded is free on $[\cot(R)]_1$ by \Cref{adiceinfgeneralprop}. 
We have used that $\cot(R)\in \mod_{k,\geq 0}^{\ft}$ as $R$ is Noetherian (cf.\ \Cref{noethcotap}). By
\Cref{cotoffilafpiscomplete},  this implies that
$  \cot( \widehat{\adic(R)})$ belongs to $\filcoh_{k, \geq 0}$ and it is therefore in particular  complete. 
Hence $\cot(\adic(R)) \rightarrow \cot(\widehat{\adic(R)})$ is  an equivalence, as it is a map between complete objects inducing an equivalence on associated gradeds.
Taking underlying objects everywhere, it follows that
$R =F^0 \adic(R) \rightarrow F^0 \widehat{\adic(R)}$ induces an equivalence on $\cot$. 
Both source and target are complete local Noetherian  
$\einf$-algebras over $k$ (the latter by \Cref{cplfilafpimpliesafp}). Hence $ R \rightarrow F^0 \widehat{\adic(R)}$ is an equivalence by \Cref{detectequivcafp}. 
\end{proof} 
We can now show that the two notions of smallness coincide: \INN{030@$\mathcal{C}_{\afp}$}
\begin{corollary} 
\label{cor:critforcafp}
Let $R \in \clg_k^{\aug}$ be a connective augmented $\einf$-algebra over $k$. Then the
following are equivalent: 
\begin{enumerate} 
\item $R$ belongs to $\mathcal{C}_{\afp}$ in the sense of
\Cref{cafp:defaxiomatic};  
\item $R$ is complete local Noetherian in the sense of \Cref{clndef}. 
\end{enumerate}
\end{corollary} 
\begin{proof} 
If $R$ is complete local Noetherian, then $\adic(R)$ is complete by \Cref{adicfiltcomplete} and 
 $\cot(R)$  belongs to $\mod_{k,\geq 0}^{\ft}$ by \Cref{noethcotap}.
Conversely, if  $R$ belongs to $\mathcal{C}_{\afp}$, then applying
\Cref{cplfilafpimpliesafp} to $\adic(R)$ shows that $R$ is complete local
Noetherian. 
\end{proof}

We can finally establish the following result:
\begin{proposition} 
\label{einfinitysatisfiesaxiomatic}
The datum of $\mathcal{C} = \clg_{k, \geq 0}^{\aug}, \mathcal{C}^{\gr}, \mathcal{C}^{\fil}$ specified in Construction~\ref{setupeinfinity}
specifies a filtered augmented monadic adjunction in the sense of \Cref{filteredrefinementdef}.  It therefore satisfies the 
 conditions of \Cref{thm:mainaxiomatic}. 
\end{proposition} 
\begin{proof} 
We have already verified conditions $(1)$, $(2)$, and $(3)$ of  \Cref{filteredrefinementdef} in \Cref{onetothree}.

For part $a)$ of the coherence condition  (\ref{coherencehyp}), we must check that if 
$A, A', A''$ are graded $\einf$-algebras over $k$ (with $k$ as degree $0$ component) which are almost
finitely presented, and if $A \rightarrow A'', A'\rightarrow A''$ are maps which induce
surjections on $\pi_0$, then $A \times_{A''} A'$ is almost finitely presented. 
This follows easily from the algebraic fact that if $S, S', S''$ are finitely
generated $k$-algebras and $S \rightarrow S'', S' \rightarrow S''$ are surjections, then $S
\times_{S''} S'$ is finitely generated as a $k$-algebra, cf.\ e.g.\ \cite[Tag
08KG]{stacks-project}.  
Part $b)$ of the coherence condition  (\ref{coherencehyp}) follows from \Cref{noethcotap},  as if $V\in \gr^{\ft} \md_{k,\geq 0}$, then $\sqz(V)$ is manifestly Noetherian. 

Part $a)$ of the completeness axiom (\ref{goodfc}) follows from \Cref{cotoffilafpiscomplete}. For part $b)$ of  axiom (\ref{goodfc}), we note that if
$A\in \mathcal{C}^{\fil}_{\afp}$ is  an augmented complete $\EE_\infty$-$k$-algebra  which is complete   almost finitely presented, then the underlying 
augmented $\EE_\infty$-$k$-algebra $F^0A$ is complete local Noetherian by \Cref{cplfilafpimpliesafp}.
 \Cref{adicfiltcomplete} and \Cref{cor:critforcafp} then imply that $\adic(F^0A)$ is complete.
\end{proof}

Hence, \Cref{thm:mainaxiomatic} applies to the present setting, and we can perform \mbox{the following construction:}
\begin{definition}[Spectral partition Lie algebras] \label{spla} 
Write $\lieps: \md_k \rightarrow \md_k$ for the unique sifted-colimit-preserving monad on $\mod_k$ satisfying 
$\lieps(V) = \cot( \sqz(V^{\vee}))^{\vee}$ for all $V\in \coh_{k, \leq 0}$. Algebras
over $\lieps$ will be called \emph{spectral partition Lie algebras.}
\end{definition}  
 
In particular,  \Cref{fmpequiv} asserts an equivalence between formal moduli problems for
$\einf$-algebras  and  the $\infty$-category
$\alg_{\lieps}$\hspace{-2pt}.  
We postpone the discussion and formulation of this result to the next section (cf.\ \Cref{mostgeneral}), where 
we will also discuss generalisations to other bases (but still augmented
over $k$). \\

Instead, we will  now give a more explicit description of partition Lie algebras. 
To begin with, we check that when $k$ is a field of  characteristic zero, we recover a familiar notion. Recall that in this case, the ordinary category of differential graded Lie algebras over $k$ carries a model structure whose weak equivalences and fibrations are transported along the forgetful functor to chain complexes. Its   underlying $\infty$-category will be denoted by $\mathrm{Alg}_{\liedg} $ (cf.\ \cite[Section 13.1]{lurie2016spectral} for a detailed treatment), and we write $\liedg$ for the corresponding monad on $\md_k$.
\begin{proposition}\label{spectraldg}
If $k$ is   of characteristic zero, then the following monads  on $\md_k$ are equivalent:
$${\lieps}    \ \ \ \ \ \ \ \ \ \ \ \ \ \ \ \Sigma\circ  {\liedg}\circ \Sigma^{-1}.$$  
As a result, the $\infty$-categories of spectral partition Lie algebras and shifted differential graded Lie algebras are all equivalent.
\end{proposition}
\begin{proof}
Writing $C^{\dg}$ for the classical Chevalley-Eilenberg complex functor, we have a pair of adjunctions (cf.\ \cite[Section 13.3]{lurie2016spectral}) given by 
\[
\begin{tikzcd}[column sep=5em]
 \clg_k^{\aug}
	\arrow[r, shift left=0.7ex, "  \mathfrak{D}^{\dg}   "]
	\arrow[r, shift right=0.7ex, "C^{\dg}"', leftarrow]
	& (\mathrm{Alg}_{\liedg} )^{op}
	\arrow[r, shift left=0.7ex, "\forget^{\dg}"]
	\arrow[r, shift right=0.7ex, "\free^{\dg}"', leftarrow]
	& \md_{k}^{op}.
\end{tikzcd}
\]

By \cite[Proposition 13.3.1.4]{lurie2016spectral}, their composite is given by 
$$ \begin{tikzpicture}[baseline=(C.base)] \label{dualadjunction}
	\node (C) at (0,0) {$\Sigma^{-1} \circ \cot(-)^\vee \  : \ \clg_k^{\aug}$};
	\node (M) at (5,0) {$  \mod_k^{op} \ :  \ \sqz(-^\vee) \circ \Sigma \  .$};
	% adjunction arrows
	\draw[->] ([yshift=2.5pt]C.east) -- node[above] { }
	([yshift=2.5pt]M.west);
	\draw[<-] ([yshift=-2.5pt]C.east) -- node[below] { }
	([yshift=-2.5pt]M.west); 
\end{tikzpicture}\ \ $$

Abstract nonsense therefore gives rise to a natural transformation of monads
$$\liedg= \forget^{\dg}\circ \free^{\dg} \ \ \xrightarrow{\ \ \ \ \ \ \ \ \ } \ \ \Sigma^{-1} \circ  \cot\left(\sqz(-^\vee)\right)^\vee\circ \Sigma\ ,$$ which is obtained by inserting the unit $\id \rightarrow \mathfrak{D}^{\dg} \circ C^{\dg}$.
The  monad $\liedg$ preserves $\mod_{k,\leq -1}^{\ft}$, and we can therefore deduce from \cite[Proposition 13.3.1.1]{lurie2016spectral} that the above transformation is an equivalence for all $V\in \mod_{k,\leq -1}^{\ft}$. By construction of $\lieps$, we obtain an equivalence of monads 
$(\Sigma \circ \liedg \circ \Sigma^{-1})|_{\mod_{k,\leq 0}^{\ft}} \simeq \lieps|_{\mod_{k,\leq 0}^{\ft}}$.
Since both $\liedg$ and $\lieps$   preserve sifted colimits (the former by \cite[Proposition 13.1.4.4]{lurie2016spectral}, the latter by construction), we in fact obtain an equivalence of monads $\Sigma \circ \liedg  \circ \Sigma^{-1} \simeq \lieps$ on $\mod_k$, applying  
\Cref{extendmonad} above. 
\end{proof}
\begin{remark}
As $\infty$-categories, $ \alg_{\liedg} $ and $\alg_{\Sigma\circ  \liedg\circ \Sigma^{-1}}$ are equivalent via a functor  whose effect on underlying $k$-modules is simply a shift by $1$.  
\end{remark}

For fields of positive characteristic, partition Lie algebras are somewhat more complicated objects. We recall the notion of partition complexes from   \Cref{partitioncomplex} above. Write  $\Sigma |\Pi_n|^\diamond$ for the  $\Sigma_n$-space given by the
reduced suspension of the unreduced 
 suspension of the $n^{th}$ partition complex. Let
$\widetilde{C}^\bullet(\Sigma |\Pi_n|^\diamond,k)$ be the cosimplicial
$k$-vector space given by its $k$-valued (reduced) singular cochains.
\begin{proposition}\label{concretespectralplie}\INN{060@$F_X,  F_X^h$}
Given any $V\in \md_k$, there is an equivalence
$$ \lieps(V) \simeq \bigoplus_{n\geq 1} (F^h_{\Sigma |\Pi_n|^\diamond})^\vee(V).$$
Here   $(F^h_{\Sigma |\Pi_n|^\diamond})^\vee$ \INN{060@$(F^h_{\Sigma |\Pi_n|^\diamond})^\vee$} is the right-left extension (cf.\ \Cref{rlext}) of the functor $\vect_k^\omega \rightarrow \mod_k$   given by $V\mapsto (\widetilde{C}^\bullet(\Sigma |\Pi_n|^\diamond,k)  \otimes V^{\otimes n} )^{h\Sigma_n}$.
If $V\in \md_{k,\leq N}$ is truncated above,   there is  an equivalence  
$$ \lieps(V) \simeq \bigoplus_{n} \left(\widetilde{C}^\bullet(\Sigma |\Pi_n|^\diamond,k)  \otimes V^{\otimes n}\right)^{h\Sigma_n}.$$

\end{proposition}
\begin{proof}
For each $n \geq 0$, we define a simplicial $\Sigma_n$-set $T(n)$ \INN{200@$T(n)$} by specifying its set of $k$-simplices as 
$$ T(n)_k= \left\{ \ [\hat{0}   =  \sigma_0 \leq \sigma_1 \leq \ldots \leq \sigma_k = \hat{1}] \  \ \bigg| \ \  \sigma_i \mbox{ are partitions of } \{1,\ldots, n\} \ \right\} \  \ \coprod \ \ \{ \ast \}$$
Degeneracy maps  insert repeated partitions into chains and  fix  $\ast$. 
Face maps delete partitions from chains whenever this yields a ``legal" chain starting in $\hat{0}$ and ending in $\hat{1}$; otherwise, they \mbox{map to   $\ast$.} 

As $\cot$ preserves geometric realisations, we obtain, for any $V\in \vect_k^\omega$,  the following equivalence: 
$$\cot(\sqz(V)) \simeq  |\Barr_\bullet(\id,\Free_{ {\enu}}, V)| \simeq  \ \ \bigg|  \ \ \ldots \ \ \substack{\longrightarrow \vspace{-4pt} \\ \leftarrow   \vspace{-4pt} \ \\ \longrightarrow \vspace{-4pt} \\ \leftarrow   \vspace{-4pt} \ \\ \longrightarrow  \vspace{-4pt} \\ \leftarrow  \vspace{-4pt} \ \\ \longrightarrow} \ \ \bigoplus_{m\geq 1}   ( \bigoplus_{n\geq 1}  V^{\otimes n}_{h\Sigma_n} )^{\otimes m}_{h\Sigma_m}\ \ \substack{\longrightarrow \vspace{-4pt} \\ \leftarrow   \vspace{-4pt} \ \\ \longrightarrow  \vspace{-4pt} \\ \leftarrow  \vspace{-4pt} \ \\ \longrightarrow} \ \ \bigoplus_{n\geq 1}  V^{\otimes n}_{h\Sigma_n} \ \ \substack{  \longrightarrow  \vspace{-4pt} \\ \leftarrow  \vspace{-4pt} \ \\ \longrightarrow} \ \ V \ \ \bigg|\ .$$
For $(X,\ast)$ a pointed set, we write $k[X]$ for the free $k$-module on $X$ subject to the relation $0\simeq \ast$.
Expanding out extended powers binomially, a well-known and elementary combinatorial observation (explained for example in  \cite{ching2005bar}) shows that $\cot(\sqz(V))$ is equivalent to 
$$ \bigg|  \ \ \ldots \ \ \substack{\longrightarrow \vspace{-4pt} \\ \leftarrow
\vspace{-4pt} \ \\ \longrightarrow \vspace{-4pt} \\ \leftarrow   \vspace{-4pt} \
\\ \longrightarrow  \vspace{-4pt} \\ \leftarrow  \vspace{-4pt} \ \\
\longrightarrow} \ \ \bigoplus_{n \geq 1} (k[T(n)_2] \otimes V^{\otimes n})_{h\Sigma_n}\ \ \substack{\longrightarrow \vspace{-4pt} \\ \leftarrow   \vspace{-4pt} \ \\ \longrightarrow  \vspace{-4pt} \\ \leftarrow  \vspace{-4pt} \ \\ \longrightarrow} \ \ \bigoplus_{n\geq 1}   (k[T(n)_1] \otimes V^{\otimes n})_{h\Sigma_n}  \ \ \substack{  \longrightarrow  \vspace{-4pt} \\ \leftarrow  \vspace{-4pt} \ \\ \longrightarrow} \ \ \bigoplus_{n\geq 1}  (k[T(n)_0] \otimes V^{\otimes n})_{h\Sigma_n} \ \ \bigg|,$$
which  is equivalent to $\bigoplus_{n\geq 1}
(\widetilde{C}_\bullet(|T(n)|,k) \otimes  V^{\otimes n})_{h \Sigma_n}$. Since both functors preserve sifted colimits, we deduce that $\cot(\sqz(V))\simeq \bigoplus_{n\geq 1}
(\widetilde{C}_\bullet(|T(n)|,k) \otimes  V^{\otimes n})_{h \Sigma_n}$ for all $V\in \md_k$.
 
We now observe that $T(n)$ can be identified with the quotient of the join $\{\hat{0}\} \ast \Pi_n \ast \{\hat{1}\}$ by the simplicial subset spanned by  all chains not containing $[\hat{0} \leq \hat{1}]$ as a subchain.  The realisation $|T(n)|$ is therefore equivalent to the  reduced suspension of the unreduced   suspension of the $n^{th}$ partition complex $\Sigma |\Pi_n|^\diamond$ (cf.\ \cite[Section 2.9]{arone2018action}). \vspace{3pt}
 
If $V\in \mod_{k,\leq 0}^{\ft}$, then $\sqz(V^\vee)$ is Noetherian. Hence $\cot(\sqz(V^\vee))  \in \mod_{k,\geq 0}^{\ft}
$, which implies that 
 $\lieps(V) = \cot(\sqz(V^\vee))^\vee$ 
belongs to $\mod_{k,\leq 0}^{\ft}$  
(cf.\ \Cref{noethcotap}).
As all homotopy groups are finite-dimensional,  
 the   map  $ \bigoplus_{n} (\widetilde{C}^\bullet(\Sigma |\Pi_n|^\diamond,k)  \otimes V^{\otimes n})^{h\Sigma_n} \xrightarrow{\simeq}   \prod_{n }(\widetilde{C}^\bullet(\Sigma |\Pi_n|^\diamond,k)  \otimes V^{\otimes n})^{h\Sigma_n}$ is an equivalence; 
 this can also be seen as a consequence of the increasing connectivity of the partition complexes $|\Pi_n|$. We therefore obtain an equivalence
$$\lieps(V)  \simeq \left(\cot (\sqz(V^\vee)\right)^\vee \simeq \bigoplus_{n} \left(\widetilde{C}^\bullet(\Sigma |\Pi_n|^\diamond,k)  \otimes V^{\otimes n}\right)^{h\Sigma_n} \simeq   \bigoplus_{n\geq 1} (F^h_{\Sigma |\Pi_n|^\diamond})^\vee(V).$$ 
Since $ \lieps(V)$ and $  \bigoplus_{n\geq 1} (F^h_{\Sigma |\Pi_n|^\diamond})^\vee$ commute with sifted colimits, 
the first claim follows.

We  now observe that   both $\lieps(-)$ and $\bigoplus_{n}  (\widetilde{C}^\bullet(\Sigma |\Pi_n|^\diamond,k)  \otimes (-)^{\otimes n} )^{h\Sigma_n}$ preserve filtered colimits in $\md_{k,\leq 0}$. For the former functor, this holds by definition.
For the latter, we note that for all $m$, the functor $\tau_{\geq -m}(-)^{h\Sigma_n} :\fun(B\Sigma_n , \mod_{k,\leq 0} ) \rightarrow 
\mod_{k,\leq 0}$ is computed by a finite limit.
The above formula therefore holds for any $V\in \md_{k,\leq 0}$.

 Both functors also preserve finite geometric realisations, which implies the formula whenever $V\in \md_{k,\leq N}$ for some $N$.
\end{proof}

\newpage

\subsection{Simplicial commutative rings}\label{SCR}
We shall now explain the modifications needed 
in order to obtain a Lie algebraic description of deformations over simplicial commutative rings over \mbox{a field $k$.} 
In particular, we will  obtain a setup as in \Cref{filteredrefinementdef} of our axiomatic section. 

For this, we will need to recall the basic homotopy theory of simplicial commutative
rings, as introduced by Quillen.  
We refer to \cite[Chapter 25]{lurie2016spectral} 
for a detailed $\infty$-categorical treatment of simplicial commutative rings. 
For our axiomatic setup, we will also need  graded and filtered versions. 
We give a quick summary below:\vspace{-1pt} 

\begin{cons}[The setup for simplicial commutative rings]\

\label{setupSCR}
\begin{enumerate}[a)]
\item Let $\mathcal{D} = \SCR_k^{\aug}$ 
be the $\infty$-category of augmented simplicial commutative $k$-algebras. 
Explicitly,  $\SCR_k^{\aug}$ can be obtained as the nonabelian derived $\infty$-category
$\mathcal{P}_{\Sigma}$ (as in \cite[Section 5.5.8]{lurie2009higher})
of the category of finitely generated augmented polynomial $k$-algebras. 
\INN{040@$\mathcal{D}^{\gr}$}
\INN{040@$\mathcal{D}^{\fil}$}

\item Let
$\mathcal{D}^{\fil}$ be the 
$\infty$-category of filtered, augmented
simplicial commutative $k$-algebras $R$ with $F^0 R/F^1 R \simeq k$. 
Specifically, $\mathcal{D}^{\fil}$   can be obtained by applying $\mathcal{P}_{\Sigma}$ 
to the category of   filtered, augmented $k$-algebras
which are free on a finite-dimensional vector space $V$ equipped with a finite
filtration with $F^0 V = F^1 V$.

\item
Let $\mathcal{D}^{\gr}$ denote the $\infty$-category of graded, augmented
simplicial commutative $k$-algebras. 
More precisely, $\mathcal{D}^{\gr}$ is obtained as $\mathcal{P}_{\Sigma}$ 
of the category of finitely generated, graded augmented polynomial algebras
of the form 
$k[x_1, \dots, x_r]$ with each $x_i$ homogeneous of \textit{positive} degree.

\item 
We have  free-forgetful adjunctions 
$\md_{k, \geq 0} \rightleftarrows \mathcal{D}$, 
$\gr( \md_{k, \geq 0}) \rightleftarrows \mathcal{D}^{\gr}$,  and \INN{060@$\free$} \INN{060@$\forget$}
\mbox{$\fil( \md_{k, \geq 0}) \rightleftarrows \mathcal{D}^{\fil}$}. The forgetful
functors act as expected on the polynomial generators of the respective $\infty$-categories of algebras (i.e.\ by taking the kernel of the \INN{060@$F^1$}
augmentation). Moreover, the forgetful functors commute with sifted colimits (cf. Construction~\ref{freeSCR} below). 
 \INN{032@$\cot_{\Delta}$} \INN{190@$\sqz$}  
 The three evident square-zero functors
 {$\sqz: \md_{k, \geq 0} \rightarrow \mathcal{D}$, 
$\sqz: \gr(\md_{k, \geq 0}) \rightarrow \mathcal{D}^{\gr},$  and
$\sqz: \fil( \md_{k, \geq 0}) \rightarrow \mathcal{D}^{\fil}$} admit left adjoints \vspace{-1pt}
$$ \cot_{\Delta}: \mathcal{D} \rightarrow \md_{k, \geq 0}, 
\
\cot_{\Delta}: \mathcal{D}^{\gr} \rightarrow \gr ( \md_{k, \geq 0}), \ 
\cot_{\Delta}: \mathcal{D}^{\fil} \rightarrow \fil( \md_{k, \geq 0})
.$$
We use the subscript $\Delta$ to contrast this with the $\einf$-cotangent fibre
construction. 
\item The underlying object functor 
$F^1: \mathcal{D}^{\fil} \rightarrow \mathcal{D}$ forgets the filtration.
On the polynomial generators, it behaves  as the name indicates; in general, it is determined by commuting with sifted colimits. 
The functor  \INN{1@$\adic$}
$\adic: \mathcal{D} \rightarrow \mathcal{D}^{\fil}$ is the left adjoint 
of the underlying functor. 

\item The associated graded functor $\mathcal{D}^{\fil} \rightarrow \mathcal{D}^{\gr}$ is constructed similarly by first defining it in the evident way on polynomial generators and then extending in a sifted-colimit-preserving manner.\vspace{-1pt}
\end{enumerate}
\end{cons}

\begin{remark}
\label{containsdiscrete}
$\mathcal{D}$ contains the 
(ordinary) category of augmented $k$-algebras as a full subcategory. 
Similarly, $\mathcal{D}^{\gr}$ and $\mathcal{D}^{\fil}$ 
contain the  categories of graded and filtered augmented $k$-algebras. \vspace{-2pt} \end{remark}

\newcommand{\lsym}{\mathbb{L}\mathrm{Sym}}

\begin{cons}[The free functors]
\label{freeSCR}  \INN{120@$\L \sym$} 
We let $\lsym^*: \md_{k, \geq 0} \rightarrow \md_{k, \geq 0}$ denote the functor which
sends $V \in \md_{k, \geq 0}$ to the 
augmentation ideal   of the free simplicial commutative \mbox{$k$-algebra  on $V$} (with its natural augmentation). 
Explicitly, if $V$ is a (discrete) $k$-vector space, then 
$\lsym^*(V) = \bigoplus_{i > 0} \mathrm{Sym}^i V$ is the (usual) nonunital symmetric
algebra on $V$; in general $\lsym^*$ is defined as the nonabelian (left) derived
functor construction (Construction~\ref{derivedfun}). 

The functors 
$\lsym^*: \gr( \md_{k, \geq 0}) \rightarrow \gr( \md_{k, \geq 0})$ and 
$\lsym^*: \fil ( \md_{k, \geq 0}) \rightarrow \fil( \md_{k, \geq 0})$ are defined in a similar way; they recover
$\lsym^*$ on underlying $k$-modules,  but keep track of the additional grading
and filtration, respectively. 

We now observe  that 
$\lsym^i$ is a polynomial functor of degree $i$ (as it preserves filtered colimits and is $i$-excisive by \Cref{JMcC}). Combining this with the
finiteness properties of symmetric powers established in \cite[Section 25.2.5]{lurie2016spectral},
it follows that $\lsym^*: \gr \md_{k, \geq 0} \rightarrow \gr \md_{k, \geq 0}$ is
admissible in the sense of \Cref{admissiblefunctors}. 
\end{cons}

\begin{example}[The adic filtration of a polynomial ring]
\label{adicdiscSCR}
Unwinding the definitions, we see that applying
the functor $\adic$ to a  free simplicial commutative ring $k[x_1, \dots , x_n] \in
\SCR_k^{\aug} = \mathcal{C}$ recovers the usual $\mathfrak{m}$-adic filtration,
where $\mathfrak{m} =
(x_1, \dots, x_n)$ is  the augmentation ideal. 
In other words, one obtains the free filtered simplicial commutative ring on
$x_1, \dots, x_n$ in filtration $1$. 

Very explicitly, the adic filtration can also be defined 
as follows: on polynomial rings, it is the $\mathfrak{m}$-adic filtration 
and in general, it is defined via left Kan extension. 
\end{example}

\begin{proposition}\label{onetwothreescr}
The setup of simplicial commutative $k$-algebras in Construction~\ref{setupSCR} satisfies conditions $(1)-(3)$ of
\Cref{filteredrefinementdef}. 
\end{proposition} 
\begin{proof} 
Conditions $(1)$ and $(2)$ are straightforward to check. 

In Construction~\ref{freeSCR}, we saw that $\lsym^*: \gr \md_{k, \geq 0} \rightarrow \gr \md_{k, \geq 0}$ 
is an admissible functor.
It remains to produce the filtration for a graded, augmented simplicial
commutative ring  $A \in \mathcal{D}^{\gr}$; for this, we will follow
the discussion in \Cref{adicdiscSCR}. 
If $A$ is free   with maximal ideal
$\mathfrak{m}_A$, the $\mathfrak{m}_A$-adic filtration gives a
natural 
convergent filtration on 
$\mathfrak{m}_A$; its associated graded is given by the symmetric algebra $\sym^*(
\mathfrak{m}_A/\mathfrak{m}_A^2)$. 
By taking left Kan extension, we conclude that for any $A \in
\mathcal{D}^{\gr}$, the augmentation ideal $\mathfrak{m}_A$ is equipped with a convergent
filtration 
with associated graded $\lsym^*( \cot_{\Delta}(A))$. This immediately implies that 
condition $(3)$ of  
\Cref{filteredrefinementdef} is satisfied. 
\end{proof}

We will now verify the coherence axiom~\ref{coherencehyp}  and the completeness
axiom~\ref{goodfc} in \Cref{filteredrefinementdef}. 
These will both be deduced from the analogous assertions involving
$\einf$-rings, which we have already checked in \Cref{einfinitysatisfiesaxiomatic} above. 

\begin{cons}[Forgetting to $\einf$-algebras]\label{forgetting}
There is a natural forgetful functor from simplicial commutative rings to
$\einf$-rings. It is characterised by the properties of acting as the forgetful functor on ordinary polynomial rings
and preserving sifted colimits (cf.\ \cite[Section
25.1.2]{lurie2016spectral}).  
This construction clearly carries over to the augmented, filtered, and graded settings, 
and we therefore obtain forgetful functors
$\mathcal{D} \rightarrow \mathcal{C}$, $\mathcal{D}^{\fil} \rightarrow \mathcal{C}^{\fil}$, and
$\mathcal{D}^{\gr} \rightarrow \mathcal{C}^{\gr}$. Here we use the notation introduced
in \Cref{setupeinfinity} and \Cref{setupSCR}. 
\end{cons}

\begin{definition} \label{SCRnoet}
We say that $A \in \SCR_k^{\aug} = \mathcal{D}$ is \emph{Noetherian} (respectively
\emph{complete local Noetherian}) if   
the underlying $\einf$-algebra of $A$ is Noetherian (respectively complete local Noetherian). 
We write $\SCR^{\cN}_k$ for  the full subcategory spanned by all 
complete
local Noetherian \INN{190@$\SCR^{\cN}$} $A \in \SCR_k^{\aug}$.
\end{definition} 
The axiomatic Definitions \ref{def:grafp}  and \ref{def:filafp}   give notions  of almost finite presentation and complete almost finite presentation 
for graded and  filtered  simplicial commutative $k$-algebras, respectively. \INN{030@$\mathcal{C}^{\gr}_{\afp}$} 
\INN{030@$\mathcal{C}^{\fil}_{\afp}$} 
\begin{proposition} 
\label{whenisSCRafp} These notions are compatible with the forgetful functor to $\EE_\infty$-$k$-algebras:
\begin{enumerate}
\item A graded (augmented) simplicial commutative $k$-algebra $A \in \mathcal{D}^{\gr}$   is almost finitely presented if and only if 
the underlying graded $\einf$-$k$-algebra (in $\mathcal{C}^{\gr}$) is almost finitely presented.
\item A filtered (augmented)  simplicial commutative $k$-algebra $A \in \mathcal{D}^{\fil}$ is complete almost finitely presented 
if and only if the underlying filtered $\einf$-$k$-algebra (in $\mathcal{C}^{\fil}$)
is complete almost finitely presented. 
\end{enumerate}
\end{proposition} 
\begin{proof} 
Both assertions follow straightforwardly from \cite[Remark 25.3.3.7]{lurie2016spectral}. 
More explicitly, 
this remark 
shows that 
there is  an associative ring spectrum $k^+$ with an augmentation $k^+ \to
k$ such that 
$\cot(A) $ is a $k^+$-module  and $\cot_{\Delta}(A) \simeq \cot(A)
\otimes_{k^+} k$. 
Moreover, $k^+$ is connective with $\pi_0(k^+) =k$, and its homotopy groups  
are finite-dimensional in each degree. 
This readily implies that $\cot(A)$ has finite-dimensional homotopy groups in
each degree if and only if $\cot_{\Delta}(A)$ does, hence proving $(1)$. Assertion $(2)$ follows
from $(1)$ as completeness is detected on underlying $k$-module spectra. 
\end{proof} 

\begin{proposition} \label{part5a}
If $R \in \mathcal{D}^{\fil}$ is complete almost finitely presented, then 
 $\cot_{\Delta}(R)\in \filcoh_{k, \geq 0}$.
\end{proposition}

\begin{proof} 
This   follows by combining  \cite[Remark 25.3.3.7]{lurie2016spectral}  with 
\Cref{cotoffilafpiscomplete}. 
Namely, we already know that $\cot(R) \in \filcoh_{k, \geq 0}$, and the identification
$\cot_{\Delta}(R) \simeq \cot(R) \otimes_{k^+} k$ (which works in the filtered
category too) together with \Cref{geomrealcomplete} allow us to conclude the claim. 
\end{proof}

\begin{proposition} 
\label{simplicialadicSCRconverge}
If $R \in \SCR_k^{\aug}$ is complete local Noetherian, then the adic filtration
converges. 
\end{proposition} 
\begin{proof} 
This follows by the argument used in the proof of
\Cref{adicfiltcomplete}, where we simply replace $\einf$-rings by simplicial commutative rings everywhere.\vspace{-2pt}
\end{proof}

\begin{corollary} 
The setup of simplicial commutative $k$-algebras of Construction~\ref{setupSCR} satisfies the axioms of
\Cref{filteredrefinementdef}.  
Consequently, 
\Cref{thm:mainaxiomatic} holds true in this context.
\end{corollary} 
\begin{proof} 
We have already checked axioms $(1)$-$(3)$ in \Cref{onetwothreescr}.
The coherence axiom $(4)$ follows immediately by combining 
\Cref{whenisSCRafp} with the corresponding result for $\einf$-algebras, which was established in \Cref{einfinitysatisfiesaxiomatic}.
Part $a)$ of the completeness axiom $(5)$ was proven in \Cref{part5a}, whereas part $b)$ follows\vspace{-3pt} by combining \Cref{simplicialadicSCRconverge} with \Cref{cplfilafpimpliesafp}. 
\end{proof}

In particular,  we can perform the following construction: 
\begin{definition}[Partition Lie algebras] \label{defliep}
Write $\liep: \md_k \rightarrow \md_k$ for the unique sifted-colimit-preserving monad on $\mod_k$ satisfying 
$\liep(V) = \cot_{\Delta}( \sqz(V^{\vee}))^{\vee}$ for all $V\in \coh_{k, \leq 0}$. Algebras\vspace{-1pt}
over $\liep$ will be called \emph{partition Lie algebras.}   \INN{120@$\lieps$}
\end{definition} 
Applying  \Cref{fmpequiv} to our setup, we obtain a classification of formal moduli problems for
augmented simplicial commutative rings as equivalent to the $\infty$-category of partition Lie algebras. We again
postpone stating the result formally until the next \vspace{-1pt}section (cf.
\Cref{mostgeneral}).
\begin{proposition}\label{deriveddg}
If $k$ is   of characteristic zero, then the monad ${\liep}$ is equivalent to the monads $\lieps$ and $\Sigma\circ  {\liedg}\circ \Sigma^{-1}$ building free spectral and free shifted differential graded Lie algebras.
As a result, the $\infty$-categories of partition Lie algebras, spectral partition Lie algebras, and shifted differential graded Lie algebras are equivalent.\vspace{-1pt}
\end{proposition}
\begin{proof}
Since $k$ has characteristic zero, the forgetful functor from simplicial commutative $k$-algebras to
connective $\EE_\infty$-$k$-algebras is an equivalence (cf.\ \cite[Proposition 25.1.2.2]{lurie2016spectral}). Together with \Cref{spectraldg}, this implies the claim.\vspace{-1pt}
\end{proof}

We proceed to establish a concrete description of $\lieps$.
As above, let $\widetilde{C}^\bullet(\Sigma |\Pi_n|^\diamond,k)$ denote the $k$-valued (reduced) singular cochains of the doubly suspended $n^{th}$ partition complex. The following result uses the genuine $\Sigma_n$-equivariant structure of this cosimplicial $k$-module:
\begin{proposition}\label{freeplie}
Given any $V\in \md_k$, there is an equivalence
$ \liep(V) \simeq \bigoplus_{n\geq 1} (F_{\Sigma |\Pi_n|^\diamond})^\vee.$ 
Here   $(F_{\Sigma |\Pi_n|^\diamond})^\vee$\INN{060@$F_X,  F_X^h$}
  is the right-left extension (cf.\ \Cref{rlext}) of the functor $\vect_k^\omega \rightarrow \mod_k$   given by $V\mapsto (\widetilde{C}^\bullet(\Sigma |\Pi_n|^\diamond,k)  \otimes V^{\otimes n} )^{\Sigma_n}$ (cf.\ \Cref{bredon} for a more formal definition). 

If $V\simeq \Tot(V^\bullet)\in \mod_{k,\leq 0}$ is represented by a 
cosimplicial $k$-vector space $V^\bullet$, then $$ \liep(V) \simeq \bigoplus_{n} \Tot \left(\widetilde{C}^\bullet(\Sigma |\Pi_n|^\diamond,k) \otimes (V^\bullet)^{\otimes n}\right)^{\Sigma_n}.$$ Here $\widetilde{C}^\bullet(\Sigma |\Pi_n|^\diamond,k)$ denotes the    $k$-valued cosimplices on the space $\Sigma |\Pi_n|^\diamond$, the functor 
$(-)^{\Sigma_n}$ takes strict fixed points,  and the tensor product is computed in cosimplicial $k$-modules.\vspace{-1pt}

\end{proposition}
\begin{proof}
We apply the same argument as in   \Cref{concretespectralplie}, replacing homotopy orbits   with strict orbit, to deduce the first statement. The second statement  then follows from \Cref{RKEisright}.
\end{proof}

\newpage
\subsection{Operads}
In this section, we fix a field $k$ and an $\infty$-operad
$\mathcal{O}$  internal to $\md_k$ (cf.\ e.g.\ \cite[Definition 4.1.4]{brantnerthesis}) satisfying the following three basic properties: 
\begin{enumerate}
\item $\mathcal{O}(0)$ is contractible; 
\item $\mathcal{O}(1)$ is equivalent to $k$ via the  unit map $k \rightarrow \mathcal{O}(1)$;
\item $\mathcal{O}(i)\in  \coh_{k, \geq 0}$ is connective and of finite type for all $i \geq 0$.   
\end{enumerate}
The $\infty$-category $\mathrm{Alg}_{\mathcal{O}}$
of $\mathcal{O}$-algebras in $\md_k$, which comes with a  free-forgetful adjunction \INN{060@$\free$}
\INN{060@$\forget$}
 \begin{equation} 
\begin{tikzcd}
	\free_{\mathcal{O}}: \md_k
	\arrow[r, shift left=0.6ex]
	& \mathrm{Alg}_{\mathcal{O}} : \forget .
	\arrow[l, shift left=0.6ex]
\end{tikzcd}
\end{equation}
The free functor is given by the formula\vspace{2pt}
$\free_{\mathcal{O}}(V) \simeq \bigoplus_{i \geq 1} ( \mathcal{O}(i) \otimes
V^{\otimes i})_{h \Sigma_i}$.

The main result of this section is that when we restrict to formal moduli problems defined on \emph{connected} $\mathcal{O}$-algebras (i.e.\ $\mathcal{O}$-algebras in $\md_{k, \geq 1}$), then our  axiomatic \Cref{fmpequiv} implies a classification of formal moduli problems. 
We stress that this will not quite recover the (harder) main result of \Cref{einfdefsectiona} above due to the stronger connectedness assumption; more on this point in \Cref{restrictiveremark} below.
Since the essential features are very similar to the previous examples, we
will be brief. Compare also the result of Ching-Harper \cite{ChingHarper},
which proves the comonadicity assertion under the stronger connectedness assumption.

\begin{cons}[The setup for connected $\mathcal{O}$-algebras]\label{consOalg} Let $k$ be a field. \INN{030@$\mathcal{C}$}\INN{030@$\mathcal{C}^{\gr}$} \INN{030@$\mathcal{C}^{\fil}$} \INN{032@$\cot$}\INN{190@$\sqz$}\INN{060@$F^1$}\INN{1@$\adic$}
\begin{enumerate}[a)]
\item  Let $\mathcal{C}$ be the $\infty$-category $\mathrm{Alg}_{\mathcal{O}}(
\md_{k, \geq 1})$ of connected $\mathcal{O}$-algebras.
\item Let $\mathcal{C}^{\fil} = \mathrm{Alg}_{\mathcal{O}}( \fil \md_{k, \geq
1})$ be the $\infty$-category of filtered $\mathcal{O}$-algebras.
\item Let  $\mathcal{C}^{\gr} = \mathrm{Alg}_{\mathcal{O}}( \gr \md_{k, \geq 1})$. Denote the $\infty$-category of graded $\mathcal{O}$-algebras.
\item We have a  free-forgetful adjunction
$\free_{\mathcal{O}}: \md_{k, \geq 1} \rightleftarrows \mathcal{C}: \forget$.
The natural augmentation map from $\mathcal{O}$ to the trivial operad
induces an adjunction 
$\cot_{\mathcal{O}}: \mathcal{C} \rightleftarrows    \md_{k, \geq 1} : \sqz$.
We define free-forgetful and cotangent fibre adjunctions in the filtered and graded context in a similar way. 
 
\item The adic filtration functor $\adic: \mathcal{C} \rightarrow \mathcal{C}^{\fil}$ is right adjoint to
the underlying functor $F^1: \mathcal{C}^{\fil} \rightarrow \mathcal{C}$.
\item The associated graded functor lifts naturally to define a functor $\gr: \mathcal{C}^{\fil} \rightarrow \mathcal{C}^{\gr}$.
\end{enumerate}
\end{cons}
Because of the connectedness
 assumption, verifying the hypotheses of
\Cref{thm:mainaxiomatic} turns out to be much simpler than before, as  convergence
works more nicely. 
To verify this, we will first show that the adic filtration converges  automatically  for
connected $\mathcal{O}$-algebras. In a second step, we then show that finiteness can be detected by (and is reflected in)
$\cot_{\mathcal{O}}$.

\begin{example}[The adic filtration on a free algebra] 
\label{freealgebraOadic}
The adic filtration on  $\free_{\mathcal{O}}(V)$ 
is given by 
$$ F^n \adic ( \free_{\mathcal{O}}(V)) \simeq 
\bigoplus_{i \geq n} (\mathcal{O}(i) \otimes V^{\otimes i})_{h \Sigma_i}
.$$
In particular, if $V \in \md_{k, \geq 1}$, then the filtration converges for 
connectivity reasons. 
\end{example}

\begin{proposition}[Convergence of the adic filtration] 
Let $A \in \mathrm{Alg}_{\mathcal{O}}( \md_{k, \geq 1})$ be a connected
$\mathcal{O}$-algebra. Then the adic
filtration on $A$ converges.  
\label{adicmodOconv}
\end{proposition} 
\begin{proof} 
This follows 
from 
\Cref{freealgebraOadic} and \Cref{geomrealcomplete}, since  any connected $\mathcal{O}$-algebra is the geometric realisation of free connected $\mathcal{O}$-algebras.
\end{proof} 

The next result, which is essentially already contained in  \cite{HarperHess}, shows that 
finite type conditions can be detected using the cotangent fibre functor $\cot_{\mathcal{O}}$ (under the assumption of connectedness).
\begin{proposition}[Finiteness, completeness, and $\cot_{\mathcal{O}}$]
\label{finitecompleteOalgebra}\ 
\begin{enumerate}
\item Let
$A \in \mathrm{Alg}_{\mathcal{O}}( \md_{k, \geq 1})$ be connected. Then $A \in \coh_{k, \geq 1}$ if and only if \mbox{$\cot_{\mathcal{O}}(A) \in \coh_{k,
\geq 1}$.} 
\item
Let $B \in \alg_{\mathcal{O}}( \gr \md_{k, \geq 1})$ be connected and graded. Then $B \in \grcoh_{k, \geq 1}$ if and only if $\cot_{\mathcal{O}}(B) \in \grcoh_{k, \geq 1}$. 
\item
Let $C \in \alg_{\mathcal{O}}( \fil \md_{k, \geq 1})$ be a connected, filtered
$\mathcal{O}$-algebra. 
Then $C \in \filcoh_{k, \geq 1}$ if and only if 
$C$ is complete and $\cot_{\mathcal{O}}(C) \in \filcoh_{k, \geq 1}$. 
\end{enumerate}
\end{proposition} 
\begin{proof} 
For part (1), let $A$ be a connected $\mathcal{O}$-algebra. Suppose that $\cot_{\mathcal{O}}(A)$ is of finite type. The fact that $A$ is of finite type (as a $k$-module spectrum) follows from the adic filtration on $A$. Indeed, this filtration converges  by \Cref{adicmodOconv}, and the terms of its associated graded are each of finite type and become arbitrarily connected.
Conversely, if $A$ is of finite type, then the bar construction can be used to express 
$\cot_{\mathcal{O}}(A)$ as a geometric realisation of a simplicial $k$-module
spectrum whose terms are of the form \mbox{$\free_{\mathcal{O}} \circ \dots \circ
\free_{\mathcal{O}}(A)$.} Each of these is connected and of finite type, so that the geometric
realisation $\cot_{\mathcal{O}}(A)$ is connected of finite type as well.

Part (2)   follows directly from part
(1) by forgetting to
underlying ungraded $\mathcal{O}$-algebras. 

Part (3) follows from part (2) together with the claim that if $C$ is a
connected filtered $\mathcal{O}$-algebra whose underlying object belongs to 
$\filcoh_{k, \geq 1}$, then $\cot_{\mathcal{O}}(C)$ is automatically complete. This last
claim again
follows from the bar construction: the key observation is that
$\free_{\mathcal{O}}$ preserves $\filcoh_{k, \geq 1}$,  and that $\filcoh_{k, \geq 1}$
is closed under geometric realisations. 
\end{proof}

We can explicitly identify the full subcategories $\mathcal{C}_{\afp}, \mathcal{C}^{\fil}_{\afp},$ \vspace{2pt} and $\mathcal{C}^{\gr}_{\afp}$ which appear when  applying the axiomatic definitions from \Cref{axiomaticsec} to the setup specified in \Cref{consOalg} The following is immediate from \Cref{finitecompleteOalgebra}:
 \INN{030@$\mathcal{C}^{\fil}_{\afp}$} \INN{030@$\mathcal{C}^{\gr}_{\afp}$} \INN{030@$\mathcal{C}_{\afp}$}
\begin{corollary} \
\label{critforOalgafp}
\begin{enumerate}
\item An object $A \in \alg_{\mathcal{O}}( \md_{k, \geq 1})$ is complete almost
finitely presented if and only if the underlying $k$-module spectrum is finite
type. 
\item  
An object $B \in \alg_{\mathcal{O}}( \gr\md_{k, \geq 1} )$ is almost finitely
presented if and only if the underlying $k$-module spectrum of $B$ is of finite type. 
\item
An object $C \in \alg_{\mathcal{O}}( \fil\md_{k, \geq 1})$ is complete almost
finitely presented if and only if the underlying 
object of $\fil\md_{k, \geq 1}$ belongs to $\filcoh_{k, \geq 1}$. 
\end{enumerate}
\end{corollary}

\begin{corollary} 
The above satisfies the conditions of \Cref{filteredrefinementdef}.   
\end{corollary} 
\begin{proof} 
Since the conditions of almost finite presentation, complete almost finite
presentation, and so forth are purely module-theoretic 
in view of \Cref{critforOalgafp}, the conditions of
\Cref{filteredrefinementdef} are evidently satisfied.
\end{proof} 

\begin{cons} 
It follows that we obtain a  monad  
$T_{\mathcal{O}}^{\vee}$ on $\md_k$ and a
Koszul duality functor 
$$ \D: \alg_{\mathcal{O}}( \md_{k, \geq 1}) \to
\alg_{T_{\mathcal{O}}^{\vee}}^{op}$$
as \Cref{thm:mainaxiomatic} and  \Cref{theadjunction}.
By construction, the monad $T_{\mathcal{O}}^{\vee}$ preserves sifted colimits and its value on  $V \in
\coh_{k, \leq -1}$ is given by $T_{\mathcal{O}}^{\vee}(V) \simeq
\cot_{\mathcal{O}}( \sqz(V^{\vee}))^{\vee}$. 
\end{cons}

\begin{remark} 
Let $\mathcal{O}^{\vee}= \Barr(\mathcal{O})$ be the Koszul dual $\infty$-operad in $\md_k$. 
Then the functor $T_{\mathcal{O}}^{\vee}$ is given, for $V \in \coh_{k, \leq -1}$,  by
the formula
$$ V \mapsto 
( \cot_{\mathcal{O}} \circ \sqz(V^{\vee}))^{\vee} =
\bigoplus_{i \geq 1} (\mathcal{O}^{\vee}(i) \otimes V^{\otimes
i})^{h \Sigma_i}. $$
Here  the product could be  interchanged with the sum for connectivity
reasons. 
Roughly speaking, we should regard 
$T_{\mathcal{O}}^{\vee}$-algebras as ``divided power'' algebras over the Koszul
dual operad $\mathcal{O}^{\vee}$. 
\end{remark}  
We introduce the following explicit definition:
\begin{definition}[Artinian $\mathcal{O}$-algebras and formal moduli problems] \label{concretedef}\ 
\begin{enumerate}
\item  
A connected $\mathcal{O}$-algebra $A$ is 
Artinian 
if $\pi_*(A)$ is a finite-dimensional $k$-vector space. 
Let $\alg_{\mathcal{O}}^{\art}$ denote the $\infty$-category of Artinian
$\mathcal{O}$-algebras. 
\item
An \emph{$\mathcal{O}$-formal moduli problem} is a functor 
$F: \alg_{\mathcal{O}}^{\art} \rightarrow \mathcal{S}$ 
such that if 
$A\rightarrow A''$, $A' \rightarrow A''$ are maps in $\alg_{\mathcal{O}}^{\art}$ inducing surjections on
$\pi_1$, then $F(A \times_{A''} A') \simeq F(A) \times_{F(A')} F(A'')$. 
\end{enumerate}
\end{definition} 
We can then use our results in \Cref{axiomaticsec}:
\begin{corollary} 
{There is an equivalence between the $\infty$-category of $\mathcal{O}$-formal moduli problems and
 the $\infty$-category $\alg_{T_{\mathcal{O}}^{\vee}}$.} 
\end{corollary} 
\begin{proof} 
In order to apply \Cref{fmpequiv}, it suffices to verify that 
Artinian $\mathcal{O}$-algebras in the sense of \Cref{concretedef}(1)
 are exactly the Artinian objects with respect to the deformation theory given by $\alg_{\mathcal{O}}( \md_{k, \geq 1})$, i.e.\ objects which can be built up inductively by pullbacks
of trivial extensions (cf.\ \Cref{axiomaticart}). It then also follows that \Cref{concretedef}(2) is an instance of \Cref
{Cfmp} in the present context.

Observe that any $\mathcal{O}$-algebra with homotopy groups concentrated in degree
$1$ is necessarily square-zero by our connectivity assumptions on $\mathcal{O}$. 
Given an algebra $A \in \alg_{\mathcal{O}}^{\art}$, we write $\tau_{\leq n}A$ for the $n^{th}$ Postnikov truncation of $A$ (cf.\ \cite[Proposition 5.5.6.18]{lurie2009higher}). Arguing as in \cite[Proposition 7.1.3.15]{lurie2014higher}, we see that the underlying $k$-module of $\tau_{\leq n}A$ is simply given by the $k$-truncation in $\mod_k$.
It then suffices to verify that if $A \in \alg_{\mathcal{O}}^{\art}$ has top homotopy
group
in degree $n$, then there exists a pullback square
of $\mathcal{O}$-algebras
$$ \xymatrix{
A \ar[d] \ar[r] & 0 \ar[d]  \\
\tau_{\leq n-1} A \ar[r] &  \sqz( (\pi_n A)[n+1])
}.$$
In the case of $\EE_\infty$-rings, this observation appears in \cite{IK1} and \cite{MR1732625}, and  is discussed in modern language in \cite[Corollary 7.4.1.28]{lurie2009higher}.

To construct the desired pullback square, we first form the pushout 
$P = \tau_{\leq n-1} A \sqcup_{A} 0$ in $\alg_{\mathcal{O}}$ and then apply the functor
$\tau_{\leq n+1}$ to it, which implies the claim as   $\tau_{\leq n+1} P \simeq
\sqz( (\pi_n A) [n+1])$. 
\end{proof} 

\begin{remark} 
In characteristic zero, $T_{\mathcal{O}}^{\vee}$ agrees with    the free $\mathcal{O}^{\vee}$-algebra monad. In particular, the assertion is that 
$\mathcal{O}$-formal moduli problems are equivalent to
$\mathcal{O}^{\vee}$-algebras under the above assumptions. This fact is
well-known to experts, but we are not aware   it has appeared \mbox{explicitly yet.}
\end{remark}

\begin{remark}[Comparison with spectral formal moduli problems] \label{restrictiveremark}
One can apply the above results when $\mathcal{O}$ is the (nonunital)
commutative operad. This gives rise to
a classification of a variant of spectral formal moduli problems (which are only   defined on \emph{connected} Artinian \mbox{$\einf$-$k$-algebras}) in terms of the same $\infty$-category of spectral partition Lie algebras. 

In particular, it follows that spectral formal moduli problems over $k$ are
determined by their restriction to connected Artinian $\einf$-$k$-algebras; they
automatically extend to all Artinian $\einf$-$k$-algebras. 
However, the arguments are  simpler when one restricts to the connected
case (as in the present section), and apply to more general operads $\mathcal{O}$. 

The fact that, for $\mathcal{O}$ the commutative operad, 
the theory extends to  some connective (rather than connected) objects requires an
additional calculation (and does not appear to \mbox{be purely formal).}

\end{remark}

\newpage

\section{Deformations over a complete local base}
Assume that $A$ is complete local Noetherian with residue field $k$ (cf.\ \Cref{clndef}), either in the setting of \mbox{$\einf$-rings}  or in the setting of  simplicial commutative rings.  The following rings will play the role of infinitesimal thickenings of $\Spec(k)$ in this mixed setting:
\begin{definition}[Artinian  rings for $A//k$] \label{Artinianrelative}   \INN{032@$\clg_{A//k}^{\art}$}
An object $B$ of $\clg_{A//k}$ is called \emph{Artinian} if it is connective and the following conditions hold:
\begin{enumerate}
\item $\pi_0(B)$ is a local Artinian ring (with residue field $k$);
\item $\bigoplus_{i \geq 0}
\pi_i(B)$ is a finitely generated module over $\pi_0(B)$ (in particular,
$\pi_i(B) = 0$ for $i \gg 0$). \end{enumerate}
We let $\clg_{A//k}^{\art} $ denote the full subcategory of $ \clg_{A//k}$ \INN{190@$\SCR_{A//k}^{\art}$}
spanned by Artinian objects. \vspace{3pt}

An object of $\SCR_{A//k}$ is \emph{Artinian} if the
underlying object of $\clg_{A//k}$ is, and we let $\SCR_{A//k}^{\art}$ be the full subcategory of $ 
\SCR_{A//k}$ spanned by all Artinian objects. 
\end{definition} 

We can now generalise the notion of a formal moduli problem to the relative context; note that this notion also appears in \cite[Section 6.1]{lurie2011derivedXII}.

\begin{definition}[Formal moduli problems for $A//k$]\label{fmpmixed}   \INN{130@$\moduli_{A//k,\einf}, \moduli_{A//k, \Delta}$}
A \emph{spectral formal moduli problem for 
$A//k$}  is 
a functor $F: \clg_{A//k}^{\art} \rightarrow \mathcal{S}$ 
such that: 
\begin{enumerate}
\item  
$F(k) $ is contractible;
\item
if $B, B', B'' \in \clg_{A//k}^{\art}$ and we have maps $B \rightarrow B''$ and $B' \rightarrow B''$ which induce surjections on $\pi_0$, then the canonical map $ F(B \times_{B''} B') \xrightarrow{ \simeq } F(B) \times_{F(B'')} F(B')  $ is an equivalence.\vspace{-1pt}
\end{enumerate} 
We denote the $\infty$-category of spectral formal moduli problems by $\moduli_{A//k,
\einf}$. \vspace{3pt}

Similarly, a  \emph{derived formal moduli problem} for $A//k$
is a functor 
$F: \SCR_{A//k}^{\art} \rightarrow \mathcal{S}$ satisfying the analogous conditions (1)
and (2) above. 
We denote the $\infty$-category of derived formal moduli problems 
for $A//k$ by $\moduli_{A//k, \Delta}$. 
\end{definition} 

\begin{example} 
Suppose $k$ is a perfect field of characteristic $p>0$. 
Let $A=W^+(k)$ \INN{230@$W(k)$,  $W^+(k)$} denote the spherical Witt vectors of $k$ (cf.\ \cite[Example 5.2.7]{lurie2018elliptic}), so that $W^+(k)$ is a
complete local Noetherian $\einf$-ring with residue field $k$. 
Then the $\infty$-category of Artinian objects of $\clg_{A//k}$ is equivalent to
a full subcategory of $\clg_{/k}$ (namely, those which are Artinian). 
It therefore follows that we can regard spectral formal moduli problems as defined
on all Artinian $\einf$-algebras augmented over $k$;  the map from $A $ 
is therefore superfluous. 

A similar statement holds for derived formal moduli problems and the classical \mbox{Witt vectors $W(k)$.}
\end{example}

The principal goal  of this section is to generalise \Cref{einfsec} to the mixed context. Our main result is \Cref{mostgeneral}
below, which gives a   Lie algebraic description of 
spectral and derived formal moduli problems for $A//k$.

\subsection{Descent properties of modules}
Let $A$ be a  complete local Noetherian $\EE_\infty$-ring with residue field $k$.
We will now establish several  convergence results on modules. This will later allow us to reduce the proof of
\Cref{mostgeneral}
 to the case 
$A =k$, which has  been handled in  \Cref{einfsec}. 

\begin{definition}[Complete $A$-modules] \label{completemodules}  
An $A$-module spectrum $M \in \md_A$ is \emph{complete} if it is derived
$\mathfrak{m}$-complete (cf.\ \cite[Theorem 7.3.4.1]{lurie2016spectral}), where $\mathfrak{m} \subset \pi_0(A)$ is the unique
maximal ideal. 
\end{definition}

\begin{proposition}[Convergence criterion for $A$-modules] 
\label{mainconvcritmodules} 
Let $M^\bullet$ be a cosimplicial object of $\md_{A, \geq 0}$. 
Suppose that each $M^i$ is complete and that 
$\mathrm{Tot}(k  \otimes_A M^\bullet  ) \in \md_k$ is connective. 
Then: 
\begin{enumerate}
\item $\mathrm{Tot}(M^\bullet)$ is connective  and complete (as an $A$-module);  
\item if $N$ is an almost perfect $A$-module, then  
$N \otimes_A \mathrm{Tot}(M^\bullet ) \xrightarrow{\simeq} \mathrm{Tot}(N \otimes_A M^\bullet)$
is an equivalence. 
\end{enumerate}
\end{proposition} 
\begin{proof} 
We begin with statement (1). Consider the class $\mathcal{T}$ of connective $A$-modules $N$ for which
$\mathrm{Tot}(N \otimes_A M^\bullet )$ is connective. 
By assumption, $\mathcal{T} $ contains   $k[i]$  for $i
\geq 0$. Moreover, $\mathcal{T}$ is closed
under extensions: given a cofibre sequence $N_1 \rightarrow N_2 \rightarrow N_3 $ with $N_1, N_3
\in \mathcal{T}$, it follows that $N_2 \in \mathcal{T}$. 

Assume that $N'$ is an almost perfect and connective $A$-module  such that the homotopy groups
$\pi_i(N')$ are  all $\mathfrak{m}$-power torsion.
By induction, the above properties of $\mathcal{T}$ show that all
the truncations $\tau_{\leq n} N'$ 
belong to
$\mathcal{T}$. 
Passage to the limit as $n \rightarrow \infty$ now shows that $N' \in \mathcal{T}$ too.

For instance, if $x_1, \dots, x_r \in \pi_0(A)$ generate the
maximal ideal, then if $N$ is any connective, almost perfect $A$-module, we
conclude that the iterated cofibre $N/(x_1, \dots, x_r)$ belongs to $\mathcal{T}$. 
It follows that $\mathrm{Tot}( N\otimes_A M^\bullet) / (x_1, \dots, x_r)$ is a connective
$A$-module. Since each 
$ N\otimes_A M^i$ is complete, it follows that 
$\mathrm{Tot}(N \otimes_A M^\bullet)$ is a complete $A$-module. Thus, it also follows
that 
$\mathrm{Tot}( N \otimes_A M^\bullet)$
is connective itself. 
Therefore, we have shown that $\mathcal{T}$ contains every connective, almost
perfect $A$-module. 
In particular, taking $N  = A$ verifies part (1) of the theorem. 

We shall now  verify part (2). The claim is clearly true in the case where $N$ is perfect. 
Suppose that $N$ is an arbitrary almost perfect $A$-module, and assume without restriction that $N$ is 
also
connective. 
Both domain and target of the map $N\otimes_A \mathrm{Tot}(M^\bullet)  \rightarrow \mathrm{Tot}(N \otimes_A M^\bullet)$
are connective (since $N \in \mathcal{T}$ by  the previous paragraph).  
Fix $n > 0$. 
We can find a perfect $A$-module $P$ and a map $P \rightarrow N$ which induces an
equivalence on $n$-truncations, so we obtain a cofibre sequence $P \rightarrow N \rightarrow N'$
where $N' \in \md_{A, \geq n+1}$. Since $N'[-n-1] \in \mathcal{T}$, it follows that 
$\mathrm{Tot}(N' \otimes_A M^\bullet) \in \md_{A, \geq n+1}$. 
In particular, 
in the commutative square
$$ \xymatrix{
  P \otimes_A \mathrm{Tot}(M^\bullet) \ar[d]  \ar[r] &  \mathrm{Tot}(P \otimes_A M^\bullet
) \ar[d]  \\
 N \otimes_A \mathrm{Tot}(M^\bullet) \ar[r] &  \mathrm{Tot}(N \otimes_A
M^\bullet )
},$$
the vertical maps 
$ P\otimes_A  \mathrm{Tot}(M^\bullet)\rightarrow N \otimes_A \mathrm{Tot}(M^\bullet)$ 
and 
$\mathrm{Tot}(P \otimes_A M^\bullet) \rightarrow \mathrm{Tot}(N \otimes_A M^\bullet )$
are equivalences on $n$-truncations. 
Since the top horizontal map is an equivalence, it follows that 
the bottom horizontal map is an equivalence on $n$-truncations. 
Since $n$ was arbitrary, we conclude that the bottom horizontal map is an
equivalence, which implies (2). 
\end{proof}

Next, we show that for connective  complete $A$-modules, the   Adams spectral sequence
converges. 
More precisely, consider the \v{C}ech nerve  of $A \rightarrow k$, i.e.\ the  augmented cosimplicial
diagram
\begin{equation} \label{Cechnerve} 
A \longrightarrow \left( \ \  k \ \  \substack{  \longrightarrow  \vspace{-4pt} \\ \leftarrow  \vspace{-4pt} \ \\ \longrightarrow}\ \    k \otimes_A k\ \  \substack{\longrightarrow \vspace{-4pt} \\ \leftarrow   \vspace{-4pt} \ \\ \longrightarrow  \vspace{-4pt} \\ \leftarrow  \vspace{-4pt} \ \\ \longrightarrow} \ \ \ldots \  \ \right)
\end{equation} 
in the $\infty$-category of $\einf$-rings. We then have:
\begin{proposition} \label{theabove}
As before, let $A$ be a complete local Noetherian $\einf$-ring with residue field $k$.  
If  $M$ is  a connective and complete $A$-module, then the  
diagram 
$M \xrightarrow{  \simeq  } \mathrm{Tot} \left(  \  M \otimes_A k \    \substack{  \longrightarrow  \vspace{-4pt} \\ \leftarrow  \vspace{-4pt} \ \\ \longrightarrow}\    M \otimes_A  k \otimes_A k\    \substack{\longrightarrow \vspace{-4pt} \\ \leftarrow   \vspace{-4pt} \ \\ \longrightarrow  \vspace{-4pt} \\ \leftarrow  \vspace{-4pt} \ \\ \longrightarrow} \   \ldots \  \right)$
(obtained by tensoring \eqref{Cechnerve} with $M$) is a limit diagram.
\end{proposition} 
\begin{proof} 
After base-change along the map $A \rightarrow k$, the augmented cosimplicial diagram
\begin{equation} \label{Cechnerve2}   M \to  M \otimes_A  k   \  \substack{
\longrightarrow  \vspace{-4pt} \\ \leftarrow  \vspace{-4pt} \ \\
\longrightarrow}\   M \otimes_A k \otimes_A k\    \substack{\longrightarrow \vspace{-4pt} \\ \leftarrow   \vspace{-4pt} \ \\ \longrightarrow  \vspace{-4pt} \\ \leftarrow  \vspace{-4pt} \ \\ \longrightarrow}   \ \ldots   \end{equation}
admits a splitting, since it becomes the \v{C}ech nerve of the map $k
\rightarrow k \otimes_A k$, which has a section, tensored with $M \otimes_A k$.
\Cref{mainconvcritmodules} therefore applies to diagram \eqref{Cechnerve2}, which implies that 
 the totalisation of the cosimplicial diagram is connective and commutes with base-change with
any almost perfect $A$-module. 
In particular, the totalisation commutes with base-change to $k$, and since the cosimplicial
diagram becomes split after base-change to $k$, we find (also by Nakayama's
lemma) that the augmented cosimplicial diagram \eqref{Cechnerve2} is a limit
diagram. 
%From this, we deduce that $M \xrightarrow{  } \mathrm{Tot} \left(     M \otimes_A k \    \substack{  \longrightarrow  \vspace{-4pt} \\ \leftarrow  \vspace{-4pt} \ \\ \longrightarrow}\    M \otimes_A  k \otimes_A k\    \substack{\longrightarrow \vspace{-4pt} \\ \leftarrow   \vspace{-4pt} \ \\ \longrightarrow  \vspace{-4pt} \\ \leftarrow  \vspace{-4pt} \ \\ \longrightarrow}    \ldots  \right)$
%is a map of connective, complete $A$-modules which becomes an equivalence after
%applying $k\otimes_A - $. Thus, it  is an equivalence by \mbox{Nakayama's lemma.}
\end{proof} 
As a strengthening of \Cref{theabove}, we 
observe the following descent theorem
for complete connective modules (it implies \Cref{theabove} by comparing  
mapping objects on both sides).
Note that descent for all modules in the 
faithfully flat case appears in \cite[Section D.6.3]{lurie2016spectral} and in the
proper surjective case in \cite[Theorem 5.6.6.1]{lurie2016spectral}. 
We thank Bhargav Bhatt for indicating the following result to us. 
\newcommand{\cpl}{\mathrm{cpl}}
\begin{theorem}\label{bhbh}
Let $A$ be a complete local Noetherian $\einf$-ring with residue field $k$. 
Writing $\md_{A, \geq 0}^{\cpl}$ \INN{130@$\md_{A, \geq 0}^{\cpl}$}
 for the full subcategory of all complete connective $A$-modules  (cf.\ \Cref{completemodules}), the natural map 
$$ 
\md_{A, \geq 0}^{\cpl} \  \xrightarrow{ \ \ \simeq \ \  } \     \mathrm{Tot}\left( \  \md_{k, \geq 0}  \    \substack{  \longrightarrow  \vspace{-4pt} \\ \leftarrow  \vspace{-4pt} \ \\ \longrightarrow}\       \md_{k\otimes_A k, \geq 0} \    \substack{\longrightarrow \vspace{-4pt} \\ \leftarrow   \vspace{-4pt} \ \\ \longrightarrow  \vspace{-4pt} \\ \leftarrow  \vspace{-4pt} \ \\ \longrightarrow}    \ldots  \ \right)$$
is an equivalence of $\infty$-categories.
\end{theorem}

\begin{proof} 
It suffices to show that 
the functor 
$k \otimes_A -: \md_{A, \geq 0}^{\cpl} \rightarrow \md_{k, \geq 0}$ is comonadic as
in \cite[Lemma D.3.5.7]{lurie2016spectral}. 
First, we observe that $  k\otimes_A - $ is conservative on connective and
complete $A$-modules by Nakayama's lemma. 
Next, let $M^\bullet$ be an object of $\md_{A, \geq 0}^{\cpl}$ 
such that the cosimplicial $k$-module $ k\otimes_A M^\bullet$ admits a splitting. 
It follows that 
$\mathrm{Tot}(k \otimes_A M^\bullet )$ (computed in $k$-modules) is connective. 
Thus, by \Cref{mainconvcritmodules}, 
$\mathrm{Tot}(M^\bullet)$ is connective and 
$k\otimes_A  \mathrm{Tot}(M^\bullet) \rightarrow \mathrm{Tot}(k\otimes_A M^\bullet ) $
is an equivalence. This verifies the hypotheses of the comonadicity
theorem. 
\end{proof} 

\newcommand{\aperf}{\mathrm{APerf}}
\begin{notation}\INN{013@$\aperf_R$} 
Let $\aperf_R$ be the full subcategory of $\md_R$ spanned by   almost perfect
$R$-modules. \end{notation}
As a consequence of \Cref{bhbh}, we observe also that almost perfect modules
satisfy descent, cf.\ \cite[Sec. 4]{HLP} for closely related results:
\begin{corollary}[Descent for almost perfect modules]\label{DAPM}
Let $A$ be a complete local Noetherian $\einf$-ring with residue field $k$. 
Then the diagram
$$ 
\aperf_A \  \xrightarrow{ \ \ \simeq \ \  } \     \mathrm{Tot}\left( \  \aperf_k  \    \substack{  \longrightarrow  \vspace{-4pt} \\ \leftarrow  \vspace{-4pt} \ \\ \longrightarrow}\     \aperf_{k \otimes_A k} \    \substack{\longrightarrow \vspace{-4pt} \\ \leftarrow   \vspace{-4pt} \ \\ \longrightarrow  \vspace{-4pt} \\ \leftarrow  \vspace{-4pt} \ \\ \longrightarrow} \    \ldots  \ \right)$$
 is an equivalence of symmetric monoidal $\infty$-categories. 
 This remains  true when we replace $\aperf$ by the corresponding $\infty$-categories $\perf$ of perfect modules. 
\end{corollary} 
\begin{proof} For the first claim, by \Cref{bhbh}, it suffices to check that if $M$ is any connective  complete $A$-module with $k
\otimes_A M\in \aperf_k$, then $M $ belongs to $\aperf_A$, which means that  $M$ has finitely generated homotopy groups.
We show inductively 
that the homotopy groups of $M$ are finitely generated. 

Indeed, if $M$ is such a module, then 
$\pi_0(k \otimes_A M ) \simeq k \otimes_{\pi_0(A)} M$ is finitely generated. 
Choose a map $A^r \rightarrow M$ which induces a surjection on $\pi_0(k \otimes_A
\cdot)$ and let $C$ be the cofibre. 
Then \cite[\href{https://stacks.math.columbia.edu/tag/09B9}{Lemma
09B9}]{stacks-project} implies that $\pi_0(A^r) \rightarrow \pi_0(M)$ is
surjective, and $\pi_0(C)=0$.
Therefore, $\pi_0(M)$ is finitely generated.

Now assume $n > 0$, and that for any connective, complete module $M$ with $k
\otimes_A M \in \aperf_k$, we have that $\pi_i(M)$ is finitely generated for $i
< n$. We will show additionally that $\pi_n(M)$ is finitely generated, which
will establish the claim by induction. 
Choose the map $A^r \rightarrow M$ as above. 
The inductive hypothesis shows that 
$\mathrm{fib}(A^r \rightarrow M)$  
has finitely generated homotopy groups $\pi_i, i < n$. 
The long exact sequence now shows that $\pi_n(M)$ is finitely generated. 

Finally, for the second claim, we observe that perfect modules are characterised
as the dualisable objects in $\aperf$, so the second claim follows from the
first. 
\end{proof}

\subsection{Constructing a deformation theory}
Let $A$ be a complete local Noetherian $\einf$-ring with residue field $k$.
Write $\clg_{A//k, \geq 0}$ \INN{032@$\clg_{A//k, \geq 0}$} for the $\infty$-category of 
connective $\einf$-$A$-algebras $B$ equipped with a map $B \rightarrow k$. 
In this subsection, we use this data to construct a deformation theory in the sense of Lurie (cf.\ \cite[Definition 12.3.3.2]{lurie2016spectral}). When $A=k$, this was done in \Cref{setupeinf} above, and we will now indicate the necessary modifications.

First, we consider the adjunction \INN{032@$\cot_A$} \INN{190@$\sqz$}  
\begin{equation} 
\begin{tikzcd}\label{cotsqzA}
	\cot_A :\ \clg_{A//k,\ge 0}
	\arrow[r, shift left=0.65ex]
	& \md_{k,\ge 0}\ : \sqz
	\arrow[l, shift left=0.65ex]
\end{tikzcd}
\end{equation}

where:
\begin{enumerate}[a)]
\item
the left adjoint $\cot_A$ sends $B \in \clg_{A//k, \geq 0}$ to $\cot_A(B)
:= k \otimes_B L_{B/A}  $;
\item  
the right adjoint $\sqz$ sends $V \in \md_{k, \geq 0}$ to 
the object $k \oplus V$, considered as a trivial square-zero $k$-algebra and equipped
with an $A$-algebra structure via $A \rightarrow k$. 
\end{enumerate}
 
\begin{remark} 
As  $\clg_{A//k, \geq 0}$ is not pointed when $A \neq k$, the mixed context does not quite fit 
 into the framework of Section~\ref{axiomaticsec}. However, it will be possible to deduce all   results from the case $A= k$. 
This is possible because the adjunction  \eqref{cotsqzA} above is the composite of   \eqref{stablecotsqz} with the adjunction
$$\clg_{A//k, \geq 0} \rightleftarrows \clg_{k//k, \geq 0}$$ given by base-change
and forgetting. 
In particular,  observe that for any $B \in \clg_{A//k, \geq 0}$, we have
\begin{equation} \label{cotAtocot} \cot_A(B) \simeq \cot( k \otimes_A B).
\end{equation}
\end{remark}

\newcommand{\cn}{\mathrm{cN}}

\begin{notation} 
Let $\clg_{A//k}^{\cn} $\INN{032@$\clg_{A//k}^{\cn} $}  
denote the full
subcategory of $ \clg_{A//k,\geq 0}$ spanned by those objects $B$ which are complete local Noetherian, i.e.\ 
such that $B$ is Noetherian (cf.\ \Cref{def:noeth}) and such that $\pi_0(B)$ is a complete local ring. 
We use similar notation $\SCR_{A//k}^{\cn}$  \INN{190@$\SCR_{A//k}^{\cn}$} when $A$ is a simplicial commutative
ring which is complete local Noetherian with residue field $k$. 
\end{notation} 

\begin{example} 
\label{freeobjectscompleted}
The completion $\widehat{A\left\{x_1, \dots, x_n\right\}}$ of a free
$\einf$-algebra in variables $x_1, \dots, x_n$ over $A$ is an object of $\clg_{A//k}^{\cn}$.
Indeed, these are the free objects of $\clg_{A//k}^{\cn}$: 
if $B \in \clg_{A//k}^{\cn}$, then  
$\hom_{\clg_{A//k}^{\cn}}( \widehat{A\left\{x_1, \dots, x_n\right\}}, B) \simeq
\Omega^\infty \mathfrak{m}_B^n$, where $\mathfrak{m}_B = \mathrm{fib}(B \rightarrow k)$
is the augmentation ideal of $B$. 
\end{example} 
We  begin by observing that any object of $\clg^{\cn}_{A//k}$ can be written
as a geometric realisation of such completed-free objects; while this will not be used in the sequel. For convenience, we state the result
as well for simplicial commutative rings. 
\begin{theorem} \label{freeresolutions}
Let $A$ be a complete local Noetherian $\einf$-algebra (resp. simplicial commutative ring)
augmented over $k$. 
Then any object of $\clg_{A//k}^{\cn}$ (resp. $\SCR_{A//k}^{\cn}$) can be
expressed as the geometric realisation of a simplicial object $X_\bullet$ in
$\clg_{A//k}^{\cn}$ (resp. $\SCR_{A//k}^{\cn}$) where each $X_i$ is the formal
completion of a free algebra over $A$ on finitely many variables in degree $0$. 
\end{theorem} 

\begin{proof} 
We give the proof for $\clg_{A//k}^{\cn}$; the simplicial commutative ring case
is similar. Here we use the notation and language of \Cref{hypercoveringlemma}
below. 
We take $\mathcal{C} = \clg_{A//k}^{\cn}$ and $S$ to be the class of maps $B \to
B'$ which induce surjections on $\pi_0$. 
Note that coproducts in $\mathcal{C}$ are given by completed tensor products. 
Similarly, we take $\mathcal{F}$ to be the class of objects in
$\clg_{A//k}^{\cn}$ which are free on a finite set of generators in degree $0$. 
These play the role of free objects in $\clg_{A//k}^{\cn}$ as in
\Cref{freeobjectscompleted}, so they have the lifting property with respect to
$S$. Thus, we can apply \Cref{hypercoveringlemma} to produce an $(\sF,
S)$-hypercover $X_\bullet$ in $\mathcal{C}$ as desired. 
This is necessarily a colimit diagram, since one can check this after
applying $\Omega^\infty$, and hypercovers are colimit diagrams in
the $\infty$-category $\mathcal{S}$ \cite[Lemma 6.5.3.11]{lurie2009higher}. 
\end{proof}

If $B \in \clg_{A//k}^{\cn}$, then the augmented $\einf$-$k$-algebra $k\otimes_A B$ is complete local Noetherian. It therefore has an almost perfect cotangent fibre by \Cref{noethcotap}, which means  that $\cot_A(B)$ is almost perfect. 
We can therefore restrict \eqref{cotsqzA}  to obtain an adjunction 
 \begin{equation} \label{Afiniteadj} 
 \begin{tikzcd}
 	\cot_A : \clg_{A//k}^{\cn}
 	\arrow[r, shift left=0.65ex]
 	& \coh_{k, \geq 0}: \sqz.
 	\arrow[l, shift left=0.65ex]
 \end{tikzcd}
  \end{equation}
with associated comonad 
$T_A: \coh_{k, \geq 0} \rightarrow \coh_{k, \geq 0}$. Pre- and postcomposing with linear duality as before,  we obtain a monad  $\liepss{A}: \coh_{k, \leq 0} \rightarrow \coh_{k, \leq 0}$  \INN{120@$\liepss{A}$} \INN{040@$\mathfrak{D}_A$}
satisfying
$ \liepss{A}(V)= \cot_A ( \sqz(V^{\vee}))^{\vee}$.

\begin{example} 
Given a complete local Noetherian algebra $B \in \clg_{A//k}^{\cn}$, we set
\mbox{$\cot(B) = L_{k/B}[-1]$.} Note that if $B$ arises from an augmented $k$-algebra by restriction of scalars along $A\rightarrow k$, then $\cot(B)$ agrees with the cotangent fibre considered before.
The natural pushout square in $\clg_{A//k}^{\cn}$ \mbox{given by}
$$\xymatrix{
 B \ar[d] \ar[r]  & k \otimes_A B  \ar[d] \\
k \ar[r] &  k \otimes_A k
}$$
induces a basic cofibre sequence  
$$ \cot(B)  \rightarrow \cot_A(B) \rightarrow \cot(k \otimes_A k). $$
Taking $B = k \oplus V$, we obtain a cofibre
sequence
\begin{equation}  \label{cofibreseqAtok}
 \cot(k \otimes_A k)^{\vee}
\rightarrow \liepss{A}(V)  \rightarrow   \lieps(V)
 , \quad V \in \coh_{k, \geq 0}. 
\end{equation} 
\end{example} 
We can use this to establish a relative version of \Cref{thm:mainaxiomatic} in the context of $\EE_\infty$-rings:  
\begin{theorem} 
\label{Akcomparestuff}
Let $A$ be a complete local Noetherian $\einf$-ring with residue field $k$. 
Then: 
\begin{enumerate}
\item The adjunction \eqref{Afiniteadj} is comonadic.  
\item The monad $\liepss{A}: \coh_{k, \leq 0} \rightarrow \coh_{k, \leq 0}$ from above  extends
to a sifted-colimit-preserving monad $\liepss{A}$ on  $\md_k$.
\item 
The induced functor $\D_A: (\clg^{\cn}_{A//k})^{op} \rightarrow \mathrm{Alg}_{\liepss{A}}$ 
carries pullbacks
of diagrams \mbox{$B \rightarrow B''$}, $B' \rightarrow B''$ inducing surjections on $\pi_0$ to pushouts
of $\liepss{A}$-algebras. 
\end{enumerate}
\end{theorem}

\begin{proof} 
For $A=k$, we have  verified the claim in
\Cref{einfinitysatisfiesaxiomatic}. We will now reduce to this case. 

For part (1), we verify the hypotheses of the comonadicity theorem (cf.\ \Cref{BBL} above). 
First, we observe that $\cot_A$ is conservative. Indeed, $\cot_A$ is the composite of the base-change functor $\clg_{A//k}^{\cn } \rightarrow \clg_{k//k}^{\cn}$ with the cotangent fibre functor $\cot:\clg_{k//k}^{\cn } \rightarrow \md_k$, both of which are  conservative by \cite[\href{https://stacks.math.columbia.edu/tag/09B9}{Lemma 09B9}]{stacks-project} and    \Cref{detectequivcafp}.

Let $B^\bullet$ be a cosimplicial object in $\clg_{A//k}^{\cn}$ such that
$\cot_A(B^\bullet)$ admits a splitting. 
Using the equivalence
$\cot_A(B^\bullet) \simeq \cot(k \otimes_A   B^\bullet)$, we conclude that
$ k \otimes_A B^\bullet$ defines a cosimplicial object of 
$\clg_{k//k}^{\cn }$ such that $\cot( k \otimes_A B^\bullet)$ is split. 
By the comonadicity already proved when $A = k$, it follows that
$\mathrm{Tot}(k \otimes_A B^\bullet)$ is  
complete local Noetherian), and the natural map 
$$\cot( \mathrm{Tot}(k \otimes_A B^\bullet ))
\xrightarrow{ \ \ \simeq \ \ } \mathrm{Tot}( \cot(k \otimes_A B^\bullet ))
$$
is an equivalence. 
Using \Cref{mainconvcritmodules}, we can conclude that
$\mathrm{Tot}(B^\bullet)$ is connective  and that 
$k \otimes_A \mathrm{Tot}(B^\bullet ) \xrightarrow{\simeq} \mathrm{Tot}(k \otimes_A B^\bullet)$
is an equivalence.
Therefore, $\mathrm{Tot}(B^\bullet)$ is also complete local Noetherian and 
$\cot_A( \mathrm{Tot}(B^\bullet)) \xrightarrow{\simeq} \mathrm{Tot}( \cot_A(B^\bullet))$ is an
equivalence. 

Part (2) follows from the case $A =k$ and the cofibre sequence
\eqref{cofibreseqAtok} constructed above.

Finally, part (3) will be proved in \Cref{Ucommutewithpushouts} below. 
\end{proof} 

We can use \Cref{Akcomparestuff} to generalise \Cref{spla} to the mixed setting:
\begin{definition}[Mixed spectral partition Lie algebras] \label{splamixed}
 Given a complete local Noetherian $\EE_\infty$-ring $A$ with residue field $k$, an
 \textit{$(A,k)$-spectral partition Lie algebra} is an algebra over 
$\liepss{A}$.
\end{definition}

\begin{cons}[The forgetful functor from $\liepss{A}$-algebras to
$\lieps$-algebras]\label{forgetfulrelative}
The base-change functor $\clg_{A//k}^{\cn} \rightarrow \clg_{k//k}^{\cn}$ given by $B \mapsto k \otimes_A B $ induces, by using the anti-equivalences established in \Cref{Akcomparestuff}, a functor
$U:  \alg_{\liepss{A}}( \coh_{k, \leq 0}) \longrightarrow \alg_{\lieps}( \coh_{k, \leq
0}).$
By construction, $U$ preserves the forgetful functors to $\coh_{k, \leq 0}$.
We extend this to a functor
$$U: \alg_{\liepss{A}} \rightarrow \alg_{\lieps}. $$
which commutes with the forgetful functor to $\md_k$. For this, we simply 
 left Kan extend from free
$\liepss{A}$-algebras on objects of $\perf_{k, \leq 0}$ (using \Cref{hypTalgebra}).
\end{cons}
\begin{proposition} 
\label{Ucommutewithpushouts}
The forgetful functor $U: \alg_{\liepss{A}} \rightarrow \alg_{\lieps}$ from \Cref{forgetfulrelative}
commutes with pushouts. 
\end{proposition} 
\begin{proof}
Suppose that we are given maps $V'' \rightarrow V$ and $V'' \to
V'$ with $V, V', V'' \in \coh_{k, \leq 0}$ such that the induced maps $\pi_0(V'') \rightarrow \pi_0(V)$ and $\pi_0(V'') \to
\pi_0(V')$ are injective. 
Consider the $\liepss{A}$-algebras $\mathfrak{g}'' \hspace{-1pt}\hspace{-1pt}=\hspace{-1pt}
\free_{\liepss{A}}\hspace{-1pt}(V''), \hspace{+1pt}
\mathfrak{g} \hspace{-1pt}\hspace{-1pt}= \hspace{-1pt}\free_{\liepss{A}}\hspace{-1pt}(V),$ and $\mathfrak{g}'\hspace{-1pt}\hspace{-1pt} =\hspace{-1pt} \free_{\liepss{A}}\hspace{-1pt}(V')$. The diagram 
$$ \xymatrix{
U(\free_{\liepss{A}}(V'')) \ar[d] \ar[r] &  U(\free_{\liepss{A}}(V)) \ar[d]  \\
U(\free_{\liepss{A}}(V')) \ar[r] &  
U(\free_{\liepss{A}}(V \sqcup_{V''} V'))
}$$
is a pushout of $\lieps$-algebras, as it is by 
construction equivalent to
$$ 
\xymatrix{
\D(k  \otimes_A  (k \oplus V''^{\vee}) ) \ar[d] \ar[r] & \D( k \otimes_A (k \oplus V^{\vee})
) \ar[d]  \\
\D( k \otimes_A (k \oplus V'^{\vee}) ) \ar[r] &  \D( k \otimes_A (k \oplus (V \sqcup_{V''}
V')^{\vee}) )
},
$$
which is a pushout by \Cref{thm:mainaxiomatic} (applied to the case
of augmented $\einf$-$k$-algebras). Here we have used that   $ \pi_0(k  \otimes_A  (k \oplus V^{\vee})) \rightarrow \pi_0(k  \otimes_A  (k \oplus V''^{\vee}))$ and $ \pi_0(k  \otimes_A  (k \oplus V'^{\vee})) \rightarrow \pi_0(k  \otimes_A  (k \oplus V''^{\vee}))$ are surjective.
As we  can write every span in $\alg_{\liepss{A}}$
as a sifted colimit of spans of the above form, the claim follows as 
$U$ preserves sifted colimits.
\end{proof} 
We now shift attention to the context of simplicial commutative rings, where we fix a 
complete local Noetherian simplicial commutative ring $A$ with residue field $k$.
Since  the arguments will be precisely analogous to the ones given in the previous paragraphs, we will simply state the results. 
\begin{definition} 
Let $\SCR_{A//k}$ be the $\infty$-category of simplicial commutative $A$-algebras $B$ with a map to $k$
factoring the augmentation $A \rightarrow k$. 
Write $\SCR_{A//k}^{\cn}  $ for  the 
full subcategory of $\SCR_{A//k}$ spanned by all complete local Noetherian objects. 
\end{definition} 

\begin{cons}\label{pliemixed}
Consider the adjunction \INN{032@$\cot_{A, \Delta}$} \INN{190@$\sqz$}
$\cot_{A, \Delta}: \SCR_{A//k} \rightleftarrows \md_{k, \geq 0}: \sqz$, whose
the right adjoint is the square-zero functor given by $V \mapsto k \oplus V$. 
This induces a comonad  on $\md_{k, \geq 0}$ and, by dualisation, a monad  \INN{120@$\liepr{A}$}
$\liepr{A}: \coh_{k, \leq 0} \rightarrow \coh_{k, \leq 0}$ satisfying $\liepr{A}(V)=\cot_{A, \Delta} ( \sqz(V^{\vee}))^{\vee}$.
\end{cons}

\begin{theorem} 
\label{AkcomparestuffDelta}
Let $A$ be a complete local Noetherian simplicial commutative ring with residue
field $k$. Then: 
\begin{enumerate}
\item  
The adjunction 
$\cot_{A, \Delta}: \SCR_{A//k}^{\cn} \rightleftarrows \coh_{k, \geq 0}:\sqz$ is
comonadic. 
\item
The induced monad $\liepr{A}: \coh_{k, \leq 0} \rightarrow \coh_{k, \leq 0}$ extends
to a sifted-colimit-preserving monad on $\md_k$. 
\item 
The induced functor
$\D_{A, \Delta}: \SCR_{A//k}^{\cn} \rightarrow \alg_{\liepr{A}}^{op}$ carries pullbacks 
of diagrams  {$B \rightarrow B''$}, $B' \rightarrow B''$ inducing surjections on $\pi_0$ to pushouts
of 
$\liepr{A}$-algebras. 
\end{enumerate}
\end{theorem} 
We can  therefore  generalise \Cref{defliep} as follows:
\begin{definition}[Mixed partition Lie algebras] \label{splamixed} Given a complete local Noetherian simplicial commutative ring $A$ with residue field $k$, an \textit{$(A,k)$-partition Lie algebra} is an algebra over 
$\liepr{A}$.
\end{definition}

\subsection{Formal moduli problems in mixed characteristic}
Finally, we can  prove that formal moduli problems in mixed characteristic are governed by (possibly spectral) partition Lie algebras. 
\newcommand{\Ds}[1]{\mathfrak{D}_{#1, \Delta}}
For this, we fix a complete local Noetherian $\einf$-ring (respectively simplicial commutative ring) 
 $A$ with residue field $k$. We define a version of the Chevalley-Eilenberg cochains functor in this context:
\begin{cons}[The adjunction $(\D_A, C_A^*)$]
The colimit-preserving functor 
$$\D_A: \clg_{A//k} \rightarrow \mathrm{Alg}_{\liepss{A}}^{op}$$ defined by 
$\D_A(B) := \cot_A(B)^{\vee}$ admits a right adjoint 
$ C^*_A: \mathrm{Alg}_{\liepss{A}}^{op}\rightarrow \clg_{A//k}. $

Similarly, if $A$ is a simplicial commutative ring, the 
cocontinuous functor
$$ \Ds{A}: \SCR_{A//k} \rightarrow \alg_{\liepr{A}}^{op}  $$
defined by  $\Ds{A}(B) := \cot_{ A, \Delta}(B)^{\vee}$ admits a right adjoint
$C^*_{ A, \Delta}: \alg_{\liepr{A}}^{op} \rightarrow \SCR_{A//k}. $
\end{cons}

\begin{theorem} Under the above assumptions, the following statements hold: \INN{040@$\D$} \INN{032@$C^*$}
\begin{enumerate}
\item  
The adjunction $(\D_A, C_A^*)$ restricts to an equivalence
$\clg_{A//k}^{\cn} \simeq \mathrm{Alg}_{\liepss{A}}(\coh_{k, \leq 0})^{op}$. 
\item
If $A$ is additionally a simplicial commutative ring, the 
adjunction $(\Ds{A}, C^*_{\Delta, A})$ restricts to an equivalence
$\SCR_{A//k}^{\cn} \simeq \alg_{\liepr{A}}( \coh_{k, \leq 0})^{op}$. 
\end{enumerate}
\label{DAcompareequiv}
\end{theorem} 
\begin{proof} In both cases we will follow the argument in \Cref{KDequivalences}. 
We will only prove (1) in detail;  assertion (2) can be established by a parallel argument.

The claim essentially follows from the comonadicity established in
\Cref{Akcomparestuff}.  
To show that $\D_A$ is fully faithful on $\clg^{\cn}_{A//k}$, it suffices to check that for any $B \in \clg^{\cn}_{A//k}$, the unit
\begin{equation} \label{compmapDCA} B \rightarrow C_A^*(\D_A(B)) \end{equation} is an equivalence. We first observe that by construction, we have 
$C_A^*( \free_{\liepss{A}}(V))  = \sqz(V^{\vee})$ for all
$V \in \coh_{k, \leq 0}$. 
Thus, the map \eqref{compmapDCA} is an equivalence for $B = \sqz(V)$. Using the comonadicity claim in \Cref{Akcomparestuff}, we can 
write any $B$ as a totalisation of  a cosimplicial diagram of square-zero
extensions via the cobar resolution. It follows that  \eqref{compmapDCA} is an
equivalence  in general. 
This shows that $\D_A$ is fully faithful; since it is also essentially
surjective by \Cref{Akcomparestuff}(1), it follows that $\D_A$ is an equivalence. 
\end{proof} 

We can now prove  that one
obtains a deformation theory in the sense of Lurie (cf.\ \cite[Definition 12.3.3.2]{lurie2016spectral})
for simplicial commutative rings and $\einf$-rings with respect to a complete
base:
\begin{theorem} \label{maina} \ 
\begin{enumerate} 
\item 
Let $A$ be a complete local Noetherian $\einf$-ring with residue field $k$. 
The $\infty$-category $\clg_{A//k}$, 
the infinite loop object  $\{\sqz(k[n]) \in \mathrm{Stab}( \clg_{A//k})\}_{n \geq 0}$, 
the adjunction $(\D_A, C_A^*)$, and the full subcategory 
$\mathrm{Alg}_{\liepss{A}}( \coh_{k, \leq 0}) \subset \mathrm{Alg}_{\liepss{A}}$
form a deformation theory. 
\item 
Let $A$ be a complete local Noetherian simplicial commutative ring, The $\infty$-category $\SCR_{A//k}$, 
the infinite loop object $\{\sqz(k[n]) \in \mathrm{Stab}( \SCR_{A//k})\}_{n \geq
0}$, the adjunction $(\D_{A,\Delta}, C_{A, \Delta}^*)$, and 
the full  subcategory
$\mathrm{Alg}_{\liepp{A}}( \coh_{k, \leq 0}) \subset \mathrm{Alg}_{\liepp{A}}$
form a deformation theory. 
\end{enumerate}
\end{theorem}
\begin{proof} 
Combine 
\Cref{DAcompareequiv} and \Cref{Akcomparestuff} (or \Cref{AkcomparestuffDelta}). 
\end{proof} 

Using the argument of \cite[Proposition 12.1.2.9]{lurie2016spectral}, it follows that the Artinian objects of $\clg_{A//k}$ (respectively $\SCR_{A//k}$) from \Cref{Artinianrelative} are exactly the ones which are Artinian in the  axiomatic deformation theory setup of \cite[Definition 12.1.2.4]{lurie2016spectral}. In other words, they are those
which can be   built from a point by taking iterated fibres of maps to square-zero extensions $\sqz(
k[n])$ \mbox{with $n > 0$.} Arguing as in \cite[Proposition 6.1.4]{lurie2011derivedXII}, we see that a morphism between two Artinian objects is small in the axiomatic sense of  \cite[Definition 12.1.2.4]{lurie2016spectral}
if and only if it is surjective on $\pi_0$. 
This allows us to conclude that \Cref{fmpmixed} agrees with the axiomatic notion of a formal moduli problem
attached to the above deformation problem  (cf.\ \cite[Definition 12.1.3.1, Proposition 12.1.3.2(3)]{lurie2016spectral}).

\begin{cons}[The tangent complex]   \INN{200@$T_X$}
Given a formal moduli problem $X$, we can construct its tangent complex $T_X \in
\md_k$ (cf.\ \cite[Definition 12.2.2.1]{lurie2016spectral}); its underlying spectrum satisfies
$\Omega^{\infty -n}T_X = X( \sqz(
k[n]))$ for all $n \geq 0$. 
\end{cons}
\hspace{-1pt}Combining \Cref{maina} with Lurie's axiomatic  \cite[Theorem 12.3.3.5]{lurie2016spectral}, we can finally deduce:
\begin{theorem} \label{mostgeneral}\ 

\begin{enumerate}
\item Let $A$ be a complete local Noetherian $\einf$-ring. There is an equivalence of $\infty$-categories
$\moduli_{A//k, \einf} \simeq \alg_{\liepss{A}}$. 
\item
Let $A$ be a complete local Noetherian simplicial commutative ring. There is an equivalence of $\infty$-categories
$\moduli_{A//k, \Delta} \simeq \alg_{\liepr{A}}$. 
\end{enumerate}
On underlying objects in $\md_k$, both equivalences send  a formal moduli problem
$X \in \moduli_{A//k}$ to its tangent complex $T_X$. 
\end{theorem}

\newpage
\section{The homology of  partition Lie algebras}
Away from characteristic zero, partition Lie algebras 
display additional subtleties: 
\begin{example}
For $A\in \SCR_{\FF_p}^{\aug}$  complete local Noetherian, the Frobenius $(x\mapsto x^p)$ on $A$ induces an endomorphism $\phi$ on the partition Lie algebra $ \cot(A)^\vee$.
While $\phi$ is zero as a map of $\FF_p$-modules (as $px^{p-1}=0$), \Cref{DAcompareequiv}  shows that $\phi$  is generally nonzero as a map of partition Lie algebras.
\end{example}
To get a better handle on our Lie algebras, we may wish to consider  Dyer-Lashof-like operations  on their homotopy groups. These are parametrised by the homotopy groups of free Lie algebras:
\begin{cons}
Given a class $\alpha \in \pi_j \left( \Lie_{\FF_p,\Delta}^\pi(\Sigma^{\ell_1} \FF_p\oplus \ldots \oplus \Sigma^{\ell_n} \FF_p)\right)$, we define a universal $n$-ary operation acting on the homotopy groups of any partition Lie algebra $\mathfrak{g}$. 
For this, we send a tuple $(x_1 \in \pi_{\ell_1}(\mathfrak{g}), \ldots , x_n \in \pi_{\ell_n}(\mathfrak{g}))$ to the element $\alpha(x_1,\ldots,x_n) \in \pi_j(\mathfrak{g})$ represented by
$$
\Sigma^{j} \FF_p \xrightarrow{\alpha}\Lie_{\FF_p,\Delta}^\pi(\Sigma^{\ell_1} \FF_p\oplus \ldots \oplus \Sigma^{\ell_n} \FF_p) \xrightarrow{\Lie_{\FF_p,\Delta}^\pi(x_1,\ldots,x_n)} \Lie_{\FF_p,\Delta}^\pi(\mathfrak{g}) \rightarrow \mathfrak{g}.
$$ 
There is a similar construction for spectral partition Lie algebras.
\end{cons}
We will now compute the homotopy groups of free (possibly spectral) partition Lie algebras  \mbox{over $\FF_p$.} Write $B(n_1,\ldots,n_m)$  \INN{020@$B(n_1,\ldots,n_k)$} for the set of Lyndon words in $m$ letters involving the $i^{th}$ letter $n_i$ times (cf.\ \Cref{Lyndon word}).
Given integers $\ell_1,\dots , \ell_m$, we   have the following \vspace{-2pt}results:
\begin{theorem}\label{finalthecohomology} 
The $\FF_p$-vector space $\pi_\ast(\Lie_{\FF_p,\Delta}^\pi(\Sigma^{\ell_1} \FF_p  \oplus \ldots \oplus \Sigma^{\ell_m} \FF_p  ))$ 
 has a basis indexed by sequences $(i_1, \ldots, i_k, e,w)$. Here $w\in B(n_1,\ldots,n_m)$ is a  Lyndon word. We have $e \in \{0,\epsilon\}$, where
 $\epsilon = 1$ if $p$ is odd and $\deg(w):= \sum_i (\ell_i-1)n_i+1$ is even. Otherwise,   $\epsilon = 0$.
  The integers $i_1, \ldots, i_k$ satisfy: 
\begin{enumerate}
\item each $|i_j|$ is congruent to $0$ or $1$ modulo $2(p-1)$; \label{congruence}
\item for all $1\le j<k$, we have   
$p i_{j+1} <i_j < -1$    or $  0 \leq i_j < pi_{j+1}$;
\item we have  $ (p-1)(1+e)\deg(w)-\epsilon \leq i_k < -1$ or $0 \leq i_k \leq (p-1) (1+e)\deg(w) -\epsilon$.
\end{enumerate}
The  sequence $(i_1, \ldots, i_k, e,w)$ sits  in homological degree  $\left((1+e)\deg(w) -e\right) + i_1+\cdots +i_k-k$ and  \mbox{multi-weight $(n_1 p^k (1+e) ,\ldots,n_m  p^k (1+e))$.} \vspace{-2pt}
\end{theorem}
  
 \begin{theorem}\label{finalthecohomologyspectral} 
The $\FF_p$-vector space $\pi_\ast(\Lie_{\FF_p,\EE_\infty}^\pi(\Sigma^{\ell_1} \FF_p  \oplus \ldots \oplus \Sigma^{\ell_m} \FF_p  ))$ 
 has a basis indexed \mbox{by sequences} $(i_1, \ldots, i_k, e,w)$. Here $w\in B(n_1,\ldots,n_m)$ is a  Lyndon word. We have $e \in \{0,\epsilon\}$, where
 $\epsilon = 1$ if $p$ is odd and $\deg(w):= \sum_i (\ell_i-1)n_i+1$ is even. Otherwise,   $\epsilon = 0$.
  The integers $i_1, \ldots, i_k$ satisfy: 
\begin{enumerate}
\item\hspace{-5pt}' \ each $i_j$ is congruent to $0$ or $1$ modulo $2(p-1)$; \label{congruence}
\item\hspace{-5pt}' \ for all $1\le j<k$, we have   $i_j<p i_{j+1}$;
\item\hspace{-5pt}' \  we have  $i_k \leq (p-1) (1+e)\deg(w) -\epsilon $.
\end{enumerate}
The homological degree of $(i_1, \ldots, i_k,e,w)$ is $\left((1+e)\deg(w) -e \right)+i_1+\cdots+i_k-k$ and its   multi-weight is \vspace{-2pt}$(n_1 p^k (1+e) ,\ldots,n_m  p^k (1+e))$. \end{theorem}

Our strategy will closely follow the proof of \cite[Theorem 8.14]{arone2018action}, which essentially computes the homotopy groups of free coconnective partition Lie algebras (for $p=2$, \cite[Theorem 8.14]{arone2018action} also follows from \cite{goerss1990andre}). Our computation relies on many  classical ingredients and insights, which we will reference in detail below. Broadly speaking, we will proceed in three steps:
\begin{enumerate}
\item First, we compute the homotopy groups of a free Lie algebra on an odd class. We use  a  bar spectral sequence and 
the known homotopy groups of symmetric or extended powers.
\item In a second step, we express the homotopy groups of a free Lie algebra  on an  even class in terms of the odd case $(1)$. We rely on the Takayasu cofibration sequence and\mbox{ its strict cousin.}
\item Finally, we give a Hilton-Milnor decomposition for free Lie algebras on many generators, thereby reducing the computation of their homotopy groups to the cases $(1)$ and $(2)$. We rely on a certain splitting of the restriction of partition complexes to Young subgroups.
\end{enumerate}

\begin{remark}
Our additive computation of the operations on spectral partition Lie algebras was later refined by Zhang \cite{zhang2022operations}, who determined their composition structure (using a general method from \cite{brantnerthesis}).
For partition Lie algebras, however, this computation has not yet been carried out.
\end{remark}

\subsection{Free partition Lie algebras on an odd generator}
\label{degeneration} 
The principal aim of this subsection is to compute the homotopy groups of free Lie algebras on a single odd class. We  will establish the following results:
\begin{theorem}\label{odda}
Let $\ell$ be an integer, assumed to be odd if $p$ is. 
Then $\Lie_{\FF_p,\Delta}^\pi(\Sigma^{\ell} \FF_p)$ has a basis given by  all sequences $(i_1, i_2,\ldots i_k)$ satisfying the following conditions:
\begin{enumerate}
\item each $|i_j|$ is congruent to $0$ or $1$ modulo $2(p-1)$;
\item for all $1 \leq j<k$ we have $pi_{j+1}<  i_j  < -1$ or $0 \leq  i_j < p i_{j+1}$;
\item we have $ (p-1) \ell \leq  i_k < -1$ or $0 \leq i_k \leq (p-1) \ell $.
\end{enumerate}
The sequence $(i_1,i_2\ldots i_k)$ lies in homological degree $\ell+i_1+i_2+\ldots + i_k-k$ and weight $p^k$.
\end{theorem}

\begin{theorem}\label{oddb}
Let $\ell$ be an integer, assumed to be odd if $p$ is. 
Then $\Lie_{\FF_p,\EE_\infty}^\pi(\Sigma^{\ell} \FF_p)$ has a basis given by  all sequences $(i_1, i_2,\ldots i_k)$ satisfying the following conditions:
\begin{enumerate}
\item\hspace{-5pt}' each $i_j$ is congruent to $0$ or $1$ modulo $2(p-1)$;
\item\hspace{-5pt}' for all $1 \leq j<k$ we have $ i_j<p i_{j+1} $;
\item\hspace{-5pt}' we have $  i_k\leq (p-1) \ell$.
\end{enumerate}
The sequence $(i_1,i_2\ldots i_k)$ lies in homological degree $\ell+i_1+i_2+\ldots + i_k-k$ and weight $p^k$.
\end{theorem}
We  carry out these two parallel computations ``weight-by-weight'' by generalising the   argument provided in  \cite[Section 9]{arone2018action}, which is  inspired by \cite[Section 3]{arone1999goodwillie}.
We outline the main steps:
\begin{enumerate}[a)]
\item By duality, it suffices to compute the homotopy groups $\pi_\ast(F_{\Sigma |\Pi_n|^\diamond}(\Sigma^{\ell} \FF_p))$ and  $\pi_\ast(F^h_{\Sigma |\Pi_n|^\diamond}(\Sigma^{\ell} \FF_p))$. The functors \INN{060@$F_X,  F_X^h$}
$  F_{\Sigma |\Pi_n|^\diamond}$ and $F_{\Sigma |\Pi_n|^\diamond}^h$ were constructed in \Cref{bredon}.
\item In a second step, we show that whenever $\ell$ is odd or $p=2$, the Bredon spectral sequences for $\pi_\ast(F_{\Sigma |\Pi_n|^\diamond}(\Sigma^{\ell} \FF_p))$ and  $\pi_\ast(F^h_{\Sigma |\Pi_n|^\diamond}(\Sigma^{\ell} \FF_p))$ degenerate. This follows by applying a result of Arone, Dwyer, and Lesh (cf.\ \cite[Theorem 1.1]{arone2016bredon}). 
\item We then use the known homotopy of symmetric and extended powers to describe 
$\pi_\ast(F_{\Sigma_n/H_+}(\Sigma^{\ell} \FF_p))$ and $\pi_\ast(F_{\Sigma_n/H_+}(\Sigma^{\ell} \FF_p))$ for all subgroups $H\subset \Sigma_n$ arising as stabilisers of points in  $|\Pi_n|^\diamond$.
\item This allows us to compute the above $E^2$-page by applying a combinatorial matching argument.
\end{enumerate}
We will now provide the details of our computation.
\subsubsection*{Duality}\label{du}
Recall that given a genuine pointed $\Sigma_n$-space $X$, we have defined  functors $$  F_X, F_X^h:\mod_{\FF_p}\rightarrow \mod_{\FF_p} $$ which extend the assignments $M \mapsto (\FF_p[X] \otimes M^{\otimes n})_{\Sigma_n}$ and $M \mapsto (\FF_p[X] \otimes M^{\otimes n})_{h\Sigma_n}$ from discrete $\FF_p$-modules to all $\FF_p$-modules in a sifted-colimit-preserving way (cf.\ \Cref{bredon}). 

Combining \Cref{concretespectralplie}, \Cref{freeplie}, and \Cref{duality}, we see that in order to prove \Cref{odda} and \Cref{oddb}, it suffices to compute $\pi_\ast(F_{  \Sigma|\Pi_n|^\diamond}(\Sigma^{\ell} \FF_p))$ and $\pi_\ast(F^h_{  \Sigma|\Pi_n|^\diamond}(\Sigma^{\ell} \FF_p))$ for all $n$, where 
$\ell$ is odd or $p=2$.

\subsubsection*{Degeneration of the Bredon Spectral Sequence}
As explained in \Cref{bredon}, the skeletal filtration on the pointed simplicial $\Sigma_n$-set $  |\Pi_n|^\diamond$ gives rise to spectral sequences converging to $\pi_\ast(F_{ |\Pi_n|^\diamond}(M))$ and $\pi_\ast(F^h_{ |\Pi_n|^\diamond}(M))$. Their $E^2$-pages are given by the  reduced Bredon homology groups of $ |\Pi_n|^\diamond$ with respect to the \INN{130@$\mu_\ast^M, \mu_\ast^{M,h}$}
  graded Mackey functors $$\mu_\ast^M, \mu_\ast^{M,h} :\mathcal{S}_\ast^{\Sigma_n} \rightarrow \vect_{\FF_p \ast}$$ sending  a   $
 \Sigma_n$-set $X$ to the graded $\FF_p$-vector spaces $\pi_\ast(F_X(M))$ and $\pi_\ast(F_X^h(M))$, respectively.

To establish degeneration of the Bredon spectral sequence, we will apply Arone-Dwyer-Lesh's \cite[Theorem 1.1]{arone2016bredon}. We begin by checking that the conditions of this theorem are satisfied:
\begin{proposition}
 For $M=\Sigma^{\ell} \FF_p$ with $\ell$   odd or $p=2$, the   functors $\mu_\ast^M, \mu_\ast^{M,h}$  satisfy:
\begin{enumerate}
\item For any Sylow $p$-subgroup $P\subset \Sigma_n$, projection induces split epimorphisms $$\mu_\ast^M(\Sigma_n/P \times -) \rightarrow \mu_\ast^M(-) \ \ \ \ \mbox{and}    \ \ \ \ \mu_\ast^{M,h}(\Sigma_n/P \times -) \rightarrow \mu_\ast^{M,h}(-)\ .$$
\item If $D\subset \Sigma_n$ is an  elementary abelian $p$-subgroup acting freely and non-transitively, then $\ker\left(C_{\Sigma_n}(D) \rightarrow \pi_0 C_{\GL_n(\RR)}(D)\right)$ acts trivially on $\mu^M_\ast(\Sigma_n/D)$ and $\mu^{M,h}_\ast(\Sigma_n/D)$, \mbox{respectively.}
\item If $p$ is odd, the odd involution  in $C_{\Sigma_n}(D)$ acts as $(-1)$ on $\mu^M_\ast(\Sigma_n/D)$ and $\mu^{M,h}_\ast(\Sigma_n/D)$.
\end{enumerate}
Here $C_G(D)$ denotes the centraliser of $D$ in $G$,  for $D$  a subgroup of $G$.
\end{proposition}
\begin{proof} We follow  \cite[Proposition 9.3]{arone2018action}\cite[Example 11.5]{arone2016bredon}, which imply the result \vspace{5pt} for $\ell\geq 0$.

If $X$ is a   finite pointed $\Sigma_n$-set, then  the  transformations of functors $\vect_{\FF_p}^\omega \rightarrow \mod_{\FF_p}$ \vspace{-3pt}   \mbox{given by}
$$
( \FF_p[X] \otimes (-)^{\otimes n})_{\Sigma_n} 
 \xrightarrow{ (\tr\otimes \id^{\otimes n})_{\Sigma_n} \ }   (\FF_p[\Sigma_n/P \times X] \otimes (-)^{\otimes n})_{\Sigma_n} \ 
 \xrightarrow{\ \ \ \ \ \ } \vspace{-3pt}   (\FF_p[X] \otimes (-)^{\otimes n})_{\Sigma_n}
 $$
 $$
( \FF_p[X] \otimes (-)^{\otimes n})_{h\Sigma_n} 
 \xrightarrow{ (\tr\otimes \id^{\otimes n})_{h\Sigma_n}}   (\FF_p[\Sigma_n/P \times X] \otimes (-)^{\otimes n})_{h\Sigma_n} 
  \xrightarrow{\ \ \ \ \ }  \vspace{3pt}   (\FF_p[X] \otimes (-)^{\otimes n})_{h\Sigma_n}
 $$
induce  multiplication by $|\Sigma_n/P|$ on homotopy groups. They are therefore \vspace{2pt} equivalences.

Taking right-left extensions  of these degree $n$ functors (cf.\ \Cref{polyrightext}), we obtain  transformations 
$F_{X} \rightarrow F_{\Sigma_n/P \times X}$ and $F^h_{X} \rightarrow F^h_{\Sigma_n/P \times X}$ such that the two composites
$F_{X} \rightarrow F_{\Sigma_n/P \times X} \rightarrow F_X $ and $F_{X}^h \rightarrow F_{\Sigma_n/P \times X}^h \rightarrow F_X^h $
are equivalences, which \vspace{5pt} clearly implies $(1)$.

For $(2)$, we begin with the  diagram drawn on the left. Its rightmost arrow takes $D$-orbits or $D$-homotopy orbits, respectively.
Freely adding  sifted colimits, we obtain the diagram on the \vspace{4pt} right.
\[ \ \hspace{-10pt}
\begin{tikzcd}
	\Set^{\Fin}_\ast
	\arrow[r, "{A \mapsto A^{\wedge n}}"]
	\arrow[d, "{A \mapsto \FF_p[A]}"']
	& (\Set^{\Fin}_\ast)^D
	\arrow[d, "{A \mapsto \FF_p[A]}"]
	\\
	\vect_{\FF_p}^{\omega}
	\arrow[r, "{V \mapsto V^{\otimes n}}"']
	& (\vect_{\FF_p}^{\omega})^D
	\arrow[r]
	& \mod_{\FF_p}
\end{tikzcd}
\qquad
\begin{tikzcd}
	\mathcal{S}_\ast
	\arrow[r, "{A \mapsto A^{\wedge n}}"]
	\arrow[d, "{A \mapsto \widetilde{C}_\ast(A)}"']
	& P_\Sigma\big((\Set^{\Fin}_\ast)^D\big)
	\arrow[d, "{A \mapsto \widetilde{C}_\ast(A)}"]
	\\
	\mod_{k,\geq 0}
	\arrow[r, "{V \mapsto V^{\otimes n}}"']
	& P_\Sigma\big((\vect_{\FF_p}^{\omega})^D\big)
	\arrow[r]
	& \mod_{\FF_p}
\end{tikzcd}\]\vspace{2pt} 

The lower composite is equivalent to  $F_{\Sigma_n/D}(-)$ or $F_{\Sigma_n/D}^h(-)$, respectively. Hence, the assignment $A \mapsto F_{ \Sigma_n/D}( \widetilde{C}_\ast(A,\FF_p))$ factors through the functor 
sending  a space $A$ to the genuine \mbox{$D$-space $A^{\wedge n}$.}
Replacing   $A\mapsto \FF_p[A]$ by the functor $A \mapsto \FF_p[A]^\vee$ and using \Cref{RKEisright}, a similar argument shows that   $A \mapsto F_{ \Sigma_n/D}( \widetilde{C}^\ast(A,\FF_p))$ factors through the functor  sending  $A$ to the  genuine \mbox{$D$-space $A^{\wedge n}$.}

We can write each $M=\Sigma^{\ell} \FF_p$ as the singular chains or the singular cochains of a sphere $X=S^\ell$ (depending on whether $\ell$ is positive or negative). The above observations therefore show that in order to prove $(2)$, it suffices to check that any  $\sigma \in \ker(C_{\Sigma_n}(D) \rightarrow \pi_0 C_{\GL_n(R)}(D))$  acts on the genuine $\Sigma_n$-space $(S^\ell)^{\wedge n}$ by a map that is $D$-equivariantly homotopic to the identity. This is clear \vspace{5pt} because  any such $\sigma$ lies in the connected component of the identity in $C_{\GL_n(R)}(D)$.

For $(3)$, we first recall that if $X$ is a spectrum with $2$ invertible in $\pi_\ast(X)$, then  \mbox{$\tau:X\rightarrow X$} acts as $(-1)$ on $\pi_\ast(X)$ if and only if $\tau-1$ is an equivalence (cf.\ \cite[Proposition 11.4]{arone2016bredon} and its proof).
Observe that if $p$ is an odd prime and $\tau \in C_{\Sigma_n}(D)$  is an odd permutation of order $2$ \mbox{centralising $D$,} then $\tau$ acts by $(-1)$ on $ \widetilde{H}_\ast\left( (S^{n\ell})^{A},k\right)$ for any subgroup $A\subset D$. This implies that $\tau-1$ induces a quasi-isomorphism on the $\FF_p$-modules  $\widetilde{C}_\ast((S^{ n \ell})^{A},\FF_p)$ and $\widetilde{C}^\ast((S^{ n \ell})^{A},\FF_p)$.  Elmendorf's theorem expresses the $D$-space $S^{n\ell}$ as a homotopy colimit of $D$-spaces of the form $(D/B)_+ \wedge ((S^{n\ell})^{A}$. By the functoriality established in the proof of $(2)$, we can therefore  express $F_{\Sigma_n/D}(M)$ and $F_{\Sigma_n/D}^h(M)$ as   homotopy colimits of $\FF_p$-modules $\widetilde{C}_\ast((S^{n\ell})^{A},\FF_p)$  if $\ell>0$ or of $\FF_p$-modules 
$\widetilde{C}^\ast((S^{n\ell})^{A}, \FF_p)$ if $\ell<0$. Hence $\tau-1$ acts as an equivalence on $F_{\Sigma/D}(M)$ and $F^h_{\Sigma/D}(M)$, which implies the thrid claim.
\end{proof}
Hence, we can apply \cite[Theorem 1.1., Corollary 1.2]{arone2016bredon} to conclude:
\begin{corollary}\label{BSSdegenerates} For $M=\Sigma^{\ell} \FF_p$ with $\ell$ even or $p=2$, the Bredon homology groups 
$$E^2_{s,t}  = \hml^{\Br}_s(   |\Pi_n|^\diamond, \mu_t^M) \ \ \ \ \ \ \ \  E^{2,h}_{s,t}  = \hml^{\Br}_s(  \Sigma|\Pi_n|^\diamond, \mu_t^{M,h})$$ vanish
unless $n=p^k$ for some $k$ and $s=k-1$.  In particular, the spectral sequence 
degenerates and  
$$\pi_{\ast }(F_{\Sigma |\Pi_n|^\diamond}(M)) 
=\pi_{\ast -1}(F_{  |\Pi_n|^\diamond}(M)) = \begin{cases} \ \  \widetilde{\hml}^{\Br}_{k-1}(|\Pi_{p^k}|^\diamond; \mu_{\ast-k}^M)  & \mbox{ if } n = p^k \\
\ \ 0 & \mbox{ else}\end{cases}$$
$$\pi_{\ast }(F^h_{\Sigma |\Pi_n|^\diamond}(M)) 
=\pi_{\ast -1}(F^h_{  |\Pi_n|^\diamond}(M)) = \begin{cases} \ \   \widetilde{\hml}^{\Br}_{k-1}(|\Pi_{p^k}|^\diamond; \mu_{\ast-k}^{M,h}) & \mbox{ if } n = p^k \\
\ \  0 & \mbox{ else}\end{cases}.$$
\end{corollary}
Hence, it suffices to compute these Bredon homology \vspace{-2pt}groups to establish Theorems \ref{odda} and \ref{oddb}.

\subsubsection*{The Bredon Homology of Stabilisers}
Next, we compute $\pi_\ast(F_{\Sigma_n/H}(M))$ and $\pi_\ast(F^h_{\Sigma_n/H}(M))$ for $H$ the stabiliser of a point in the partition complex $|\Pi_n|$  and $M\in \md_{\FF_p}$  any $\FF_p$-module. This extends  \cite[Section 9.2]{arone2018action}, which is inspired by \cite[Section 3]{arone1999goodwillie}, 
to the coconnective setting. 

We  need several auxiliary additive  \vspace{-2pt}functors:\INN{060@$\mathcal{F}_k,  \mathcal{F}_k^h$}
\begin{definition}\label{fk}
Given  $k\geq 0$,  the functor $\mathcal{F}_k$  sends a graded $\FF_p$-vector space $V$ to the graded $\FF_p$-vector space $\mathcal{F}_k(V)$  generated by symbols   $(i_1, . . . , i_k; v)$, where $v$ is a homogeneous element of $V$ and $i_1, \ldots,  i_k  $ are  integers satisfying the following conditions:
\begin{enumerate}
\item  each $|i_j|$ is congruent to $0$ or $1$ modulo  $2(p-1)$;
\item $i_j \geq p i_{j+1}>p$ \textit{or}  $i_j \leq p i_{j+1} \leq 0$    for all $1\leq j<k$;
\item  If $p$ is odd, \ then $1<i_1 <(p-1)(|v|+i_2+\ldots +i_k)$ \textit{or} 
$ 0 \geq i_1 > (p-1)(|v|+i_2+\ldots +i_k)$;\\ 
If $p$ is even,   then $1<i_1 \leq(p-1)(|v|+i_2+\ldots +i_k)$ \textit{or} $ 0 
\geq i_1  \geq (p-1)(|v|+i_2+\ldots +i_k)$. 
\end{enumerate}
We divide out by the relation $(i_1,...,i_k;u)+(i_1,...,i_k;v)=(i_1,...,i_k;u+v)$. There is  a homological grading on $\mathcal{F}_k(V)$, which puts $(i_1, \ldots, i_k; v)$ in degree $|v| + i_1 + \ldots + i_k$ whenever $v$ is homogeneous of degree $|v|$. Moreover, there is a weight grading putting $(i_1, \ldots, i_k; v)$ in weight $p^k$. 
\end{definition}

\begin{remark}
Observe that either all $i_j$ are strictly larger  than $1$  or all $i_j$ are nonpositive.
\end{remark}

\begin{definition}\label{fhk}
For $k\geq 0$,  the functor $\mathcal{F}_k^h$  sends a graded $\FF_p$-vector space $V$ to the graded $\FF_p$-vector space $\mathcal{F}_k^h(V)$  generated by symbols   $(i_1, . . . , i_k; v)$, where $v$ is a homogeneous element of $V$ and $i_1, \ldots,  i_k  $ are  integers satisfying the following conditions:
\begin{enumerate}
\item\hspace{-5pt}'  each $i_j$ is congruent to $0$ or $-1$ modulo $2(p-1)$;
\item\hspace{-5pt}'   $i_j \leq p i_{j+1} $    for all $1\leq j<k$;
\item\hspace{-5pt}'  if $p$ is odd, \ then $ i_1 >(p-1)(|v|+i_2+\ldots +i_k)$;\\ 
If $p$ is even,   then $ i_1 \geq(p-1)(|v|+i_2+\ldots +i_k)$. 
\end{enumerate}
We again divide out by the relation $(i_1,...,i_k;u)+(i_1,...,i_k;v)=(i_1,...,i_k;u+v)$. The homological grading and the weight grading are  as in \Cref{fk}. \INN{190@$S,  S_n$}
\end{definition}
\begin{definition}Given a homologically graded $\FF_p$-vector space, let $S(V) = \bigoplus_{n\geq 0} S_n(V) $ be the free graded-commutative algebra on $V$ if $p$ is odd and   the free exterior algebra on $V$ if $p=2$. 

Observe that if $V$ is equipped with an additional weight grading, then $S(V)$ is   naturally bigraded.\end{definition}

We will now use the functors  $\mathcal{F}_k$,  $\mathcal{F}_k^h$, and $S$ to give a simple formula for homotopy groups of the symmetric and exterior powers of a given $M\in \md_k$, thereby summarising 
computations of Dold \cite{dold1958homology},  Nakaoka \cite{nakaoka1957cohomology} \cite{nakaoka1957cohomology2}, Milgram \cite{milgram1969homology}, and  Priddy \cite{priddy1973mod} in the ``strict'' case, as well as of Adem \cite{MR0050278}, Serre \cite{MR0060234}, Araki--Kudo \cite{kudo1956topology}, Cartan \cite{MR0065161}  \cite{MR0068219}, Dyer-Lashof \cite{MR0141112}, May, and Steinberger \cite{MR836132} in the ``homotopy orbits" case:

\begin{proposition}\label{priddys} For any   $M\in \md_k$,  there are (unnatural) isomorphisms
$$ \pi_\ast\left( \bigoplus_n F_{\Sigma_n/\Sigma_n}(M)\right) \cong  S\left(\bigoplus_k \mathcal{F}_k(\pi_\ast(M))\right) \ \ \ \ \ \ \ \ \  \pi_\ast\left( \bigoplus_n F^h_{\Sigma_n/\Sigma_n}(M)\right) \cong  S\left(\bigoplus_k \mathcal{F}^h_k(\pi_\ast(M))\right)$$
which respect   the homological   and the weight grading. Here $\mathcal{F}_k(\pi_\ast(M)), \mathcal{F}_k^h(\pi_\ast(M))$ sit in \mbox{weight $p^k$.} 
\end{proposition} 
\begin{remark}Note that the functor  $ F_{\Sigma_n/\Sigma_n}(M)$ computes the (suitably derived) $n^{th}$ symmetric power of $M$, whereas $F^h_{\Sigma_n/\Sigma_n}(M)$ computes its $n^{th}$ extended power.
\end{remark}
\begin{warning}
These are isomorphisms of bigraded vector spaces; they \textit{do  not} respect the multiplicative structure. In fact, they are not even functorial in $M$, as we should really use divided power functors on the left. However, this will not cause any problems for us, since we will only need  a dimension count of the weighted pieces. We therefore adopt this simpler approach for notational convenience.   
\end{warning}
\begin{proof}[Proof of Proposition  \ref{priddys}]
After picking a basis, we may identify $M$ with a direct sum of shifts of $\FF_p$. Since all functors commute with filtered colimits and send finite direct sums to tensor products, it   suffices to check the claim for $\FF_p$-module spectra of the form $M=\Sigma^{\ell} \FF_p$, where $\ell$ is any integer.

For $\ell> 0$,  the vector space $\mathcal{F}_k(\Sigma^{\ell} \FF_p)$ is the $k^{th}$ summand of the free simplicial commutative \mbox{$\FF_p$-algebra} on one generator in degree $\ell$. The work of Nakaoka (cf.\ \cite{nakaoka1957cohomology} \cite{nakaoka1957cohomology2})  therefore shows that it has a basis given by all sequences $(i_1,\ldots,i_k)$ satisfying the following conditions: 
\begin{enumerate}
\item  each $i_j$ is congruent to $0$ or $1$ modulo $2(p-1)$;
\item $i_j \geq p i_{j+1}>p$;
\item  If $p$ is odd, \ then $1<i_1 <(p-1)(v+i_2+\ldots +i_k)$;\\
If $p$ is even,   then $1<i_1 \leq(p-1)(v+i_2+\ldots +i_k)$. 
\end{enumerate} 

For $\ell \leq   0$, the $\FF_p$-vector space $\mathcal{F}_k(\Sigma^{\ell} \FF_p)$ agrees with the $k^{th}$ summand in the free cosimplicial $\FF_p$-vector space on a generator in degree $\ell$. The work of Priddy (cf.\  \cite[Theorem 4.1]{priddy1973mod}) 
therefore shows that $\mathcal{F}_k(\Sigma^{\ell} \FF_p)$ has a basis given by  all sequences $(i_1,\ldots,i_k)$ satisfying the following conditions: 
\begin{enumerate}
\item  each $i_j$ is congruent to $0$ or $-1$ modulo $2(p-1)$;
\item   $i_j \leq p i_{j+1} \leq 0$    for all $1\leq j<k$;
\item  if $p$ is odd, \ then  
$ 0 \geq i_1 > (p-1)(v+i_2+\ldots +i_k)$;\\ 
If $p$ is even,   then $ 0 
\geq i_1  \geq (p-1)(v+i_2+\ldots +i_k)$. 
\end{enumerate}

Corresponding statements for  $G^h$ are described on p.298 of \cite{MR836132} and \mbox{p.16 of \cite{cohen1978homology}.}
\end{proof}
\begin{remark}
Note that for $p=2$, the cited sources state the result in a slightly different,
yet equivalent, form, which uses a strict inequality for the excess and the free symmetric algebra functor.
\end{remark}
Write $P(n)$ \INN{160@$P(n)$}for the set of sequences $(a_0,a_1, \dots)$  of natural numbers satisfying  $n= \sum_{k\geq 0} a_k p^k$. Restricting  attention to a specific weight and expanding \Cref{priddys} binomially, we deduce:
\begin{corollary}\label{sf}
For each $n\geq 0$ and any $M\in \md_k$, there are isomorphisms
$$ \pi_\ast\left( F_{\Sigma_n/\Sigma_n}(M)\right) \cong  \bigoplus_{(a_0,a_1,\dots) \in P(n)} \left(\bigotimes_{k\geq 0} S_{a_k} \left(  \mathcal{F}_k(\pi_\ast(M))\right)\right)$$
$$ \pi_\ast\left( F^h_{\Sigma_n/\Sigma_n}(M)\right) \cong  \bigoplus_{(a_0,a_1,\dots) \in P(n)}\left( \bigotimes_{k\geq 0} S_{a_k} \left(  \mathcal{F}^h_k(\pi_\ast(M))\right)\right).$$
\end{corollary}

We   compute $\pi_\ast\left( F_{\Sigma_n/K_\sigma}(M)\right)$ and $\pi_\ast\left( F_{\Sigma_n/K_\sigma}^h(M)\right)$ for $K_\sigma$ the stabiliser of any simplex $$\sigma = [\ \hat{0}<x_1<\ldots<x_i<\hat{1}\ ]$$ in the doubly suspended partition complex  $\Sigma |\Pi_n|^\diamond$, where the integer $n \geq 1$ is fixed throughout. We will   use \cite[Definition 9.12]{arone2018action}, which is a variant of \cite[Definition 1.10]{arone1999goodwillie}:

\begin{definition}
A \textit{$p$-enhancement} of a chain of partitions  $\sigma = [ \ \hat{0}<x_1<\ldots<x_i<\hat{1}\ ]$  consists of a refining chain 
$$ \Theta= [\ \hat{0} \leq e_1  \leq x_1 \leq \ldots \leq e_i \leq x_i \leq e_{i+1} \leq \hat{1}\ ] $$
such that the following two conditions hold true: 
\begin{enumerate}
\item The number of $x_a$-classes contained in a given $e_{a+1}$-class is a power of $p$.
\item Given   $x_a$-classes $S_1$ and $S_2$, we can define chains of partitions of $S_1$ and $S_2$ by \mbox{restricting $\Theta$.} If $S_1,S_2$ lie in the same $e_{a+1}$-class, then these restricted chains are isomorphic, by which we mean that they lie in the same $\Sigma_n$-orbit.
\end{enumerate}
\end{definition}

Two $p$-enhancements   are said to be \textit{isomorphic} if they lie in the same $\Sigma_n$-orbit.
We can then define endofunctors from  enhancements as follows (cf.\ \cite[Definition 9.13]{arone2018action}):
\begin{definition}
Assume  we are given a chain $\sigma = [\ \hat{0}<x_1<\ldots<x_i<\hat{1}\ ]$ and an isomorphism class of \mbox{$p$-enhancements} of $\sigma$ represented by 
  $\Theta =    [\ \hat{0} \leq e_1  \leq x_1 \leq \ldots \leq e_i \leq x_i \leq e_{i+1} \leq \hat{1}\ ] $. 
\INN{200@$[\Theta],  [\Theta]^h$}
We define  endofunctors $[\Theta]$ and $[\Theta]^h$  
on graded $\FF_p$-vector spaces by the following rules:
\begin{itemize}
\item If $i=0$ and $[\Theta] = [\ \hat{0} \leq e_1 \leq \hat{1}\ ]$ with $e_1$ having $a_j$ classes of size $p^j$ for all  $j\geq 0$, we  set $$ [\Theta](C) := \bigotimes_j S_{a_j}(\mathcal{F}_j(V)) \  \ \ \  \  \ \ \ \ \ \ \ \ \ \ \ \ \ \ \ [\Theta]^h(C) := \bigotimes_j S_{a_j}(\mathcal{F}_j^h(V)). \ \ \ \ \  \  \ \ \ \  $$
\item If $i>0$,  assume that restricting the chain $\Theta$ to the classes of $e_{i+1}$ gives $a_1$ chains of isomorphism type $1$, $a_2$ chains of isomorphism type $2$, etc\ldots .
Suppose that each $e_{i+1}$-class of type $t$ contains $p^{b_t}$ many $x_i$-classes, and write $\Theta_t$ for the restriction of $\Theta$ to any $x_i$-class contained in an $e_{i+1}$-class of type $t$. We then define $$[\Theta](V) := \bigotimes_t S_{a_t}( \mathcal{F}_{b_t}([\Theta_t](V)))
\ \ \ \ \ \  \ \ \ \  \ \ [\Theta]^h(V) := \bigotimes_t S_{a_t}( \mathcal{F}^h_{b_t}([\Theta_t]^h(V))) .\ \ $$

\end{itemize}
The functors $[\Theta]$ and $[\Theta]^h$ are well-defined as the above construction only depends on the isomorphism class of the $p$-enhancement $\Theta$.
\end{definition}

Using this notation, we can generalise \cite[Proposition 9.14]{arone2018action} and describe the Bredon homology of stabilisers in the partition complex:
\begin{proposition}\label{stabiliserdec}
Let $K_\sigma\subset \Sigma_n$ be  the stabiliser of a simplex $\sigma = [\ \hat{0}<x_1<\dots<x_i<\hat{1}\ ]$ in $\Sigma |\Pi_n|^\diamond$ and write $E[\sigma]$ for the set of isomorphism classes of its $p$-enhancements.

For any $M\in \md_k$, there are isomorphisms
$$ \pi_\ast(F_{\Sigma_n/K_\sigma}(M)) = \bigoplus_{[\Theta] \in E[\sigma]} [\Theta](\pi_\ast(M)) \ \ \ \ \ \ \ \ \ \ \ \ \ \  \pi_\ast(F^h_{\Sigma_n/K_\sigma}(M)) = \bigoplus_{[\Theta] \in E[\sigma]} [\Theta]^h(\pi_\ast(M))$$
\end{proposition}
\begin{proof}
This follows formally   from \Cref{sf} by precisely the same argument as used in the proof of  \cite[Proposition 9.14]{arone2018action}.
\end{proof} 

\subsubsection*{The Bredon Homology of the Partition Complex}
Let $\mathcal{P}_n$ be the poset of partitions of $\{1,\ldots,n\}$. Observe that
$|\Pi_n|^\diamond$ is $\Sigma_n$-equivariantly equivalent to the  realisation of the pointed simplicial set 
 $$T_\bullet := N_\bullet(\mathcal{P}_n- \{\hat{0}\})/N_\bullet(\mathcal{P}_n- \{\hat{0},\hat{1}\}).$$
Its nondegenerate $i$-simplices are either the basepoint or correspond to chains of partitions $$\sigma = [\ \hat{0}< x_1<\ldots<x_i < \hat{1}\ ].$$

The  groups $\widetilde{\hml}^{\Br}_\ast(|\Pi_n|^\diamond, \mu_t^M)$ and $\widetilde{\hml}^{\Br}_\ast(|\Pi_n|^\diamond, \mu_t^{M,h})$ are   given by  
 the homology  of the normalised chain complexes  $\widetilde{C}^{\Br}_\ast(|\Pi_n|^\diamond, \mu_t^M)$, $\widetilde{C}^{\Br}_\ast(|\Pi_n|^\diamond, \mu_t^{M,h})$ of the  
  simplicial abelian groups $ \mu_t^M(T_\bullet)$, $ \mu_t^{M,h}(T_\bullet)$.
  
The $i^{th}$ degree of the chain complexes $\widetilde{C}^{\Br}_\ast(|\Pi_n|^\diamond , \mu_t^M)$ and $\widetilde{C}^{\Br}_\ast(|\Pi_n|^\diamond, \mu_t^{M,h})$ can be decomposed with the help of  \Cref{stabiliserdec} as a direct sum   indexed by isomorphism classes of $p$-enhancements $$\Theta =  [\ \hat{0} \leq e_1  \leq x_1 \leq \ldots \leq e_i \leq x_i \leq e_{i+1} \leq \hat{1}\ ].$$

 In fact, we can discard most $p$-enhancements. Let us  call a $p$-enhancement as above   \textit{pure} if $e_j = x_j$ for all $1\leq j \leq i+1$.
Extending \cite[Proposition 9.19]{arone2018action} to our setting, we have:
\begin{proposition}\label{puresummands} For each $t$, the  Euler characteristics
of $\widetilde{\hml}^{\Br}_\ast(|\Pi_n|^\diamond, \mu_t^M)$  and
$\widetilde{\hml}^{\Br}_\ast(|\Pi_n|^\diamond, \mu_t^{M,h})$ agrees with the
Euler characteristic of the submodules of
$\widetilde{C}^{\Br}_\ast(|\Pi_n|^\diamond, \mu_t^M)$ and
$\widetilde{C}^{\Br}_\ast(|\Pi_n|^\diamond, \mu_t^{M,h})$ spanned by all
summands corresponding to pure $p$-enhancements.
\end{proposition}

We can now establish the main results of this section:
\begin{proof}[Proof of   \Cref{odda} and \Cref{oddb}]
By the observations on duality made in the beginning of this section on page \pageref{du}, it suffices to compute  $\pi_t(F_{\Sigma|\Pi_n|^\diamond}(\Sigma^{l} \FF_p))$ and $\pi_t(F^h_{\Sigma|\Pi_n|^\diamond}(\Sigma^{l} \FF_p))$ for $l = -\ell$.

By \Cref{BSSdegenerates}, these groups vanish if $n$ is not a power of $p$. If $n=p^k$, then the dimension of these groups is given by  the Euler characteristics of 
$H^{\Br}_\ast(|\Pi_{p^k}|^\diamond, \mu_{t-k}^M)$  and $\hml^{\Br}_\ast(|\Pi_{p^k}|^\diamond, \mu_{t-k}^{M,h})$.

\Cref{puresummands} shows that these dimensions agree with the Euler
characteristics of the \textit{submodules}   of
$\widetilde{C}^{\Br}_\ast(|\Pi_{p^k}|^\diamond, \mu_{t-a}^M)$ and
$\widetilde{C}^{\Br}_\ast(|\Pi_{p^k}|^\diamond, \mu_{t-a}^{M,h})$ spanned by all
summands corresponding to pure $p$-enhancements. 
As in \cite[Theorem 9.1]{arone2018action}, we are therefore reduced to computing the Euler characteristics of the following bigraded abelian groups in  ``Bredon-direction''.\vspace{5pt}

\noindent
\begin{tikzcd}[column sep=2.2em]
	\mathbf{(1)} &
	\mathcal{F}_k(\Sigma^{l}\FF_p) &
	\bigoplus_{k_1+k_2=k} \mathcal{F}_{k_1}\mathcal{F}_{k_2}(\Sigma^{l}\FF_p) &
	\cdots &
	\mathcal{F}_{1}\cdots\mathcal{F}_{1}(\Sigma^{l}\FF_p)
\end{tikzcd}

\noindent
\begin{tikzcd}[column sep=2.2em]
	\mathbf{(2)} &
	\mathcal{F}^h_k(\Sigma^{l}\FF_p) &
	\bigoplus_{k_1+k_2=k} \mathcal{F}^h_{k_1}\mathcal{F}^h_{k_2}(\Sigma^{l}\FF_p) &
	\cdots &
	\mathcal{F}^h_{1}\cdots\mathcal{F}^h_{1}(\Sigma^{l}\FF_p)
\end{tikzcd}

To begin with, we use \Cref{fk} to see that the summand $\mathcal{F}_{k_1}\ldots \mathcal{F}_{k_r}(\Sigma^{l} \FF_p)$ in $\mathbf{(1)}$   has a basis consisting of all sequences $(i_1,\ldots, i_k)$ satisfying the following properties:
\begin{enumerate}
\item  each $|i_j|$ is congruent to $0$ or $1$ modulo $2(p-1)$;
\item $i_j \geq p i_{j+1}>p$ \textit{or}  $i_j \leq p i_{j+1} \leq 0$    for all $1\leq j<k$ with $j \neq k_1, k_1+k_2, \ldots$;
\item  if $p$ is odd,  then for $t=0,\ldots, r-1$, we have \\
 $_{\ }$ \  \  $1<i_{k_1+\ldots + k_t+1} <(p-1)(l+i_{k_1+\ldots + k_t+2}+\ldots +i_{k_1+\ldots + k_r})$\\ \textit{or} 
$ 0 \geq i_{k_1+\ldots + k_t+1} > (p-1)(l+i_{k_1+\ldots + k_t+2}+\ldots +i_{k_1+\ldots + k_r})$;\\ 
If $p$ is even,  then for $t=0,\ldots, r-1$, we have \\
 $_{\ }$ \  \     $1<i_{k_1+\ldots + k_t+1} \leq(p-1)(l+i_{k_1+\ldots + k_t+2}+\ldots +i_{k_1+\ldots + k_r})$\\ \textit{or} $ 0 
\geq i_{k_1+\ldots + k_t+1}  \geq (p-1)(l+i_{k_1+\ldots + k_t+2}+\ldots +i_{k_1+\ldots + k_r})$.
\end{enumerate}
Observe that if $l > 0$, then $i_j  > 1$ for all $j$, whereas if $l \leq 0$, then $i_j\leq 0$ for all $j$.\\

Fix a sequence $(i_1,\ldots,i_k)$ satisfying the  conditions above, but such that for  $j = k_1, k_1+k_2, \ldots$, we have $1< i_j < p i_{j+1} $ \textit{or}  $0 \geq i_j > p i_{j+1}$. Informally speaking, $(2)$ is violated whenever possible.

This sequence $(i_1,\ldots,i_k)$ appears exactly once as a basis element in
$\mathcal{F}_{m_1} \ldots \mathcal{F}_{m_s}(\Sigma^{l} \FF_p)$ for any ordered
partition $m_1+\ldots +m_s = k$ of the number $k$ refining the ordered partition
$k_1+\ldots+k_r=k$.  Counting the number of such refinements, we see that the
total contribution of  $(i_1,\ldots,i_k)$ to  the Euler characteristic in
``Bredon direction'' is $$\sum_{s=r}^k (-1)^{s-1} {{k-r}\choose{s-r}}.$$ This
alternating sum is zero  for $k\neq r$. For $k=r$, it is equal to $(-1)^{k-1}$.
In this case,  $(i_1,\ldots,i_k)$ indexes a basis element in $\mathcal{F}_1
\ldots \mathcal{F}_1 (\Sigma^{l} \FF_p)$. Hence, the absolute value of the Euler
characteristic ``in Bredon direction" is  equal to the number of sequences $(i_1,\ldots,i_k)$
 satisfying the following:
 \begin{enumerate}
\item  each $|i_j|$ is congruent to $0$ or $1$ modulo $2(p-1)$;
\item $1<i_j < p i_{j+1} $ \textit{or}  $0 \geq i_j > p i_{j+1} $    for all $1\leq j<k$;
\item  if $p$ is odd,  then for $t=0,\ldots, k-1$, we have \\
 $_{\ }$ \  \  $1<i_{t+1} <(p-1)(l+i_{t+2}+\ldots +i_{k})$ \textit{or} 
$ 0 \geq i_{t+1} > (p-1)(l+i_{t+2}+\ldots +i_{k})$;\\ 
If $p$ is even,  then for $t=0,\ldots, k-1$, we have \\
 $_{\ }$ \  \     $1<i_{t+1} \leq(p-1)(l+i_{t+2}+\ldots +i_{k})$  \textit{or} $ 0 
\geq i_{t+1}  \geq (p-1)(l+i_{t+2}+\ldots +i_{k})$.
\end{enumerate}
We observe that   these conditions can be rephrased as follows:
\begin{enumerate}
\item each $|i_j|$ is congruent to $0$ or $1$ modulo $2(p-1)$;
\item for all $1 \leq j<k$ we have $1 < i_j < pi_{j+1}$ or $p i_{j+1} < i_j \leq 0$;
\item we have $1 < i_k \leq  (p-1) l$ or $ (p-1) l \leq  i_k\leq 0$.
\end{enumerate} 
 \Cref{odda} follows by replacing each $i_j$ by its   inverse for the sake of notational convenience.

 \vspace{5pt}
We compute the ``Bredon Euler characteristic'' of the complex  $\mathbf{(2)}$  above by a similar method. Using   \Cref{fhk}, we see that the summand $\mathcal{F}^h_{k_1}\ldots \mathcal{F}^h_{k_r}(\Sigma^{l} \FF_p)$  has a basis consisting of all sequences $(i_1,\ldots, i_k)$ satisfying the following properties:
\begin{enumerate}
 \item\hspace{-5pt}'  each $i_j$ is congruent to $0$ or $-1$ modulo $2(p-1)$;
\item\hspace{-5pt}' $i_j \leq p i_{j+1}$     for all $1\leq j<k$ with $j \neq k_1, k_1+k_2, \ldots$;
\item\hspace{-5pt}'  if $p$ is odd,  then for $t=0,\ldots, r-1$, we have \\
 $_{\ }$ \  \  $i_{k_1+\ldots + k_t+1} > (p-1)(l+i_{k_1+\ldots + k_t+2}+\ldots +i_{k_1+\ldots + k_r})$;\\ 
if $p$ is even,  then for $t=0,\ldots, r-1$, we have \\
 $_{\ }$ \  \     $i_{k_1+\ldots + k_t+1} \geq(p-1)(l+i_{k_1+\ldots + k_t+2}+\ldots +i_{k_1+\ldots + k_r})$.
\end{enumerate}

We fix a sequence $(i_1,\ldots,i_k)$ satisfying the four conditions above,  such that for  $j = k_1, k_1+k_2, \ldots$, we have $i_j > p i_{j+1}$. Again, the sequence 
  appears  exactly once as a basis element in $\mathcal{F}^h_{m_1} \ldots \mathcal{F}^h_{m_s}(\Sigma^{l} \FF_p)$ for any ordered partition $m_1+\ldots +m_s = k$ of $k$ refining $k_1+\ldots+k_r=k$. As above, we see that   these copies have vanishing contribution to the Euler characteristic unless $k=r$, in which case they contribute  $(-1)^{k-1}$. Hence, the 
absolute value of the Euler characteristic in Bredon direction is equal to the number of sequences $(i_1,\ldots,i_k)$
 satisfying:
 \begin{enumerate}
\item\hspace{-5pt}' each $i_j$ is congruent to $0$ or $-1$ modulo $2(p-1)$;
\item\hspace{-5pt}' $ i_j > p i_{j+1} $ for all $1\leq j<k$;
\item\hspace{-5pt}' if $p$ is odd,  then for $t=0,\ldots, k-1$, we have 
  $i_{ t+1} > (p-1)(l+i_{  t+2}+\ldots +i_{k})$;\\ 
if $p$ is even,  then for $t=0,\ldots,k-1$, we have 
     $i_{t+1} \geq(p-1)(l+i_{t+2}+\ldots +i_{k})$.
\end{enumerate}
To conclude the proof of \Cref{oddb}, we check that   these conditions are equivalent to \mbox{the following:}
\begin{enumerate}
\item\hspace{-5pt}'  each $i_j$ is congruent to $0$ or $-1$ modulo $2(p-1)$;
\item\hspace{-5pt}'  for all $1 \leq j<k$ we have $p i_{j+1} < i_j $;
\item\hspace{-5pt}'  we have $ (p-1) l \leq  i_k$.\vspace{-10pt}
\end{enumerate}  
\end{proof} 
\subsection{Free partition Lie algebras on an even generator}\label{evengenerator}
In the last section, we have   computed the homotopy groups of free  partition Lie algebras on an {odd} \mbox{generator (cf.\ Theorems \ref{odda}, \ref{oddb}).}\vspace{3pt}

 We will now shift attention to the even degree case. For this, recall that given a  pointed space $X$ and a positive \mbox{integer $d$,} Theorem $8.5.$ in  \cite{arone2018action} constructs a\vspace{-1pt} natural sequence of spaces 
\begin{equation} \label{EHP}\Sigma^2 |\Pi_{\frac{d}{2}}|^\diamond \mywedge{\Sigma_{\frac{d}{2}}} (\Sigma X^{\wedge 2})^{\wedge \frac{d}{2}} 
\rightarrow  \Sigma^2 |\Pi_d|^\diamond \mywedge{\Sigma_d} X^{\wedge d} \rightarrow  \Sigma |\Pi_d|^\diamond \mywedge{\Sigma_d} (\Sigma X)^{\wedge d} \vspace{-1pt}\end{equation}
which varies naturally in $X$.
 If $X=S^n$ is an even-dimensional sphere, then this sequence is in fact a  \textit{cofibration sequence}.
 Applying $\FF_p$-valued cohomology to this sequence, we can decompose the free partition Lie algebra on a  class in negative  even degree $-n $ in terms of free  partition Lie algebras on  odd classes $-n-1$ and $-2n-1$ (using  \Cref{freeplie}).

To extend this decomposition to \textit{all} even integers, we will need to mildly generalise the above sequence (\ref{EHP}) and construct it naturally in topological vector spaces rather than just spaces. 
A minor modification of our argument will also allow us to decompose the free \textit{spectral} partition Lie algebras on an even class (using \Cref{concretespectralplie}), thereby reproving   the classical Takayasu cofibration sequence (cf.\ \cite{takayasu1999stable}) and its ``dual'' (cf \cite[Theorem 3.2]{arone2006note}) by a \vspace{-1pt}new argument.
 
\subsubsection*{Topological  Vector Spaces}
Let $\mathbf{Top}$ \INN{200@$\mathbf{Top}$}
be the category of compactly generated topological spaces (henceforth simply called ``spaces'') with its Quillen model structure. This is a well-fibred topological cartesian closed category over sets in the sense of \cite[Definitions 21.7., 27.20]{adamek2004abstract}. By \cite[Proposition 2.2, Proposition 4.6]{seal2005cartesian}, we can therefore lift  the usual tensor product on $\FF_p$-vector spaces  to a closed symmetric monoidal structure $\otimes$ on the category  $\mathbf{tMod}_{\FF_p}$ \INN{200@$\mathbf{tMod_{\FF_p}},  \mathbf{tMod}_{\FF_p}^J$}
 \INN{060@$\FF_p\{X\}$}     
of (compactly generated) topological $\FF_p$-vector spaces. This topological tensor product satisfies the expected universal property with respect to continuous bilinear maps.

Given a pointed completely regular space $(X,x)$, we can form the \textit{free topological   $\FF_p$-vector space} $\FF_p\{X\}$ on $X$ with $x = 0$ satisfying the obvious universal property and containing $X$ as a closed subset (cf.\ \cite[Theorem 6.2.2]{arnautov1996introduction}).  The underlying $\FF_p$-vector space of $\FF_p\{X\}$ is simply given by the free $\FF_p$-vector space on $X$ \vspace{-1pt}with $x=0$.  

\subsubsection*{Gradings}
For $J$   (commutative) indexing monoid  in sets, work of Schw{\"a}nzl-Vogt \cite{schwanzl1991categories} shows that  the category $\mathbf{tMod}_{\FF_p}^J$ of functors   from $J$ to (compactly generated) topological $\FF_p$-vector spaces carries  a cofibrantly generated  model structure. Its underlying fibrations and weak equivalences are given by pointwise fibrations and weak equivalences on underlying spaces. 

Moreover,   $\mathbf{tMod}_{\FF_p}^J$ has  a symmetric monoidal structure given by  Day convolution. It sends $V,W $ to the $J$-graded topological $\FF_p$-vector space with
$(V\otimes W)_k  : = \bigoplus_{a+b = k} (V_a \otimes  W_b)$.
For $(X,x)$ a  pointed completely regular space,  we equip $\FF_p\{X\}$ with a grading \vspace{-1pt}concentrated\mbox{ in degree $0$.}

\subsubsection*{Topological  Algebras}\label{topologicalalgebras}
Write $\mathbf{tAlg}_{\FF_p}^J$ for the category of  \INN{200@$\mathbf{tAlg}_{\FF_p}^J,  \mathbf{tAlg}^{aug,J}_{\FF_p}$}
 commutative algebra objects in the symmetric monoidal category $\mathbf{tMod}_{\FF_p}^J$; these are
   $\textit{J}$-graded  (compactly generated) topological  commutative  $\FF_p$-algebras.
Again,   the work of Schw{\"a}nzl-Vogt \cite{schwanzl1991categories} equips $\mathbf{tAlg}_{\FF_p}^J$ with a cofibrantly generated model structure in which  a map is a fibration or weak equivalence if the underlying map in $\mathbf{tMod}_{\FF_p}^J$ has the corresponding property. We denote the augmented variant   by $\mathbf{tAlg}_{\FF_p}^{J,aug}$.
There is a natural functor $  \mathbf{tMod}_{\FF_p}^J \rightarrow \mathbf{tAlg}^{aug,J}_{\FF_p}$ sending  $V$ to the trivial square-zero extension $\FF_p\oplus V$ on $V$.  
\subsubsection*{Simplicial Variants}\label{simplicialvariants}
We can also 
 define model categories $\mathbf{sMod}_{\FF_p}^J$ and $\mathbf{sAlg}_{\FF_p}^J$ of $J$-graded simplicial $\FF_p$-modules and simplicial commutative $\FF_p$-algebras, respectively.
The standard Quillen equivalence $|-|: \sSet \leftrightarrows  \mathbf{Top}: \Sing$ preserves finite products and therefore  induces Quillen equivalences 
$$|-|: \mathbf{sMod}_{\FF_p}^J \leftrightarrows \mathbf{tMod}_{\FF_p}^J: \Sing \mbox{ \ \ \ and \ \ \ }  |-|:  \mathbf{sAlg}_{\FF_p}^J \leftrightarrows \mathbf{tAlg}_{\FF_p}^J:\Sing.$$

Given two $J$-graded simplicial $\FF_p$-vector spaces $V_\bullet$ and $W_\bullet$, it is straightforward to check that there is an isomorphism $|V_\bullet|\otimes |W_\bullet| \cong |V_\bullet \otimes  W_\bullet|$.
Geometric realisation  intertwines  square-zero extensions   in $\mathbf{sAlg}_{\FF_p}^{J,aug}$ with the corresponding construction in $\mathbf{tAlg}_{\FF_p}^{J,aug}$.
Finally, the functor $|-|$  sends free simplicial $\FF_p$-modules to   free topological $\FF_p$-modules.

\subsubsection*{Homotopy Pushouts of Algebras}
As expected, pushouts in $\mathbf{tAlg}_{\FF_p}^J$  are simply  computed by relative tensor products. More precisely, given a span  $B \leftarrow A \rightarrow C$ of $J$-graded topological $\FF_p$-algebras, the pushout is given by the coequaliser  $B\myotimes{A} C := \coequ(B\otimes A \otimes C \rightrightarrows B\otimes C)$.

\begin{cons}\label{suspensionconstruction}
We  describe an explicit model for the \textit{homotopy pushout} of   $B \leftarrow A \rightarrow C$ in $\mathbf{tAlg}_{\FF_p}^J$. For  $k\geq 0$, the topological $\FF_p$-vector space 
$\FF_p\{\Delta^k\}  \otimes B \otimes A^{\otimes n} \otimes  C$ is generated by symbols
\[
\Biggl(
\begin{tikzpicture}[baseline=(m.center)]
	\matrix (m) [matrix of math nodes,
	row sep=1pt,
	column sep=4pt] {
		0 & \leq & t_1 & \leq & t_2 & \leq & \dots & \leq & t_k & \leq & 1 \\
		b & \otimes & a_1 & \otimes & a_2 & \otimes & \ldots & \otimes & a_k & \otimes & c \\
	};
\end{tikzpicture}
\Biggr)
\]
with $t_i \in [0,1]$, $a_i \in A$, $b\in B$, and $c\in C$, subject to the standard multilinear relations.

Define an object in $ \mathbf{tMod}_{\FF_p}^J$ by $B \myotimestwo{A}{\hobased} C = |\Barr_\bullet(B,A,C)| =\bigg( \bigoplus_{k \geq 0}\FF_p\{\Delta^k\}  \otimes  B \otimes A^{\otimes n} \otimes  C   \bigg)\bigg/\hspace{-2pt}\sim,$ where $\sim$ denotes the quotient by the $\FF_p$-linear subspace generated\vspace{-10pt}  by the following relations: \\ 
 \[
 \Biggl(
 \begin{tikzpicture}[baseline=(m.center)]
 	\matrix (m) [matrix of math nodes, row sep=2pt, column sep=4pt] {
 		0 & \leq & \dots & \leq & t_i & = & t_{i+1} & \leq & \dots \\
 		b & \otimes & \dots & \otimes & a_i & \otimes & a_{i+1} & \otimes & \dots \\
 	};
 \end{tikzpicture}
 \Biggr)
 \ \sim\
 \Biggl(
 \begin{tikzpicture}[baseline=(m.center)]
 	\matrix (m) [matrix of math nodes, row sep=2pt, column sep=4pt] {
 		0 & \leq & \dots & \leq & t_i & \leq & t_{i+2} & \leq & \dots \\
 		b & \otimes & \dots & \otimes & a_i \cdot a_{i+1} & \otimes & a_{i+2} & \otimes & \dots \\
 	};
 \end{tikzpicture}
 \Biggr)
 \] 
 \[
 \Biggl(
 \begin{tikzpicture}[baseline=(m.center)]
 	\matrix (m) [matrix of math nodes, row sep=2pt, column sep=4pt] {
 		0 & \leq & \dots & \leq &\hspace{2pt}  t_i \hspace{2pt} & \leq & t_{i+1} & \leq & \dots \\
 		b & \otimes & \dots & \otimes & a_i & \otimes & 1 & \otimes & \dots \\
 	};
 \end{tikzpicture}
 \Biggr)
 \ \sim\
 \Biggl(
 \begin{tikzpicture}[baseline=(m.center)]
 	\matrix (m) [matrix of math nodes, row sep=2pt, column sep=4pt] {
 		0 & \leq & \dots & \leq & \ \   \ \  t_i  \ \   \ \ & \leq & t_{i+2} & \leq & \dots \\
 		b & \otimes & \dots & \otimes & a_i & \otimes & a_{i+2} & \otimes & \dots \\
 	};
 \end{tikzpicture}
 \Biggr).
 \]
We endow $B \myotimestwo{A}{\hobased} C$ with  the unique multiplication \ {$  (B \myotimestwo{A}{\hobased} C)\otimes ( B \myotimestwo{A}{\hobased} C ) \rightarrow (B \myotimestwo{A}{\hobased} C)$ \mbox{satisfying}}
\\
\\ 
\\
\[
\Bigg(
\begin{tikzpicture}[baseline=(current bounding box.center)]
	\matrix (m1) [matrix of math nodes,row sep=2pt,column sep=3pt]
	{
		0   & \leq & t_{i_1} & \leq & \dots & \leq & t_{i_n} & \leq & 1 \\
		b_1 & \otimes & a_{i_1} & \otimes & \dots & \otimes & a_{i_n} & \otimes & c_1 \\
	};
\end{tikzpicture}
\Bigg)
\;\cdot\;
\Bigg(
\begin{tikzpicture}[baseline=(current bounding box.center)]
	\matrix (m2) [matrix of math nodes,row sep=2pt,column sep=3pt]
	{
		0   & \leq & t_{j_1} & \leq & \dots & \leq & t_{j_m} & \leq & 1 \\
		b_2 & \otimes & a_{j_1} & \otimes & \dots & \otimes & a_{j_m} & \otimes & c_2 \\
	};
\end{tikzpicture}\vspace{-8pt}
\Bigg)\] \[ \hspace{173pt} 
\;=\;
\Bigg(
\begin{tikzpicture}[baseline=(current bounding box.center)]
	\matrix (m3) [matrix of math nodes,row sep=2pt,column sep=3pt]
	{
		0        & \leq & t_{1} & \leq & \dots & \leq & t_{n+m} & \leq & 1 \\
		b_1 b_2  & \otimes & a_{1} & \otimes & \dots & \otimes & a_{n+m} & \otimes & c_1 c_2 \\
	};
\end{tikzpicture}
\Bigg)
\]
for any disjoint union 
$\{1<\dots< m+n\} \   \ =\  \  \{i_1 < \dots < i_n\} \coprod \{j_1 < \dots < j_m\}  $.
Simple checks show that  this multiplication is well-defined, commutative and associative.
\end{cons}
The following  is proven by an argument entirely parallel\vspace{-2pt} to \mbox{the proof of \cite[Proposition 7.15]{arone2018action}:}
\begin{proposition}\label{explicitconstructionworks}
If the unit $\FF_p\rightarrow A$ is a cofibration, then $B \myotimestwo{A}{\hobased} C$ is a homotopy pushout of the span $B \leftarrow A \rightarrow C$ of topological $J$-graded $\FF_p$-algebras.\vspace{-2pt}
\end{proposition}

\begin{definition}\label{suspension1}
The \textit{suspension} $\Sigma^{\otimes} A$ of some  $A\in \mathbf{tAlg}_{\FF_p}^{J, aug}$ is given by $ \FF_p \myotimestwo{A}{\hobased} \FF_p$.\vspace{-2pt} \INN{190@$\Sigma^{\otimes}$}\INN{270@$\Omega^{\otimes}$}
\end{definition}

\subsubsection*{Homotopy Pullbacks of Algebras}
A much simpler construction gives us  explicit models for homotopy pullbacks of $J$-graded topological $\FF_p$-algebras. For this, let $D^I = \Map_{\mathbf{Top}}([0,1], D)$ denote the space of paths in a\vspace{-3pt} given space $D$.

\begin{definition} If $B \xrightarrow{f} A \xleftarrow{g} C$ is a diagram of $J$-graded  topological  commutative $\FF_p$-algebras, we equip the $J$-graded space $B\mytimestwo{A}{\hobased} C $ determined by\vspace{-4pt}
$$(B\mytimestwo{A}{\hobased} C)_j := \left\{(b, \upalpha, c) \in B_j\times A^I_j \times C_j \  |  \    \upalpha(0) = f(b), \upalpha(1) = g(c) \right\}\vspace{-4pt} $$
with an $\FF_p$-algebra structure by setting \vspace{-2pt}
$$\lambda_1 (b_1, \upalpha_1, c_1) +  \lambda_2(b_2, \upalpha_2, c_2) = (\lambda_1 b_1 + \lambda_2 b_2, \lambda_1\upalpha_1  + \lambda_2\upalpha_2, \lambda_1c_1 + \lambda_2 c_2) \vspace{-2pt}$$ 
$$(b_1, \upalpha_1, c_1) \cdot (b_2, \upalpha_2, c_2) = (b_1 b_2, \upalpha_1 \upalpha_2, c_1c_2), $$
where the paths $\lambda_1\upalpha_1  + \lambda_2\upalpha_2$ and $\upalpha_1 \upalpha_2 $ are defined using pointwise operations.    
\end{definition} 
An entirely parallel argument  to\vspace{-5pt} \mbox{the proof of \cite[Proposition 7.24]{arone2018action} then shows:}
\begin{proposition}
The homotopy pullback of a diagram $B \xrightarrow{f} A \xleftarrow{g} C$ in $\mathbf{tAlg}_{\FF_p}^J$\vspace{-3pt}
 \mbox{is given by $B\mytimestwo{A}{\hobased} C$.}
\end{proposition} 
\begin{definition}
The \textit{loop space} of some $A\in \mathbf{tAlg}_{\FF_p}^J$ is given by 
$\Omega^{\otimes} A:= \FF_p\mytimestwo{A}{\hobased }\FF_p  $.\vspace{-2pt}
\end{definition} 
\begin{remark}The underlying space of $\Omega^{\otimes} A$ is given by the space of all paths $[0,1]\rightarrow A$ which start and end at the same point in $\FF_p \subset A$.
\end{remark}

\subsubsection*{Suspension-Loops Adjunction}
We can link the two constructions above by  setting up an adjunction
$$\Sigma^{\otimes} : \mathbf{tAlg}_{\FF_p}^J \leftrightarrows \mathbf{tAlg}_{\FF_p}^J : \Omega^{\otimes}.$$
Its unit $\eta $ is defined
by the following explicit formula:\vspace{-5pt}
$$
 A \xrightarrow{\eta_A} \Omega^{\otimes} \Sigma^{\otimes} A
 \ \ \text{ sends } \ \
 a \in A
 \ \ \text{ to } \ \
 \upalpha_a :=
 \Bigg(
 (0 \le s \le 1) \mapsto
 \Bigg(
 \tikz[baseline=(m.center)]{
 	\matrix (m) [matrix of math nodes, row sep=0.3em, column sep=0.6em]{
 		0 & \le & s & \le & 1 \\
 		1 & \otimes & a & \otimes & 1 \\
 	};
 }
 \Bigg)
 \Bigg). \vspace{-6pt}
 $$
Here $s\in [0,1]$ denotes a parameter for a loop, and it is not hard to check that the   map $\eta_A$ respects the grading and is both linear and multiplicative. The unit $\epsilon$   is specified as \vspace{-2pt} follows:
\[ 
\Sigma^{\otimes} \Omega^{\otimes} A \xrightarrow{\epsilon_A} A
\ \text{ sends }\ 
\Bigg(
\tikz[baseline=(m.center)]{
	\matrix (m) [matrix of math nodes, row sep=0.3em, column sep=0.25em]{
		0 & \le & t_1 & \le & \dots & \le & t_n & \le & 1 \\
		\lambda & \otimes & \upalpha_1 & \otimes & \dots & \otimes & \upalpha_n & \otimes & \mu \\
	};
}
\Bigg)
\in \Sigma^{\otimes} \Omega^{\otimes} A
\ \text{ to }\ 
\lambda \cdot \upalpha_1(t_1) \cdot \ldots \cdot \upalpha_n(t_n) \cdot \mu.
\]

Here $\alpha_1,\ldots,\alpha_n \in \Omega^{\otimes}A$ are given paths, and we it is again straightforward to check that this assignment gives a map in $\mathbf{tAlg}_{\FF_p}^J$. Observe that $ \epsilon_{ \Sigma^{\otimes}} \circ \Sigma^{\otimes}\eta $ and $\Omega^{\otimes} \epsilon \circ \eta_{\Omega^{\otimes}}$ are indeed given by the identity transformations, and we have therefore defined an adjunction.

The following result is proven by an argument parallel to the proof of \cite[Proposition 7.28]{arone2018action}:
\begin{lemma} 
The adjunction $ \Sigma^{\otimes} : \mathbf{tAlg}_{\FF_p}^J \leftrightarrows \mathbf{tAlg}_{\FF_p}^J : \Omega^{\otimes} $ 
is Quillen. 
\end{lemma}

\subsubsection*{The $\EHP$-sequence for Topological Vector Spaces.}
We proceed to generalise the $\EHP$-sequence for (strictly commutative) monoid spaces (cf.\ \cite[Definition 7.42]{arone2018action}) to the setting of topological $\FF_p$-algebras. We begin with the following observation, which is immediate from \Cref{suspension1}:
\begin{proposition}
Given a $J$-graded topological $\FF_p$-vector space $V$, we let $\FF_p \oplus V \in \mathbf{tAlg}_{\FF_p}^J$ denote the trivial square-zero extension of $\FF_p$ by $V$. There is an isomorphism  of $J$-graded topological $\FF_p$-vector spaces
$$\Sigma^{\otimes}(\FF_p \oplus V) \simeq \bigoplus_k \FF_p\{S^k\} \otimes V^{\otimes k}.$$
\end{proposition}
Using this splitting, we define  a natural transformation of functors $\mathbf{tMod}_{\FF_p}^J  \rightarrow \mathbf{tAlg}_{\FF_p}^J$ as follows:
\begin{definition}\label{Einhaengung}
Given $V\in \mathbf{tMod}_{\FF_p}^J$, the \textit{Einh{\"a}ngung} $E_V: \Sigma^{\otimes}( \FF_p \oplus V) \rightarrow  \FF_p \oplus \left(\FF_p\{S^1\} \otimes V\right)$ is the map of $J$-graded topological $\FF_p$-algebras obtained by projecting to the first two summands.
\end{definition}
The construction of the Hopf map is somewhat more interesting. To this date, we do not know of a definition staying in the realm of higher category theory, and this is in fact the reason why we had to work with strict models.

First, given any $V\in \mathbf{tMod}_{\FF_p}^J$, we define a map of $J$-graded topological $\FF_p$-vector spaces 
$$\Phi : \FF_p\{S^1\} \otimes V^{\otimes 2} \ \ \   \longrightarrow \ \ \    \Omega^{\otimes} \Sigma^{\otimes}( \FF_p \oplus V) \ \ \   \ \ \   \ \ \   \ \ \    \ \ \   \ \ \   \ \ $$  
\[ \ \ \   \ \ \    \ \ \   \ \ \   \ \ \   \hspace{0.5pt}
\bigl(0 \le t \le 1\bigr) \otimes v \otimes w
\;\longmapsto\;
\Bigg(
s \mapsto
\Bigg(
\tikz[baseline=(m.center)]{
	\matrix (m) [
	matrix of math nodes,
	row sep=0.25em,
	column sep=0.6em
	]{
		0 & \leq & ts & \leq & s & \leq & 1 \\
		1 & \otimes & v & \otimes & w & \otimes & 1 \\
	};
}
\Bigg)
\Bigg).
\]

Since $\Phi(x)\cdot \Phi(y) = 0$ for all $x,y$ by inspection, we obtain a map of $J$-graded topological $\FF_p$-algebras
$$\FF_p \oplus (\FF_p\{S^1\} \otimes V^{\otimes 2}) \longrightarrow  \Omega^{\otimes}\Sigma^{\otimes}( \FF_p \oplus V).$$
\begin{definition}\label{Hopfmap}
For $V\in \mathbf{tMod}_{\FF_p}^J$, the \textit{Hopf map} $H_V: \Sigma^{\otimes}( \FF_p \oplus (\FF_p\{S^1\} \otimes V^{\otimes 2})) \longrightarrow \Sigma^{\otimes}( \FF_p \oplus V) $ is adjoint to the map $\FF_p \oplus (\FF_p\{S^1\} \otimes V^{\otimes 2})  \longrightarrow  \Omega^{\otimes}\Sigma^{\otimes}( \FF_p \oplus V) $ specified above.
\end{definition}
The Hopf map varies naturally in the $\FF_p$-module $V$.\vspace{3pt}

We now fix the   indexing monoid
of nonnegative integers  $J = \NN$, considered under addition. Observe that the $\infty$-category $\mathcal{D}^{\gr}$ from  \Cref{setupSCR} arises as a full subcategory of the underlying $\infty$-category of $ \mathbf{tAlg}_{\FF_p}^{\NN}$. Namely, it is spanned by all algebras which are equal to $\FF_p$ in weight $0$.

Given a finite-dimensional discrete $\FF_p$-vector space $V\in \vect_{\FF_p}^\omega$, we write $V_1$ for the 
 $\NN$-graded topological $\FF_p$-vector space consisting of $V$ concentrated in degree $1$.
Combining \Cref{Einhaengung} and \Cref{Hopfmap}, we obtain a sequence of $\NN$-graded topological $\FF_p$-algebras
\begin{equation}\Sigma^{\otimes}( \FF_p \oplus (\FF_p\{S^1\} \otimes V_1^{\otimes 2})) \xrightarrow{H}\Sigma^{\otimes}( \FF_p \oplus V_1) \xrightarrow{E}  \FF_p \oplus \left(\FF_p\{S^1\} \otimes V_1\right).\label{EHPmodel}\end{equation}
which  varies naturally in $V$. Inverting weak equivalences in $ \mathbf{tMod}_{\FF_p}^{\NN}$, we obtain:

\begin{proposition}\label{strictsequence}
There is a natural sequence of functors 
$\vect_{\FF_p}^{\omega} \rightarrow \mathcal{D}^{\gr}$ (cf.\ \Cref{setupSCR}) sending $V\in \vect_k^{\omega}$ to a sequence 
$$ \Sigma^{\otimes}(\sqz(\Sigma V_1^{\otimes 2}))  \xrightarrow{H} \Sigma^{\otimes}(\sqz(V_1)) \xrightarrow{E} \sqz(\Sigma V_1)
$$
Here $\Sigma^{\otimes}$ denotes the suspension functor in the pointed $\infty$-category $\mathcal{D}^{\gr}$, whereas $\Sigma$ denotes the suspension functor in $\md_k$, i.e.\ the shift in $\md_k$.
\end{proposition}
\begin{proof}
Writing $\FF_p\{S^1_\bullet\}$ for the free simplicial $\FF_p$-module on the  simplicial circle $S^1_\bullet$ with the basepoint equal to $0$, we verify the universal property to deduce that $|\FF_p\{S^1_\bullet\}|\cong \FF_p\{S^1\}$. Since the geometric realisation functor $|-|: \mathbf{sMod}_{\FF_p} \rightarrow \mathbf{tMod}_{\FF_p}$ also respects tensor products, we deduce that $|\FF_p\{S^1_\bullet\} \otimes V_1| \cong \FF_p\{S^1\}\otimes V_1$ is equivalent to the chain complex $\Sigma V\in \md_k$, concentrated in weight $1$. 
The claim  follows from  \Cref{explicitconstructionworks}, since all appearing units   are  cofibrations.
\end{proof}

\subsubsection*{Decomposing  Lie Algebras on an Even Class.}
The  preceding section allows us to decompose even Lie algebras in terms of odd ones. In the terminology of \Cref{bredon} and  \Cref{partitioncomplex}, we obtain:
\begin{proposition}\label{partitionsequence}
For  $w\geq 0$, there is a  natural  sequence of functors  $\vect_k^{\omega} \rightarrow \md_k$  \mbox{sending $V$ to}
$$\Sigma F_{\Sigma |\Pi_{\frac{w}{2}}|^\diamond}(\Sigma V^{\otimes 2}) \longrightarrow 
 \Sigma F_{\Sigma |\Pi_{w}|^\diamond}( V)  \longrightarrow F_{\Sigma |\Pi_{w}|^\diamond} (\Sigma V),$$
where the leftmost module is interpreted as zero whenever $w$ is odd.  
\end{proposition}
\begin{proof}
First, we apply the cotangent fibre functor $\cot$ (for simplicial commutative rings) to the sequence appearing in \Cref{strictsequence}. In a second step, we note that since the left adjoint $\cot_{\Delta}$ preserves colimits, there is a natural equivalence $\Sigma \circ \cot_{\Delta} \simeq \cot_{\Delta} \circ \Sigma^{\otimes}$. Finally, we proceed as in the proof of \Cref{freeplie} to evaluate the functor $\cot_{\Delta}$ on a trivial square-zero extension, thereby keeping track of the weights.
\end{proof}
By  \Cref{polyrightext}, we can in fact take the right-left Kan extension and obtain a sequence of $w$-excisive functors $\mod_k \rightarrow \mod_k$ sending $V$ to \begin{equation}\label{cofseq}\Sigma F_{\Sigma |\Pi_{\frac{w}{2}}|^\diamond}(\Sigma V^{\otimes 2}) \rightarrow 
 \Sigma F_{\Sigma |\Pi_{w}|^\diamond}( V)  \rightarrow F_{\Sigma |\Pi_{w}|^\diamond} (\Sigma V).\end{equation}

Applying the reduced singular chains functor $ \widetilde{C}_\ast( - , \FF_p)$ to \cite[Theorem 8.5]{arone2018action}, we see that  \eqref{cofseq} is a cofibre sequence   when evaluated on modules $V=\Sigma^n \FF_p$ with $n\geq 0 $ even.
By \cite[Proposition 4.6]{arone1999goodwillie}, this implies that \eqref{cofseq}  is in fact a cofibre sequence  on 
all modules of the form $V=\Sigma^n \FF_p$ with $n$ an even integer.
Applying linear duality and using \Cref{freeplie}, we deduce:
 \begin{theorem}\label{ehpa}
For all even integers $n$  and all weights $w\geq 0$, there is a cofibre
sequence in $\md_k$
$$ \Sigma \Free_{\Lie_{\FF_p,\Delta}^\pi}[w](\Sigma^{n-1}\FF_p )\rightarrow  \Free_{\Lie_{\FF_p,\Delta}^\pi}[w](\Sigma^n \FF_p)\rightarrow  \Free_{\Lie_{\FF_p,\Delta}^\pi}\left[\frac{w}{2}\right](\Sigma^{2n-1}\FF_p ).$$
\end{theorem}
The forgetful functor $\mathcal{D}^{\gr} \rightarrow \mathcal{C}^{\gr}$ from graded simplicial algebras to graded $\EE_\infty$-algebras described in \Cref{forgetting}  preserves  pushouts and trivial square-zero extensions. We may therefore   interpret \Cref{strictsequence} as a natural sequence of $\EE_\infty$-algebras. Repeating the argument in the proof of \Cref{strictsequence} in this context, we conclude:
  \begin{theorem}\label{ehpb}
For all even integers $n$   and all weights $w\geq 0$, there is a cofibre
sequence in $\md_k$
$$ \Sigma \Free_{\Lie_{\FF_p,\EE_\infty}^\pi}[w](\Sigma^{n-1}\FF_p )\rightarrow  \Free_{\Lie_{\FF_p,\EE_\infty}^\pi}[w](\Sigma^n \FF_p)\rightarrow  \Free_{\Lie_{\FF_p,\EE_\infty}^\pi}
\left[\frac{w}{2}\right](\Sigma^{2n-1}\FF_p ).$$\end{theorem} 
 \subsection{Free partition Lie algebras on many generators}\label{freemany}
We can express free Lie algebras on many classes in terms of free Lie algebras on  a single generator.
Recall the following   terminology:
\begin{definition}\label{Lyndonword}  \INN{020@$B(n_1,\ldots,n_k)$}
A \textit{Lyndon word} in  letters $x_1,\ldots,x_k$ is a word   which is lexicographically (strictly) minimal among all its cyclic rotations. Write $B_k$ for the set of Lyndon words in $k$ letters and let $B(m_1, \ldots, m_k)\subset B_k$ be the subset consisting of all   words involving each $x_i$ precisely $m_i$ times.

Given a Lyndon word $w\in B_k$, we write $|w|_i$ for the number of occurrences of the letter $x_i$ in $w$.
\end{definition}
Our decomposition will follow from   \cite[Theorem 5.10]{arone2018action}, which we will now recall:
\begin{theorem}\label{AB} Given a decomposition $n=n_1+ \ldots + n_k$, there is a $\Sigma_{n_1} \times \ldots \times \Sigma_{n_k}$-equivariant (simple) homotopy equivalence 
$$\Sigma|\Pi_n|^\diamond \xrightarrow{\ \ \  \simeq \ \ } \bigvee_{\substack{d | \gcd(n_1,\ldots,n_k)\\ w\in B(\frac{n_1}{d}, \ldots, \frac{n_k}{d})} }
\Ind_{\Sigma_d}^{\Sigma_{n_1} \times \ldots \times \Sigma_{n_k}} \left( (S^{\frac{n}{d}-1})^{\wedge d} \wedge \Sigma|\Pi_d|^\diamond\right).$$ 
\end{theorem} 
From this, we can  obtain the following decomposition:
\begin{proposition}\label{frommanytoone}
Given integers $\ell_1,\ldots, \ell_m$, there are isomorphisms of $\NN^m$-graded $\FF_p$-modules
$$\bigoplus_{w \in B_m} \Free_{\Lie_{\FF_p,\Delta}^\pi}\left(\Sigma^{1+\sum_i (\ell_i-1) |w|_i}(\FF_p) \right)
\cong 
\Free_{\Lie_{\FF_p,\Delta}^\pi}\left( {\Sigma^{\ell_1}\FF_p  \oplus \ldots \oplus \Sigma^{\ell_m}\FF_p  }\right)  $$
$$ \bigoplus_{w \in B_m} \Free_{\Lie_{\FF_p,\EE_\infty}^\pi}\left(\Sigma^{1+\sum_i (\ell_i-1) |w|_i}(\FF_p) \right)
  \cong
\Free_{\Lie_{\FF_p,\EE_\infty}^\pi}\left( {\Sigma^{\ell_1}\FF_p  \oplus \ldots \oplus \Sigma^{\ell_m}\FF_p  }\right)$$
The ``multinomial" grading by $\NN^m$ will be constructed in the course of the proof.
\end{proposition} 
\begin{proof} Recall the colimit-preserving functors  \mbox{$F_{(-)} , F^h_{(-)} : \mathcal{S}_\ast^{\Sigma_n} \rightarrow \End_\Sigma^n(\mod_{\FF_p})$} from  \Cref{bredon}. \\
For  $X\in \Set^{\Fin}_\ast$, $V\in \md_k^{\omega}$, and $\ell_1, \ldots, \ell_m \in \ZZ$, 
expanding ``binomially" gives a natural equivalence
$$ \FF_p[X] \myotimes{\Sigma_n} (\Sigma^{-\ell_1} V \oplus \ldots \oplus \Sigma^{-\ell_m} V)^{\otimes n} 
\simeq \bigoplus_{n_1 + \ldots + n_k = n}\Sigma^{-\ell_1 n_1 - \ldots  - \ell_m n_m}  \left(\FF_p[X] \myotimes{\Sigma_{n_1} \times \ldots \times \Sigma_{n_m}} V^{\otimes n}\right).$$
Taking the right-left extension of these degree $n$ functors (cf.\ \Cref{polyrightext}), we obtain
\begin{equation}\label{abcdefg}  F_{X}(\Sigma^{-\ell_1} V \oplus \ldots \oplus \Sigma^{-\ell_m} V) \simeq 
\bigoplus_{n_1 + \ldots + n_k = n}  \Sigma^{-\ell_1 n_1 - \ldots  - \ell_m n_m}
 F_{ \Ind^{\Sigma_n}_{\Sigma_{n_1} \times \ldots \times \Sigma_{n_k}} (X)}(V).\end{equation}
Since $F_{(-)}: \mathcal{S}_\ast^{\Sigma_n} \rightarrow \End_\Sigma^n(\mod_{\FF_p})$ was defined by freely extending from finite $\Sigma_n$-sets to genuine $\Sigma_n$-spaces under sifted colimits, this equivalence in fact holds for general $\Sigma_n$-spaces $X$ in $\mathcal{S}_\ast^{\Sigma_n}$.

We will further analyse the right hand side of the above equivalence in the case $X = \Sigma |\Pi_n|^\diamond$. Here,
\Cref{AB}   gives rise to an equivalence  of $\Sigma_n$-spaces 
$$ \Ind^{\Sigma_n}_{\Sigma_{n_1} \times \ldots \times \Sigma_{n_k}} (\Sigma |\Pi_n|^\diamond) \simeq 
\bigvee_{\substack{d | \gcd(n_1,\ldots,n_k)\\ w\in B(\frac{n_1}{d}, \ldots, \frac{n_k}{d})} }
\Ind_{\Sigma_d}^{\Sigma_n} \left( (S^{\frac{n}{d}-1})^{\wedge d} \wedge \Sigma|\Pi_d|^\diamond\right)$$

Plugging this equivalence into \eqref{abcdefg}, we obtain an identification
$$ \ \ \ \ \ \ \F_{\Sigma | \Pi_n|^\diamond}(\Sigma^{-\ell_1} V \oplus \ldots \oplus \Sigma^{-\ell_m} V)   \ \simeq  \  \displaystyle   \bigoplus_{\substack{n_1+\ldots + n_m = n\\d | \gcd(n_1,\ldots,n_k)\\ w\in B(\frac{n_1}{d}, \ldots, \frac{n_k}{d})} } F_{\Sigma |\Pi_d|^\diamond}\left(\Sigma^{1+(1-\ell_1)\frac{n_1}{d} +\ldots + (1-\ell_m)\frac{n_m}{d} }V  \right).\ \ \ \ \   \ \ \ \ \ \ \ \ \   \ \ \   \ \   \ \ \ \ \ \ \ \ \   \ \ \ \ \vspace{5pt}$$ 
Combining this equivalence with  \Cref{freeplie}, we can deduce the first claim:   the \mbox{weight $n$ piece of}   \mbox{$\Free_{\Lie_{\FF_p,\Delta}^\pi}\left( {\Sigma^{\ell_1}\FF_p  \oplus \ldots \oplus \Sigma^{\ell_m}\FF_p  }\right)$ is }  $\left( F_{\Sigma | \Pi_n|^\diamond}(\Sigma^{-\ell_1} \FF_p \oplus \ldots \oplus \Sigma^{-\ell_m} \FF_p)\right)^\vee$,
whereas the \mbox{weight $n$ piece} of $\displaystyle \bigoplus_{w \in B_m} \Free_{\Lie_{\FF_p,\Delta}^\pi}\left(\Sigma^{1+\sum_i (\ell_i-1) |w|_i}(\FF_p ) \right)$ is  $\displaystyle \bigoplus_{\substack{n_1+\ldots + n_m = n\\d | \gcd(n_1,\ldots,n_k)\\ w\in B(\frac{n_1}{d}, \ldots, \frac{n_k}{d})} } \left(F_{\Sigma |\Pi_d|^\diamond}\left(\Sigma^{1+(1-\ell_1)\frac{n_1}{d} +\ldots + (1-\ell_m)\frac{n_m}{d} }\FF_p   \right)\right)^\vee.$\\  \\
The second claim follows by a parallel argument using the construction $F^h_{(-)}$ instead of $F_{(-)}$.
\end{proof}\ 
 We  combine  our results to prove the main claim of this section.
 \begin{proof}[Proof of \Cref{finalthecohomology}  and \Cref{finalthecohomologyspectral}]
We first consider the case $m=1$. \\
If $p=2$ or $\ell_1$ odd, both statements   can be read off from \Cref{odda} and \Cref{oddb}, respectively.\vspace{5pt}
If $p$ is odd and $\ell_1$ is even, we recall the two cofibre sequences of weight graded $\FF_p$-module spectra established in \Cref{ehpa} and \Cref{ehpb}:
$$ \Sigma \Free_{\Lie_{\FF_p,\Delta}^\pi}[w](\Sigma^{{\ell_1}-1}\FF_p )\rightarrow  \Free_{\Lie_{\FF_p,\Delta}^\pi}[w](\Sigma^{\ell_1} \FF_p)\rightarrow  \Free_{\Lie_{\FF_p,\Delta}^\pi}\left[\frac{w}{2}\right](\Sigma^{2{\ell_1}-1}\FF_p ).$$
$$ \Sigma \Free_{\Lie_{\FF_p,\EE_\infty}^\pi}[w](\Sigma^{{\ell_1}-1}\FF_p )\rightarrow  \Free_{\Lie_{\FF_p,\EE_\infty}^\pi}[w](\Sigma^{\ell_1} \FF_p)\rightarrow  \Free_{\Lie_{\FF_p,\EE_\infty}^\pi}
\left[\frac{w}{2}\right](\Sigma^{2\ell_1-1}\FF_p ).$$

If $w=p^k$ for some $k$, then the right terms vanish. This implies by the ``odd case'' that in both cases, the middle terms have a basis consisting of all sequences $(i_1,\ldots, i_k)$ satisfying conditions $(1), (2)$ and $(1)', (2)'$ respectively,  together with the respective conditions
\begin{enumerate} \setcounter{enumi}{2}
\item  $(p-1) (\ell_1-1) \leq i_k <-1$      or   $ 0 \leq i_k \leq (p-1)   (\ell_1-1)$  ;
\end{enumerate}
\begin{enumerate} \setcounter{enumi}{2}
\item\hspace{-5pt}'   $ i_k \leq  (p-1)  (\ell_1-1)$.
\end{enumerate}
Since $\ell_1$ is even, these conditions are (in light of the congruences $(1)$ or $(1)'$) equivalent to 
\begin{enumerate} \setcounter{enumi}{2}
\item $(p-1) \ell_1-1 \leq i_k < -1$ or $0 \leq i_k \leq (p-1) \ell_1-1 $;
\end{enumerate}
 
\begin{enumerate} \setcounter{enumi}{2}
\item\hspace{-5pt}' $  i_k \leq   (p-1) \ell_1-1$.
\end{enumerate}
This agrees with the assertions made in the two theorems (where $\epsilon = 1$ and $e=0$ in this case).\vspace{5pt}

If $w = 2p^k$ for some $k$, then the respective left terms in the above cofibre sequences vanish.
By the ``odd cases'', the middle terms have a basis consisting of all sequences $(i_1,\ldots, i_k)$ satisfying conditions $(1), (2)$ or $(1)', (2)'$ ,  together with the respective conditions
\begin{enumerate} \setcounter{enumi}{2}
\item  $(p-1) (2\ell_1-1) \leq i_k<-1$ or $   0 \le i_k \leq (p-1)   (2\ell_1-1)$;
\end{enumerate}
 
\begin{enumerate} \setcounter{enumi}{2}
\item\hspace{-5pt}'   $i_k \leq (p-1)  (2\ell_1-1)   $.
\end{enumerate}
In light of the congruence conditions $(1)$ or $(1)'$, these conditions are  in turn equivalent to 
\begin{enumerate} \setcounter{enumi}{2}
\item $(p-1) (2\ell_1)-1\leq i_k <-1$ or $0 \leq i_k \leq (p-1)   ( 2\ell_1)  - 1$;
\end{enumerate}
 
\begin{enumerate} \setcounter{enumi}{2}
\item\hspace{-5pt}' $i_k \leq  (p-1) (2\ell_1)-1$.
\end{enumerate}
Again, this agrees with the assertions made in the two theorems (with $\epsilon = 1$ and $e=1$ in this case).

If $w\neq p^k, 2p^k$ for all $k$, then the outer summands in the above cofibre sequences vanish, which implies that the middle term must also vanish. This verifies the two claims in these weights.
We have finally verified the two theorems whenever there is just a single generator.

The statement for $m>1$ follows immediately from the single generator case by  the direct sum decomposition established in \Cref{frommanytoone}.
 \end{proof}

\newpage

\section{Appendix: Hypercoverings and Kan extensions}
In higher algebra, simplicial resolution arguments often  proceed by writing a given object $X$ (which we want to control) as a geometric realisation  of a simplicial diagram $X_\bullet$ consisting of  simpler objects   (which we can control). 
The theory of hypercoverings gives a general tool for building such simplicial resolutions. It goes back to Verdier's Expos\'{e} V  in   \cite{artin1972theorie}, and was studied in a higher categorical context by Dugger-Hollander-Isaksen \cite{dugger2004hypercovers}, To{\"e}n-Vezzosi \cite[Section 3.2]{toen2005homotopical}, Lurie \cite[Section 6.5.3]{lurie2009higher} \cite[Section Prop. 7.2.1]{lurie2014higher}, and many others.

In this appendix, we will develop a variant  of these ideas  which will allow us to  construct completed-free resolutions of  complete Noetherian algebras in \Cref{freeresolutions} above.  Moreover, we will   use hypercoverings to explicitly describe 
certain left Kan extensions of algebras; this technical result is needed in \Cref{forgetfulrelative} in the main body of this article.

\subsection{Construction of hypercoverings}
We begin by recalling the following classical definition:
\begin{definition}[Matching and latching objects]   
Let $X_\bullet$ be a simplicial object in an \mbox{$\infty$-category $\mathcal{C}$.}
\begin{enumerate} \INN{130@$M_n(X_\bullet)$}\INN{120@$L_n(X_\bullet)$}
\item  
The $n$th \emph{matching object} $M_n(X_\bullet)$ is given by the limit 
$M_n(X_\bullet) =  \varprojlim_{[m] \rightarrow [n], m < n} X_m$, if this limit exists.  
The limit is taken over the opposite
of the full  subcategory of $\Delta^{}_{/[n]}$ spanned by arrows
$[m] \rightarrow [n]$ with $m < n$. 
Equivalently, by a classical cofinality argument, we can take the limit over the opposite of the
poset
of proper \vspace{3pt} subsets of $[n]$. 
\item The $n$th \emph{latching object} $L_n(X_\bullet)$ is given by the
colimit $\varinjlim_{[n] \rightarrow [m], m < n} X_m$, \vspace{2pt}if it exists. 
By cofinality, we  can also take 
the colimit over the poset of surjections $[n] \twoheadrightarrow [m]$ with  $m <
n$. 
\end{enumerate}
For each $n$, we have natural maps $L_n(X_\bullet) \rightarrow X_n \rightarrow M_n(X_\bullet).$

\end{definition} 
We recall a criterion for contractibility, together with its relative variant:
\begin{example} 
\label{hypspace0}
Let $X_\bullet$ be a simplicial space. Suppose the map 
$X_n \rightarrow M_n(X_\bullet)$ induces a surjection on $\pi_0$ for all $n\geq 0$. Then
$|X_\bullet|$ is contractible. 
This is proven in  \cite[Lemma 6.5.3.11]{lurie2009higher}. 
\end{example}  
\begin{example} 
\label{hypspace}
Let $X_\bullet$ be a simplicial space augmented over a space $Z$, i.e.\ a  simplicial \mbox{object of $\mathcal{S}_{/Z}$.}
Consider the $n$th mapping object $M_n(X_\bullet)$ in $\mathcal{S}_{/Z}$
(by computing the relevant limit internal to $\mathcal{S}_{/Z}$). 
If the map $X_n \rightarrow M_n(X_\bullet)$ induces a surjection on $\pi_0$ for all $n\geq 0$, then 
 $|X_\bullet| \simeq Z$. This reduces to the previous example by taking
homotopy fibre products over points of $Z$. 
\end{example}

The theory of hypercoverings provides a generalisation of the last example:  we will look for (possibly augmented) simplicial objects
such that the map $X_n \rightarrow M_n(X_\bullet)$ has some type of
surjectivity. We will study hypercoverings in the following general context:

\begin{definition} \label{wop} 
Let $\mathcal{C}$ be an $\infty$-category which admits (finite
nonempty) coproducts and a
terminal object $\ast$,  $S$ a class of morphisms in
$\mathcal{C}$, and $\mathcal{F} \subset \mathcal{C}$ a class of objects. 
We say that $(\mathcal{F}, S)$ forms a \emph{weakly orthogonal pair} if:
\begin{enumerate}
\item $S$ is  
closed under composition and contains all equivalences. Moreover, we have the
following property:
given composable arrows $g, f$ with $g \circ f \in S$, we have $g \in
S$ too. 
\item
 Pullbacks of morphisms
 in $S$
 along morphisms in $S$ 
exist and belong to $S$.  
\item $\mathcal{F}$ is closed under coproducts. 
\item For each $F \in \mathcal{F}$, the map $F \to \ast$ belongs to $S$. 
\item Given  $F \in \mathcal{F}$ and a morphism $f: Y \rightarrow Y'$ in $S$, the map 
\mbox{$\Map_{\mathcal{C}}(F, Y) \rightarrow \Map_{\mathcal{C}}(F, Y')$} is surjective on
$\pi_0$. Hence objects in $F$ have the left lifting property with respect to
$S$. 
\item Given an object $Y \in \mathcal{C}$, there exists a map $f: F \rightarrow Y$ in $S$ with $F
\in \mathcal{F}$. \label{factcondition}
\end{enumerate}

\end{definition}

\begin{remark} 
Let $\mathcal{C}, \mathcal{F}, S$ be as in \Cref{wop} and fix an object $Z \in \mathcal{C}$. 
Consider the full subcategory $(\mathcal{C}_Z)' \subset \mathcal{C}_{/Z}$ 
consisting of those maps $Y \to Z$ which belong to $S$. 
Our assumptions imply that $(\mathcal{C}_Z)'$ contains a terminal object as well
as finite nonempty coproducts. 
Then 
$(\mathcal{C}_Z)'$ 
admits a weakly orthogonal pair $(\mathcal{F}_Z, S_Z)$
as follows. 
The class $\mathcal{F}_Z$ consists of those objects in $(\mathcal{C}_{/Z})'$  
whose underlying object of $\mathcal{C}$ belongs to $\mathcal{F}$. 
The class
$S_{/Z}$ consists of those morphisms whose underlying morphism in $\mathcal{C}
$ belongs to $S$.
\end{remark} 

We will now define the notion of a hypercovering and prove an existence statement.
This is essentially a classical result  from  SGA4; the $\infty$-categorical
treatment is a slight modification of \cite[Proposition 7.2.1.5]{lurie2014higher}, except that we do not assume the
existence of finite limits. 

\begin{lemma}[General hypercovering lemma] 
\label{hypercoveringlemma}
Assume that $( \mathcal{F}, S)$ is a weakly orthogonal pair in an
$\infty$-category $\mathcal{C}$ which admits finite nonempty coproducts and a terminal object $\ast$.
Then there exists a simplicial object $X_\bullet $ such that for all $n\geq 0$, we have:
\begin{enumerate}
\item  the object $X_n$ belongs to $\mathcal{F}$;
\item  the 
matching object $M_n(X_\bullet)$  exists in
$\mathcal{C}_{}$;
\item the latching object $L_n(X_\bullet)$ exists in $\mathcal{C}_{}$; 
\item the map $X_n \rightarrow M_n(X_\bullet)$ belongs to $S$
(When $n = 0$, this is the map $X_0 \rightarrow \ast$);
\item the map $L_n(X_\bullet) \rightarrow X_n$ 
expresses $X_n$ as a coproduct of the source and an object in $\mathcal{F}$. 
\end{enumerate}
\label{existenceofhyp}
\end{lemma}

\begin{definition}[$(\sF, S)$-hypercoverings] 
Fix a weakly orthogonal pair $(\mathcal{F}, S)$ on an $\infty$-category $\mathcal{C}$.
\begin{enumerate}
\item A simplicial object $X_\bullet$ is said to be an \emph{$(\sF, S)$-hypercovering} of the terminal object
$\ast$ if it satisfies  conditions (1)--(5) of \Cref{hypercoveringlemma}.
\item An augmented simplicial object $X_\bullet \rightarrow Z$
is called an \emph{$(\sF, S)$-hypercovering} of $Z$ if each $X_i \to Z$
belongs to $S$, and, when  considered as a
simplicial object of $(\mathcal{C}_{/Z})'$, it is an $(\mathcal{F}_Z, S_Z)$ hypercovering of the
terminal object. 
\end{enumerate}
\end{definition} 

To prove  \Cref{hypercoveringlemma}, we will need the following technical result:\vspace{-1pt}
\begin{proposition} 
\label{colimlemma}
Let $\mathcal{P}$ be a finite poset. Let $\mathcal{D}$ be an
$\infty$-category containing an initial object and suppose that $T$ is a class of morphisms in $\mathcal{D}$ which is closed under
composition and contains all equivalences.  Let $G: \mathcal{P} \rightarrow \mathcal{D}$ be any functor.
Suppose that: 
\begin{enumerate}
\item  
pushouts of morphisms in $T$ along morphisms in $T$ exist, and remain in $T$; 
\item
for any $x \in \mathcal{P}$, the functor $G|_{\mathcal{P}_{< x}} 
: \mathcal{P}_{< x}
\rightarrow \mathcal{D}$ admits a colimit in $\mathcal{D}$; 
\item for any $x \in \mathcal{P}$, the morphism 
$\varinjlim_{y\in \mathcal{P}_{< x}} G(y) \rightarrow G(x)$ belongs to $T$. 
\end{enumerate}
Then $G$ admits a colimit in $\mathcal{D}$, and  the canonical map from the initial object to $\varinjlim_{\mathcal{P}} G$\vspace{-1pt}
 \mbox{belongs to $T$.}
\end{proposition} 
\begin{proof} 
Let $\mathcal{Q} \subset \mathcal{P}$ be a  downward-closed subset; this means that 
 if $y \leq x$, then $x \in \mathcal{Q}$ implies  $y \in
\mathcal{Q}$. 
We claim that if $\mathcal{Q}' \subset \mathcal{Q}$ is 
downward closed, then the colimits of $G$ over
$\mathcal{Q}, \mathcal{Q}'$ exist, and the morphism
$\varinjlim_{\mathcal{Q}'} G\rightarrow \varinjlim_{\mathcal{Q}} G$ belongs to $T$. 
Taking $\mathcal{Q} = \mathcal{P}$ (and $\mathcal{Q}' = \emptyset$)   then implies the result.

Suppose $\mathcal{Q}$ is maximal among downward closed subsets for which the
above claim  holds true.\\ If $\mathcal{Q} \neq \mathcal{P}$, let $z \in \mathcal{P}$ be an element minimal subject to the condition that $z
\notin \mathcal{Q}$; this means that any $z'$ with  $z' < z$ belongs to $\mathcal{Q}$. 
In particular, $\widetilde{\mathcal{Q}} := \mathcal{Q} \cup \left\{z\right\}$ is a downward
closed subset as well.
The poset $\widetilde{\mathcal{Q}}$ is the union of $\mathcal{Q}$ and $
\mathcal{P}_{\leq z}$, with common intersection being given by $\mathcal{P}_{< z}$. 
Moreover, we have
a pushout, and in fact a homotopy pushout in the Joyal model structure, 
of simplicial sets
$$ \N(\mathcal{\widetilde{Q}}) = \N( \mathcal{Q}) \sqcup_{\N( \mathcal{P}_{< z})}
\N(\mathcal{P}_{\leq z}).  $$
By assumption, the colimit $\varinjlim_{y\in \mathcal{P}_{< z}} G(y)$ exists, and  
$\varinjlim_{y\in \mathcal{P}_{< z}} G(y) \rightarrow G(z)  = \varinjlim_{y\in \mathcal{P}_{\leq  z}} G(y)
$ belongs to $T$. 
By the defining hypothesis on $\mathcal{Q}$, we know that
$\varinjlim_{y\in  \mathcal{P}_{<z}} G(y) \rightarrow \varinjlim_{y\in \mathcal{Q}} G(y)$ belongs
to $T$ as well. 
Using \cite[Corollary 4.2.3.10]{lurie2009higher}, we deduce that $G|_{\widetilde{\mathcal{Q}}}$
admits a colimit as desired, which is given as the pushout of the restricted colimits. 
It follows from (1) that both
$\varinjlim_{y\in \mathcal{Q}} G(y) \rightarrow \varinjlim_{y\in\mathcal{\widetilde{Q}}} G(y)$  and $\varinjlim_{y\in\mathcal{P}_{\leq z}}  G(y) \rightarrow
\varinjlim_{y\in\widetilde{\mathcal{Q}}} G(y)$ belong to $T$. 
Since any proper subposet $\mathcal{Q}''$ of $\mathcal{\widetilde{Q}}$ is contained in either
$\mathcal{P}_{\leq z}$ or $\mathcal{Q}$, we conclude\vspace{1pt} that
$\varinjlim_{y\in \mathcal{Q}''} G(y) \rightarrow \varinjlim_{y \in \widetilde{\mathcal{Q}}} G(y)$
belongs to $T$ by  using the defining hypothesis of $\mathcal{Q}$ and the fact that $T$ is closed under
composition. 
\end{proof} 
We are now in a position to construct hypercoverings:\vspace{-2pt}
\begin{proof}[Proof of \Cref{existenceofhyp}] 
Using condition (5) in \Cref{wop}, we can chose an object  $X_0 \in \mathcal{F}$ such that  $X_0 \rightarrow \ast$
belongs to $S$. We will now construct a simplicial object  by a recursive construction.

Suppose that we have defined $X$ on $\Delta^{op}_{\leq r}$ so that it \vspace{1pt}satisfies 
conditions (1)--(5) of \Cref{existenceofhyp}, 
for all $n \leq r$. 
In order to extend $X$ to $\Delta^{op}_{\leq r+1}$, we first observe that the \vspace{1pt}colimit
$\varinjlim_{[r+1] \rightarrow [m], m < r+1} X_m$ (i.e.\ the {latching} object, which
is already defined for the $r$-truncated \vspace{1pt} simplicial object) and 
the limit 
$\varprojlim_{ [m] \rightarrow [r+1], m < r+1} X_m$
(i.e.\ the \vspace{1pt}corresponding {matching} object)  both exist. 
Furthermore, we claim that the latching object belongs to $\mathcal{F}$.
To verify these claims,  we   apply Proposition~\ref{colimlemma} as follows:
\begin{enumerate}
\item  
The matching object $M_{r+1}(X_\bullet)$ (if it exists) can be computed as the  limit
$\varprojlim_{U \subsetneq [r+1]} X_U$, taken over  \vspace{1pt}
the opposite of 
the poset of proper subsets $U \subsetneq [r+1]$. 
Given a proper  \vspace{1pt} subset $U  \subsetneq [r+1]$, say $U = [m]$, the limit
$\varprojlim_{U' \subsetneq [m]} X_{U'}$  \vspace{1pt} exists and $X_{m} \rightarrow \varprojlim_{U' \subsetneq [m]} X_{U'}$  \vspace{1pt} belongs to $S$
by the inductive hypothesis.  
Therefore, the matching object exists \vspace{2pt}  by \Cref{colimlemma}.
\item
We apply a dual argument for the latching object. Indeed,  define $T$  to be the class of 
 morphisms which are equivalent to $Y \rightarrow Y \sqcup X$ with  $
X \in \mathcal{F}$. By \Cref{colimlemma}, it then follows that the latching object exists and belongs to
$\mathcal{F}$. 
\end{enumerate}

To construct $X$ on $\Delta^{op}_{\leq r+1}$, 
by \cite[Proposition A.2.9.15]{lurie2009higher} and the surrounding discussion, it suffices to provide an object $X_{r+1}$ and a factorisation 
$$\varinjlim_{[r+1] \rightarrow [m], m < r+1} X_m \ \ \ \longrightarrow \ \ \ 
X_{r+1} \ \ \ \longrightarrow \ \ \ 
\varprojlim_{[m] \rightarrow [n], m < n} X_m
.$$
Define $X_{r+1}$ as the coproduct of the left-hand-side
with an object $F \in \mathcal{F}$ with a map $F \rightarrow \varprojlim_{[m] \rightarrow [n], m < n} X_m
$ in $S$. 
Then, $X_{r+1} \to \varprojlim_{[m] \rightarrow [n], m < n} X_m$ lies in $S$
by the two-out-of-three property.
This extends $X$ to $\Delta^{op}_{\leq r+1}$,  and\vspace{2pt} we observe  that the conditions (1)--(5) hold for all $n \leq r+1$. 
\end{proof}

\subsection{Kan extensions}
Hypercoverings will allow us to compute certain 
left Kan extensions via geometric realisations.  For this, we will  need a general way of producing weakly orthogonal pairs:
\begin{cons} \label{newwops}
Let $\mathcal{C}$ be a presentable $\infty$-category, and assume that \INN{060@$\mathcal{F}^0$}
$\mathcal{F}^0$ is a set of objects which is  closed under finite coproducts. 
\begin{enumerate}[a)]
\item  
Let $\sF$ denote the class of objects of $\mathcal{C}$ which are (possibly infinite) coproducts of objects in
$\mathcal{F}^0$. 
\item
Let $S$ denote
the class of morphisms $f: X \rightarrow Y$ in $\mathcal{C}$ such that for all $F \in \mathcal{F}^0$, 
the map of sets
$\pi_0\Map_{\mathcal{C}}(F, X) \rightarrow \pi_0\Map_{\mathcal{C}}(F, Y)$ is surjective. 
\end{enumerate}
It is then straightforward to check that $(\sF, S)$ forms a weakly orthogonal pair in $\mathcal{C}$. 
Part \ref{factcondition} of \Cref{wop} follows from a compactness argument. 
\end{cons}

\begin{proposition} \label{realLKE}
Suppose $\mathcal{C}$ and $( \sF, S)$ are specified as in \Cref{newwops}. 
Let $\mathcal{D}$ be a presentable $\infty$-category and assume that  $G: \mathcal{C} \to
\mathcal{D}$ is a functor which is left Kan extended from
$\mathcal{F}^0$. Given any $(\sF, S)$-hypercovering $X_\bullet$ of an object $Y \in
\mathcal{C}$, we have $|G(X_\bullet)| \simeq G(Y).$
\end{proposition} 
\begin{proof} 
Let $\mathcal{F}^1$  be a small subcategory 
with $\mathcal{F}^0 \subset \mathcal{F}^1  \subset \mathcal{F}$ such that the image of
$X_\bullet$ is contained in $\mathcal{F}^1$. By assumption, the functor $G$ is left Kan
extended from $\mathcal{F}^1$ too; in fact, the sole purpose of introducing $\mathcal{F}^1$ is  to
avoid discussing Kan extensions from non-small subcategories.

Recall (cf.\ \cite[Section 4.3.2]{lurie2009higher}) that the left Kan extension can be computed by the formula
$$ G(Y) \simeq \varinjlim_{Z \in \mathcal{F}^1_{/Y}} G(Z).  $$
We have a natural functor 
$\Delta^{op} \rightarrow \mathcal{F}^1_{/Y}$ given by the simplicial object $X_\bullet$,
and it therefore suffices to check  that this functor is left cofinal. 

Using the $\infty$-categorical version of Quillen's Theorem A 
\cite[Theorem 4.1.3.1]{lurie2009higher}, we are reduced to proving that 
 for any $Y_1 \in \mathcal{F}^1_{/Y}$, the 
homotopy pullback 
$\Delta^{op} \times_{\mathcal{F}^1_{/Y}} \mathcal{F}^1_{Y_1//Y}$ has a weakly
contractible nerve. By the Grothendieck construction, it in fact suffices to show that the geometric
realisation of the simplicial space
$\hom_{\mathcal{F}^1_{/Y}}(Y_1, X_\bullet)$ is weakly contractible. 
This is true because  $X_\bullet$ being an $(\sF,
S)$-hypercovering of $Y$  implies that  
$\hom_{\mathcal{F}^1_{/Y}}(Y_1, X_\bullet)$
satisfies the conditions of \Cref{hypspace0}. 
\end{proof} 
We now illustrate \Cref{realLKE}  in two examples of interest:
\begin{example}[Left Kan extension from $\perf_{k, \leq 0}$] 
\label{hypmodules}
Suppose that $k$ is a field. We can then take $\mathcal{C}$ to be the \mbox{$\infty$-category} $\md_k$ and  $\sF_0$ to be the full  subcategory $\perf_{k, \leq 0}$.
It is then  not hard to check that $S$ becomes the class of
morphisms in $\perf_{k, \leq 0}$ which induce surjections on $\pi_i$ for all $i \leq 0$. 

We now claim that any 
$(\sF, S)$-hypercovering $X_\bullet$ of $Y \in \md_k$ is a colimit diagram. 
Indeed, applying the functors $\hom_{\md_k}( k[-n], -)$, and combining condition (4) of \Cref{hypercoveringlemma} with 
\Cref{hypspace}, we see  that $| \Omega^{\infty-n} X_\bullet| \simeq \Omega^{\infty - n} Y $ is an equivalence 
for all $n \geq 0$.
As we can write any spectrum $Z$ as a canonical colimit $Z \simeq \varinjlim_{n}
\Sigma^{\infty -n} \Omega^{\infty - n} Z$, we deduce\vspace{2pt} our claim. 

We can therefore explicitly describe  the 
procedure of left Kan extension along the inclusion $\perf_{k, \leq 0} \to
\md_k$, even for functors which do not preserve finite coconnective geometric
realisations. 
For this, let $G_0: \perf_{k, \leq 0} \rightarrow \mathcal{D}$ be any functor. To compute its  left Kan
extension $G: \md_k \rightarrow \mathcal{D} $, we first extend $G_0$ in a filtered-colimit-preserving way to 
a functor $G_1: \md_{k, \leq 0} \rightarrow \mathcal{D}$.

Given an arbitrary $k$-module $Y \in \md_k$, we can pick an $(\sF, S)$-hypercovering $X_\bullet \rightarrow Y$ by applying
\Cref{hypercoveringlemma}. By construction this means that each $X_i $ belongs to $\md_{k, \leq 0}$. By \Cref{realLKE}, we obtain an equivalence  $G(Y) \simeq |G_1(X_\bullet)|$. Note in particular that while $G$ need
not preserve all geometric realisations,   it can still be computed in this fashion. 
\end{example}  
\begin{example}[Left Kan extension for $\infty$-categories of algebras] 
\label{hypTalgebra}
Let $T: \md_k \rightarrow \md_k$ be a monad which preserves sifted colimits. 
Write $\sF_0 \subset \alg_T$ for the full subcategory 
spanned by all  free $T$-algebras of the form $T(V)$ with $V \in \perf_{k, \leq 0}$. 
It is again not difficult to check that the associated class  $S$ consists of those maps of $T$-algebras
which induce surjections on $\pi_i$ for all $i \leq 0$.

Let now $\mathcal{D}$ be a presentable $\infty$-category, and suppose that we are given a
functor $G_0: \mathcal{F}_0 \rightarrow \mathcal{C}.$
We can then ask: what is the left Kan extension 
$G: \alg_T \rightarrow \mathcal{C}$ of $G_0$ to all of $\alg_T$?

We observe that the  associated $\infty$-category $\sF$ is spanned by all free $T$-algebras of the form $T(W)$
with  $W \in \md_{k, \leq 0}$. 
Since $G$ is left Kan extended from its values on compact objects, it follows
that $G$ commutes with filtered colimits, which determines its values on all objects in  $\sF$. 

Given an arbitrary $T$-algebra $A$, we  use \Cref{hypercoveringlemma} to
find an $(\sF, S)$-hypercovering \mbox{$X_\bullet$ of
$A$}. 
It follows as in
\Cref{hypmodules} that $|X_\bullet| \simeq A$
in $T$-algebras,  and \Cref{realLKE} gives an equivalence
$ |G(X_\bullet)|\simeq G(A).$
For each $i$,  $G(X_i)$ is determined  as $X_i$ is free on a coconnective
$k$-module spectrum. 

As a simple consequence, we deduce that any sifted-colimit-preserving functor $\alg_T \rightarrow \mathcal{C}$
is left Kan extended from $\mathcal{F}_0$. The forgetful functor
$\alg_T \rightarrow \md_k$ is therefore left Kan extended \mbox{from $\mathcal{F}_0$.}
\end{example}

\printindex 
 \newpage
\bibliographystyle{amsalpha}
\bibliography{There}

\end{document}